%% file: main.tex
\documentclass[11pt,twoside,reqno]{amsart}
\usepackage{import}
\include{preamble.tex}

\usepackage{tikz}
\usetikzlibrary{positioning}

\usetikzlibrary{arrows.meta, decorations.markings}

\title{The Gamma-disordered Aztec diamond}

\author{Maurice Duits}
\author{Roger Van Peski}

\DeclareMathOperator{\Var}{{\mathrm Var}}
\date{\today}

\begin{document}

\begin{abstract}
    We introduce a multi-parameter family of random edge weights on the Aztec diamond graph, given by certain Gamma variables, and prove several results about the corresponding random dimer measures.
    
    Firstly, we show there is no phase transition at the level of the free energy. This provides rigorous backing for the physics predictions of Zeng-Leath-Hwa \cite{zeng1999thermodynamics} and later works that dimer models with random weights are in the glassy `super-rough' phase at all temperatures with no phase transition.

    Secondly, we show that the random dimer covers themselves enjoy exact distributional equalities of certain marginals with path locations in new `hybrid' integrable polymers. These reduce to the stationary log-Gamma, strict-weak, and Beta polymer in random environment in certain cases, allowing transfer of known results from integrable polymers to dimers with random weights. As an example application, we prove that the turning points at the boundaries of the Aztec diamond exhibit fluctuations of order $n^{2/3}$, in contrast to the $n^{1/2}$ fluctuations for deterministic weights.

    Underlying all these is a key integrability property of the weights: they are the unique family for which independence is preserved under the shuffling algorithm. 
\end{abstract}

\maketitle

\tableofcontents

\section{Introduction}

The purpose of this work is to introduce a new dimer model with random weights, about which we make two claims:
\begin{enumerate}
    \item It exhibits interesting probabilistic behaviors not present for dimer models with deterministic weights.
     \item It is exactly solvable, and contains several other key exactly solvable models as special cases.
\end{enumerate}

The Aztec diamond of size $n$ is the bipartite graph of the form of \Cref{fig:our_weights_intro}, originally introduced in 1992 by Elkies-Kuperberg-Larsen-Propp \cite{elkies1992alternatingI}, which we further equip with positive edge weights $a_{i,j},b_{i,j}$ as shown. The dimer model on such a graph is the probability measure on perfect matchings (also called \emph{dimer covers}) which gives each perfect matching probability proportional to the product of edge weights over all edges in the matching, meaning that edges with higher weight are more likely to be contained in the matching. With fixed deterministic weights, the dimer model on arbitrary planar bipartite graphs goes back to Kasteleyn \cite{kasteleyn1961statistics,kasteleyn1963dimer} and Temperley-Fisher \cite{temperley1961dimer} in the 1960s. The mathematical literature on the model has ballooned in the last three decades, and it remains an active topic, see for instance the textbook \cite{gorin_2021} and the references therein.

\begin{figure}[H]
    \centering

    \begin{tikzpicture}[scale=1]
  
\foreach \x in {0,...,2} {
  \foreach \y in {1,...,3} {
   \draw (2*\x,2*\y)-- (2*\x+1,2*\y+1);
    \draw (2*\x,2*\y)-- (2*\x+1,2*\y-1);
\draw (2*\x+2,2*\y)-- (2*\x+1,2*\y+1);
    \draw (2*\x+2,2*\y)-- (2*\x+1,2*\y-1);
  }
}

\foreach \x in {1,...,3} { \foreach \y in {1,...,3} {
 \draw (2*\x-2+0.2,-2*\y+8+.7) node {\small $a_{\text{\x,\y}}$};
\draw (2*\x-2+0.2,-2*\y+8-.7) node {\small $b_{\text{\x,\y}}$};
}}
\foreach \x in {0,...,2} { \foreach \y in {1,...,3} {
 \draw (2*\x+1.35,2*\y+.4) node {\small$1$};
\draw (2*\x+1.35,2*\y-.4) node {\small $1$};
}}

  \draw[-,red,line width=4pt] (0,6)--(1,7);
\draw[-,yellow,line width=4pt] (0,4)--(1,3);
\draw[-,yellow,line width=4pt] (0,2)--(1,1);
\draw[-,red,line width=4pt] (2,6)--(3,7);
\draw[-,red,line width=4pt] (2,2)--(3,3);
\draw[-,green,line width=4pt] (1,5)--(2,4);

\draw[-,blue,line width=4pt] (3,1)--(4,2);
\draw[-,blue,line width=4pt] (3,5)--(4,6);

\draw[-,red,line width=4pt] (4,4)--(5,5);

\draw[-,blue,line width=4pt] (5,1)--(6,2);
\draw[-,blue,line width=4pt] (5,3)--(6,4);
\draw[-,green,line width=4pt] (5,7)--(6,6);

\foreach \x in {0,...,3} {
  \foreach \y in {1,...,3} {
   \filldraw (2*\x,2*\y) circle(3pt);
    \draw[fill=white] (2*\y-1,2*\x+1) circle(3pt);
  }
}
\end{tikzpicture}
    \caption{A perfect matching of the Aztec diamond graph $G_n^{\Az}$ with $n=3$, where each included edge is colored in red, blue, yellow or green according to its orientation---edges in the same direction are differentiated by whether the leftmost vertex is white or black. The probability of the matching is proportional to the product of the weights of the edges, in this case $a_{1,1}a_{2,1}a_{2,3} a_{3,2}b_{1,2}b_{1,3}$.}
    \label{fig:our_weights_intro}
\end{figure}
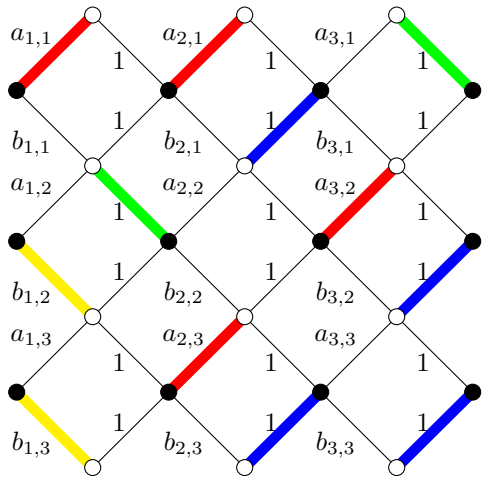

\begin{figure}[t]
  \begin{subfigure}[b]{\textwidth}
    \centering
    \begin{minipage}{0.48\textwidth}
      \centering
\includegraphics[scale=.6]{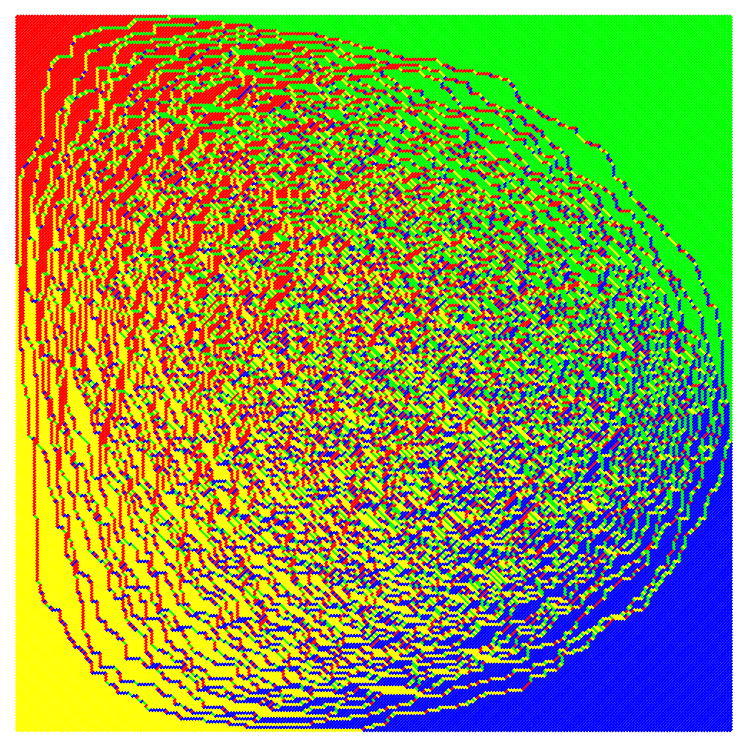}
      \caption{Dimer cover, random weights.}\label{subfig:dimer_cover_rand}
    \end{minipage}\hfill
    \begin{minipage}{0.48\textwidth}
      \centering
\includegraphics[scale=.6]{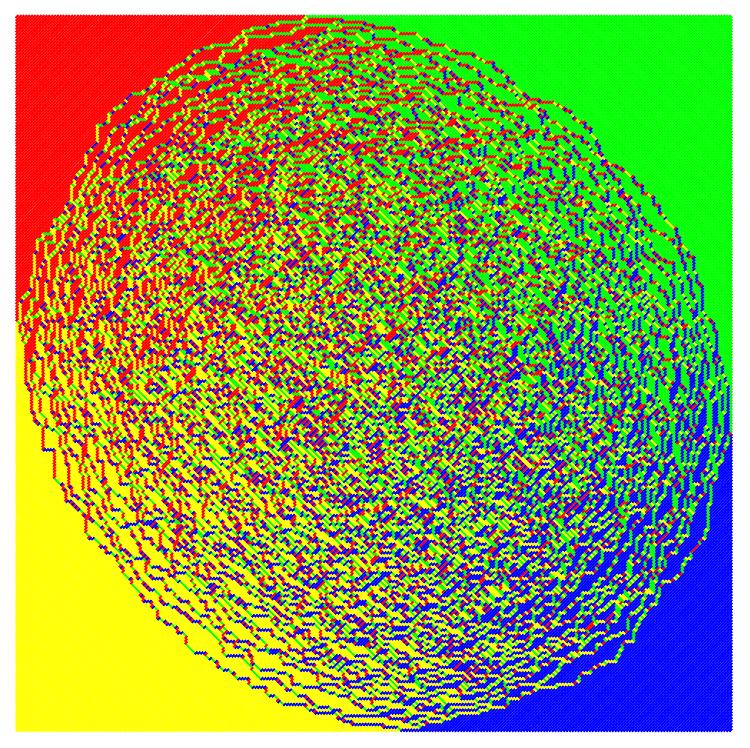}
 \caption{Dimer cover, deterministic weights.}\label{subfig:dimer_cover_det}
    \end{minipage}
  \end{subfigure}

  \medskip

  \begin{subfigure}[b]{\textwidth}
    \centering
    \begin{minipage}{0.48\textwidth}
      \centering
\includegraphics[scale=.6]{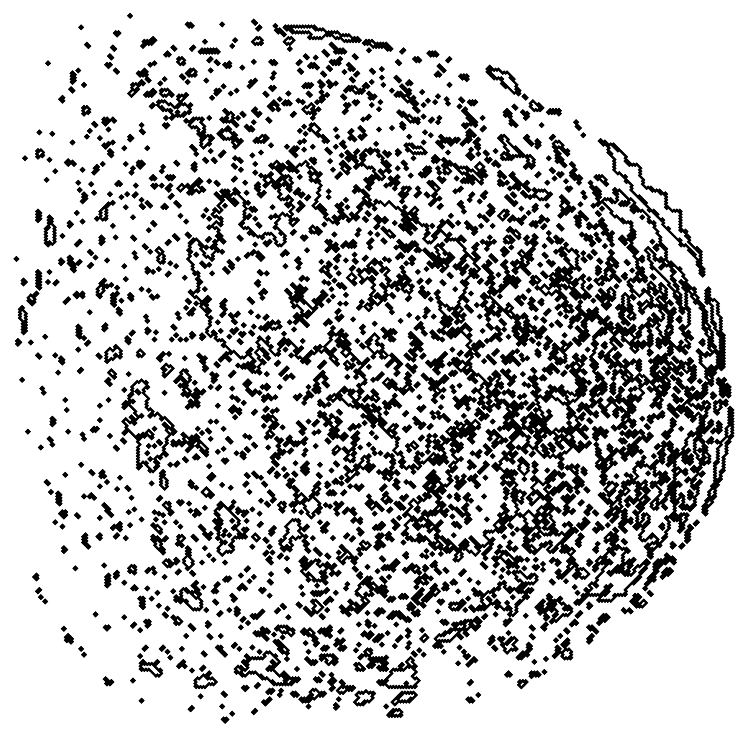}
\caption{Double-dimer configuration, random weights.}\label{subfig:dd_rand}
    \end{minipage}\hfill
    \begin{minipage}{0.48\textwidth}
      \centering
\includegraphics[scale=.6]{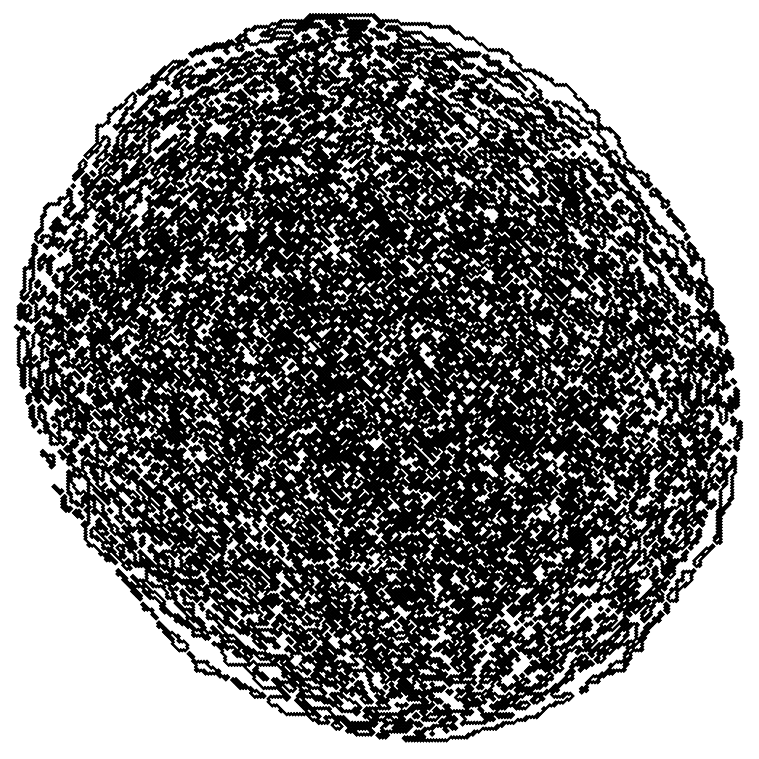}
    \caption{Double-dimer, deterministic weights.}\label{subfig:dd_det}
    \end{minipage}
  \end{subfigure}
    \caption{Perfect matchings of $G_{500}^{\Az}$ taken from the dimer measure with random weights as in \Cref{def:gamma_weights_intro} with $\alpha=.2,\beta=.25$ (left), and deterministic weights $a_{i,j} = .2,b_{i,j}=.25$  (right).  Top figures depict a single matching with edges colored as in \Cref{fig:our_weights_intro}. Bottom figures depict a pair of matchings taken from the dimer measure with the same weights, overlaid, with the edges in both matchings removed to leave a collection of loops. It is important that on the bottom left, both matchings are taken from the \emph{same} sample of random weights rather than two independent samples.}\label{fig:simulations}
\end{figure}

One may also first choose the weights independently at random from some distribution(s), and then sample a perfect matching according to the resulting dimer measure. Such random weights are referred to as \emph{quenched disorder}, meaning that they determine a random probability measure on states of the model, which in our case are perfect matchings. Our model is of this type:

\begin{defi}
    \label{def:gamma_weights_intro}
    Fix $n \geq 1$ and let $\alpha,\beta > 0$. The \emph{biased Gamma-disordered Aztec diamond} of size $n$ with parameters $\alpha,\beta$ is the random dimer measure with independent random weights $\{a_{i,j},b_{i,j}: 1 \leq i,j \leq n\}$, placed on the Aztec diamond graph of size $n$ as in \Cref{fig:our_weights_intro}, distributed by
    \begin{align}
        a_{i,j} &\sim \Gamma(\alpha, 1) \\ 
        b_{i,j} &\sim \Gamma(\beta, 1).
    \end{align}
    Here $\Gamma(\chi,s)$ is the Gamma distribution with shape parameter $\chi > 0$ and scale parameter $s > 0$, see \Cref{def:gamma_var}.
\end{defi}

We discuss canonical reasons for this choice of weights later in \Cref{subsec:shuffling_and_uniqueness}, but for now let us take the model as given and observe its probabilistic behavior. In simulations, matchings sampled from such a measure appear qualitatively similar to those with fixed deterministic weights: the four corners have a `frozen' region where all edges are of a single type, while toward the center all four types of edges coexist, see \Cref{fig:simulations} (top). Differences become noticeable if one instead samples two matchings from the same quenched dimer measure and overlays them, removing all edges which are the same in both to leave a collection of loops, as in \Cref{fig:simulations} (bottom). Because edges with high weight are more likely to be chosen by both matchings while edges of low weights are more likely to be avoided, with random weights the two samples `overlap': many edges are the same (and hence deleted), and the remaining edges appear to form fewer and smaller loops than with uniform weights. This difference is particularly pronounced toward the left side of \Cref{subfig:dd_rand}, which is mostly whitespace even inside the disordered region.

\subsection{History and physical motivation} Nontrivial behavior of the distribution of overlaps between different samples with the same disorder is one of the hallmarks of \emph{spin glasses}, a loose class of models believed to include dimers with independent random weights. A more classic example is the \emph{Sherrington-Kirkpatrick model} introduced in \cite{sherrington1975solvable}, which is the Ising model on a complete graph with quenched random interaction strengths between spins which depend on a temperature parameter. This model is now known to exhibit a phase transition as the temperature varies. Above the critical temperature, two spin samples with the same disorder are asymptotically independent, while below the critical temperature their overlaps have a nontrivial ultrametric structure, a phenomenon known as replica symmetry breaking---see the surveys by Talagrand \cite{talagrand2010mean} and Panchenko \cite{panchenko2013sherrington} and the references therein.

Spin glass models with spatial structure are much harder to analyze. One such model is the \emph{Edwards-Anderson model} \cite{edwards1975theory,binder1976monte}, which is the Ising model with random interaction strengths on a lattice instead of a complete graph. It was only shown recently by Chatterjee \cite{chatterjee2023spin} to exhibit spin glass behavior. By contrast, rigorous results supporting spin glass behavior for the Sherrington-Kirkpatrick model (which does not have spatial structure) date back to Aizenman-Lebowitz-Ruelle \cite{aizenman1987some}, and the main features of the spin glass picture predicted by Parisi \cite{parisi1979infinite} were proven in the early 2000s---see \cite{talagrand2010mean,panchenko2013sherrington}. 

Another spatial model, relevant to the story here, is the $XY$ model with quenched random background field studied by Cardy-Ostlund \cite{cardy1982random}. Using nonrigorous renormalization group calculations, they found a glassy phase at low temperatures as well, and a phase transition as the temperature is raised. Dotsenko and Feigelman \cite{dotsenko19812d}, around the same time, found no glassy phase or phase transition. Various conflicting predictions were published, see for instance the discussions of prior work in Cule-Shapir \cite{cule1995nonergodic} and Le Doussal-Giamarchi \cite{le1995replica}, both of which predicted a form of replica symmetry breaking.

One way of shedding light on these conflicts was to consider discrete models of random height functions which were believed to lie in the same universality class. Toner-DiVincenzo \cite{toner1990super} predicted analogous phenomena for models of random crystal interfaces, most notably that these interfaces have `super-rough' fluctuation on the scale of $(\log n)^2$ at low temperature, transitioning at high temperature to the $\log n$ fluctuations seen in models without quenched disorder; this was supported by Carpentier-Le Doussal \cite{carpentier1997glass}. 

Subsequent works studied discrete models such as the solid-on-solid model with quenched disorder numerically. Cule-Shapir \cite{cule1995glassy} used Markov chain Monte Carlo methods to explicitly sample from these measures, but were hampered by the slow convergence of these dynamics, another feature of glassy models. Other works, by Zeng-Middleton-Shapir \cite{zeng1996ground}, Rieger-Blasum \cite{rieger1997ground}, and Middleton \cite{middleton1999numerical}, considered zero-temperature models where the analogue of the dimer measure concentrates on a point mass at the (analogue of a) matching with highest weight, which could be found efficiently by combinatorial optimization algorithms. However, it was---and still is---not at all clear that such zero-temperature models exhibit the same behaviors as positive-temperature ones.

A remedy for the sampling issue came from a then-recent mathematical advance, the polynomial-time \emph{shuffling algorithm} for the Aztec diamond. Dimers with iid random weights are believed to lie in the same universality class as the low-temperature Cardy-Ostlund model, so this finally provided an efficient and provably-perfect way to probe these phenomena numerically. The algorithm was developed by Elkies-Kuperberg-Larsen-Propp \cite{elkies1992alternatingII} and Jokusch-Propp-Shor \cite{jockusch1998random} for the Aztec diamond with uniform edge weights, and extended to arbitrary weights by Propp \cite{propp2003generalized}, which allowed sampling with random edge weights. It was in fact used in physics by Zeng-Leath-Hwa \cite{zeng1999thermodynamics,zeng2000universal} before it was formally written down: \cite{propp2003generalized} thanked ``Chen Zeng, who has made use of the generalized shuffling algorithm in his own work on [\emph{sic}] and thereby encouraged me to write up this algorithm for publication''. 

Dimer models with random weights also naturally come with a temperature parameter. In the original setting of \cite{zeng1999thermodynamics,zeng2000universal}, this was done by taking independent weights of the form $e^{X/T}$ with $X$ uniform in $(-\tfrac{1}{2},\tfrac{1}{2})$ and $T$ a temperature parameter. When $T$ is small, the weights have very high variance even when rescaled to have mean $1$. For fixed $n$, with high probability there will be a single matching with much larger weight than the others, and the dimer measure will be close to a point mass on this matching. When $T$ is large, the weights are essentially deterministic, and there are many matchings with roughly equal weight. In our model, taking $\alpha = T \balpha, \beta = T \bbeta$ in \Cref{def:gamma_weights_intro} similarly yields weights which (when normalized to have constant means $\balpha,\bbeta$ independent of $T$) have high variance at low $T$ and converge to constants at high $T$.

Zeng-Leath-Hwa observed in their numerics that the dimer model with quenched random weights appeared to be in a glassy phase at all temperatures, with no phase transition. This contrasts with the models of \cite{cardy1982random} and \cite{toner1990super}, see also the article by Emig-Bogner \cite{emig2003there} devoted entirely to this question of (non)existence of a phase transition. Later works by Bogner-Emig-Taha-Zeng \cite{bogner2004test} and Perret-Ristivojevic-Le Doussal-Schehr-Wiese \cite{perret2012super} carried out more simulations using the shuffling algorithm, supporting the view that the model is glassy at all temperatures. Works such as Le Doussal-Schehr \cite{le2007disordered} and Ristivojevic-Le Doussal-Wiese \cite{ristivojevic2012super} also carried out sufficiently fine renormalization group analysis of the Cardy-Ostlund model at low-temperature to get convincing agreement with shuffling numerics in \cite{perret2012super}, so there is now much more consensus on the model in the physics literature. However, we are not aware of any rigorous work on dimer models with independent\footnote{Several recent mathematical works by Bufetov-Petrov-Zografos \cite{bufetov2025domino}, Zografos \cite{zografos2025quenched}, and Moulard-Toninelli \cite{moulard2025dimers} treat dimers with random weights which are not independent, but rather are equal along rows or columns of the graph, so there are $\mathcal{O}(n)$ distinct random variables rather than $\mathcal{O}(n^2)$. Such models determined by nonlocal random parameters are believed not to lie in the Cardy-Ostlund universality class studied by the above-mentioned physics works. Indeed the results of \cite{bufetov2025domino} and \cite{zografos2025quenched} show that when the distribution of the noise does not depend on $n$, the height fluctuations are $\sqrt{n}$, much larger than the `super-rough' $(\log n)^2$ ones mentioned above. This is essentially because each one of the $\mathcal{O}(n)$ random parameters has too much influence on the edge weights and $1$D Gaussian statistics take over.} random weights which would support or contradict any of these predictions. 

In this paper, we obtain rigorous results for the Gamma-disordered Aztec diamond, to our knowledge the first for any model believed to lie in the Cardy-Ostlund universality class. We show that there is no phase transition in the free energy fluctuations (\Cref{thm:free_energy_fluctuations_intro}), and that the difference between averaged quenched and annealed free energy is positive as in known spin glass models (\Cref{thm:free_energy_difference_intro}) at all temperatures.

We also go beyond the partition function and study statistics of the random matchings sampled from these random dimer measures. Specifically, we show that some of the simplest observable features---the so-called \emph{turning points}---have $n^{2/3}$ fluctuation exponents characteristic of the Kardar-Parisi-Zhang universality class (\Cref{thm:turning_points_intro}), unlike their counterparts for deterministic weights which have Gaussian fluctuations at scale $n^{1/2}$. Again, this occurs at all temperatures. This result is explained (and proved) by making an exact matching to known and new integrable polymer models, which we explain in \Cref{subsec:stationary_single_path_intro} and \Cref{subsec:hybrid_polymer_intro}. 

Finally, we show that any random weights satisfying a natural integrability property with respect to the shuffling algorithm must be Gamma weights belonging to a multi-parameter family which generalizes \Cref{def:gamma_weights_intro}, see \Cref{subsec:shuffling_and_uniqueness}; our main exact results are shown for this general family as well.

\subsection{Free energy and lack of phase transition}

To probe for a phase transition, it is natural to take weights $\alpha = T \balpha, \beta = T \bbeta$ in \Cref{def:gamma_weights_intro} as mentioned, where $T$ is a temperature parameter. Phase transitions of statistical mechanics models are typically accompanied by sharp transitions in the behavior of the free energy with respect to the parameters in the model. In our case, the partition function of a dimer measure is the sum over all possible matchings of the product of edge weights, and we denote it by $Z_n$ for the size $n$ Gamma-disordered Aztec diamond when the other parameters $\balpha,\bbeta,T$ are clear from context. The (negative of) \emph{quenched free energy} is 
\begin{equation}
    F_n = T \log Z_n.
\end{equation}
Note that, since our edge weights are random variables, the partition function and the free energy are also random variables. A related deterministic quantity is the \emph{annealed} free energy
\[
F_n^a = T \log \mathbb{E}[Z_n].
\]
By Jensen's inequality, we always have \( F_n^a \geq \E[F_n] \). It is easy to see that if the weights are deterministic, $F_n^a = \E[F_n]$, so the difference $F_n^a - \E[F_n] \geq 0$ relates directly to the randomness of the weights. 

The normalized difference $F_n^a - \E[F_n]$ is a very good proxy for different physical behavior, as we recall in the example of the Sherrington-Kirkpatrick model. There, Aizenman-Lebowitz-Ruelle \cite{aizenman1987some} proved (after appropriately normalizing the free energy) that $\lim_{n \to \infty} F_n^a - \E[F_n]$ is $0$ in the high-temperature phase. They also proved that $\lim_{n \to \infty} F_n^a - \E[F_n] > 0$ at sufficiently low temperatures, strengthened to positivity at all subcritical temperatures by Toninelli \cite{toninelli2002almeida}; the precise positive value predicted by Parisi \cite{parisi1979infinite} was proven later by Guerra \cite{guerra2003broken} and Talagrand \cite{talagrand2006parisi}. On the probabilistic side, \cite{aizenman1987some} proved that overlaps between two samples with the same quenched disorder are asymptotically trivial above the critical temperature, while Panchenko \cite{panchenko2013parisi} proved nontrivial ultrametric structure of overlaps below it, so the critical temperature also separates two phases of probabilistic behavior in the model.

By contrast, our next result shows that $\lim_{n \to \infty}n^{-2}\left(F_n^a - \E[F_n]\right)> 0$ for every temperature for the Gamma-disordered Aztec diamond. This is exactly the behavior which should be expected from the prediction of \cite{zeng1999thermodynamics} that dimers with quenched random weights are glassy at all temperatures with no phase transition. 

Here and later,
\begin{equation*}
    \Psi_0(z)=\frac{d}{dz} \log \Gamma(z)
\end{equation*}
and
    \begin{equation*}
        \Psi_1(z) = \frac{d^2}{dz^2} \log \Gamma(z)
    \end{equation*}
denote the digamma and trigamma functions respectively.

\begin{thm}
    \label{thm:free_energy_difference_intro}
For the Gamma-disordered Aztec diamond of size $n$ with parameters $\alpha = \balpha T,\beta = \bbeta T$, the normalized annealed and averaged quenched free energy have limits 
\begin{equation}
\begin{split}
    \lim_{n \to \infty } \frac{1}{n^2} F_n^a&=\frac{1}{2}T\log(T(\balpha+\bbeta))\\
     \lim_{n \to \infty }\frac{1}{n^2} \mathbb E [F_n]&=\frac{1}{2}T\Psi_0(T(\balpha+\bbeta)),
    \end{split}
\end{equation}
which in particular implies 
\begin{equation}
    \lim_{n \to \infty}\frac{1}{n^2}\left(F_n^a - \E[F_n]\right)  > \frac{1}{4(\balpha+\bbeta)}.
\end{equation}
Moreover, for finite $n$ we have the nonasymptotic bounds
\begin{equation}
    \frac{1}{4(\balpha+\bbeta)}\left(1+\frac{1}{n}\right) < \frac{1}{n^2}\left(F_n^a - \E[F_n]\right)< \frac{1}{2(\balpha+\bbeta)}\left(1+\frac{1}{n}\right).
\end{equation}
\end{thm}

A related way to probe a model's behavior through the free energy is through the asymptotic fluctuations of the random variable $F_n$. Continuing with our running example of the Sherrington-Kirkpatrick model, there at below-critical temperatures the fluctuations are believed to be higher-order than above the critical temperature\footnote{This is not currently known as far as we are aware: only the high-temperature phase and upper bounds in the low-temperature phase are known, by the aforementioned \cite{aizenman1987some} and Chatterjee \cite{chatterjee2009disorder} respectively.}. Our model, however, has Gaussian fluctuations with variance depending smoothly on the parameters, again consistent with the prediction that there is no phase transition:

\begin{thm}[Special case of {\Cref{prop:CLT_free_energy}}] \label{thm:free_energy_fluctuations_intro}
    The free energy of the Gamma-disordered Aztec diamond with parameters $\alpha = \balpha T,\beta = \bbeta T$ satisfies
\begin{equation}\label{eq:CLT_free_energy}
\mathbb{P}\left(\frac{F_n - \E[F_n]}{\sqrt{2n(n+1)T^2\Psi_1(T(\balpha + \bbeta))}} <  s \right)
= \frac{1}{\sqrt{2\pi}} \int_{-\infty}^s e^{-x^2/2} \, dx + \mathcal{O}(n^{-1})
\end{equation}
as \( n \to \infty \), where the error term is uniform for \( T \in (0,\infty) \), and $\balpha,\bbeta$ from compact subsets of \( (0,\infty) \).
\end{thm}

Note that the uniformity in $T$ is not over compact subsets of $(0,\infty)$, but over the whole interval. This is important because in principle, if we did not have uniform error terms over $(0,\infty)$, it might be possible that there is still a phase transition in the variance and distribution of the fluctuations if $T$ is varied simultaneously with $n$, for instance $T=\tau/n$ for a normalized temperature $\tau$. The uniformity shows that this cannot be the case, and even in such a regime there is no phase transition. The explicit error bounds in \Cref{thm:free_energy_difference_intro} likewise show that there is no phase transition in the difference between quenched and annealed free energy in regimes where $T$ is varied with $n$.

We mention also that from \eqref{eq:CLT_free_energy} and basic estimates on the trigamma function given in \eqref{eq:trigamma_bounds}, it follows that the variance of \( F_n \) behaves asymptotically as \( \sim n^2 \) for small \( T > 0 \), and as \( \sim n^2 T \)  when both \( n^2 \) and \( T \) are large.

For our weights, Theorems~\ref{thm:free_energy_difference_intro} and~\ref{thm:free_energy_fluctuations_intro} are easy consequences of the surprising fact that in this exactly-solvable model, the partition function $Z_n$ is a product of $\binom{n+1}{2}$ independent random variables with distribution $\Gamma((\balpha+\bbeta)T,1)$; we prove this for a more general family of Gamma weights in \Cref{thm:compute_Z_cor}. It is worth noting that from such an expression one may derive essentially any other desired property of the free energy such as large deviations, though we do not do so here.

\subsection{Asymptotics of turning points}

After understanding the free energy, it is natural to try to get some concrete probabilistic results about the behavior of the matchings themselves. A natural observable of a matching is the place on a given boundary where one type of edge gives way to another, called a \emph{turning point}. In these results we do not consider varying temperature, so we just take the parameters $\alpha,\beta$ in \Cref{def:gamma_weights_intro}.

\begin{defi}
    \label{def:all_turning_points}
    Let $n \in \N$ and let $M$ be a perfect matching of $G_n^{\Az}$. The edges come in four types based on direction and whether the leftmost vertex is white or black, colored red, green, blue and yellow respectively in \Cref{fig:our_weights_intro}. We call these Northwest edges, Northeast edges, Southeast edges, and Southwest edges respectively.
    
    Now we define the turning points $T^{North}(M), T^{East}(M), T^{South}(M), T^{West}(M) \in \{0,\ldots,n\}$ as follows:
    \begin{enumerate}
        \item $T^{North}(M)$ is the number of Northwest edges incident to a white vertex in the top row.
        \item $T^{East}(M)$ is the number of Northeast edges incident to a black vertex in the rightmost column.
        \item $T^{South}(M)$ is the number of Southwest edges incident to a white vertex in the bottom row.
        \item $T^{West}(M)$ is the number of Northwest edges incident to a black vertex in the leftmost column.
    \end{enumerate}
\end{defi}

\begin{example}
    For the matching $M$ of \Cref{fig:our_weights_intro}, we have 
    \begin{align}
        \begin{split}
            T^{North}(M) &= 2 \\ 
            T^{East}(M) &= 1 \\ 
            T^{South}(M) &= 1 \\ 
            T^{West}(M) &= 1.
        \end{split}
    \end{align}
 
\end{example}

The early results of Jockusch-Propp-Shor \cite{jockusch1998random} imply that when all edge weights are $1$, the turning points all concentrate around $n/2$ with Gaussian fluctuations of order $\sqrt{n}$. For random weights, we see fluctuations of order $n^{2/3}$ around explicit locations given in terms of trigamma functions. Specifically, \eqref{eq:tails_intro} shows that the turning points lie within a large enough $n^{2/3}$-size interval around the given location with high probability, while \eqref{eq:anticoncentration_intro} shows that they lie outside an $n^{2/3}$-size interval with non-negligible probability. Together, these show that they fluctuate on the scale $n^{2/3}$.

\begin{thm}
    \label{thm:turning_points_intro}
    Fix $\alpha,\beta >  0$, and let $M_1,M_2,\ldots$ be perfect matchings of independent biased Gamma-disordered Aztec diamonds of sizes $1,2,\ldots$ and parameters $\alpha,\beta$. Then there exist constants $b_0,C,n_0,C_0,C_1$ depending only on $\alpha,\beta$, such that:
    \begin{enumerate}
        \item For every $b \geq b_0$ and $n \in \N$,
    \begin{align}\label{eq:tails_intro}
        \begin{split}
            \P\left(\abs*{T^{West}(M_n) - \frac{\Psi_1(\beta )}{\Psi_1(\alpha )+\Psi_1(\beta )} \cdot n} > bn^{2/3}\right) & < \frac{C}{b^3} \\ 
            \P\left(\abs*{T^{South}(M_n) - \frac{\Psi_1(\alpha  + \beta )}{\Psi_1(\beta )}\cdot n} > bn^{2/3}\right) &< \frac{C}{b^3} \\ 
            \P\left(\abs*{T^{North}(M_n) - \frac{\Psi_1(\alpha  + \beta )}{\Psi_1(\alpha )}\cdot n} > bn^{2/3}\right) &< \frac{C}{b^3}.
        \end{split}
    \end{align}
    \item For every $n \geq n_0$, 
        \begin{align}\label{eq:anticoncentration_intro}
        \begin{split}
            \P\left(\abs*{T^{West}(M_n) - \frac{\Psi_1(\beta )}{\Psi_1(\alpha )+\Psi_1(\beta )}\cdot n} > C_1 n^{2/3}\right) & > C_0 \\ 
            \P\left(\abs*{T^{South}(M_n) - \frac{\Psi_1(\alpha  + \beta )}{\Psi_1(\beta )}\cdot n} > C_1 n^{2/3}\right) &> C_0 \\ 
            \P\left(\abs*{T^{North}(M_n) - \frac{\Psi_1(\alpha  + \beta )}{\Psi_1(\alpha )}\cdot n} > C_1 n^{2/3}\right) & > C_0.
        \end{split}
    \end{align}
    \end{enumerate}
Here $\P$ is the probability on matchings with respect to both the random weights and the dimer measure determined by those weights.
\end{thm}

However, the East turning point is special, and exhibits order $n^{1/2}$ Gaussian fluctuations around $\tfrac{\alpha}{\alpha+\beta} \cdot n$. Here we in fact show that the shuffling algorithm produces a random walk in random environment which converges to a Brownian motion, see \Cref{thm:beta_rwre_limit_dynamical}. This contrasts with the other turning points, where the randomness of the weights plays a role visible in the order of fluctuations. One can see a reflection of the special nature of the East side of the figure in \Cref{subfig:dd_rand}: the loops toward the right of the figure appear more numerous and similar to those in \Cref{subfig:dd_det} with deterministic weights, while away from the right of the figure one sees fewer and smaller loops, characteristic of `freezing' of the matching due to random weights. We suspect that the $n^{2/3}$ behavior is the generic universal one for other classes of random edge weights and the $n^{1/2}$ fluctuations likely occur only in very special cases.

The scaling exponent $n^{2/3}$ of \Cref{thm:turning_points_intro} characteristic of the Kardar-Parisi-Zhang (KPZ) universality class \cite{kardar1986dynamic}. We do not prove \Cref{thm:turning_points_intro} from first principles, but instead prove an exact distributional equality with known exactly-solvable directed polymer models in this class, and use results already proven for these. This means, in particular, that as finer information becomes available for transverse fluctuations of integrable polymer paths, it will transfer mutatis mutandis to finer results on the fluctuations of these turning points, see \Cref{rmk:port_results}.

\subsection{Polymer models and turning points} \label{subsec:stationary_single_path_intro}

Directed polymers in random media are another well-studied model with quenched disorder. Here, the random object sampled on top of the quenched disorder is not a random matching, but a random directed path, commonly with fixed endpoints. A concrete example is the following exactly solvable case due to Sepp{\"a}l{\"a}inen \cite{seppalainen2012scaling}.

\begin{defi}
    \label{def:stationary_log_gamma_intro}
    Let $\alpha,\beta > 0$. Place an independent random weight $Y_{i,j}$ on every point $(i,j) \in \Z_{\geq 0}^2$, where $Y_{0,0}=1$ and those with $(i,j) \neq (0,0)$ have distributions 
    \begin{equation}
        Y_{i,j} \sim \begin{cases}
             \Gamma^{-1}(\beta ,1) & j=0 \\ 
             \Gamma^{-1}(\alpha,1) & i = 0 \\ 
             \Gamma^{-1}(\alpha+\beta,1) & i,j > 0
        \end{cases}
    \end{equation}
    where we write $X \sim \Gamma^{-1}(\chi,s)$ if $1/X \sim \Gamma(\chi,s)$. These are known as \emph{stationary log-Gamma polymer} weights. 
    
    Let $\Pi_{n,n}$ be the set of up-right paths with steps $(i,j) \to (i+1,j)$ and $(i,j) \to (i,j+1)$, beginning at $(0,0)$ and ending at $(n,n)$. Then the associated random \emph{partition function} is 
    \begin{equation}
        Z^{stat-\Gamma}_{n,n} := \sum_{\pi \in \Pi_{n,n}} \prod_{(i,j) \in \pi} Y_{i,j},
    \end{equation}
    where the product is over the weights on all edges in the path. The associated \emph{quenched} polymer measure on $\Pi_{n,n}$ is 
    \begin{equation}
        Q^{stat-\Gamma}_{n,n}(\pi) = \frac{\prod_{(i,j) \in \pi} Y_{i,j}}{Z^{stat-\Gamma}_{n,n}},
    \end{equation}
    a random probability measure depending on the weights. The associated \emph{annealed} polymer measure on $\Pi_{m,n}$ is 
    \begin{equation}
        \P_{n,n}^{stat-\Gamma}(\pi) = \E[Q^{stat-\Gamma}_{n,n}(\pi)],
    \end{equation}
    where the expectation is with respect to the weights.

    For $\pi \in \Pi_{n,n}$, we define $x_{mid}(\pi)$ to be the $x$-coordinate of the unique intersection of $\pi$ with the line $y=n-x$. 
\end{defi}

\begin{figure}[H]
\centering
\begin{tikzpicture}[scale=1.15]
\draw (0,0) -- (5,0) -- (5,5) -- (0,5) -- (0,0);

\draw [line width=2.5pt]
  (0,0) -- (0.25,0) -- (0.25,0.5) -- (1,0.5) -- (1,0.75) -- (1.25,0.75)
  -- (1.25,1.5) -- (1.5,1.5) -- (1.5,1.75) -- (2.25,1.75) -- (2.25,2)
  -- (2.5,2) -- (2.5,2.75) -- (3,2.75) -- (3,3) -- (3.5,3) -- (3.5,3.75)
  -- (4.25,3.75) -- (4.25,4) -- (4.5,4) -- (4.5,4.5) -- (5,4.5) -- (5,5);

\draw [dashed] (2.5,0) -- (2.5,2.5);
\node [below] at (2.5,0) {$x_{mid}(\pi)$};
\draw[blue,ultra thick] (0,5) -- (5,0); 

\node [right] at (5,0) {$(n,0)$};

\node [left] at (0,5) {$(0,n)$};

\node [left] at (0,0) {$(0,0)$};

\node [right] at (5,5) {$(n,n)$};

\draw[step=0.25, dotted, draw=black!60, opacity=0.45, line cap=round] (0,0) grid (5,5);

\end{tikzpicture}
\caption{A path $\pi \in \Pi_{n,n}$ with $n=20$. Here $x_{mid}(\pi) = 10$.}
\label{fig:polymer_xmid_intro}
\end{figure}
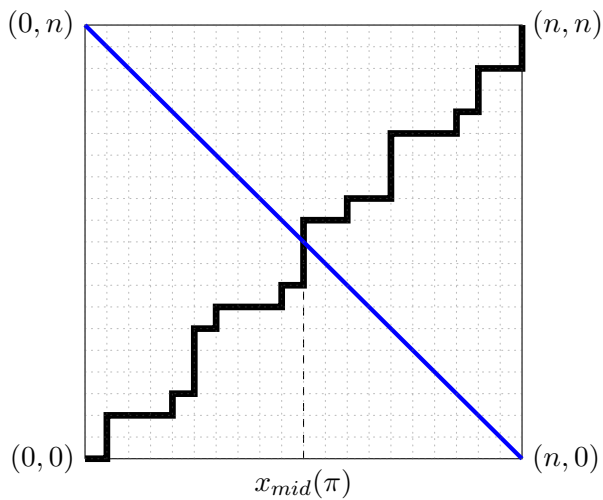

In other words, the annealed probability measure just gives the distribution of a random path given by first sampling the polymer environment of Gamma variables, then sampling a path with respect to the quenched polymer measure $Q^{stat-\Gamma}_{n,n}$ associated to those weights. The reason for the different weights at the boundary is that they yield a natural invariance property for the partition function, the \emph{Burke property} (\Cref{thm:burke}), which underlies the exact solvability and is the reason for the name `stationary'.

Directed polymer models like this one, both discrete and continuum, have been studied in the physics literature for some time, see e.g. Kardar-Zhang \cite{kardar1987scaling} and the references therein. Rigorous mathematical work dates back to Imbrie-Spencer \cite{imbrie1988diffusion}. A later wave of activity came from exactly solvable models in the plane such as the log-Gamma polymer above. Exact formulas for various observables of these models, many coming from Macdonald processes \cite{borodin2014macdonald}, played a key role in subsequent rigorous results concerning the KPZ universality class.

Four basic discrete exactly-solvable models are known: the log-Gamma polymer above, and the later strict-weak polymer, Beta polymer/random walk in random environment, and inverse Beta polymer, introduced by Corwin-Sepp{\"a}l{\"a}inen-Shen \cite{corwin2015strict} and O'Connell-Ortmann \cite{oconnell2015tracy}, Barraquand-Corwin \cite{barraquand2017random}, and Thiery-Le Doussal \cite{thiery2015integrable} respectively. Exact solvability is quite rigid: apart from extensions of these models to other types of planar domains like strips, half-spaces, and cylinders, this list of four has not grown since 2015. Both Thiery-Le Doussal \cite{thiery2015integrable} and Chaumont-Noack \cite{chaumont2018characterizing} further show uniqueness characterizations which single these models out as the only ones with natural types of exact solvability.

Amazingly, all four of these exactly-solvable models appear naturally in the Gamma-disordered Aztec diamond. One of the four cases is the following:

\begin{thm}
    \label{thm:west_matching_intro}
    Fix $n \in \N$ and $\alpha,\beta> 0 $. Let $M$ be a perfect matching of the Gamma-disordered Aztec diamond of size $n$ with these parameters. Let $\pi$ be distributed according to the polymer measure $Q^{stat-\Gamma}_{n,n}$ of an independent stationary log-Gamma polymer with parameters $\alpha,\beta$. Then 
    \begin{equation}
        T^{West}(M) = x_{mid}(\pi)\quad \quad \quad \quad \text{in (annealed) distribution,}
    \end{equation}
    where $T^{West}$ is as in \Cref{def:all_turning_points}.
\end{thm}

For the South turning point, and by symmetry the North, the result is essentially the same. For these the matching, given in \Cref{thm:south_endpoint}, is with the \emph{stationary strict-weak polymer} which is recalled in \Cref{def:stat_sw}. These matchings of \Cref{thm:west_matching_intro} and \Cref{thm:south_endpoint} are the source of the KPZ scaling exponents in \Cref{thm:turning_points_intro}, which we prove in \Cref{sec:fluctuations_of_turning_points} using these matchings together with results of Sepp{\"a}l{\"a}inen \cite{seppalainen2012scaling} and Chaumont-Noack \cite{chaumont2018fluctuation} on fluctuations of polymer paths. 

Random walks in a random environment are formally a special case of directed polymer models, but display different path behavior. The reason for the different scalings at the East turning point is that it matches with the Beta random walk in random environment of \cite{barraquand2017random}, see \Cref{def:beta_rwre} for definitions and \Cref{thm:right-slice_to_rwre} for the distributional equality. The inverse-beta polymer appears in more general results beyond the turning points, see \Cref{rmk:relation_to_inverse_beta}.

\begin{rmk}\label{rmk:port_results}
    These exact distributional equalities mean that whenever stronger results become available for transverse fluctuations of polymer paths, they can immediately be ported to turning points as well. For instance, we note that improvements of the $1/b^3$ decay in \Cref{thm:turning_points_intro} were recently proved for the log-Gamma polymer which controls the West turning point by Rassoul-Agha-Sepp{\"a}l{\"a}inen-Shen \cite{rassoul2024coalescence} (we thank Timo Sepp{\"a}l{\"a}inen and Xiao Shen for pointing this out). We use the weaker results of \cite{chaumont2018fluctuation} simply to keep the statement uniform for the sake of exposition, as the stronger ones are not proven to our knowledge for the strict-weak polymer which controls the North and South turning points.
    
    These matchings also provide a route to prove the limiting distribution of fluctuations of the turning points, which we do not address here. The physics papers of Le Doussal \cite{le2017maximum} and Maes-Thiery \cite{maes2017midpoint} derived that the midpoint fluctuations of the log-Gamma polymer are governed by position of the maximum of an Airy process minus a parabola and an independent Brownian motion (and this should also hold for the strict-weak polymer). Formulas for this scaling function $f_{KPZ}(y)$ were derived much earlier by Pr{\"a}hofer-Spohn \cite{prahofer2004exact}. These results have to our knowledge not yet been proven rigorously, but should be accessible to current methods with some technical work. We thank Pierre Le Doussal and Evan Sorensen for comments on these points.
\end{rmk}

\subsection{Beyond turning points: hybrid polymers} \label{subsec:hybrid_polymer_intro}

The distributional equalities between stationary polymer models and turning points, such as \Cref{thm:west_matching_intro}, come from special cases of a more general match with what we call \emph{hybrid} integrable polymer models. These feature polymer paths in a random environment of two types, each type being one of the four mentioned above, grafted along a boundary. We are not aware of these models appearing in previous literature, though they have natural relations to known stationary multi-path polymer models which will be explained in upcoming work, see \Cref{rmk:multi-path_future}.

These polymer matchings give complete information about the distribution of a matching along a fixed vertical or horizontal slice. This directly generalizes results such as \Cref{thm:west_matching_intro} on turning points, which correspond to the extremal slices on the edges of the Aztec diamond. The information of a matching at a general slice is captured by the following statistics.

\begin{defi}\label{def:aztec_slices_intro}
    For an Aztec diamond $G_n^{\Az}$ of size $n$, let $\mathcal{M}$ be the set of perfect matchings of $G_n^{\Az}$. For any $1 \leq \ell \leq n$, define functions 
    \begin{equation}
        X_\ell^{\Az}: \mathcal{M} \to \binom{[n+1]}{n-\ell+1}
    \end{equation}
    of the $\ell\tth$ column as follows. There are $n$ columns of white vertices in $G_n^{\Az}$; consider the $\ell\tth$ one from left to right. For each perfect matching $M \in \mathcal{M}$, there will be exactly $n-\ell+1$ white vertices in the $\ell\tth$ column which are matched to a black vertex to their left (either up-left or down-left). Labeling the white vertices in this column by $1,\ldots,n+1$ from top to bottom, we let $X_\ell^{\Az}(M)$ be the set of labels of the white vertices which are matched with a black vertex to their left (again, either up-left or down-left). 
\end{defi}

Note that the extreme cases of $X_\ell^{\Az}$ encode turning points: $X_n^{\Az}(M) = \{T^{East}(M)+1\}$ and $X_1^{\Az}(M) = [n+1] \setminus \{T^{West}(M)+1\}$. See \Cref{fig:mixed_polymer_intro} for a non-extreme example.

\begin{defi}
    \label{def:beta_sw_intro}
    Let $p,m \geq 1$ be integers, and let $\alpha,\beta > 0$. Consider the weighted, directed graph $G^{\beta \Gamma}_{p,m}$ (\Cref{fig:mixed_polymer_intro} (right)) with all vertices lying in the set 
    \begin{multline}\label{eq:bsw_support_set_intro}
        \{(x,y) \in \Z^2: -m \leq x \leq -1, -m-p \leq y \leq -1, x+y \geq -m-p\} \\ 
         \cup \{(x,y) \in \Z^2: -m-p \leq y \leq -1, 0 \leq x \leq \min(p-1, y + m + p)\} 
    \end{multline}
    and edges given as follows:
    \begin{enumerate}
\item For each vertex at $(x,y)$ with $x \leq -1$, there are right and down-right directed edges to vertices $(x+1,y)$ and $(x+1,y-1)$, with weights $\beta_{x,y}$ and $1-\beta_{x,y}$ respectively, where $\beta_{x,y} \sim \Beta(\alpha , \beta )$ are independent Beta variables (see \Cref{def:beta_var}).  
\item For each vertex at $(x,y)$ with $x \geq 0$ lying in the set \eqref{eq:bsw_support_set_intro}, there are down and right directed edges to vertices $(x,y-1)$ and $(x+1,y)$ (if they also lie in the set \eqref{eq:bsw_support_set_intro}). The downward edges have weight $1$, while the down-right edges have independent Gamma weights $\gamma_{x,y} \sim \Gamma(\alpha+\beta,1)$.
\end{enumerate}
Then the associated \emph{$\beta$-$\Gamma$ polymer partition function} is 
    \begin{equation}
        Z^{\beta \Gamma}_{p,m} := \sum_{\substack{\pi_j:(-m,-j) \to (p-j,-m-j) \\ 1 \leq j \leq p}} \prod_{j=1}^p \wt(\pi_j),
    \end{equation}
    where the sum is over $p$-tuples of paths $\pi_1,\ldots,\pi_p$ on $G^{\beta \Gamma}_{p,m}$ with no vertices in common, where $\pi_j$ has start point $(-m,-j)$ and end point $(p-j,-m-j)$, and $\wt(\pi_j)$ denotes the product of edge weights over the edges in $\pi_j$.

    The associated \emph{$\beta$-$\Gamma$ polymer measure} is a probability measure on such $p$-tuples $(\pi_1,\ldots,\pi_p)$ of nonintersecting paths which assigns to each one the probability
    \begin{equation}
        \frac{1}{Z^{\beta \Gamma}_{p,m}} \prod_{j=1}^p \wt(\pi_j).
    \end{equation}
\end{defi}

The left and right portion of the graph $G^{\beta \Gamma}_{p,m}$, on their own, have been studied in the literature: they are the Beta random walk in random environment and the strict-weak polymer mentioned earlier. The polymer statistic which matches with $X^{\Az}_\ell$ is the following.

\begin{defi}\label{def:X_polymer_intro}
    Each of the polymer paths $\pi_1,\ldots,\pi_p$ in \Cref{def:beta_sw} contains at least one vertex $(0,-y)$, and sometimes several. Let 
    \begin{equation*}
        \pi_j(m) = \min(\{y: (0,-y) \in \pi_j\})
    \end{equation*}
    be the distance of the closest such vertex to the $x$-axis. Then we define
    \begin{equation}
        X^{poly}(\pi_1,\ldots,\pi_p) = \{\pi_j(m): 1 \leq j \leq p\}.
    \end{equation}
\end{defi}

At the risk of stating the obvious, $|X^{poly}(\pi_1,\ldots,\pi_p)| = p$ since the paths are non-intersecting.

\begin{thm}[Special case of \Cref{thm:vert_slice_polymer} in text]
    \label{thm:multi-path_intro_vert}
    Fix $n \in \Z_{\geq 1}$ and $\alpha,\beta > 0$. Consider a biased Gamma-disordered Aztec diamond (\Cref{def:gamma_weights_intro}) of size $n$ with these parameters, and let $M$ be a matching distributed by the corresponding random dimer measure.

    Fix $\ell$ with $1 \leq \ell \leq n$, and consider also an independent polymer on $G^{\beta \Gamma}_{p,m}$ with $p=n-\ell+1$ paths and $m=\ell$ layers in the notation of \Cref{def:beta_sw_intro} with the same parameters $\alpha,\beta$. Let $\pi_1,\ldots,\pi_p$ be paths (ordered from top to bottom as in \Cref{def:beta_sw_intro}) distributed according to the corresponding polymer measure. Then
    \begin{equation}
        X_\ell^{\Az}(M) = X^{poly}(\pi_1,\ldots,\pi_p) \quad \quad \quad \quad \text{ in distribution,}
    \end{equation}
    where $X_\ell^{\Az}$ and $X^{poly}$ are as in Definitions \ref{def:aztec_slices_intro} and \ref{def:X_polymer_intro}.
\end{thm}

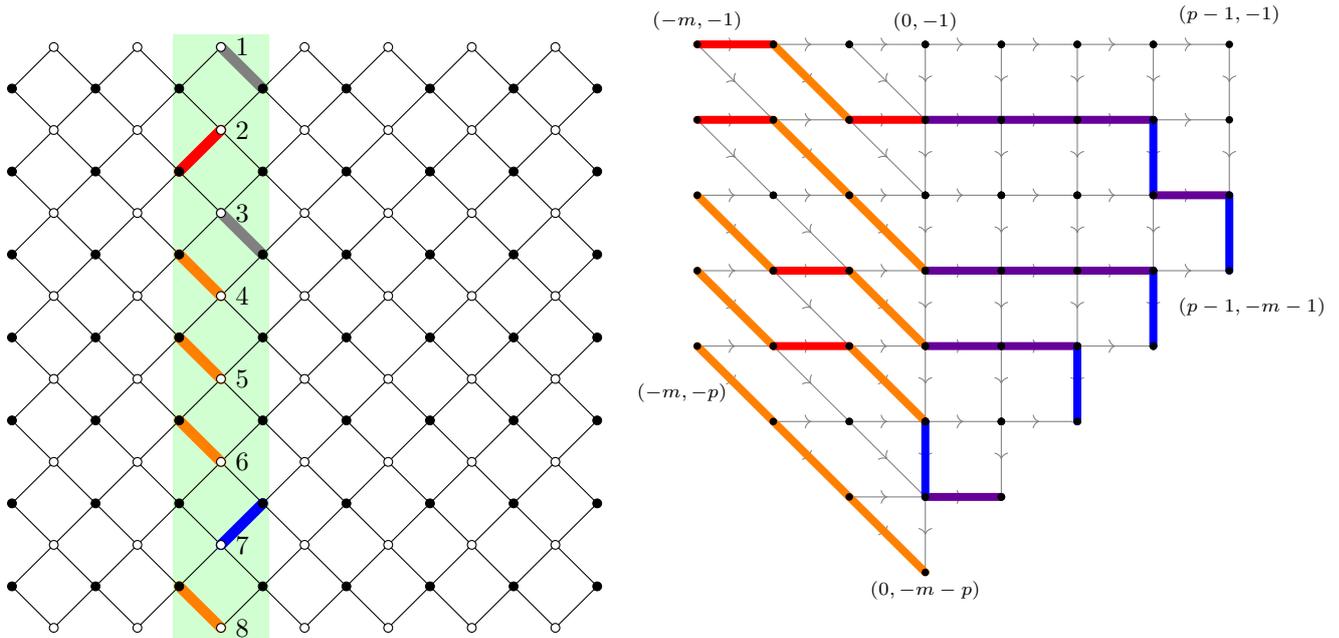
\begin{figure}[t]
    \centering

  \begin{subfigure}{0.48\textwidth}
    \centering
\begin{tikzpicture}[scale=.55, line cap=round, line join=round]
  \def\n{7}

  \tikzset{
    matchlu/.style={line width=4pt, red},            
    matchlh/.style={line width=4pt, orange},         
    matchru/.style={line width=4pt, blue},           
    matchrh/.style={line width=4pt, blue!60!red},    
    matchblank/.style={line width=4pt, gray},        
  }

  \pgfmathtruncatemacro{\k}{2}
  \path[fill=green!60, fill opacity=0.30, draw=none]
    (2*\k-.15, .7) rectangle (2*\k+2.15, 2*\n+2-.7);

  \foreach \x in {0,...,\numexpr\n-1\relax} {
    \foreach \y in {1,...,\n} {
      \draw (2*\x,   2*\y) -- (2*\x+1, 2*\y+1);
      \draw (2*\x,   2*\y) -- (2*\x+1, 2*\y-1);
      \draw (2*\x+2, 2*\y) -- (2*\x+1, 2*\y+1);
      \draw (2*\x+2, 2*\y) -- (2*\x+1, 2*\y-1);
    }
  }

  \newcommand{\LL}[1]{(2*\k,   2*#1)}
  \newcommand{\RR}[1]{(2*\k+2, 2*#1)}
  \newcommand{\WU}[1]{(2*\k+1, 2*#1+1)}
  \newcommand{\WD}[1]{(2*\k+1, 2*#1-1)}


  \foreach \j in {7,5} { \draw[matchblank] \RR{\j} -- \WU{\j}; }

  \foreach \j in {1,3,4,5} { \draw[matchlh] \LL{\j} -- \WD{\j}; }

  \foreach \j in {6} { \draw[matchlu] \LL{\j} -- \WU{\j}; }


  \foreach \j in {2} { \draw[matchru] \RR{\j} -- \WD{\j}; }

  \foreach \x in {0,...,\n} {
    \foreach \y in {1,...,\n} {
      \filldraw (2*\x, 2*\y) circle (3pt);                     
      \draw[fill=white] (2*\y-1, 2*\x+1) circle (3pt);         
    }
  }

  \foreach \t in {1,...,\numexpr\n+1\relax} {
    \pgfmathtruncatemacro{\yodd}{2*\n - 2*\t + 3} 
    \node[anchor=west, font=\small] at (2*\k+1+.1, \yodd) {\t};
  }

\end{tikzpicture}

  \end{subfigure}\hfill
  \begin{subfigure}{0.48\textwidth}
    \centering
    
 \begin{tikzpicture}[
    scale=1,
    line cap=round, 
    every node/.style={circle, fill=black, inner sep=1pt},
    decoration={markings, mark=at position 0.5 with {\arrow{>}}},
    arrowedge/.style={postaction={decorate}, gray},
    infdot/.style={circle, draw=none, fill=gray!50, inner sep=0.5pt}
  ]
\def\n{7}
\def\l{3}

\def\jmax{5}

\node[anchor=south, circle=none, fill=none, inner sep=0pt] at (0,4.7) {\tiny{$(-m,-1)$}};
\node[anchor=south, circle=none, fill=none, inner sep=0pt] at (3,4.85) {\tiny{$(0,-1)$}};
\node[anchor=south, circle=none, fill=none, inner sep=0pt] at (7,4.7) {\tiny{$(p-1,-1)$}};
\node[anchor=north, circle=none, fill=none, inner sep=0pt] at (-.2,1) {\tiny{$(-m,-p)$}};
\node[anchor=north, circle=none, fill=none, inner sep=0pt] at (3,-1.5) {\tiny{$(0,-m-p)$}};
\node[anchor=north, circle=none, fill=none, inner sep=0pt] at (7.3,2.5) {\tiny{$(p-1,-m-1)$}};

\def\n{7}
\def\l{3}

\def\jmax{5}

\foreach \i in {0,...,2} {%
  \pgfmathtruncatemacro{\ii}{\i + 1}%
  \pgfmathtruncatemacro{\jmin}{-\i+1}%
  \foreach \j in {\jmin,...,\jmax} {%
    \pgfmathtruncatemacro{\jm}{\j - 1}%
    \path
      (\i,\j)   coordinate (v-\i-\j)
      (\ii,\j)  coordinate (v-\ii-\j)
      (\ii,\jm) coordinate (v-\ii-\jm);
    \draw[arrowedge] (v-\i-\j) -- (v-\ii-\j);
    \draw[arrowedge] (v-\i-\j) -- (v-\ii-\jm);
  }%
}

\foreach \i in {3,...,6} {%
  \pgfmathtruncatemacro{\ii}{\i + 1}%
  \pgfmathtruncatemacro{\jmin}{-4+\i}%
  \foreach \j in {\jmin,...,\jmax} {%
    \pgfmathtruncatemacro{\jm}{\j - 1}%
      \pgfmathtruncatemacro{\jmm}{\j - 1}
    \path
       (\ii,\j) coordinate (v-\ii-\j);
       (\i,\j)   coordinate (v-\i-\j)
      (\i,\jm)  coordinate (v-\ii-\jm)
     
    \draw[arrowedge] (v-\i-\j) -- (v-\i-\jm);
    \draw[arrowedge] (v-\i-\j) -- (v-\ii-\j);
  }%
}%
\draw[arrowedge] (7,5)--(7,4);
\draw[arrowedge] (7,4)--(7,3);


\draw[line width=3pt, red] (0,5)--(1,5);
\draw[line width=3pt, orange] (1,5)--(2,4);
\draw[line width=3pt, red] (2,4)--(3,4);
\draw[line width=3pt, blue!60!red] (3,4)--(6,4);
\draw[line width=3pt, blue] (6,4)--(6,3);
\draw[line width=3pt, blue!60!red] (6,3)--(7,3);
\draw[line width=3pt, blue] (7,3)--(7,2);

\draw[line width=3pt, red] (0,4)--(1,4);
\draw[line width=3pt, orange] (1,4)--(3,2);
\draw[line width=3pt, blue!60!red] (3,2)--(6,2);
\draw[line width=3pt, blue] (6,2)--(6,1);

\draw[line width=3pt, orange] (0,3)--(1,2);
\draw[line width=3pt,red] (1,2)--(2,2);
\draw[line width=3pt, orange] (2,2)--(3,1);
\draw[line width=3pt, blue!60!red] (3,1)--(5,1);
\draw[line width=3pt, blue] (5,1)--(5,0);

\draw[line width=3pt, orange] (0,2)--(1,1);
\draw[line width=3pt,red] (1,1)--(2,1);
\draw[line width=3pt, orange] (2,1)--(3,0);
\draw[line width=3pt, blue!60!red] (3,-1)--(4,-1);
\draw[line width=3pt, blue] (3,0)--(3,-1);

\draw[line width=3pt,orange] (0,1)--(3,-2);

\foreach \i in {3,...,6} {%
  \pgfmathtruncatemacro{\ii}{\i + 1}%
  \pgfmathtruncatemacro{\jmin}{-4+\i}%
  \foreach \j in {\jmin,...,\jmax} {%
    \pgfmathtruncatemacro{\jm}{\j - 1}%
      \pgfmathtruncatemacro{\jmm}{\j - 1}
    \path
       (\ii,\j) coordinate (v-\ii-\j);
       (\i,\j)   coordinate (v-\i-\j)
      (\i,\jm)  coordinate (v-\ii-\jm)
     
    \node (v\i\j) at (\i,\j) {};
    \node (v\i\jm) at (\i,\jm) {};
    \node (v\ii\j) at (\ii,\j) {};
  }%
}%

\foreach \i in {0,...,2} {%
  \pgfmathtruncatemacro{\ii}{\i + 1}%
  \pgfmathtruncatemacro{\jmin}{-\i+1}%
  \foreach \j in {\jmin,...,\jmax} {%
    \pgfmathtruncatemacro{\jm}{\j - 1}%
    \path
      (\i,\j)   coordinate (v-\i-\j)
      (\ii,\j)  coordinate (v-\ii-\j)
      (\ii,\jm) coordinate (v-\ii-\jm);
    \node (v\i\j) at (\i,\j) {};
    \node (v\ii\j) at (\ii,\j) {};
  }%
}

\node (x,y) at (6,3) {};
\end{tikzpicture}
  \end{subfigure}

    \caption{Part of a matching $M$ of $G_7^{\Az}$ (left), and a path configuration $\pi_1,\ldots,\pi_5$ on the directed graph $G^{\beta \Gamma}_{5,3}$ (right). These satisfy $X_3^{\Az}(M) = \{2,4,5,6,8\} = X^{poly}(\pi_1,\ldots,\pi_5)$ (see Definitions~\ref{def:aztec_slices_intro} and \ref{def:X_polymer_intro}). \Cref{thm:multi-path_intro_vert} in this case $n=7,\ell=3$ implies that $M$ and $\pi_1,\ldots,\pi_5$ may be coupled so that this equality $X_3^{\Az}(M) = X^{poly}(\pi_1,\ldots,\pi_5)$ always holds. The parts of $\pi_1,\ldots,\pi_5$ to the left and right of the central line correspond to the history of $M$ under the shuffling algorithm, as we discuss later in \Cref{sec:dynamical_vert}. }
    \label{fig:mixed_polymer_intro}
\end{figure}

Rather than looking at the $\ell\tth$ vertical slice of the Aztec diamond as above, one may also consider the $\ell\tth$ horizontal slice. We prove a completely analogous result to \Cref{thm:multi-path_intro_vert} for this statistic, matching to a mixed polymer consisting of the log-Gamma and strict-weak polymers, see \Cref{def:horiz_polymer_digraph} for the polymer definition and \Cref{thm:hor_slice_polymer} for the matching theorem. It is perhaps surprising that one dimer model naturally contains all of these integrable polymers simultaneously.

\Cref{thm:turning_points_intro} shows that the Gamma-disordered Aztec diamond has nontrivially different behavior from Aztec diamonds with deterministic weights at the turning points, but there are many features of matchings with deterministic weights for which we do not yet know the analogues for random weights. These include the Gaussian free field fluctuations of the height function, Tracy-Widom fluctuations of the frozen boundary, and GUE corners process limits close to the turning points. We hope that the match with multi-polymers helps begin to address these questions; in particular the analogue of the GUE corners process fluctuations seems achievable if the turning point fluctuations can be established (\Cref{rmk:port_results}). 

The fact that matchings of the Aztec diamond correspond to nonintersecting paths is well-known \cite[Exercise 6.49]{stanley2024enumerative}, and many of the above-mentioned physics works have noted that dimer models with random weights correspond to multi-path directed polymers, often calling the paths `flux lines' or `vortices'. What distinguishes \Cref{thm:multi-path_intro_vert} (and the horizontal version \Cref{thm:hor_slice_polymer}) from those prior works is that the polymers appearing in these theorems are combinations of known integrable polymers. This is particularly surprising because we initially chose the weights of the model for reasons having only to do with the shuffling algorithm, with no intention to try to match with polymers, as we explain next.

\subsection{Shuffling and the uniqueness characterization of Gamma weights}\label{subsec:shuffling_and_uniqueness}

First note that, while the weights of \Cref{fig:our_weights_intro} appear non-generic because half of them are $1$, in fact there is no loss of generality in taking weights of this form. For any fixed vertex $v$, all perfect matchings include exactly one edge incident to $v$, so multiplying the weights of all such incident edges by the same constant factor does not change the dimer measure, only the partition function. By making a sequence of such gauge transforms along columns, any collection of (nonzero) weights gives an equivalent dimer measure to one as in \Cref{fig:our_weights_intro}.

The shuffling algorithm is designed to give a random sample from the Aztec diamond dimer model of size $n$ with given fixed weights, which without loss of generality are given by a collection $\{a_{i,j}\}_{i,j=1}^n$ and $\{b_{i,j}\}_{i,j=1}^n$ placed as in \Cref{fig:our_weights_intro}. The algorithm proceeds by generating a coupled sequence of matchings of the Aztec diamonds of sizes $1,2,\ldots,n$, where the $k\tth$ is generated from the $(k-1)\tth$ by certain local random moves. For example, suppose that the matching at $k-1=3$ is the one from \Cref{fig:our_weights_intro} earlier. First embed this Aztec diamond inside one of size $k=4$, after interchanging white with black colors so that the convention of black vertices on sides and white vertices on top and bottom is preserved. Then perform the following two steps, illustrated in \Cref{fig:shuffling_steps}:
\begin{enumerate}
    \item \textbf{(Deterministic slide and destruction):} Move each Northeast, Northwest, Southeast, and Southwest edge by $1$ unit in the direction prescribed by its name. If there is any pair of a Northeast and Southwest edge which would exchange places, remove both edges, and similarly for Northwest/Southeast pairs.
    \item \textbf{(Random creation):} At each square with black vertices on its sides and white vertices on its top and bottom, independently fill in either a Northeast and Southwest edge or a Northwest and Southeast edge.
\end{enumerate}
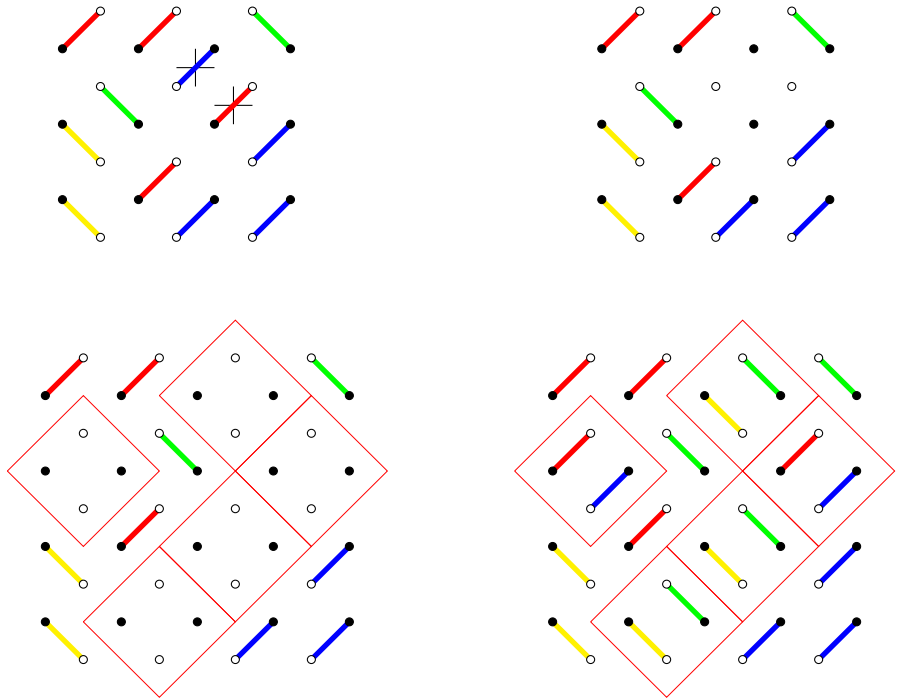
\begin{figure}[t]
\begin{tikzpicture}[scale=.5]

\draw (3.5,5)--(3.5,6);
\draw (3,5.5)--(4,5.5);
\draw (4.5,4)--(4.5,5);
\draw (4,4.5)--(5,4.5);
     \draw[-,red,line width=2pt] (0,6)--(1,7);
\draw[-,yellow,line width=2pt] (0,4)--(1,3);
\draw[-,yellow,line width=2pt] (0,2)--(1,1);
\draw[-,red,line width=2pt] (2,6)--(3,7);
\draw[-,red,line width=2pt] (2,2)--(3,3);
\draw[-,green,line width=2pt] (1,5)--(2,4);

\draw[-,blue,line width=2pt] (3,1)--(4,2);
\draw[-,blue,line width=2pt] (3,5)--(4,6);

\draw[-,red,line width=2pt] (4,4)--(5,5);

\draw[-,blue,line width=2pt] (5,1)--(6,2);
\draw[-,blue,line width=2pt] (5,3)--(6,4);
\draw[-,green,line width=2pt] (5,7)--(6,6);

\foreach \x in {0,...,3} {
  \foreach \y in {1,...,3} {
   \filldraw (2*\x,2*\y) circle(3pt);
    \draw[fill=white] (2*\y-1,2*\x+1) circle(3pt);
  }
}
\end{tikzpicture} \qquad \qquad  \qquad \qquad \qquad
\begin{tikzpicture}[scale=.5]
     \draw[-,red,line width=2pt] (0,6)--(1,7);
\draw[-,yellow,line width=2pt] (0,4)--(1,3);
\draw[-,yellow,line width=2pt] (0,2)--(1,1);
\draw[-,red,line width=2pt] (2,6)--(3,7);
\draw[-,red,line width=2pt] (2,2)--(3,3);
\draw[-,green,line width=2pt] (1,5)--(2,4);

\draw[-,blue,line width=2pt] (3,1)--(4,2);

\draw[-,blue,line width=2pt] (5,1)--(6,2);
\draw[-,blue,line width=2pt] (5,3)--(6,4);
\draw[-,green,line width=2pt] (5,7)--(6,6);

\foreach \x in {0,...,3} {
  \foreach \y in {1,...,3} {
   \filldraw (2*\x,2*\y) circle(3pt);
    \draw[fill=white] (2*\y-1,2*\x+1) circle(3pt);
  }
}

\end{tikzpicture}
\vspace*{1cm}

\begin{tikzpicture}[scale=.5]
     \draw[-,red,line width=2pt] (-1,7)--(0,8);
\draw[-,yellow,line width=2pt] (-1,3)--(0,2);
\draw[-,yellow,line width=2pt] (-1,1)--(0,0);
\draw[-,red,line width=2pt] (1,7)--(2,8);
\draw[-,red,line width=2pt] (1,3)--(2,4);
\draw[-,green,line width=2pt] (2,6)--(3,5);

\draw[-,blue,line width=2pt] (4,0)--(5,1);

\draw[-,blue,line width=2pt] (6,0)--(7,1);
\draw[-,blue,line width=2pt] (6,2)--(7,3);
\draw[-,green,line width=2pt] (6,8)--(7,7);

\draw[red, very thin] (0,1)-- ++(2,2)--++(2,-2)--++(-2,-2)--++(-2,2); 

\draw[red, very thin] (2,3)-- ++(2,2)--++(2,-2)--++(-2,-2)--++(-2,2); 
\draw[red, very thin] (4,5)-- ++(2,2)--++(2,-2)--++(-2,-2)--++(-2,2); 

\draw[red, very thin] (2,7)-- ++(2,2)--++(2,-2)--++(-2,-2)--++(-2,2); 

\draw[red, very thin] (-2,5)-- ++(2,2)--++(2,-2)--++(-2,-2)--++(-2,2); 
\foreach \x in {-1,...,3} {
  \foreach \y in {1,...,4} {
   \filldraw (2*\x+1,2*\y-1) circle(3pt);
    \draw[fill=white] (2*\y-2,2*\x+2) circle(3pt);
  }
}
\end{tikzpicture}
\qquad  \qquad 
\begin{tikzpicture}[scale=.5]
     \draw[-,red,line width=2pt] (-1,7)--(0,8);
\draw[-,yellow,line width=2pt] (-1,3)--(0,2);
\draw[-,yellow,line width=2pt] (-1,1)--(0,0);
\draw[-,red,line width=2pt] (1,7)--(2,8);
\draw[-,red,line width=2pt] (1,3)--(2,4);
\draw[-,green,line width=2pt] (2,6)--(3,5);

\draw[-,blue,line width=2pt] (4,0)--(5,1);

\draw[-,blue,line width=2pt] (6,0)--(7,1);
\draw[-,blue,line width=2pt] (6,2)--(7,3);
\draw[-,green,line width=2pt] (6,8)--(7,7);

\draw[-,blue,line width=2pt] (0,4)--(1,5);
\draw[-,red,line width=2pt] (-1,5)--(0,6);

\draw[-,blue,line width=2pt] (6,4)--(7,5);
\draw[-,red,line width=2pt] (5,5)--(6,6);

\draw[-,yellow,line width=2pt] (3,7)--(4,6);
\draw[-,green,line width=2pt] (4,8)--(5,7);
\draw[-,yellow,line width=2pt] (3,3)--(4,2);
\draw[-,green,line width=2pt] (4,4)--(5,3);

\draw[-,yellow,line width=2pt] (1,1)--(2,0);
\draw[-,green,line width=2pt] (2,2)--(3,1);

\draw[red, very thin] (0,1)-- ++(2,2)--++(2,-2)--++(-2,-2)--++(-2,2); 

\draw[red, very thin] (2,3)-- ++(2,2)--++(2,-2)--++(-2,-2)--++(-2,2); 
\draw[red, very thin] (4,5)-- ++(2,2)--++(2,-2)--++(-2,-2)--++(-2,2); 

\draw[red, very thin] (2,7)-- ++(2,2)--++(2,-2)--++(-2,-2)--++(-2,2); 

\draw[red, very thin] (-2,5)-- ++(2,2)--++(2,-2)--++(-2,-2)--++(-2,2); 
\foreach \x in {-1,...,3} {
  \foreach \y in {1,...,4} {
   \filldraw (2*\x+1,2*\y-1) circle(3pt);
    \draw[fill=white] (2*\y-2,2*\x+2) circle(3pt);
  }
}
\end{tikzpicture}
\caption{Illustration of the shuffling algorithm starting from the dimer configuration in Figure \ref{fig:our_weights_intro}. The deletion step is illustrated from top-left to top-right, where one pair of a Northwest (red) and Southeast (blue) edge is removed. From top-right to bottom-left, we perform the slide step. This yields a partial matching of an Aztec diamond graph of size $4$, with some empty faces indicated by red boxes. For each box, there are two possible pairs that complete the matching: a pair of Northwest/Southeast (red/blue) edges or a pair of Northeast/Southwest (green/yellow) edges, and each one is chosen independently according to the probabilities \eqref{eq:creation_probabilities}.}\label{fig:shuffling_steps}
\end{figure}

The only piece of information not specified above is with what probabilities the two options in the random creation step should be taken, and this in fact is one of the most interesting features of the shuffling algorithm. To explain, it is first important to note that the matchings on the Aztec diamonds $G_1^{\Az},G_2^{\Az},\ldots,G_n^{\Az}$ produced by the shuffling algorithm are each themselves distributed by the dimer measure on these graphs, but with different weights. Concretely, the graph $G_k^{\Az}$ has weights $\{a_{i,j}^{[k]}\}_{i,j=1}^k$ and $\{b_{i,j}^{[k]}\}_{i,j=1}^k$ placed as in \Cref{fig:our_weights_intro}, which are defined in terms of the weights on $G_n^{\Az}$ by the recursion

\begin{align}\label{eq:map_iterating_down:intro}
    \begin{split}
        a_{i,j}^{[\ell-1]} &= \frac{a_{i,j}^{[\ell]}}{a_{i,j}^{[\ell]}+b_{i,j}^{[\ell]}}(a_{i+1,j}^{[\ell]}+b_{i+1,j}^{[\ell]}) \\ 
        b_{i,j}^{[\ell-1]} &= \frac{b_{i,j+1}^{[\ell]}}{a_{i,j+1}^{[\ell]}+b_{i,j+1}^{[\ell]}}(a_{i+1,j+1}^{[\ell]}+b_{i+1,j+1}^{[\ell]})
    \end{split}
\end{align}
where when $\ell=n$ the weights $a_{i,j}^{[n]},b_{i,j}^{[n]}$ are our original weights $\{a_{i,j}\}_{i,j=1}^n$ and $\{b_{i,j}\}_{i,j=1}^n$. At the random creation step described above, each square has weights $a_{i,j}^{[k]},b_{i,j}^{[k]},1,1$, and the two possible options are chosen with probabilities

\begin{equation}\label{eq:creation_probabilities}
    \mathbb P\left(\tikz{ 
     \useasboundingbox (-.5,0) rectangle (2.5,1.25);
    \draw (0,0)--(1,1)--(2,0)--(1,-1)--(0,0);
    \draw[red,line width=2pt] (0,0)--(1,1);
    \draw[blue,line width=2pt] (1,-1)--(2,0);
    \draw[fill=white] (1,1) circle(3pt);
    \draw[fill=white] (1,-1) circle(3pt);
    \draw[fill=black] (0,0) circle(3pt);
    \draw[fill=black] (2,0) circle(3pt); 
    \draw (0.2,.7) node {\small $a_{i,j}^{[k]}$};
       \draw (0.2,-.7) node {\small $b_{i,j}^{[k]}$};
    }\right)=\frac{a_{i,j}^{[k]}}{a_{i,j}^{[k]}+b_{i,j}^{[k]}}, \qquad  \mathbb P\left(\tikz{ 
    \useasboundingbox (-.5,0) rectangle (2.5,1.25);
    \draw (0,0)--(1,1)--(2,0)--(1,-1)--(0,0);
    \draw[yellow,line width=2pt] (0,0)--(1,-1);
    \draw[green,line width=2pt] (1,1)--(2,0);
    \draw[fill=white] (1,1) circle(3pt);
    \draw[fill=white] (1,-1) circle(3pt);
    \draw[fill=black] (0,0) circle(3pt);
    \draw[fill=black] (2,0) circle(3pt); 
    \draw (0.2,.7) node {\small $a_{i,j}^{[k]}$};
       \draw (0.2,-.7) node {\small $b_{i,j}^{[k]}$};
    }\right)=\frac{b_{i,j}^{[k]}}{a_{i,j}^{[k]}+b_{i,j}^{[k]}}.
\end{equation}

If one simply wishes to numerically sample a perfect matching of an Aztec diamond of some fixed size such as $n=100$ and random weights, one need only sample the weights $\{a_{i,j}^{[n]},b_{i,j}^{[n]}: 1 \leq i,j \leq n\}$ of the size $n$ Aztec diamond, then deterministically compute the weights $a_{i,j}^{[k]},b_{i,j}^{[k]}$ for each $1 \leq k \leq n$ recursively via \eqref{eq:map_iterating_down:intro}, and then follow the above-described algorithm to generate a matching. This is what was done in \cite{zeng1999thermodynamics,bogner2004test,perret2012super}\footnote{Strictly speaking, \cite{perret2012super} used a variant due to Janvresse-de la Rue-Velenik \cite{janvresse2006note} which allows edges of weight $0$.}.

However, because the shuffling algorithm yields a matching distributed by the dimer measure at each step, not just at the final $n\tth$ step, it is natural not to privilege a specific $n$ but rather view the algorithm as a Markov process running in discrete time $n \in \N$ with fixed-time marginals given by the dimer measure on the appropriate Aztec diamond. This perspective of the shuffling algorithm for all $n$ as an interacting particle system was used initially in Jockusch-Propp-Shor \cite{jockusch1998random} and expanded in works such as Borodin-Ferrari \cite{borodin2014anisotropic}. The recurrence \eqref{eq:map_iterating_down:intro} may be inverted to yield a recurrence 

\begin{align}\label{eq:map_iterating_up:intro}
    \begin{split}
        a_{i,j}^{[k+1]} &= \frac{a_{i,j}^{[k]}}{a_{i,j}^{[k]}+b_{i,j-1}^{[k]}}(a_{i-1,j}^{[k]}+b_{i-1,j-1}^{[k]}) \\ 
        b_{i,j}^{[k+1]} &= \frac{b_{i,j-1}^{[k]}}{a_{i,j}^{[k]}+b_{i,j-1}^{[k]}}(a_{i-1,j}^{[k]}+b_{i-1,j-1}^{[k]})
    \end{split}
\end{align}
for the weights on the larger Aztec diamond in terms of those on the smaller. However, some weights such as $a_{i,1}^{[k+1]}$ depend on weights $b_{i-1,0}^{[k]}$ which do not actually appear in the smaller Aztec diamond. To define all the random weights $a_{i,j}^{[n]},b_{i,j}^{[n]}$ with $1 \leq i,j \leq n$ for every $n$, it is best to begin with a full plane worth of weights $\{a_{i,j}^{[0]},b_{i,j}^{[0]}: i,j \in \Z\}$ and then iterate \eqref{eq:map_iterating_up:intro} to produce weights $\{a_{i,j}^{[n]},b_{i,j}^{[n]}: i,j \in \Z\}$ for each $n \in \N$, of which those with $1 \leq i,j \leq n$ will be used for the Aztec diamond of size $n$.

The above-mentioned works \cite{zeng1999thermodynamics,bogner2004test,perret2012super} treat various choices of weights which are independent at the size $n$ being sampled, but for which the weights $\{a_{i,j}^{[k]},b_{i,j}^{[k]}: 1 \leq i,j \leq k\}$ at $k < n$ are no longer independent. While this poses no issue for numerically sampling, from the above perspective of the shuffling algorithm as a Markov process on all $n$, it is strange to privilege a specific $n$ in this way. This is one motivation to look for independent random weights which remain independent at all steps.

Another motivation comes from work on dimer models with deterministic weights. Here, the strongest asymptotic results have only been established for various `probabilistically integrable' families of weights, for which exact contour integral formulas for the correlation functions of edges allow asymptotic analysis. For one such family, the Schur process weights, such formulas were derived by Johansson \cite{johansson2006random}. For doubly-periodic weights, similar formulas were derived by computation of the inverse Kasteleyn matrix by Chhita-Young and Chhita-Johansson \cite{ChhitaYoung2014,ChhitaJohansson2016}, using Wiener-Hopf factorizations by Berggren-Duits \cite{BerggrenDuits2019}, Borodin-Duits \cite{BorodinDuits2023} and Berggren-Borodin \cite{berggren2025geometry}, and using matrix-valued orthogonal polynomials by Duits-Kuijlaars \cite{DuitsKuijlaars2021} and Kuijlaars-Piorkowski \cite{kuijlaars2025wiener}. For Fock's weights introduced in \cite{fock2015inverse}, formulas were established by Boutillier-de Tili\`ere \cite{BoutillierDeTiliere2024_FockDimerAztec}. In each case, the discrete updates \eqref{eq:map_iterating_up:intro} are integrable, meaning that they take a particularly simple form after introducing appropriate coordinates. 

Hence the heuristic equivalence
\[
\text{Integrability of weights under \eqref{eq:map_iterating_down:intro} and \eqref{eq:map_iterating_up:intro} }\quad \quad \Leftrightarrow \quad \quad \text{Probabilistic integrability}
\]
has appeared to hold so far among known models. So, if one desires probabilistic integrability for a dimer model with random weights, it is natural to look for a version of integrability under \eqref{eq:map_iterating_up:intro}. Perhaps the simplest is that independence be preserved; one may hope in addition that the distribution of weights changes in a traceable fashion. The following thus explains our choice of model:

\begin{thm}
    \label{thm:shuffling_char_intro}
    Let 
    \begin{equation}
        \{a_{i,j}^{[0]},b_{i,j}^{[0]}: i,j \in \Z\}
    \end{equation}
    be a collection of mutually independent, positive, nonconstant random variables, such that for each $n \in \mathbb Z$, the collections of random variables \begin{equation}
        \{ a_{i,j}^
        {[n]}, b_{i,j}^{[n]}: i,j \in \Z\}
    \end{equation}
    defined by \eqref{eq:map_iterating_down:intro} and \eqref{eq:map_iterating_up:intro} are mutually independent\footnote{Note that we do not require independence of variables with different upper index $n$. Correlations between different levels $n$ are unavoidable in view of \eqref{eq:map_iterating_up:intro}.}. Then there exist sequences of real parameters $(\psi_j)_{j \in \Z},(\phi_j)_{j \in \Z}, (s_i)_{i \in \Z}, (\theta_i)_{i \in \Z}$ such that for every $i,j,n \in \Z$, the distribution of $ a_{i,j}^{[n]}$ and $b_{i,j}^{[n]}$ for each are given by  
    \begin{align}
    \begin{split} \label{eq:n_level_parameters_intro}
            a_{i,j}^{[n]} &\sim \Gamma(\psi_j+\theta_i, s_{i-n}) \\ 
            b_{i,j}^{[n]} &\sim \Gamma(\phi_{j-n}-\theta_i, s_{i-n}).
        \end{split}
    \end{align}
\end{thm}

\Cref{thm:shuffling_char_intro} is a slightly weaker version of \Cref{thm:3-layer_implies_all-layer} proven in-text.

\begin{rmk}\label{rmk:scale_doesn't_matter}
    The scale parameters $s_i$ are just constants in front of the weights. It is a combinatorial fact that for each $i$, every matching includes exactly $n-i+1$ edges with $a_{i,j}^{[n]}$ or $b_{i,j}^{[n]}$ weights, so changing the constant multiple $s_i$ changes the partition function but does not change the dimer measure. In our results on the dimer measure in the body of the paper, with the exception of \Cref{appendix:deterministic}, we will set the scale parameter to $1$ for convenience. 
\end{rmk}

Hence it is natural to define the following weights on the finite Aztec diamond.

\begin{defi}
    \label{def:gamma_weights_intro_general}
    Fix $n \geq 1$ and let $\{\psi_j\}_{j=1}^n,\{\phi_{j-n}\}_{j=1}^n,\{\theta_i\}_{i=1}^n $ be real parameters such that $\psi_j+\theta_i>0$ and $\phi_{j-n}-\theta_i>0$ for $i,j=1,\ldots,n$. The \emph{Gamma-disordered Aztec diamond} of size $n$ (with general parameters) is the dimer measure with independent random weights $\{a_{i,j}^{[n]},b_{i,j}^{[n]}: 1 \leq i,j \leq n\}$, placed on the Aztec diamond graph $G_n^{\Az}$ of size $n$ as in \Cref{fig:our_weights_intro}, distributed by
    \begin{align}
        a_{i,j}^{[n]} &\sim \Gamma(\psi_j +\theta_i, 1) \\ 
        b_{i,j}^{[n]} &\sim \Gamma(\phi_{j-n}-\theta_i, 1).
    \end{align}
    Here $\Gamma(\chi,s)$ is the Gamma distribution with shape parameter $\chi > 0$ and scale parameter $s > 0$, see \Cref{def:gamma_var}.
\end{defi}

The special case
$$\psi_j = \alpha, \phi_j = \beta, \theta_i = 0$$
yields the weights of \Cref{def:gamma_weights_intro}. All of our exact results on hybrid polymers hold in the generality of the weights in \Cref{def:gamma_weights_intro_general}, though the probabilistic results we use have only been established in this special case because the polymer results which we need as input are only proven in this case.

It is not obvious \emph{a priori} that any family of weights should remain independent under the shuffling update. What is perhaps even more surprising is that the family of weights in \Cref{thm:shuffling_char_intro}, arrived at through this elementary probabilistic consideration, recovers many known models of algebraic origin. In the match with stationary polymers, these general weights yield deformed versions with more parameters. In the case of the Beta random walk in random environment, which is relevant to the East turning point, this yields the deformed version of Korotkikh \cite{korotkikh2022hidden}, which was arrived at by degenerating a stochastic colored vertex model. The deterministic limit is also worth mentioning. Replacing $\psi_j,\phi_j,\theta_i,s_i$ by $\psi_j T, \phi_j T, \theta_i T, 1/T$ in the weights of \Cref{thm:shuffling_char_intro} and taking the $T \to \infty$ limit yields deterministic weights 
    \begin{align}
    \begin{split} \label{eq:deterministic_n_level_parameters_intro}
            a_{i,j}^{[n]} &=\psi_j+\theta_i \\ 
            b_{i,j}^{[n]} &=\phi_{j-n}-\theta_i.
        \end{split}
    \end{align}
When $\theta_i =0$ for all $i$, these are exactly the Schur process weights mentioned above. For general $\theta_i$, they are a subfamily of the genus $0$ Fock's weights \cite{fock2015inverse}, which in general have four (not three) families of parameters; we make this match explicit in \Cref{appendix:deterministic}. 
The genus $0$ Fock's weights  appeared in the context of critical weights for isoradial graphs \cite{Kenyon2002Laplacian}, see also the more recent discussion of \cite{BoutillierDeTiliere2024} and the references therein. Surprisingly, probabilistic considerations on random weights lead naturally to this rich known family of deterministic ones. 

There is also a nontrivial limit as $T \to 0$, after rescaling the weights, though we do not treat it in this work. The limiting weights become negative exponential variables (see \eqref{eq:loggamma_to_exp}) and the dimer measure concentrates on a single ground state matching with highest weight. The update rules \eqref{eq:map_iterating_down:intro} and \eqref{eq:map_iterating_up:intro} are replaced by nontrivial tropical limits, and the shuffling algorithm becomes a deterministic (conditional on the weights) algorithm to produce the ground state. The corresponding polymer models become certain multi-path last-passage percolation models with exponential weights and different geometry than the usual such models. This seems an interesting model to explore, and similar zero-temperature models were studied numerically by Rieger-Blasum \cite{rieger1997ground} and Zeng-Middleton-Shapir \cite{zeng1996ground} in physics prior to \cite{zeng1999thermodynamics}.

Finally, it is worth mentioning the apparent similarity of the birational maps \eqref{eq:map_iterating_down:intro} and \eqref{eq:map_iterating_up:intro} with those appearing in the geometric Robinson-Schensted-Knuth (RSK) correspondence of Noumi-Yamada \cite{noumi2004tropical}, which underlies the integrability of the log-Gamma and strict-weak polymers (see Corwin-O’Connell-Sepp{\"a}l{\"a}inen-Zygouras \cite{corwin2014tropical}). In particular, \cite{noumi2004tropical} features a matrix commutation perspective on geometric RSK which appears very similar to the matrix commutation perspective on our maps which we discuss in \Cref{appendix:lgv}, and \cite[Figure 1]{noumi2004tropical} features nonintersecting paths on a mixed graph similar to the one in \Cref{fig:mixed_polymer_intro}. It is an interesting problem to understand these connections, and we hope they will be clarified in the future.

\subsection{Outline of proofs and rest of the paper} 

On the dimer side, the shuffling algorithm is powered by a certain local transformation graphs, known variously as the spider move, square move, or urban renewal. The weight updates in \eqref{eq:map_iterating_down:intro} and \eqref{eq:map_iterating_up:intro} come from this local move and gauge transformations. This, together with needed properties of Gamma random variables, is explained in \Cref{sec:prelim}.

The property of Gamma variables behind their appearance in \Cref{thm:shuffling_char_intro} is Lukacs' theorem \cite{lukacs1955characterization} (\Cref{thm:lukacs} below) that if $X$ and $Y$ are independent nonconstant positive random variables, then $X+Y$ and $X/(X+Y)$ are independent if and only if $X$ and $Y$ are Gamma-distributed random variables with equal scale parameters. In \Cref{sec:full_plane_gamma} we use this to prove \Cref{thm:shuffling_char_intro}, in fact proving a slightly stronger form (\Cref{thm:3-layer_implies_all-layer}) that the mutual independence of the collections with $n=0,-1,1$ only is enough to show that the weights have distributions as in \eqref{eq:n_level_parameters_intro}. This in particular implies mutual independence of the collection for every $n \in \mathbb Z$ rather than just $0,\pm 1$. It is however necessary to consider $3$ steps of the shuffling algorithm rather than $2$, see \Cref{rmk:2_steps_insufficient}.

The Aztec diamond oriented as in \Cref{fig:our_weights_intro} can be viewed as consisting of two types of columns alternating, with $n$ columns of each type, see \Cref{thm:tG}. Applying the spider move along an adjacent pair of such columns swaps the two and updates the weights; this column swap is reminiscent of zipper arguments using the local Yang-Baxter relation, and we believe it should be a degeneration of such for an appropriate vertex model. As we explain in \Cref{sec:dynamical_vert}, swapping all columns simultaneously corresponds to the update of the shuffling algorithm. However, if one only swaps some of the columns, one obtains different graphs. Each local swap does not affect the marginal distribution of matchings outside the columns swapped (this is the important property of the local square move), so these naturally give distributional equalities between the marginals along columns of related graphs. After doing all possible swaps on either side of a fixed `observation' column, one obtains a graph for which matchings are naturally in bijection with paths on the graph of \Cref{def:beta_sw_intro}. In \Cref{sec:vert_polymer} this is explained and the general version of \Cref{thm:multi-path_intro_vert} is proven. This is all upgraded to distributional equalities across multiple steps of the shuffling algorithm (which we have not stated above) in \Cref{sec:dynamical_vert}. The story for rows rather than columns is structurally the same, and is given in \Cref{sec:hor_slice}.

These arguments first swap columns and then use bijections between dimers and non-intersecting paths. However, one may also do this in the other order, first relating to non-intersecting paths moving horizontally through a graph with columns of two types, and then justifying that these columns can be swapped. This perspective has been used in probability since Johansson's seminal work \cite{johansson2002non}, and a version of the above column-swapping perspective on shuffling in terms of Wiener-Hopf factorizations of transition matrices was given by Chhita and the first author in \cite{chhita2023domino}. For completeness, we show how to prove the results of \Cref{sec:vert_polymer} in this way in \Cref{appendix:lgv}.

The hybrid polymers of \Cref{thm:multi-path_intro_vert} (and the horizontal version given in \Cref{thm:hor_slice_polymer}) are not obviously related to stationary polymers, even in the turning point case. In \Cref{sec:turning_points} we give background on the stationary log-Gamma polymer and Beta random walk in random environment, and prove \Cref{thm:west_matching_intro} and the corresponding version for the East turning point. The corresponding result for the South turning point and stationary strict-weak polymer is given later in \Cref{thm:south_endpoint}, and holds analogously for the North turning point by symmetry. These crucially use the Burke property of these polymer models, an independence property under local moves which also comes from Lukacs' theorem.

The last two sections prove our probabilistic results. In \Cref{sec:fluctuations_of_turning_points} we prove \Cref{thm:turning_points_intro} and the analogous Gaussian fluctuation result for the East turning point, \Cref{thm:beta_rwre_limit_dynamical}. In \Cref{sec:free_energy} we prove Theorems~\ref{thm:free_energy_difference_intro} and \ref{thm:free_energy_fluctuations_intro} and the analogous result for the general weights of \Cref{thm:shuffling_char_intro}. These proofs are straightforward computations, in view of the result \Cref{thm:compute_Z_cor} that the partition function is a product of independent Gamma variables.

\Cref{appendix:lgv} was already described above, and \Cref{appendix:deterministic} shows that the deterministic limit of our weights yields a large subfamily of genus $0$ Fock weights.

\begin{rmk}\label{rmk:multi-path_future}
    In this paper, we only match our hybrid polymers with stationary polymers in the single-path case. To our knowledge, this is the only case where probabilistic results on polymer path fluctuations are currently proven, so we did not choose to extend beyond it in this paper because it would gain no new asymptotic results for the Aztec diamond. However, it is natural to ask whether the exact results matching hybrid polymers with stationary polymers extend to the multi-path case. The answer is yes, but requires some additional arguments and will be presented in a subsequent paper. In particular, these results will give discrete versions of the recent result of Barraquand-Le Doussal \cite{barraquand2023stationary} producing stationary measures of the $k$-path continuum directed polymer using the semi-discrete Brownian polymer. It was surprising to us that considering random weights on the Aztec diamond led naturally to new results featuring only polymers. We also hope that these matches provide some impetus to the directed polymer community to extend existing single-path fluctuation results to multi-path polymers, as these now would translate to natural probabilistic results on matchings of the Gamma-disordered Aztec diamond.
\end{rmk}

\textbf{Acknowledgments.} We thank Amol Aggarwal, Theodoros Assiotis, Tomas Berggren, Alexei Borodin, C\'edric Boutillier, Sunil Chhita, Ivan Corwin, Milind Hegde, Kurt Johansson, Timo Sepp{\"a}l{\"a}inen, Xiao Shen, and Evan Sorensen for various helpful conversations and comments. RVP wishes to particularly thank Guillaume Barraquand for the suggestion to gauge-transform the weights by $a_{i,j}^{[n]}+b_{i,j}^{[n]}$ and look for relations to the Beta polymer, which provided the key to discovering the matching with hybrid polymers. We are deeply grateful to Pierre Le Doussal and Gregory Schehr for helping us understand their work and other physics literature on dimers with random weights through many conversations, and for comments on the text. Finally, we thank Sunil Chhita for providing his code for implementing the shuffling algorithm, and Leonid Petrov for adding our weights to his applet collection, where matchings and double-dimer configurations as in \Cref{fig:simulations} from the Gamma-disordered Aztec diamond may be sampled in-browser at \href{https://lpetrov.cc/double-dimer-gamma/}{https://lpetrov.cc/double-dimer-gamma/}. The authors were supported by the European Research Council (ERC), Grant Agreement No.101002013. Part of this work was completed while MD held a Chaire d’Excellence from the Fondation Sciences Mathématiques de Paris (FSMP). MD thanks the LPSM at Sorbonne University for its hospitality during this period.

\section{Preliminaries on dimers, urban renewal, shuffling, and Lukacs' theorem} \label{sec:prelim}

\subsection{Dimer generalities} We begin with some standard notation and lemmas.

\begin{defi}\label{def:dimer_meas}
    Let $G = (V,E)$ be a weighted finite graph with nonnegative edge weights, which we write as $\nu(e) \in \R_{\geq 0}$ for each $e \in E$. We define its \emph{dimer partition function}
    \begin{equation}
        Z_G := \sum_{M} \prod_{e \in M} \nu(e)
    \end{equation}
    where the sum is over all perfect matchings $M$ of $G$. If $Z_G \neq 0$, we define the \emph{dimer measure} as the probability measure on perfect matchings $M$ with
    \begin{equation}\label{eq:dimer_measure}
        \mathbb{P}_G(M) = \frac{1}{Z_G}  \prod_{e \in M} \nu(e).
    \end{equation}
    We similarly write $\E_G$ for the expectation value with respect to this probability measure; we will sometimes drop the $G$ subscript when it is clear from context.
\end{defi}

\begin{rmk}
    For any vertex $v \in V$, multiplying the weights of all edges incident to $v$ by the same constant $c$ does not change the dimer measure, since any perfect matching contains exactly one such edge. Such a local change in edge weights is called a \emph{gauge transform}.
\end{rmk}

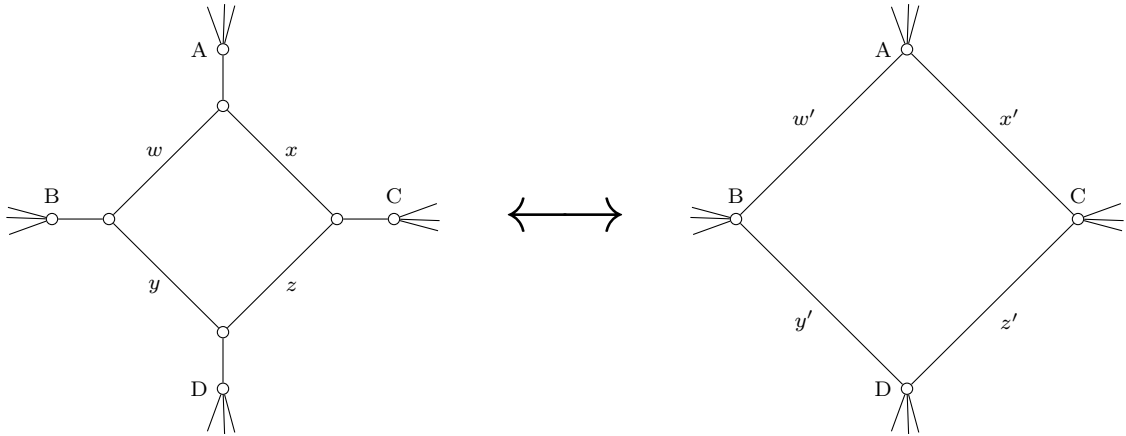
\begin{figure}
    \centering

    \begin{tikzpicture}[scale=1.5, every node/.style={draw, circle, inner sep=1.5pt, fill=white}, font=\scriptsize]

\node[draw, circle, label=left:A] (A) at (0, 1.5) {};
\node[draw, circle, label=above:B] (B) at (-1.5, 0) {};
\node[draw, circle, label=above:C] (C) at (1.5, 0) {};
\node[draw, circle, label=left:D] (D) at (0, -1.5) {};

\node (AA) at (0, 1) {};
\node (BB) at (-1, 0) {};
\node (CC) at (1, 0) {};
\node (DD) at (0, -1) {};

\draw (A) -- (AA);
\draw (B) -- (BB);
\draw (C) -- (CC);
\draw (D) -- (DD);

\draw (BB) -- (AA);
\draw (DD) -- (BB);
\draw (DD) -- (CC);
\draw (AA) -- (CC);

\draw (A) -- ++(75:.4);
\draw (A) -- ++(88:.4);
\draw (A) -- ++(110:.4);

\draw  (D) -- ++(-75:.4);
\draw (D) -- ++(-88:.4);
\draw (D) -- ++(-110:.4);

\draw (B) -- ++(75+90:.4);
\draw (B) -- ++(88+90:.4);
\draw (B) -- ++(110+90:.4);

\draw (C) -- ++(75-90:.4);
\draw (C) -- ++(88-90:.4);
\draw (C) -- ++(110-90:.4);

\node[draw=none] at (-0.6, 0.6) {$w$};
\node[draw=none] at (.6,.6) {$x$};
\node[draw=none] at (-0.6, -0.6) {$y$};
\node[draw=none] at (0.6, -0.6) {$z$};


\begin{scope}[shift={(6,0)}]

\node[draw, circle, label=left:A] (A) at (0, 1.5) {};
\node[draw, circle, label=above:B] (B) at (-1.5, 0) {};
\node[draw, circle, label=above:C] (C) at (1.5, 0) {};
\node[draw, circle, label=left:D] (D) at (0, -1.5) {};

\draw (B) -- (A);
\draw (D) -- (B);
\draw (D) -- (C);
\draw (A) -- (C);

\draw (A) -- ++(75:.4);
\draw(A) -- ++(88:.4);
\draw (A) -- ++(110:.4);

\draw (D) -- ++(-75:.4);
\draw (D) -- ++(-88:.4);
\draw (D) -- ++(-110:.4);

\draw (B) -- ++(75+90:.4);
\draw (B) -- ++(88+90:.4);
\draw(B) -- ++(110+90:.4);

\draw (C) -- ++(75-90:.4);
\draw (C) -- ++(88-90:.4);
\draw (C) -- ++(110-90:.4);

\node[draw=none] at (-0.9, 0.9) {$w'$};
\node[draw=none] at (.9,.9) {$x'$};
\node[draw=none] at (-0.9, -0.9) {$y'$};
\node[draw=none] at (0.9, -0.9) {$z'$};
\end{scope}

	\node[draw=none,rotate=0] at (3,0) {\Huge $\longleftrightarrow$};


\end{tikzpicture}

    \caption{The spider move.}
    \label{fig:spider_notiles}
\end{figure}

\begin{prop}[spider move/square move/urban renewal]\label{thm:spider_Z}
    Let $G$ be a weighted graph containing a local pattern as in \Cref{fig:spider_notiles} (left), where the partial edges going out from vertices $A,B,C,D$ may occur in arbitrary number and with arbitrary weights. Let $G'$ be the same graph where the local pattern is replaced by the one in \Cref{fig:spider_notiles} (right), with updated weights given by
    \begin{align}
        \begin{split}
            w'&=\frac{z}{wz+xy} \\
            x'&=\frac{y}{wz+xy} \\
            y'&=\frac{x}{wz+xy} \\
            z'&=\frac{w}{wz+xy} \\
        \end{split}
    \end{align}
    (all full edges without weight pictured have weight $1$, and all partial edges pictured have the same weights in both figures). Then the partition functions of $G$ and $G'$ are related by
        \begin{equation}\label{eq:spider_partition}
            Z_G = (wz+xy)Z_{G'}.
        \end{equation}
\end{prop}
\begin{proof}
    See \cite[Section 5]{propp2003generalized}.
\end{proof}

The following is also standard. Since it gives the coupling used in the shuffling algorithm, we will walk through the details for the sake of exposition.

\begin{figure}
    \centering
    \begin{tikzpicture}[scale=1.5, every node/.style={draw, circle, inner sep=1.5pt, fill=white}, font=\scriptsize]

\node[draw, circle, label=left:A] (A) at (0, 1.5) {};
\node[draw, circle, label=above:B] (B) at (-1.5, 0) {};
\node[draw, circle, label=above:C] (C) at (1.5, 0) {};
\node[draw, circle, label=left:D] (D) at (0, -1.5) {};

\node (AA) at (0, 1) {};
\node (BB) at (-1, 0) {};
\node (CC) at (1, 0) {};
\node (DD) at (0, -1) {};

\draw (A) -- (AA);
\draw (B) -- (BB);
\draw (C) -- (CC);
\draw (D) -- (DD);

\draw[very thick,red] (BB) -- (AA);
\draw (DD) -- (BB);
\draw[very thick,red] (DD) -- (CC);
\draw (AA) -- (CC);

\draw (A) -- ++(75:.4);
\draw [very thick,blue] (A) -- ++(88:.4);
\draw (A) -- ++(110:.4);

\draw[very thick,blue] (D) -- ++(-75:.4);
\draw (D) -- ++(-88:.4);
\draw (D) -- ++(-110:.4);

\draw (B) -- ++(75+90:.4);
\draw (B) -- ++(88+90:.4);
\draw[very thick,blue] (B) -- ++(110+90:.4);

\draw (C) -- ++(75-90:.4);
\draw (C) -- ++(88-90:.4);
\draw[very thick,blue] (C) -- ++(110-90:.4);

\node[draw=none] at (-0.6, 0.6) {$w$};
\node[draw=none] at (.6,.6) {$x$};
\node[draw=none] at (-0.6, -0.6) {$y$};
\node[draw=none] at (0.6, -0.6) {$z$};


\begin{scope}[shift={(5,-2.5)}]

\node[draw, circle, label=left:A] (A) at (0, 1.5) {};
\node[draw, circle, label=above:B] (B) at (-1.5, 0) {};
\node[draw, circle, label=above:C] (C) at (1.5, 0) {};
\node[draw, circle, label=left:D] (D) at (0, -1.5) {};

\draw (B) -- (A);
\draw (D) -- (B);
\draw (D) -- (C);
\draw (A) -- (C);

\draw (A) -- ++(75:.4);
\draw[very thick,blue] (A) -- ++(88:.4);
\draw (A) -- ++(110:.4);

\draw[very thick,blue] (D) -- ++(-75:.4);
\draw (D) -- ++(-88:.4);
\draw (D) -- ++(-110:.4);

\draw (B) -- ++(75+90:.4);
\draw (B) -- ++(88+90:.4);
\draw[very thick,blue] (B) -- ++(110+90:.4);

\draw (C) -- ++(75-90:.4);
\draw (C) -- ++(88-90:.4);
\draw[very thick,blue] (C) -- ++(110-90:.4);

\node[draw=none] at (-0.9, 0.9) {$w'$};
\node[draw=none] at (.9,.9) {$x'$};
\node[draw=none] at (-0.9, -0.9) {$y'$};
\node[draw=none] at (0.9, -0.9) {$z'$};
\end{scope}

\begin{scope}[shift={(0,-5)}]
	\node[draw, circle, label=left:A] (A) at (0, 1.5) {};
\node[draw, circle, label=above:B] (B) at (-1.5, 0) {};
\node[draw, circle, label=above:C] (C) at (1.5, 0) {};
\node[draw, circle, label=left:D] (D) at (0, -1.5) {};

\node (AA) at (0, 1) {};
\node (BB) at (-1, 0) {};
\node (CC) at (1, 0) {};
\node (DD) at (0, -1) {};

\draw (A) -- (AA);
\draw (B) -- (BB);
\draw (C) -- (CC);
\draw (D) -- (DD);

\draw (BB) -- (AA);
\draw[very thick,red] (DD) -- (BB);
\draw (DD) -- (CC);
\draw[very thick,red] (AA) -- (CC);

\draw (A) -- ++(75:.4);
\draw [very thick,blue] (A) -- ++(88:.4);
\draw (A) -- ++(110:.4);

\draw[very thick,blue] (D) -- ++(-75:.4);
\draw (D) -- ++(-88:.4);
\draw (D) -- ++(-110:.4);

\draw (B) -- ++(75+90:.4);
\draw (B) -- ++(88+90:.4);
\draw[very thick,blue] (B) -- ++(110+90:.4);

\draw (C) -- ++(75-90:.4);
\draw (C) -- ++(88-90:.4);
\draw[very thick,blue] (C) -- ++(110-90:.4);

\node[draw=none] at (-0.6, 0.6) {$w$};
\node[draw=none] at (.6,.6) {$x$};
\node[draw=none] at (-0.6, -0.6) {$y$};
\node[draw=none] at (0.6, -0.6) {$z$};
\end{scope}

	\node[draw=none,rotate=-25] at (2,-1.5) {\Huge $\longrightarrow$};

	\node[draw=none,rotate=25] at (2,-3.5) {\Huge $\longrightarrow$};

\end{tikzpicture}

    \caption{Case 1 of the mapping $\phi$ in \Cref{thm:spider_coupling}. The two possible dimer covers $M \in \mathcal{M}(G)$ of the graph $G$ contained in Case 1 are on the left, and their (deterministic) image $\phi(M) \in \mathcal{M}(G')$ is on the right.}
    \label{fig:spider_case_1}
\end{figure}
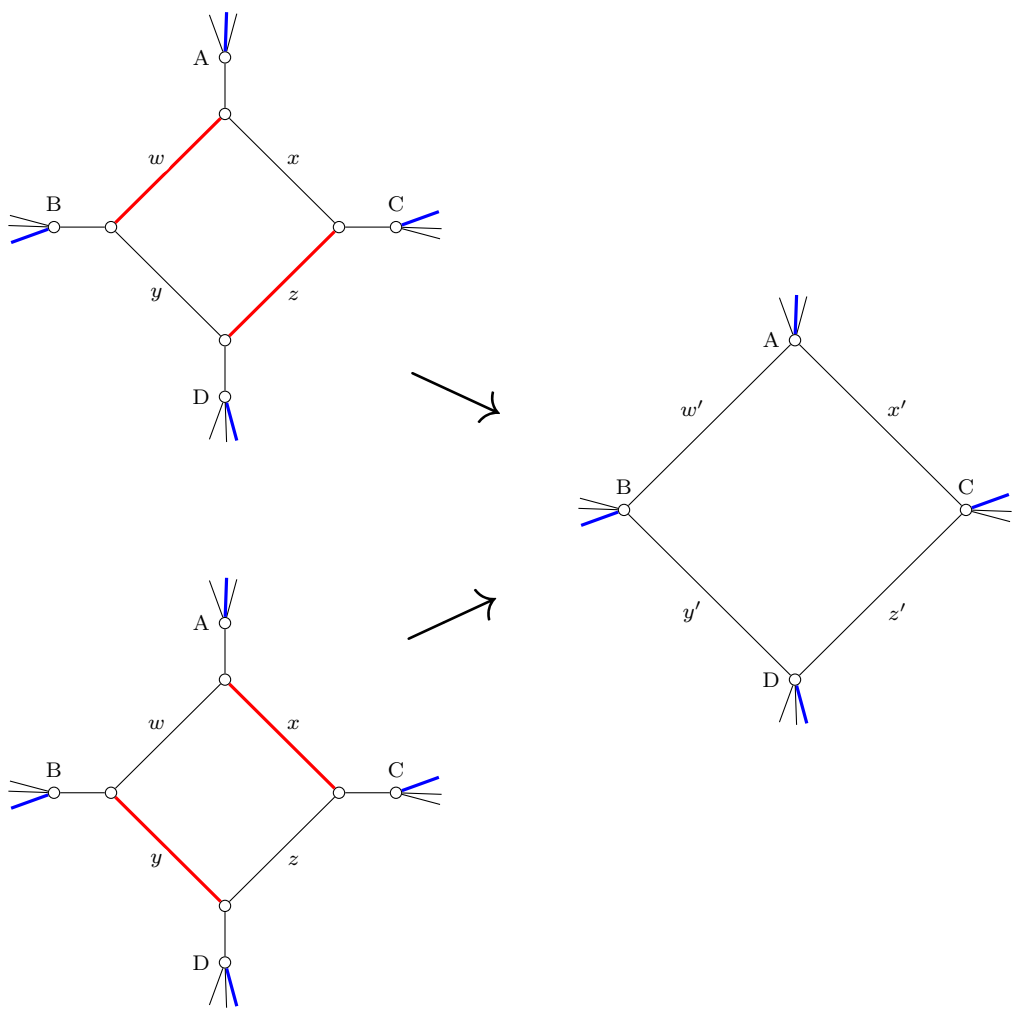

\begin{figure}
    \centering

    \begin{tikzpicture}[scale=1.5, every node/.style={draw, circle, inner sep=1.5pt, fill=white}, font=\scriptsize]

\node[draw, circle, label=left:A] (A) at (0, 1.5) {};
\node[draw, circle, label=above:B] (B) at (-1.5, 0) {};
\node[draw, circle, label=above:C] (C) at (1.5, 0) {};
\node[draw, circle, label=left:D] (D) at (0, -1.5) {};

\node (AA) at (0, 1) {};
\node (BB) at (-1, 0) {};
\node (CC) at (1, 0) {};
\node (DD) at (0, -1) {};

\draw (A) -- (AA);
\draw (B) -- (BB);
\draw[very thick,red] (C) -- (CC);
\draw[very thick,red] (D) -- (DD);

\draw[very thick,red] (BB) -- (AA);
\draw (DD) -- (BB);
\draw (DD) -- (CC);
\draw (AA) -- (CC);

\draw (A) -- ++(75:.4);
\draw [very thick,blue] (A) -- ++(88:.4);
\draw (A) -- ++(110:.4);

\draw  (D) -- ++(-75:.4);
\draw (D) -- ++(-88:.4);
\draw (D) -- ++(-110:.4);

\draw (B) -- ++(75+90:.4);
\draw (B) -- ++(88+90:.4);
\draw[very thick,blue] (B) -- ++(110+90:.4);

\draw (C) -- ++(75-90:.4);
\draw (C) -- ++(88-90:.4);
\draw (C) -- ++(110-90:.4);

\node[draw=none] at (-0.6, 0.6) {$w$};
\node[draw=none] at (.6,.6) {$x$};
\node[draw=none] at (-0.6, -0.6) {$y$};
\node[draw=none] at (0.6, -0.6) {$z$};


\begin{scope}[shift={(6,0)}]

\node[draw, circle, label=left:A] (A) at (0, 1.5) {};
\node[draw, circle, label=above:B] (B) at (-1.5, 0) {};
\node[draw, circle, label=above:C] (C) at (1.5, 0) {};
\node[draw, circle, label=left:D] (D) at (0, -1.5) {};

\draw (B) -- (A);
\draw (D) -- (B);
\draw[very thick,red] (D) -- (C);
\draw (A) -- (C);

\draw (A) -- ++(75:.4);
\draw[very thick,blue] (A) -- ++(88:.4);
\draw (A) -- ++(110:.4);

\draw (D) -- ++(-75:.4);
\draw (D) -- ++(-88:.4);
\draw (D) -- ++(-110:.4);

\draw (B) -- ++(75+90:.4);
\draw (B) -- ++(88+90:.4);
\draw[very thick,blue] (B) -- ++(110+90:.4);

\draw (C) -- ++(75-90:.4);
\draw (C) -- ++(88-90:.4);
\draw (C) -- ++(110-90:.4);

\node[draw=none] at (-0.9, 0.9) {$w'$};
\node[draw=none] at (.9,.9) {$x'$};
\node[draw=none] at (-0.9, -0.9) {$y'$};
\node[draw=none] at (0.9, -0.9) {$z'$};
\end{scope}

	\node[draw=none,rotate=0] at (3,0) {\Huge $\longrightarrow$};


\end{tikzpicture}

    \caption{Case 2 of the mapping $\phi$ in \Cref{thm:spider_coupling}. We show one of the dimer covers in Case 1 on the left, and the other three are related to it by rotation. The (deterministic) image $\phi(M) \in \mathcal{M}(G')$ is on the right, and the other three are rotations of it.}
    \label{fig:spider_case_2}
\end{figure}
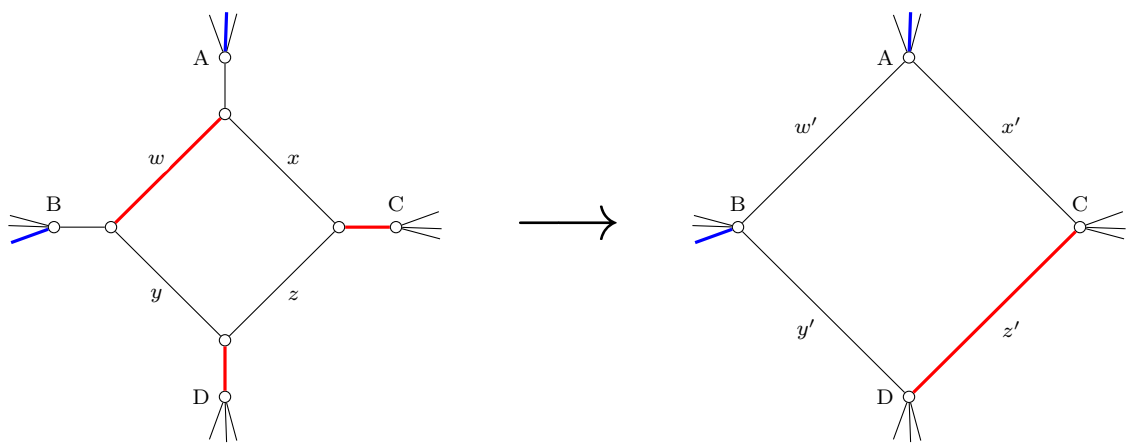

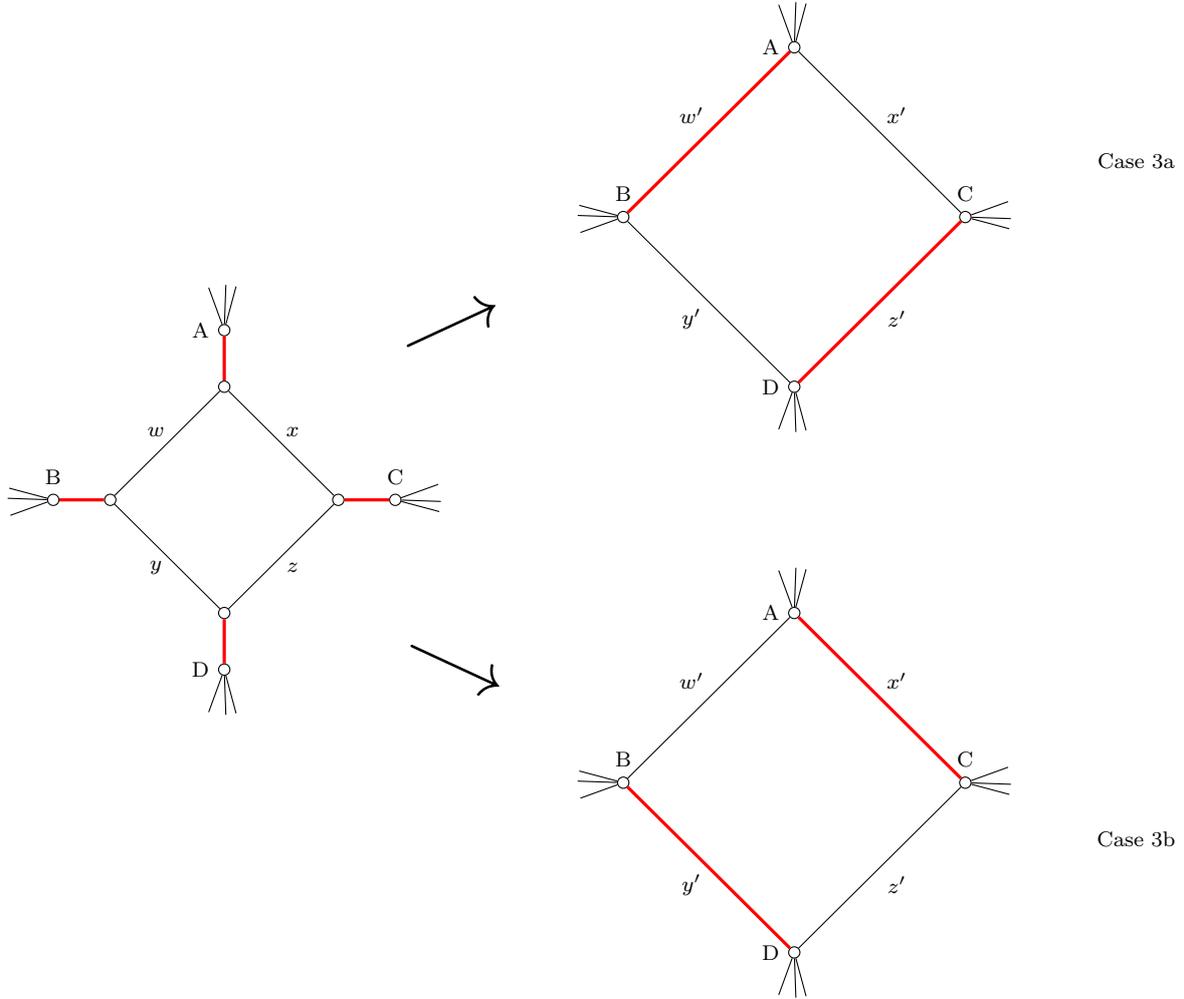
\begin{figure}
    \centering

  \begin{tikzpicture}[scale=1.5, every node/.style={draw, circle, inner sep=1.5pt, fill=white}, font=\scriptsize]

\node[draw, circle, label=left:A] (A) at (0, 1.5) {};
\node[draw, circle, label=above:B] (B) at (-1.5, 0) {};
\node[draw, circle, label=above:C] (C) at (1.5, 0) {};
\node[draw, circle, label=left:D] (D) at (0, -1.5) {};

\node (AA) at (0, 1) {};
\node (BB) at (-1, 0) {};
\node (CC) at (1, 0) {};
\node (DD) at (0, -1) {};

\draw[very thick,red] (A) -- (AA);
\draw[very thick,red] (B) -- (BB);
\draw[very thick,red] (C) -- (CC);
\draw[very thick,red] (D) -- (DD);

\draw  (BB) -- (AA);
\draw (DD) -- (BB);
\draw (DD) -- (CC);
\draw (AA) -- (CC);

\draw (A) -- ++(75:.4);
\draw  (A) -- ++(88:.4);
\draw (A) -- ++(110:.4);

\draw (D) -- ++(-75:.4);
\draw (D) -- ++(-88:.4);
\draw (D) -- ++(-110:.4);

\draw (B) -- ++(75+90:.4);
\draw (B) -- ++(88+90:.4);
\draw (B) -- ++(110+90:.4);

\draw (C) -- ++(75-90:.4);
\draw (C) -- ++(88-90:.4);
\draw(C) -- ++(110-90:.4);

\node[draw=none] at (-0.6, 0.6) {$w$};
\node[draw=none] at (.6,.6) {$x$};
\node[draw=none] at (-0.6, -0.6) {$y$};
\node[draw=none] at (0.6, -0.6) {$z$};


\begin{scope}[shift={(5,+2.5)}]

\node[draw, circle, label=left:A] (A) at (0, 1.5) {};
\node[draw, circle, label=above:B] (B) at (-1.5, 0) {};
\node[draw, circle, label=above:C] (C) at (1.5, 0) {};
\node[draw, circle, label=left:D] (D) at (0, -1.5) {};

\draw[red,very thick] (B) -- (A);
\draw (D) -- (B);
\draw[red,very thick] (D) -- (C);
\draw (A) -- (C);

\draw (A) -- ++(75:.4);
\draw (A) -- ++(88:.4);
\draw (A) -- ++(110:.4);

\draw (D) -- ++(-75:.4);
\draw (D) -- ++(-88:.4);
\draw (D) -- ++(-110:.4);

\draw (B) -- ++(75+90:.4);
\draw (B) -- ++(88+90:.4);
\draw(B) -- ++(110+90:.4);

\draw (C) -- ++(75-90:.4);
\draw (C) -- ++(88-90:.4);
\draw (C) -- ++(110-90:.4);

\node[draw=none] at (-0.9, 0.9) {$w'$};
\node[draw=none] at (.9,.9) {$x'$};
\node[draw=none] at (-0.9, -0.9) {$y'$};
\node[draw=none] at (0.9, -0.9) {$z'$};
\end{scope}

\begin{scope}[shift={(5,-2.5)}]
	
\node[draw, circle, label=left:A] (A) at (0, 1.5) {};
\node[draw, circle, label=above:B] (B) at (-1.5, 0) {};
\node[draw, circle, label=above:C] (C) at (1.5, 0) {};
\node[draw, circle, label=left:D] (D) at (0, -1.5) {};

\draw (B) -- (A);
\draw[red,very thick] (D) -- (B);
\draw (D) -- (C);
\draw[red,very thick] (A) -- (C);

\draw (A) -- ++(75:.4);
\draw (A) -- ++(88:.4);
\draw (A) -- ++(110:.4);

\draw (D) -- ++(-75:.4);
\draw (D) -- ++(-88:.4);
\draw (D) -- ++(-110:.4);

\draw (B) -- ++(75+90:.4);
\draw (B) -- ++(88+90:.4);
\draw(B) -- ++(110+90:.4);

\draw (C) -- ++(75-90:.4);
\draw (C) -- ++(88-90:.4);
\draw (C) -- ++(110-90:.4);

\node[draw=none] at (-0.9, 0.9) {$w'$};
\node[draw=none] at (.9,.9) {$x'$};
\node[draw=none] at (-0.9, -0.9) {$y'$};
\node[draw=none] at (0.9, -0.9) {$z'$};
\end{scope}

	\node[draw=none,rotate=25] at (2,1.5) {\Huge $\longrightarrow$};

	\node[draw=none,rotate=-25] at (2,-1.5) {\Huge $\longrightarrow$};

  \node[draw=none] at (8,3) { Case 3a};
  \node[draw=none] at (8,-3) { Case 3b};
\end{tikzpicture}
    
    \caption{Case 3 of the mapping $\phi$ in \Cref{thm:spider_coupling}. Here the image $\phi(M)$ is random, and equal to either of the dimer covers on the right (3a) and (3b) with probabilities $\frac{wz}{wz+xy} = \frac{w'z'}{w'z'+x'y'}$ and $\frac{xy}{wz+xy} = \frac{x'y'}{w'z'+x'y'}$ respectively.}
    \label{fig:spider_case_3}
\end{figure}

\begin{prop}\label{thm:spider_coupling}
    Let $G,G'$ be as in \Cref{thm:spider_Z}. Given any perfect matching $M$ on $G$, let $\phi(M)$ be the random perfect matching of $G'$ such that all edges outside the above-mentioned local patterns are the same, and the edges in the local patterns are as described in Figures \ref{fig:spider_case_1}, \ref{fig:spider_case_2} and \ref{fig:spider_case_3}. Then $(M,\phi(M))$ defines a coupling of the dimer measures on $G$ and $G'$, and furthermore it is the unique coupling for which all edges of $G$ and $G'$ not featured in \Cref{fig:spider_notiles} are the same.
\end{prop}

\begin{proof}
    To show that $(M,\phi(M))$ defines a coupling of dimer measures, since $M$ is by definition distributed according to the dimer measure on $G$, we need only check that $\phi(M)$ is distributed by the dimer measure on $G'$. In other words we must show
    \begin{equation}\label{eq:coupling_wts}
        \sum_M \mathbb{P}_G(M) \mathbb{P}(\phi(M)=M') = \mathbb{P}_{G'}(M'),
    \end{equation}
    where the sum is over all perfect matchings of $G$.

    Let us write the edge set $E$ of $G$ as $E = E_1 \sqcup E_2$, where $E_2$ is the set of edges fully pictured in \Cref{fig:spider_notiles} (left) and $E_1$ is the complementary set (including the edges partially pictured in \Cref{fig:spider_notiles} (left)). Then the edges set $E'$ of $G'$ may be written as $E' = E_1 \sqcup E_2'$ where $E_2'$ is the set of edges fully pictured in \Cref{fig:spider_notiles} (right) and we identify the rest of the edges with $E_1$ since the graphs $G,G'$ are the same outside of the subsets pictured.

    Consider a matching $M'$ of $G'$. Let 
    \begin{equation}
        \Pi = \prod_{e \in E_1 \cap M'} \nu(e),
    \end{equation}
    so that
    \begin{equation}\label{eq:relate_G_G'}
        \mathbb{P}_{G'}(M') = \frac{\Pi}{Z_{G'}} \prod_{e \in E_2' \cap M'} \nu(e) = \frac{(wz+xy)\Pi}{Z_G} \prod_{e \in E_2' \cap M'} \nu(e)
    \end{equation}
    by \Cref{thm:spider_Z}. 

    The portion $M' \cap E_2'$ of $M'$ on the local pattern of \Cref{fig:spider_notiles} (right) will be either the one in Case 1, one of the four rotations of the one in Case 2, or one of the two matchings 3a and 3b in Case 3. So it suffices to check \eqref{eq:coupling_wts} in all such cases, which we do now.

    \textbf{Case 1:} The probabilities of the two matchings of $G$ in this case are $\frac{\Pi}{Z_G}wz$ and $\frac{\Pi}{Z_G}xy$, so 
    \begin{equation}
         \sum_M \mathbb{P}_G(M) \mathbb{P}(\phi(M) = M') = \frac{\Pi}{Z_G}wz + \frac{\Pi}{Z_G}xy = \mathbb{P}_{G'}(M') 
     \end{equation} 
     by \eqref{eq:relate_G_G'}.

     \textbf{Case 2:} This is really four cases, but by symmetry it suffices to consider the one depicted in \Cref{fig:spider_case_2}. We check
     \begin{equation}
         \sum_M \mathbb{P}_G(M) \mathbb{P}(\phi(M) = M') = \frac{\Pi}{Z_G}w = \frac{\Pi}{Z_{G'}} \frac{w}{wz+xy} =  \frac{\Pi}{Z_G}w = \frac{\Pi}{Z_{G'}} w' =  \mathbb{P}_{G'}(M').
     \end{equation} 

     \textbf{Case 3a:} Here
     \begin{equation}
         \sum_M \mathbb{P}_G(M) \mathbb{P}(\phi(M) = M') = \frac{\Pi}{Z_G} \cdot \frac{wz}{wz+xy} = \frac{\Pi}{Z_{G'}} \frac{w}{wz+xy} \frac{z}{wz+xy} = \frac{\Pi}{Z_{G'}} w'z' = \mathbb{P}_{G'}(M').
     \end{equation} 

     \textbf{Case 3b:} The same as Case 3a.

     This completes the casework proof of \eqref{eq:coupling_wts}, so $\phi$ gives a coupling. Let us prove uniqueness. In Case 1 and Case 2, note that given a matching $M$ of $G$ as pictured, there is only one matching $M'$ of $G'$ for which $M \cap E_1 = M' \cap E_1$, so any coupling---viewed as a random transition map from matchings on $G$ to matchings on $G'$---must agree with $\phi$ in those cases. For Case 3, the only freedom is to choose the probabilities of the two matchings 3a and 3b, and the map $\phi$ given above is clearly the only one which will satisfy \eqref{eq:coupling_wts} (i.e. which will give the dimer measure on $G'$). This shows uniqueness.
\end{proof}

\begin{rmk}\label{rmk:cases_of_spider_and_shuffling}
    In the formulation of the shuffling algorithm given in \Cref{subsec:shuffling_and_uniqueness}, Case 1 corresponds to the deletion step of the shuffling algorithm, Case 2 corresponds to the deterministic slides, and Case 3 corresponds to the random creation step. See the proof of \Cref{thm:Z_recurrence} for how shuffling on the Aztec diamond comes from the spider move.
\end{rmk}

There is one other useful---but more trivial---local graph relation:

\begin{prop}
    \label{thm:vertex_expansion}
    Let $G=(V,E)$ be a weighted graph, $v \in V$ any vertex, and $\{e \in E: e \text{ contains }v\} = S \sqcup S'$ a partition of the edges incident to $v$ into two subsets. Let $G'$ be the graph identical to $G$ except for the local vertex-dilation move
    \begin{figure}[H]
        \centering
       \begin{tikzpicture}[scale=1.5, every node/.style={draw, circle, inner sep=1.5pt, fill=white}, font=\scriptsize]

\node[draw, circle, label=above:$v$] (A) at (-2, 0) {};
\node[draw=none] () at (-2.4,.4)  {$S$};
\node[draw=none] () at (-1.6,.4)  {$S'$};

\draw (A) -- ++(20:.7);
\draw (A) -- ++(-32:.7);
\draw (A) -- ++(-7:.7);

\draw (A) -- ++(167:.7);
\draw (A) -- ++(183:.7);

\node[draw=none] (iff) at (.5,0)  {\(\iff\)};

\node[draw, circle, label=above:v] (A') at (4, 0) {};
\node[draw, circle,] (Aleft) at (3, 0) {};
\node[draw, circle] (Aright) at (5, 0) {};

\node[draw=none] () at (2.4,.4)  {$S$};
\node[draw=none] () at (5.4,.4)  {$S'$};

\draw (Aleft) -- (A');
\draw (Aright) -- (A');

\draw (Aright) -- ++(20:.7);
\draw (Aright) -- ++(-32:.7);
\draw (Aright) -- ++(-7:.7);

\draw (Aleft) -- ++(167:.7);
\draw (Aleft) -- ++(183:.7);
\end{tikzpicture}
        \caption{Vertex-dilation.}
        \label{fig:vertex-dilation}
    \end{figure}
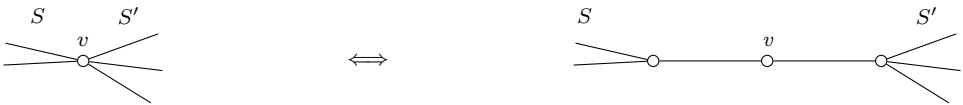
\noindent where the new edges have weight $1$. Then $Z_G=Z_{G'}$, and there is a weight-preserving bijection between perfect matchings of $G$ and $G'$.
\end{prop}
\begin{proof}
    The equality $Z_G=Z_{G'}$ is implied by the weight-preserving bijection (blue edges represent those included in the dimer cover)

        \begin{center}
       \begin{tikzpicture}[scale=1.5, every node/.style={draw, circle, inner sep=1.5pt, fill=white}, font=\scriptsize]

\node[draw, circle, label=above:$v$] (A) at (-2, 0) {};
\node[draw=none] () at (-2.4,.4)  {$S$};
\node[draw=none] () at (-1.6,.4)  {$S'$};

\draw[very thick,blue] (A) -- ++(20:.7);
\draw (A) -- ++(-32:.7);
\draw (A) -- ++(-7:.7);

\draw (A) -- ++(167:.7);
\draw (A) -- ++(183:.7);

\node[draw=none] (iff) at (.5,0)  {\(\iff\)};

\node[draw, circle, label=above:$v$] (A') at (4, 0) {};
\node[draw, circle,] (Aleft) at (3, 0) {};
\node[draw, circle] (Aright) at (5, 0) {};

\node[draw=none] () at (2.4,.4)  {$S$};
\node[draw=none] () at (5.4,.4)  {$S'$};

\draw[very thick,blue] (Aleft) -- (A');
\draw (Aright) -- (A');

\draw[very thick,blue] (Aright) -- ++(20:.7);
\draw (Aright) -- ++(-32:.7);
\draw (Aright) -- ++(-7:.7);

\draw (Aleft) -- ++(167:.7);
\draw (Aleft) -- ++(183:.7);
\end{tikzpicture}

       \begin{tikzpicture}[scale=1.5, every node/.style={draw, circle, inner sep=1.5pt, fill=white}, font=\scriptsize]

\node[draw, circle, label=above:$v$] (A) at (-2, 0) {};
\node[draw=none] () at (-2.4,.4)  {$S$};
\node[draw=none] () at (-1.6,.4)  {$S'$};

\draw (A) -- ++(20:.7);
\draw (A) -- ++(-32:.7);
\draw (A) -- ++(-7:.7);

\draw (A) -- ++(167:.7);
\draw[very thick,blue] (A) -- ++(183:.7);

\node[draw=none] (iff) at (.5,0)  {\(\iff\)};

\node[draw, circle, label=above:$v$] (A') at (4, 0) {};
\node[draw, circle,] (Aleft) at (3, 0) {};
\node[draw, circle] (Aright) at (5, 0) {};

\node[draw=none] () at (2.4,.4)  {$S$};
\node[draw=none] () at (5.4,.4)  {$S'$};

\draw (Aleft) -- (A');
\draw[very thick,blue] (Aright) -- (A');

\draw (Aright) -- ++(20:.7);
\draw (Aright) -- ++(-32:.7);
\draw (Aright) -- ++(-7:.7);

\draw (Aleft) -- ++(167:.7);
\draw[very thick,blue] (Aleft) -- ++(183:.7);
\end{tikzpicture}
\end{center}
i.e. if a dimer cover of $G$ has an edge in $S'$, add the edge from $v$ to the new vertex from which the $S$ edges sprout to the dimer cover in $G'$, while if it has an edge in $S$, add the other edge for the corresponding cover in $G'$. The edges added have weight $1$.
\end{proof}
\subsection{The Aztec diamond} The local transformations in the previous subsection are the basis of the shuffling algorithm for the Aztec diamond. For arbitrary positive edge weights, one obtains a dimer probability measure on perfect matchings as in \Cref{def:dimer_meas}. However, any arbitrary edge weights are equivalent by a series of local gauge transformations (which, recall, do not change the dimer measure) to edge weights as in \Cref{fig:our_weights}, where half of the edge weights (the unmarked ones) are $1$, see e.g. \cite{chhita2023domino}. Hence, from the perspective of the dimer measure on the Aztec diamond, there is no loss in studying weights as in \Cref{fig:our_weights} rather than arbitrary ones.

\begin{figure}
    \centering

    \begin{tikzpicture}[scale=1]
\foreach \x in {0,...,2} {
  \foreach \y in {1,...,3} {
   \draw (2*\x,2*\y)-- (2*\x+1,2*\y+1);
    \draw (2*\x,2*\y)-- (2*\x+1,2*\y-1);
\draw (2*\x+2,2*\y)-- (2*\x+1,2*\y+1);
    \draw (2*\x+2,2*\y)-- (2*\x+1,2*\y-1);
  }
}
\foreach \x in {0,...,3} {
  \foreach \y in {1,...,3} {
   \filldraw (2*\x,2*\y) circle(3pt);
    \draw[fill=white] (2*\y-1,2*\x+1) circle(3pt);
  }
}

\foreach \x in {1,...,3} { \foreach \y in {1,...,3} {
 \draw (2*\x-2+0.2,-2*\y+8+.7) node {\small $a_{\text{\x,\y}}^{[3]}$};
\draw (2*\x-2+0.2,-2*\y+8-.7) node {\small $b_{\text{\x,\y}}^{[3]}$};
}}
\foreach \x in {0,...,2} { \foreach \y in {1,...,3} {
 \draw (2*\x+1.35,2*\y+.4) node {\small$1$};
\draw (2*\x+1.35,2*\y-.4) node {\small $1$};
}}

\end{tikzpicture}
    \caption{An Aztec diamond of size $3$ with our standard edge weights.}
    \label{fig:our_weights}
\end{figure}
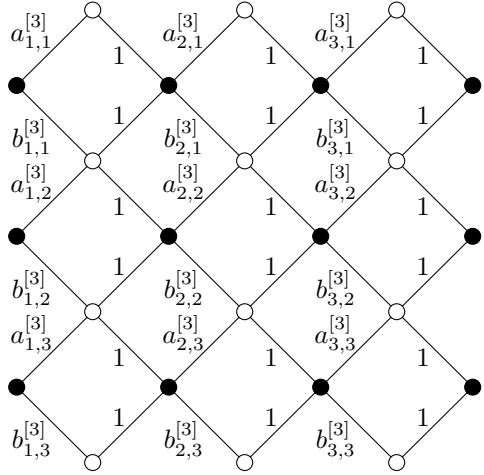

The partition function of the Aztec diamond of size $n$ can be computed recursively in terms of the partition function of the Aztec diamond of size $n-1$ by an algorithm of Propp \cite[Section 2]{propp2003generalized}, which is proven by reducing one graph to the other by the local move of \Cref{thm:spider_Z}. For the sake of exposition we give this argument now, though it will also fall naturally out of the graph manipulations in \Cref{sec:vert_polymer}, see the second proof of \Cref{thm:compute_Z_cor} given in that section.

\begin{prop}
    \label{thm:Z_recurrence}
    Let $n \geq 1$ and let $Z_n$ (resp. $Z_{n+1}$) be the partition functions of Aztec diamonds of size $n$ (resp. $n+1$) with edge weights $\{a_{i,j}^{[n]},b_{i,j}^{[n]}: 1 \leq i,j \leq n\}$ (resp. $\{a_{i,j}^{[n+1]},b_{i,j}^{[n+1]}: 1 \leq i,j \leq n+1\}$) as in \Cref{fig:our_weights}, where we define the weights of the size $n$ Aztec in terms of those of the size $n+1$ Aztec via
    \begin{align}\label{eq:downshuffle_weights_finite}
    \begin{split}
        a_{i,j}^{[n]} &= \frac{a_{i,j}^{[n+1]}}{a_{i,j}^{[n+1]}+b_{i,j}^{[n+1]}}(a_{i+1,j}^{[n+1]}+b_{i+1,j}^{[n+1]}) \\ 
        b_{i,j}^{[n]} &= \frac{b_{i,j+1}^{[n+1]}}{a_{i,j+1}^{[n+1]}+b_{i,j+1}^{[n+1]}}(a_{i+1,j+1}^{[n+1]}+b_{i+1,j+1}^{[n+1]}).
    \end{split}
\end{align} 
Then 
    \begin{equation}\label{eq:Z_recurrence}
        Z_{n+1} = \left( \prod_{j=1}^{n+1} (a_{1,j}^{[n+1]} + b_{1,j}^{[n+1]})\right) Z_{n}.
    \end{equation}
\end{prop}
\begin{proof}
    We begin with the Aztec diamond of size $n+1$ with weights $\{a_{i,j}^{[n+1]},b_{i,j}^{[n+1]}: 1 \leq i,j \leq n+1\}$. First, apply the (reverse of the) square move on every other square (those with black leftmost vertices) as shown (for concreteness we show pictures for $n=2$ in this proof) as in \Cref{fig:square_downshuffle}.

    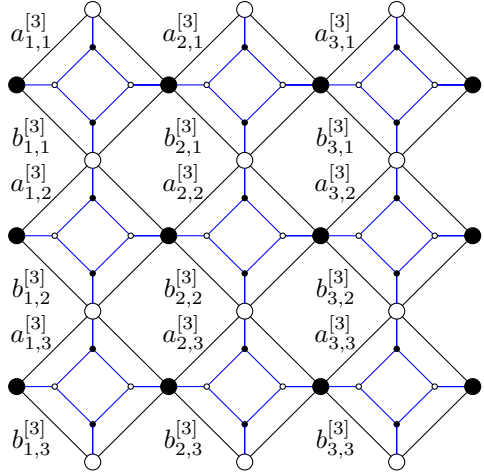
\begin{figure}[H]
    \centering
        \begin{tikzpicture}[scale=1]

\foreach \x in {0,...,2} {
  \foreach \y in {1,...,3} {
   \draw (2*\x,2*\y)-- (2*\x+1,2*\y+1);
    \draw (2*\x,2*\y)-- (2*\x+1,2*\y-1);
\draw (2*\x+2,2*\y)-- (2*\x+1,2*\y+1);
    \draw (2*\x+2,2*\y)-- (2*\x+1,2*\y-1);
  }
}

\foreach \x in {1,...,3} { \foreach \y in {1,...,3} {
 \draw (2*\x-2+0.2,-2*\y+8+.7) node {\small $a_{\text{\x,\y}}^{[3]}$};
\draw (2*\x-2+0.2,-2*\y+8-.7) node {\small $b_{\text{\x,\y}}^{[3]}$};
}}

\foreach \x in {0,...,2} {
  \foreach \y in {1,...,3} {
   \draw[blue] (2*\x,2*\y)--(2*\x+.5,2*\y)--(2*\x+1,2*\y+.5)--(2*\x+1,2*\y+1)--(2*\x+1,2*\y+.5)--(2*\x+1.5,2*\y)--(2*\x+2,2*\y)--(2*\x+1.5,2*\y)--(2*\x+1,2*\y-.5)--(2*\x+1,2*\y-1)--(2*\x+1,2*\y-.5)--(2*\x+.5,2*\y);
     \filldraw (2*\x+1,2*\y+.5) circle(1pt);
     \filldraw (2*\x+1,2*\y-.5) circle(1pt);
    \draw[fill=white] (2*\x+.5,2*\y) circle(1pt);
    \draw[fill=white] (2*\x+1.5,2*\y) circle(1pt);
}}

\foreach \x in {0,...,3} {
  \foreach \y in {1,...,3} {
   \filldraw (2*\x,2*\y) circle(3pt);
    \draw[fill=white] (2*\y-1,2*\x+1) circle(3pt);
  }
}

\end{tikzpicture}
    \caption{Black edges are those present before square move is applied, blue are those after. Note that we have not indicated the weights on the blue edges.}
    \label{fig:square_downshuffle}
\end{figure}
    This changes the partition function by a factor of
    \begin{equation}\label{eq:partition_fn_update_shuffling}
    \frac{1}{\prod_{1 \leq i,j \leq n+1} (a_{i,j}^{[n+1]}+b_{i,j}^{[n+1]})},    
    \end{equation}
    by \Cref{thm:spider_Z}. After applying the vertex-contraction move of \Cref{thm:vertex_expansion} at each inner black vertex, the resulting weighted graph is as in \Cref{fig:downshuffled_pendant}:
    \begin{figure}[H]
    \centering
     \begin{tikzpicture}[scale=.7]

\foreach \x in {0,...,2} {
  \foreach \y in {1,...,3} {
   \draw (2*\x,2*\y)-- (2*\x+1,2*\y+1);
    \draw (2*\x,2*\y)-- (2*\x+1,2*\y-1);
\draw (2*\x+2,2*\y)-- (2*\x+1,2*\y+1);
    \draw (2*\x+2,2*\y)-- (2*\x+1,2*\y-1);
  }
}



\foreach \y in {1,...,3} {
  \draw[blue,thick] (-1,2*\y)--(0,2*\y);
  \filldraw  (-1,2*\y) circle(3pt);
}

\foreach \y in {1,...,3} {
  \draw[blue,thick] (7,2*\y)--(6,2*\y);
  \filldraw  (7,2*\y) circle(3pt);
}

\foreach \x in {0,...,2} {
  \draw[blue,thick] (2*\x+1,7)--(2*\x+1,8);
  \filldraw[fill=white]  (2*\x+1,8) circle(3pt);
}

\foreach \x in {0,...,2} {
  \draw[blue,thick] (2*\x+1,0)--(2*\x+1,1);
  \filldraw[fill=white]  (2*\x+1,0)  circle(3pt);
}

\foreach \x in {0,...,3} {
  \foreach \y in {1,...,3} {
   \draw[fill=white](2*\x,2*\y) circle(3pt);
    \filldraw  (2*\y-1,2*\x+1) circle(3pt);
  }
}

\end{tikzpicture}
    \caption{The Aztec diamond after spider moves and vertex contraction. Any dimer cover must include the pendant edges, highlighted in blue, which have weight $1$. Edge weights are not pictured, but the weights on edges which can appear in matchings are in \Cref{fig:downshuffled_no_pendant}.} \label{fig:downshuffled_pendant}
    \end{figure}
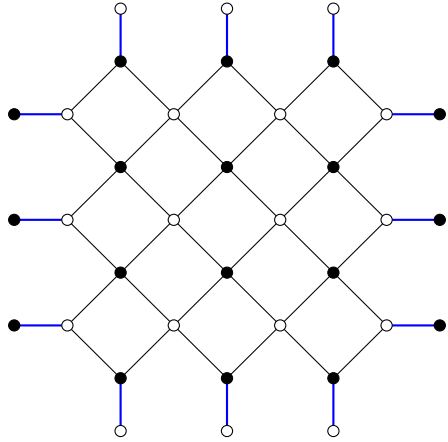
    Because every dimer cover of this graph must include the `pendant' edges which are incident to a vertex of degree $1$, its partition function is the same as that of the graph in \Cref{fig:downshuffled_no_pendant}:

    \begin{figure}[H]
    \centering
     \begin{tikzpicture}[scale=1.5]

\foreach \x in {1,...,2} {
  \foreach \y in {2,...,3} {
   \draw (2*\x,2*\y)-- (2*\x+1,2*\y+1);
    \draw (2*\x,2*\y)-- (2*\x+1,2*\y-1);
\draw (2*\x+2,2*\y)-- (2*\x+1,2*\y+1);
    \draw (2*\x+2,2*\y)-- (2*\x+1,2*\y-1);
  }
}

\foreach \x in {1,...,3} {
  \foreach \y in {2,...,3} {
   \draw[fill=black](2*\x,2*\y) circle(3pt);
    \filldraw[fill=white]  (2*\y-1,2*\x+1) circle(3pt);
  }
}
\node[anchor=south east] at (2.5,6.5) {\small{$\frac{a_{11}^{[3]}}{a_{11}^{[3]}+b_{11}^{[3]}}$}};

\node[anchor=south east] at (2.5,4.35) {\small{$\frac{a_{12}^{[3]}}{a_{12}^{[3]}+b_{12}^{[3]}}$}};

\draw[anchor=north east] (2.5,5.75) node {\small{$\frac{b_{12}^{[3]}}{a_{12}^{[3]}+b_{12}^{[3]}}$}};

\draw[anchor=north east] (2.5,3.75) node {\small{$\frac{b_{13}^{[3]}}{a_{13}^{[3]}+b_{13}^{[3]}}$}};

\node[anchor=south east] at (4.7,6.5) {\small{$\frac{a_{21}^{[3]}}{a_{21}^{[3]}+b_{21}^{[3]}}$}};

\node[anchor=south east] at (4.7,4.45) {\small{$\frac{a_{22}^{[3]}}{a_{22}^{[3]}+b_{22}^{[3]}}$}};

\draw[anchor=north east] (4.7,5.65) node {\small{$\frac{b_{12}^{[3]}}{a_{22}^{[3]}+b_{22}^{[3]}}$}};

\draw[anchor=north east] (4.75,3.65) node {\small{$\frac{b_{23}^{[3]}}{a_{23}^{[3]}+b_{23}^{[3]}}$}};
¨

\node[anchor=north east] at (5.6,6.6) {\small{$\frac{1}{a_{31}^{[3]}+b_{31}^{[3]}}$}};

\node[anchor=north east] at (5.6,4.55) {\small{$\frac{1}{a_{32}^{[3]}+b_{32}^{[3]}}$}};

\draw[anchor=south east] (5.6,5.55) node {\small{$\frac{1}{a_{32}^{[3]}+b_{32}^{[3]}}$}};

\draw[anchor=south east] (5.6,3.45) node {\small{$\frac{1}{a_{33}^{[3]}+b_{33}^{[3]}}$}};

\node[anchor=north east] at (3.6,6.6) {\small{$\frac{1}{a_{21}^{[3]}+b_{21}^{[3]}}$}};

\node[anchor=north east] at (3.6,4.55) {\small{$\frac{1}{a_{22}^{[3]}+b_{22}^{[3]}}$}};

\draw[anchor=south east] (3.6,5.55) node {\small{$\frac{1}{a_{22}^{[3]}+b_{22}^{[3]}}$}};

\draw[anchor=south east] (3.6,3.55) node {\small{$\frac{1}{a_{23}^{[3]}+b_{23}^{[3]}}$}};
\end{tikzpicture}
        \caption{The graph after removal of the pendant edges and the edges adjacent to them which are never covered.} \label{fig:downshuffled_no_pendant}
    \end{figure}
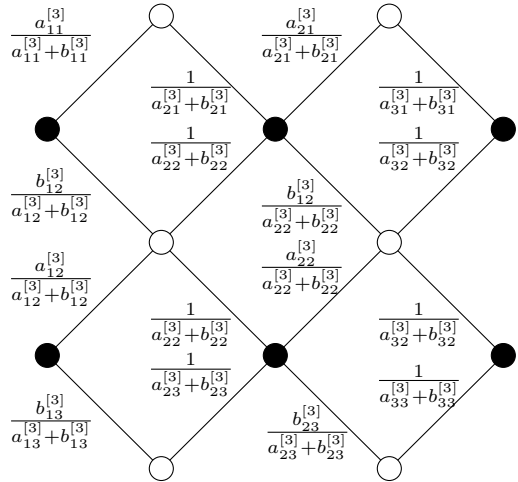
    Finally, gauge transform at each white vertex by the factor $(a_{i,j}^{[n+1]}+b_{i,j}^{[n+1]}), 2 \leq i \leq n+1, 1 \leq j \leq n+1$ necessary to make the Northeast and Southeast edges incident to it have weight $1$. The resulting graph has weights $\{a_{i,j}^{[n]},b_{i,j}^{[n]}: 1 \leq i,j \leq n\}$ defined as in \eqref{eq:downshuffle_weights_finite} and its partition function is therefore $Z_n$. We have changed the partition function $Z_{n+1}$ by a factor of
    \begin{equation}
    \frac{\prod_{\substack{2 \leq i \leq n+1 \\ 1 \leq j \leq n+1}} (a_{i,j}^{[n+1]}+b_{i,j}^{[n+1]})}{\prod_{1 \leq i,j \leq n+1} (a_{i,j}^{[n+1]}+b_{i,j}^{[n+1]})} = \frac{1}{\prod_{1 \leq j \leq n+1} (a_{1,j}^{[n+1]}+b_{1,j}^{[n+1]})}
    \end{equation}
    through the spider move and then the gauge transformation, which exactly yields \eqref{eq:Z_recurrence}.
\end{proof}

The recurrence of \Cref{thm:Z_recurrence} of course gives an explicit formula for the whole partition function $Z_n$ when applied repeatedly. Since the proof relates the Aztec diamond graphs of sizes $n$ and $n+1$ via the spider move and vertex-dilation, \Cref{thm:spider_coupling} and \Cref{thm:vertex_expansion} give not just a relation between partition functions but a coupling between matchings of the Aztec diamonds of size $n$ and size $n+1$. This coupling, viewed as a random transition map from matchings of size $n$ to matchings of size $n+1$, is exactly a step of the shuffling algorithm discussed in the Introduction. As noted in \Cref{rmk:cases_of_spider_and_shuffling}, the three cases of the coupling correspond to the deletion, slide, and creation steps of the algorithm; for a longer discussion of the equivalence of the two descriptions see e.g. \cite[Section 12.3]{assiotis2023some}.

\subsection{Relevant random variables} The Gamma and Beta distributions are characterized by certain invariance properties which make them natural from the perspective of the spider move on the Aztec diamond. We give background on these random variables now.

\begin{defi}
    \label{def:gamma_var}
    The \emph{Gamma distribution} $\Gamma(\chi,s)$ with \emph{shape parameter} $\chi > 0$ and \emph{scale parameter} $s > 0$ is the probability distribution on $\R_{+}$ with density
    \begin{equation}
        \frac{1}{\Gamma(\chi)s^\chi} x^{\chi-1} e^{-x/s}
    \end{equation}
    with respect to the Lebesgue measure $dx$, where $\Gamma(\chi)$ is the Gamma function. We write $X \sim \Gamma(\chi,s)$ for Gamma random variables $X$.

    We write $\Gamma^{-1}(\chi,s)$ for the distribution of $1/X$ where $X \sim \Gamma(\chi,s)$.
\end{defi}

Note that the scale parameter just multiplies the random variable by a constant $s$. This does not affect the dimer measure as mentioned in \Cref{rmk:scale_doesn't_matter}. 

\begin{defi}
    \label{def:beta_var}
    The \emph{Beta distribution} $\Beta(\alpha,\beta)$ is the probability distribution on $[0,1]$ with density
    \begin{equation}
        \frac{1}{\beta(\alpha,\beta)}x^{\alpha-1}(1-x)^{\beta-1}
    \end{equation}
    with respect to the Lebesgue measure $dx$, where $\beta(\alpha,\beta)$ is the beta function.

    We write $\Beta^{-1}(\alpha,\beta)$ for the distribution of $1/X$ where $X \sim \Beta(\alpha,\beta)$.
\end{defi}

We rely on the following classical result due to Lukacs \cite{lukacs1955characterization}\footnote{Note that the original version was stated for independence of $X+Y$ with $Y/X$ instead of $X/(X+Y)$, but these are equivalent since $X/(X+Y) = 1/(1+Y/X)$.}.

\begin{thm}\label{thm:lukacs} Let $X,Y$ be independent, positive, and nonconstant random variables. Then the random variables $X+Y$ and $X/(X+Y)$ are independent if and only if there exist $x,y,s \in \R_{>0}$ such that
    \begin{align}
        X &\sim \Gamma(x,s) \\ 
        Y &\sim \Gamma(y,s).
    \end{align}
    In this case,
    \begin{align} \label{eq:gamma_to_gammabeta}
        X+Y &\sim \Gamma(x+y,s) \\ 
        \frac{X}{X+Y} &\sim \Beta(x,y).
    \end{align}
\end{thm}

\begin{cor}
    \label{thm:XY_lukacs_cor}
    Let $A,B$ be two independent, nonconstant random variables such that $B$ is supported in $[0,1]$ and $A$ is positive. Then $BA$ and $(1-B)A$ are independent if and only if there exist $x,y,s \in \R_{>0}$
    \begin{align}\label{eq:AB_lukacs_hyp}
    \begin{split}
        A &\sim \Gamma(x+y,s) \\ 
        B &\sim \Beta(x,y).
    \end{split}
    \end{align}
    In this case,
    \begin{align} \label{eq:AB_lukacs_cons}
        BA &\sim \Gamma(x,s) \\ 
        (1-B)A &\sim \Gamma(y,s).
    \end{align}
\end{cor}
\begin{proof}
    ($\Rightarrow$) Let $X=BA, Y = (1-B)A$. Then these are both positive, nonconstant random variables, and the random variables $A=X+Y$ and $B=X/(X+Y)$ are independent. Hence the forward direction of \Cref{thm:lukacs} yields that there exist $x,y,s \in \R_{>0}$ such that
    \begin{align}
    \begin{split}\label{eq:xy_intermediate}
        X &\sim \Gamma(x,s) \\ 
        Y &\sim \Gamma(y,s).
    \end{split}
    \end{align}
    Deducing \eqref{eq:AB_lukacs_hyp} is a standard computation.

    ($\Leftarrow$) One has equality in joint distribution $(A,B) \sim (X+Y,X/(X+Y))$ for $X,Y$ independent and distributed as in \eqref{eq:xy_intermediate}, and it is a standard computation that 
    \begin{equation}
        (X+Y,X/(X+Y)) \sim \Gamma(x+y,s) \times \Beta(x,y).
    \end{equation}
\end{proof}

\section{Characterization of shuffle-independent full-plane weights} \label{sec:full_plane_gamma}

Let us recall and expand upon the motivation given in the Introduction for considering a full plane worth of weights. The recursion used in \Cref{thm:Z_recurrence} gives the weights of a size $n$ Aztec diamond in terms of those of a size $n+1$ Aztec diamond. If considering random weights, it is natural to ask what families of independent random weights on the size $n$ Aztec diamond have the property that the weights of each Aztec diamond of size $n-1,n-2,\ldots$, expressed in terms of them via 
\begin{align}\label{eq:downshuffle_weights}
    \begin{split}
        a_{i,j}^{[n-1]} &= \frac{a_{i,j}^{[n]}}{a_{i,j}^{[n]}+b_{i,j}^{[n]}}(a_{i+1,j}^{[n]}+b_{i+1,j}^{[n]}) \\ 
        b_{i,j}^{[n-1]} &= \frac{b_{i,j+1}^{[n]}}{a_{i,j+1}^{[n]}+b_{i,j+1}^{[n]}}(a_{i+1,j+1}^{[n]}+b_{i+1,j+1}^{[n]})
    \end{split}
\end{align}
(given previously in \eqref{eq:downshuffle_weights_finite}) are also independent. 

We will, however, answer a slightly different version of this question in order to characterize the Gamma weights treated here. It is natural, particularly from the perspective of sampling via the shuffling algorithm, to invert these recurrences \eqref{eq:downshuffle_weights} and define the weights of the size $n+1$ Aztec diamond in terms of the one of size $n$. The formula \eqref{eq:downshuffle_weights} directly yields

\begin{align}\label{eq:upshuffle_weights}
    \begin{split}
        a_{i,j}^{[n+1]} &= \frac{a_{i,j}^{[n]}}{a_{i,j}^{[n]}+b_{i,j-1}^{[n]}}(a_{i-1,j}^{[n]}+b_{i-1,j-1}^{[n]}) \\ 
        b_{i,j}^{[n+1]} &= \frac{b_{i,j-1}^{[n]}}{a_{i,j}^{[n]}+b_{i,j-1}^{[n]}}(a_{i-1,j}^{[n]}+b_{i-1,j-1}^{[n]}).
    \end{split}
\end{align}
However, there is a problem: to get weights $a_{i,j}^{[n+1]}$ or $b_{i,j}^{[n+1]}$ when $i$ or $j$ is $1$, \eqref{eq:upshuffle_weights} expresses them in terms of weights $a_{0,j}^{[n]}$ and $b_{i,0}^{[n]}$, which do not appear in the size $n$ Aztec diamond. Similarly, the weights $a_{i,j}^{[n+2]}$ or $b_{i,j}^{[n+2]}$ with $i$ or $j$ equal to $1$ will require weights $a_{-1,j}^{[n]}$ and $b_{i,-1}^{[n]}$. Hence it is really most natural to consider a full plane worth of random weights 
\begin{equation}\label{eq:full_plane_weights}
    \{a_{i,j}^{[0]},b_{i,j}^{[0]}: i,j \in \Z\},
\end{equation}
which `preprocesses' all randomness needed to generate the random weights of the Aztec diamond of every size $n$. All weights $a_{i,j}^{[n]},b_{i,j}^{[n]}$ for any $n \geq 1$ may be expressed in terms of these through repeated application of the recurrences \eqref{eq:upshuffle_weights}, and conversely each weight in \eqref{eq:full_plane_weights} appears in the formula for $a_{i,j}^{[n]},b_{i,j}^{[n]}$ for some $n$. One can also define weights with $n \leq 0$ via \eqref{eq:downshuffle_weights}. This full-plane perspective, which is also taken e.g. in \cite{chhita2023domino}, is the one we take in this section.

\begin{defi}
    \label{def:gen_weights}
    We say that real parameter sequences $(\psi_j)_{j \in \Z},(\phi_j)_{j \in \Z}, (s_i)_{i \in \Z}, (\theta_i)_{i \in \Z}$ are \emph{admissible} if  
    \begin{align}\label{eq:weight_positivity}
    \begin{split}
        s_i &> 0 \\ 
        \psi_j+\theta_i & > 0 \\ 
        \phi_j -\theta_i & > 0 \\ 
    \end{split}
    \end{align}
    for each $i,j \in \Z$.
\end{defi}

\begin{thm}
    \label{thm:gamma_to_gamma}
    Let $(\psi_j)_{j \in \Z},(\phi_j)_{j \in \Z}, (s_i)_{i \in \Z}, (\theta_i)_{i \in \Z}$ be admissible parameter sequences in the sense of \Cref{def:gen_weights}, and $\{a_{i,j}^{[0]},b_{i,j}^{[0]}: i,j \in \Z\}$ be a collection of mutually independent random variables with distribution
    \begin{align}
        \begin{split} \label{eq:0_level_parameters}
            a_{i,j}^{[0]} &\sim \Gamma(\psi_j+\theta_i, s_{i}) \\ 
            b_{i,j}^{[0]} &\sim \Gamma(\phi_{j}-\theta_i, s_{i}).
        \end{split}
    \end{align}
    For every $n \in \Z$, define the random variables $\{a_{i,j}^{[n]},b_{i,j}^{[n]}: i,j \in \Z\}$ recursively by the recurrences \eqref{eq:upshuffle_weights} for $n \geq 1$ and \eqref{eq:downshuffle_weights} for $n \leq 1$. Then for each $n$, the random variables $\{a_{i,j}^{[n]},b_{i,j}^{[n]}: i,j \in \Z\}$ are mutually independent, with distribution given by
    \begin{align}
    \begin{split} \label{eq:n_level_parameters}
            a_{i,j}^{[n]} &\sim \Gamma(\psi_j+\theta_i, s_{i-n}) \\ 
            b_{i,j}^{[n]} &\sim \Gamma(\phi_{j-n}-\theta_i, s_{i-n}).
        \end{split}
    \end{align}
\end{thm}
\begin{proof}
The proof proceeds by induction, using the recurrence relation \eqref{eq:upshuffle_weights}. We prove the case of $n \geq 1$, and the case of $n \leq -1$ is the same argument.

Assume that \eqref{eq:n_level_parameters} holds for some fixed \( n \in \mathbb{N} \). Consider the collection of variables $\{A_{i,j},B_{i,j} : (i,j) \in \Z^2\}$ defined by 
\begin{align}
    \begin{split}
        A_{i,j} &= a_{i,j}^{[n]} + b_{i,j-1}^{[n]} \sim \Gamma(\psi_j + \phi_{j-n-1}, s_{i-n}) \\ 
        B_{i,j} &=  \frac{a_{i,j}^{[n]}}{a_{i,j}^{[n]} + b_{i,j-1}^{[n]}} \sim \Beta(\psi_j + \theta_i, \phi_{j-n-1} - \theta_i).
    \end{split}
\end{align}
This collection is mutually independent, because the only pairs in this collection which depend on the same randomness are pairs $A_{i,j},B_{i,j}$ for a given $i,j$, and these are independent by \Cref{thm:lukacs}.

Now, the recurrence \eqref{eq:upshuffle_weights} in terms of the $A_{i,j}$ and $B_{i,j}$ reads
\begin{align}\label{eq:a_from_A}
    \begin{split}
        a_{i,j}^{[n+1]} &= B_{i,j}A_{i-1,j} \sim \Gamma(\psi_j + \theta_i, s_{i-n-1}) \\ 
        b_{i,j}^{[n+1]} &= (1-B_{i,j})A_{i-1,j} \sim  \Gamma(\phi_{j-n-1} - \theta_i, s_{i-n-1}),      
    \end{split}
\end{align}
where the $\sim \Gamma(\cdots)$ comes from the `in this case' part of \Cref{thm:XY_lukacs_cor}. It remains to show that the variables $\{a_{i,j}^{[n+1]},b_{i,j}^{[n+1]}: (i,j) \in \Z^2\}$ are mutually independent.

For each $i_0,j_0$, by \eqref{eq:a_from_A} the variables $a_{i_0,j_0}^{[n+1]}$ and $b_{i_0,j_0}^{[n+1]}$ are the only random variables in the collection $\{a_{i,j}^{[n+1]},b_{i,j}^{[n+1]}: (i,j) \in \Z^2\}$ which depend on $B_{i_0,j_0}$, and they are also the only random variables which depend on $A_{i_0-1,j}$. So, the collection of $\R^2$-valued random variables $\{(a_{i,j}^{[n+1]},b_{i,j}^{[n+1]}) : (i,j) \in \Z^2\}$ is mutually independent. By \Cref{thm:XY_lukacs_cor}, $a_{i,j}^{[n+1]}$ is independent of $b_{i,j}^{[n+1]}$, hence the collection $\{a_{i,j}^{[n+1]},b_{i,j}^{[n+1]}: (i,j) \in \Z^2\}$ is mutually independent.

Hence, the claim holds for all \( n \geq 0 \). The case \( n \leq 0 \) follows by a similar argument, using \eqref{eq:downshuffle_weights}.
\end{proof}

\begin{thm}
    \label{thm:independent_implies_gamma}
    Let 
    \begin{equation}
        \{a_{i,j}^{[0]},b_{i,j}^{[0]}: i,j \in \Z\}
    \end{equation}
    be a collection of mutually independent, positive, nonconstant random variables, such that each of the collections of random variables 
    \begin{equation}
        \{a_{i,j}^{[1]},b_{i,j}^{[1]}: i,j \in \Z\}
    \end{equation}
    defined by \eqref{eq:upshuffle_weights} and 
    \begin{equation}
        \{a_{i,j}^{[-1]},b_{i,j}^{[-1]}: i,j \in \Z\}
    \end{equation}
    defined by \eqref{eq:downshuffle_weights} is also mutually independent. Then there exist admissible parameter sequences $(\psi_j)_{j \in \Z},(\phi_j)_{j \in \Z}, (s_i)_{i \in \Z}, (\theta_i)_{i \in \Z}$ (in the sense of \Cref{def:gen_weights}) such that the variables $\{a_{i,j}^{[0]},b_{i,j}^{[0]}: i,j \in \Z\}$ are distributed according to \eqref{eq:0_level_parameters}.
\end{thm}

\begin{proof}[Proof of \Cref{thm:independent_implies_gamma}]
    We first show that the random variables $\{a_{i,j}^{[0]},b_{i,j}^{[0]}: i,j \in \Z\}$ have Gamma distribution, then pin down the parameters. 

    If the random variables $\{a_{i,j}^{[1]},b_{i,j}^{[1]}: i,j \in \Z\}$ are mutually independent, then in particular for each $i,j \in \Z$ we have that
    \begin{equation}
        a^{[1]}_{i,j} = \frac{a_{i,j}^{[0]}}{a_{i,j}^{[0]}+b_{i,j-1}^{[0]}}(a_{i-1,j}^{[0]}+b_{i-1,j-1}^{[0]})
    \end{equation}
    is independent of 
    \begin{equation}
        a_{i+1,j}^{[1]} = \frac{a_{i+1,j}^{[0]}}{a_{i+1,j}^{[0]}+b_{i+1,j-1}^{[0]}}(a_{i,j}^{[0]}+b_{i,j-1}^{[0]}).
    \end{equation}
    Since all the variables with superscript $[0]$ are independent, the above implies that the laws of $a_{i,j}^{[1]}$ and $a_{i+1,j}^{[1]}$ are conditionally independent after conditioning on $a_{i+1,j}^{[0]},b_{i+1,j-1}^{[0]},a_{i-1,j}^{[0]},b_{i-1,j-1}^{[0]}$, hence
    \begin{equation}
        \frac{a_{i,j}^{[0]}}{a_{i,j}^{[0]}+b_{i,j-1}^{[0]}} \quad \quad \quad \text{is independent of} \quad \quad \quad a_{i,j}^{[0]}+b_{i,j-1}^{[0]}.
    \end{equation}
    By \Cref{thm:lukacs}, since $a_{i,j}^{[0]}$ and $b_{i,j}^{[0]}$ are assumed to be nonconstant, it follows that for all $i,j \in \Z$,
    \begin{align}\label{eq:[0]_vars_parameters}
        \begin{split}
            a_{i,j}^{[0]} &\sim \Gamma(\alpha_{i,j},s_{i,j}) \\ 
            b_{i,j-1}^{[0]} &\sim \Gamma(\beta_{i,j-1},s_{i,j})
        \end{split}
    \end{align}
    for some $\alpha_{i,j},\beta_{i,j-1}, s_{i,j} \in \R_+$; note that the scale parameters $s_{i,j}$ are the same for both. 
    
    It remains to show that the parameters have the structure as stated in the theorem. That is, we need to show that the scale parameters $s_{i,j}$ are independent of $j$,  $\alpha_{i,j}+\beta_{i,j}$ is independent of $i$ and  $\alpha_{i,j}-\beta_{i,j}$ is independent of $j$. 
    
    We start with the statement on the scale parameter and begin  by writing
    $$
    a^{[-1]}_{i,j}=BA, \quad \text{and}\quad 
    b^{[-1]}_{i,j-1}=(1-B)A,
    $$
   where 
  \begin{equation} \label{eq:[0]_vars_AB_A}
    A= a^{[0]}_{i+1,j}+b^{[0]}_{i+1,j},
    \end{equation}
    and 
     \begin{equation} \label{eq:[0]_vars_AB_B}
B=\frac{a^{[0]}_{i,j}}{a^{[0]}_{i,j}+b^{[0]}_{i,j}}.    \end{equation}
    Since the variables $a^{[0]}_{i,j}, b^{[0]}_{i,j}$ for $i,j\in \mathbb Z$ are mutually independent, we see that $A$ and $B$ are independent. By assumption, we also have that $a^{[-1]}_{i,j}$  and $b^{[-1]}_{i,j}$ are independent, and thus by using \Cref{thm:XY_lukacs_cor}, we find that there exist parameters $x,y,s$ such that 
    \begin{equation} \label{eq:[0]_vars_AB}
    A\sim \Gamma(x+y,s), \quad B\sim \Beta(x,y)
    \end{equation}
    which by  \Cref{thm:lukacs} gives  $a^{[0]}_{i,j} \sim \Gamma(x,s)$ and $b_{i,j}^{[0]}\sim \Gamma(y,s)$, and thus they  must have the same scale parameters. By \eqref{eq:[0]_vars_parameters} this implies that $s_{i,j}=s_{i,j+1}$. Since this holds for all $i,j\in \mathbb Z$ we thus indeed find that $s_{i,j}=s_i$ does not depend on $j$.

    This  argument also shows that $\alpha_{i,j}+\beta_{i,j}$ does not depend on $i$. Indeed, by \eqref{eq:[0]_vars_AB_B} and \eqref{eq:[0]_vars_parameters}, we find that $x,y$ in \eqref{eq:[0]_vars_AB} satisfies $x+y=\alpha_{i,j}+\beta_{i,j}$. However, from \eqref{eq:[0]_vars_AB} and \eqref{eq:[0]_vars_AB_A} we also find $x+y=\alpha_{i+1,j}+\beta_{i+1,j}$, and thus 
    \begin{equation} \label{eq:[0]_vars_iii} \alpha_{i+1,j}+\beta_{i+1,j}=\alpha_{i,j}+\beta_{i,j},
    \end{equation}
    for $i,j\in \mathbb Z$,
    which implies that $\alpha_{i,j}+\beta_{i,j}$ does not depend on $i$, but only on $j$.

    Similarly,  from \eqref{eq:[0]_vars_parameters},
    \begin{equation}\frac{a_{i,j}^{[0]}}{a_{i,j}^{[0]}+b_{i,j-1}^{[0]}} \sim \Beta(\alpha_{i,j},\beta_{i,j-1})
    \end{equation}
    and
    \begin{equation}
        a_{i,j}^{[0]}+b_{i,j-1}^{[0]} \sim \Gamma(\alpha_{i,j}+\beta_{i,j-1},s_{i,j}).
    \end{equation}
    So since
    \begin{equation}
        \hat A := a_{i-1,j}^{[0]}+b_{i-1,j-1}^{[0]} \sim \Gamma(\alpha_{i-1,j}+\beta_{i-1,j-1},s_{i-1,j-1})
    \end{equation}
    and 
    \begin{align}
        \begin{split}
            \hat B := \frac{a_{i,j}^{[0]}}{a_{i,j}^{[0]}+b_{i,j-1}^{[0]}} \sim \Beta(\alpha_{i,j},\beta_{i,j-1})
        \end{split}
    \end{align}
    are independent, and also $a_{i,j}^{[1]}=\hat B\hat A$ is independent of $b_{i,j}^{[1]}=(1-\hat B)\hat A$, \Cref{thm:XY_lukacs_cor} tells us that
    \begin{equation}
        \label{eq:first_alpha_relation}
        \alpha_{i,j}+\beta_{i,j-1} = \alpha_{i-1,j}+\beta_{i-1,j-1}
    \end{equation}
    (indeed,  the sum of the Beta variable parameters $x,y$ is equal to the Gamma variable's shape parameter $x+y$ in that result). By starting from a reordering of \eqref{eq:first_alpha_relation} and then  inserting \eqref{eq:[0]_vars_iii}  (with $i$ replaced by $i-1$), we find, for $i,j \in \mathbb Z$, 
    $$
    \beta_{i,j-1}-\beta_{i-1,j-1}=\alpha_{i-1,j}-\alpha_{i,j}=\beta_{i,j}-\beta_{i-1,j}.
    $$
 Hence  \begin{equation}\label{eq:def_kappa}
         \alpha_{i-1,j}-\alpha_{i,j} = \beta_{i,j-1}-\beta_{i-1,j-1} =: \kappa_i
    \end{equation}
    is independent of $j$. We may choose (non-uniquely, but unique up to an overall constant shift) parameters $\theta_i \in \R$ such that 
    \begin{equation}
        \kappa_i = \theta_{i-1}-\theta_i.
    \end{equation}
    With this notation, \eqref{eq:def_kappa} implies 
    \begin{equation}\label{eq:alpha_theta}
        \alpha_{i,j} - \theta_i = \alpha_{i-1,j} - \theta_{i-1},
    \end{equation}
    hence the right hand side of \eqref{eq:alpha_theta} is independent of $i$ and we may define
    \begin{equation}
        \psi_j := \alpha_{i,j} - \theta_i.
    \end{equation}
    Similarly, $\beta_{i,j} + \theta_i$ is independent of $i$ and hence we may define 
    \begin{equation}
        \phi_j := \beta_{i,j} + \theta_i.
    \end{equation}
    With this notation,
    \begin{align}
        \alpha_{i,j} &= \psi_j + \theta_i \\ 
        \beta_{i,j} &= \phi_j - \theta_i.
    \end{align}
    Finally, note that \eqref{eq:[0]_vars_parameters} in our new parameters gives
    \begin{align}\label{eq:[0]_vars_parameters2}
        \begin{split}
            a_{i,j}^{[0]} &\sim \Gamma(\psi_j+\theta_i,s_i) \\ 
            b_{i,j}^{[0]} &\sim \Gamma(\phi_j-\theta_i,s_{i}),
        \end{split}
    \end{align}
    and this proves the statement.\end{proof}
\begin{cor}
    \label{thm:3-layer_implies_all-layer}
    Let 
    \begin{equation}
        \{a_{i,j}^{[0]},b_{i,j}^{[0]}: i,j \in \Z\}
    \end{equation}
    be a collection of mutually independent nonconstant random variables, such that each of the collections of random variables 
    \begin{equation}
        \{a_{i,j}^{[1]},b_{i,j}^{[1]}: i,j \in \Z\}
    \end{equation}
    defined by \eqref{eq:upshuffle_weights} and 
    \begin{equation}
        \{a_{i,j}^{[-1]},b_{i,j}^{[-1]}: i,j \in \Z\}
    \end{equation}
    defined by \eqref{eq:downshuffle_weights} is also mutually independent. Then for any $n \in \Z$, the random variables $\{a_{i,j}^{[n]},b_{i,j}^{[n]}: i,j \in \Z\}$ are mutually independent. Furthermore, there is an admissible (in the sense of \Cref{def:gen_weights}) parameter collection $(\psi_j)_{j \in \Z},(\phi_j)_{j \in \Z}, (s_i)_{i \in \Z}, (\theta_i)_{i \in \Z}$ such that for every $i,j,n \in \Z$, the weights have distribution as in \eqref{eq:n_level_parameters}.
\end{cor}
\begin{proof}
    Follows directly from \Cref{thm:independent_implies_gamma} and \Cref{thm:gamma_to_gamma}.
\end{proof}

\begin{rmk}
    \label{rmk:need_full_plane}
    Rather than passing to a full $\Z^2$ worth of weights, one might ask if simply requiring independence of the weights $\{a_{i,j}^{[n]},b_{i,j}^{[n]} : 0 \leq i,j \leq n-1\}$ and $\{a_{i,j}^{[n-1]},b_{i,j}^{[n-1]} : 0 \leq i,j \leq n-2\}$ at two consecutive steps of the shuffling algorithm for an actual finite-size Aztec diamond is enough to conclude that all weights are $\Gamma$ variables. This fails: for instance, if one chooses any independent weights for an Aztec diamond of size two, then the corresponding weights
    \begin{align}
        \begin{split}
            a_{0,0}^{[1]} &= \frac{a_{0,0}^{[2]}}{a_{0,0}^{[2]}+b_{0,0}^{[2]}}(a_{1,0}^{[2]}+b_{1,0}^{[2]}) \\ 
            b_{0,0}^{[1]} &= \frac{b_{0,1}^{[2]}}{a_{0,1}^{[2]}+b_{0,1}^{[2]}}(a_{1,1}^{[2]}+b_{1,1}^{[2]})
        \end{split}
    \end{align}
    are independent simply because they depend on different variables. In general, one can show that the subset of weights $a_{i,j}^{[n]},b_{i,j}^{[n]}$ for $0 \leq i \leq n-2$ and $1 \leq j \leq n-2$ are $\Gamma$-distributed if they are independent and the corresponding weights of the size $n-1$ Aztec diamond are independent, but cannot conclude this for all weights.
\end{rmk}

\begin{rmk}\label{rmk:2_steps_insufficient}
    Similarly, since \Cref{thm:3-layer_implies_all-layer} requires independence of weights at $3$ steps of the shuffling algorithm, one might ask if only two steps suffice. While this is enough to obtain that the weights are Gamma variables distributed by \eqref{eq:[0]_vars_AB}, if the parameters $s_{i,j}$ there actually depend nontrivially on $j$ then the weights will become dependent for $n \geq 2$.
\end{rmk}

As mentioned in the Introduction, \Cref{thm:3-layer_implies_all-layer} makes it natural to define the multi-parameter Gamma weights of \Cref{def:gamma_weights_intro_general} on the Aztec diamond of size $n$. Before ending this section, we prove the key independence of factors in the partition function for these weights, which will be used to prove asymptotic results on free energy in \Cref{sec:free_energy}. It also comes as a corollary of some of the more involved graph manipulations in \Cref{sec:vert_polymer}, and we give a second proof there, but for the sake of exposition we give a direct proof here as well.

\begin{cor}\label{thm:compute_Z_cor}
    The partition function of an Aztec diamond of size $n$, with weights $a_{i,j}^{[n]},b_{i,j}^{[n]}, 1 \leq i,j \leq n$ as in \Cref{fig:our_weights} is 
    \begin{equation}\label{eq:partition_fn}
         Z_{G_n^{\Az}}  = \prod_{k=1}^n \prod_{j=1}^{k} (a_{1,j}^{[k]} + b_{1,j}^{[k]})
    \end{equation}
    Furthermore, if the weights $a_{i,j}^{[n]},b_{i,j}^{[n]}, 1 \leq i,j \leq n$ are Gamma variables with parameters as in \Cref{def:gamma_weights_intro_general}, then the random variables $a_{1,j}^{[k]} + b_{1,j}^{[k]}$ appearing in \eqref{eq:partition_fn} are mutually independent, with distributions
    \begin{equation}\label{eq:Z_product_marginals}
        a_{1,j}^{[k]} + b_{1,j}^{[k]} \sim \Gamma(\psi_j+\phi_{j-k},1).
    \end{equation}
\end{cor}

\begin{proof}
    The explicit form \eqref{eq:partition_fn} follows immediately from \Cref{thm:Z_recurrence}, and the marginal distributions \eqref{eq:Z_product_marginals} likewise follow from \Cref{thm:gamma_to_gamma}, so it suffices to show they are mutually independent. By \Cref{thm:gamma_to_gamma}, for each fixed $k$ the collection of variables $\{a_{1,j}^{[k]}+b_{1,j}^{[k]}: 1 \leq j \leq k\}$ are mutually independent, but \emph{a priori} one would expect that there exist dependencies between them for different $k$. To show that they do not, we will prove something stronger:
    
    \noindent \textbf{Claim:} For each $k$, the weights $\{a_{i,j}^{[k]},b_{i,j}^{[k]}: 1 \leq i,j \leq k\}$ are independent of $\{a_{1,j}^{[k+1]}+b_{1,j}^{[k+1]}: 1 \leq j \leq k+1\}$.

    By \eqref{eq:downshuffle_weights}, the only weights in the collection $\{a_{i,j}^{[k]},b_{i,j}^{[k]}: 1 \leq i,j \leq k\}$ which might depend on weights $a_{1,j}^{[k+1]}$ or $b_{1,j}^{[k+1]}$ at all are those with $i=1$, since all others are expressed in terms of $a_{i,j}^{[k+1]}$s and $b_{i,j}^{[k+1]}$s with $i > 1$. However,
    \begin{align}
    \begin{split}
        a_{1,j}^{[k-1]} &= \frac{a_{1,j}^{[k]}}{a_{1,j}^{[k]}+b_{1,j}^{[k]}}(a_{1+1,j}^{[k]}+b_{1+1,j}^{[k]}) \\ 
        b_{1,j}^{[k-1]} &= \frac{b_{1,j+1}^{[k]}}{a_{1,j+1}^{[k]}+b_{1,j+1}^{[k]}}(a_{1+1,j+1}^{[k]}+b_{1+1,j+1}^{[k]}),
    \end{split}
\end{align}
    and $\frac{a_{1,j}^{[k]}}{a_{1,j}^{[k]}+b_{1,j}^{[k]}}$ is independent of $a_{1,j}^{[k]}+b_{1,j}^{[k]}$ by \Cref{thm:lukacs}. This proves the claim.

    Because all weights $\{a_{1,j}^{[k]}+b_{1,j}^{[k]}: 1 \leq j \leq k, 1 \leq k \leq n-1\}$ can be expressed in terms of the weights $\{a_{i,j}^{[n-1]},b_{i,j}^{[n-1]}: 1 \leq i,j \leq n-1\}$, our claim (with $k=n-1$) immediately yields that all variables $\{(a_{1,j}^{[k]} + b_{1,j}^{[k]}): 1 \leq k \leq n-1, 1 \leq j \leq k\}$ are independent of $\{(a_{1,j}^{[n]} + b_{1,j}^{[n]}): 1 \leq j \leq n\}$. Similarly, all variables $\{(a_{1,j}^{[k]} + b_{1,j}^{[k]}): 1 \leq k \leq n-2, 1 \leq j \leq k\}$ are independent of $\{(a_{1,j}^{[n-1]} + b_{1,j}^{[n-1]}): 1 \leq j \leq n\}$, and so on; this completes the proof.
\end{proof}

\section{The $\beta$-$\Gamma$ polymer and matching along a vertical slice} \label{sec:vert_polymer}

Having singled out the Gamma-disordered Aztec diamond as special from the perspective of shuffling in the previous section, we begin investigating its other properties. The main result of this section is \Cref{thm:vert_slice_polymer}, which is a general version of \Cref{thm:multi-path_intro_vert} for the weights of \Cref{def:gamma_weights_intro_general} which arose in the previous section. Because the parameters $s_i$ do not affect the dimer measure, we will set them to $1$ throughout this section and the rest of the paper except for \Cref{sec:free_energy} and \Cref{appendix:deterministic}.

The structure of the section is as follows. In \Cref{subsec:aztec_to_bgbsw}, we show a distributional equality (\Cref{thm:aztec_to_bgbsw}) between the dimer measure on $G_n^{\Az}$ on the $\ell\tth$ vertical slice and the dimer measure on a related graph $\bG_{n,\ell}^{v-swap}$, obtained from the former by a sequence of local moves. In \Cref{subsec:distrbeta} we relate dimer covers on the latter graph to the $\beta$-$\Gamma$ polymer (\Cref{thm:vert_slice_polymer_det_weights}).

One feature of these results is that we obtain not just a distributional equality between matchings of the Aztec diamond and an independent $\beta$-$\Gamma$ polymer paths, but also an explicit mapping from the weights of the Aztec diamond to the weights in the $\beta$-$\Gamma$ polymer.

\subsection{Distributional equalities between the Aztec diamond and related graphs} \label{subsec:aztec_to_bgbsw} The local urban renewal move (\Cref{thm:spider_coupling}), together with the more trivial vertex-expansion move (\Cref{thm:vertex_expansion}), gives distributional equalities between dimer models on related graphs. The goal of this subsection is to prove \Cref{thm:aztec_to_bgbsw}, which gives such a distributional equality between the Aztec diamond and a dimer model on a certain graph $\bgbsw_{n,\ell}$ which is transparently related to the $\beta$-$\Gamma$ polymer (as we will see next). Along the way, by keeping track of the independence properties of random weights under these transformations, we also give another proof of the independence statement in \Cref{thm:compute_Z_cor} from earlier.

We first alter the Aztec diamond slightly, defining a related graph $G_n^{vert}$.

\begin{lemma}\label{thm:tG}
    Denote by $G_n^{vert}$ the graph as in \Cref{fig:tilde_aztec}, with weights related to those on the original Aztec graph $G_n^{\Az}$ (recall from \Cref{fig:our_weights}) as shown there. Then there is a weight-preserving bijection between perfect matchings of the two graphs, which furthermore does not change the edges surrounding the square faces\footnote{The correspondence between the square faces on $G_n^{\Az}$ and $G_n^{vert}$ is clear.}.
\end{lemma}

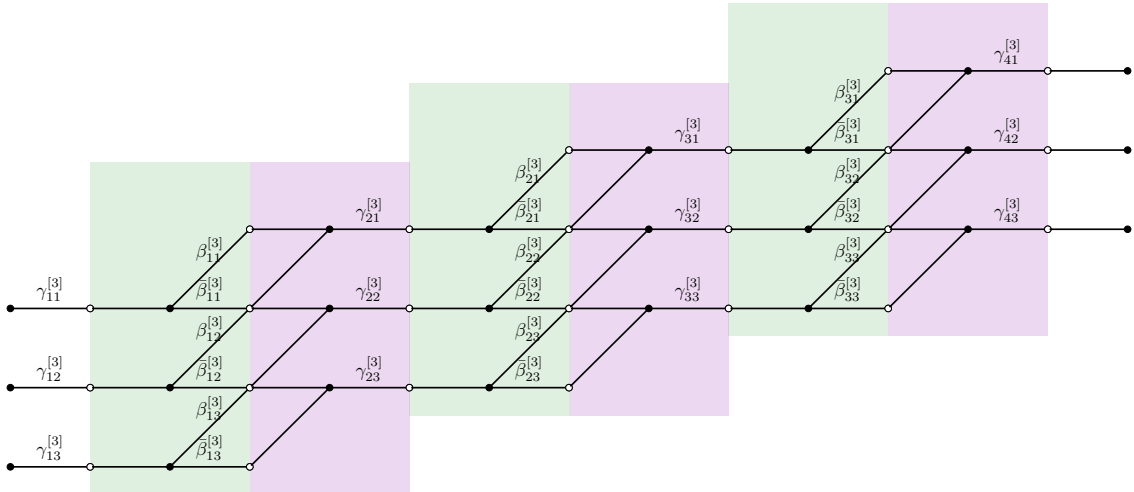
\begin{figure}
        \centering
        \scalebox{.7}{
        \begin{tikzpicture}[scale=1.5,
    whiteV/.style={circle,draw,fill=white,inner sep=1.25pt},
    blackV/.style={circle,draw,fill=black,inner sep=1.25pt},
    edge/.style={line width=0.9pt, shorten <=1pt, shorten >=1pt}
]

\draw (0.5,0) node[anchor=south] {\small{$\gamma_{11}^{[3]}$}};
\draw (0.5,-1) node[anchor=south] {\small{$\gamma_{12}^{[3]}$}};
\draw (0.5,-2) node[anchor=south] {\small{$\gamma_{13}^{[3]}$}};

\draw (2.5,0) node[anchor=south] {\small{$\bar \beta_{11}^{[3]}$}};
\draw (2.5,-1) node[anchor=south] {\small{$\bar \beta_{12}^{[3]}$}};
\draw (2.5,-2) node[anchor=south] {\small{$\bar \beta_{13}^{[3]}$}};

\draw (2.5,0.5) node[anchor=south] {\small{$\beta_{11}^{[3]}$}};
\draw (2.5,-.5) node[anchor=south] {\small{$\beta_{12}^{[3]}$}};
\draw (2.5,-1.5) node[anchor=south] {\small{$\beta_{13}^{[3]}$}};

\draw (4.5,1) node[anchor=south] {\small{$\gamma_{21}^{[3]}$}};
\draw (4.5,0) node[anchor=south] {\small{$\gamma_{22}^{[3]}$}};
\draw (4.5,-1) node[anchor=south] {\small{$\gamma_{23}^{[3]}$}};

\draw (6.5,1) node[anchor=south] {\small{$\bar \beta_{21}^{[3]}$}};
\draw (6.5,0) node[anchor=south] {\small{$\bar\beta_{22}^{[3]}$}};
\draw (6.5,-1) node[anchor=south] {\small{$\bar \beta_{23}^{[3]}$}};

\draw (6.5,1.5) node[anchor=south] {\small{$\beta_{21}^{[3]}$}};
\draw (6.5,.5) node[anchor=south] {\small{$\beta_{22}^{[3]}$}};
\draw (6.5,-.5) node[anchor=south] {\small{$\beta_{23}^{[3]}$}};

\draw (8.5,2) node[anchor=south] {\small{$\gamma_{31}^{[3]}$}};
\draw (8.5,1) node[anchor=south] {\small{$\gamma_{32}^{[3]}$}};
\draw (8.5,0) node[anchor=south] {\small{$\gamma_{33}^{[3]}$}};

\draw (10.5,2) node[anchor=south] {\small{$\bar \beta_{31}^{[3]}$}};
\draw (10.5,1) node[anchor=south] {\small{$\bar \beta_{32}^{[3]}$}};
\draw (10.5,0) node[anchor=south] {\small{$\bar \beta_{33}^{[3]}$}};

\draw (10.5,2.5) node[anchor=south] {\small{$\beta_{31}^{[3]}$}};
\draw (10.5,1.5) node[anchor=south] {\small{$\beta_{32}^{[3]}$}};
\draw (10.5,.5) node[anchor=south] {\small{$\beta_{33}^{[3]}$}};

\draw (12.5,3) node[anchor=south] {\small{$\gamma_{41}^{[3]}$}};
\draw (12.5,2) node[anchor=south] {\small{$\gamma_{42}^{[3]}$}};
\draw (12.5,1) node[anchor=south] {\small{$\gamma_{43}^{[3]}$}};
\def\n{3}

\definecolor{pluscol}{RGB}{76,175,80}   
\definecolor{minuscol}{RGB}{156,39,176} 
\tikzset{
  colPlus/.style={fill=pluscol,  opacity=0.18},
  colMinus/.style={fill=minuscol,opacity=0.18}
}
\pgfmathsetmacro{\ypad}{0.35} 

\foreach \y in {0,...,\numexpr-\n+1\relax}{
  \node[blackV] at (0,\y) {};
  \node[whiteV] at (1,\y) {};
  \draw[edge] (0,\y) -- (1,\y);
}

\foreach \p in {0,...,\numexpr\n-1\relax}{

  \pgfmathsetmacro{\xplus}{1+4*\p}        
  \pgfmathtruncatemacro{\top}{\p}         

  \pgfmathtruncatemacro{\ybottomplus}{\top - (\n - 1)}
  \pgfmathtruncatemacro{\ytopplus}{\top + 1}
  \begin{scope}[on background layer]
    \fill[colPlus] (\xplus, \ybottomplus - \ypad) rectangle (\xplus + 2, \ytopplus + \ypad+.5);
  \end{scope}

  \foreach \i in {0,...,\numexpr\n-1\relax}{
    \pgfmathtruncatemacro{\y}{\top - \i}
    \node[whiteV] at (\xplus,\y) {};
  }

  \pgfmathtruncatemacro{\topplus}{\top + 1}
  \node[whiteV] at ({\xplus+2},\topplus) {};

  \foreach \i in {0,...,\numexpr\n-1\relax}{
    \pgfmathtruncatemacro{\y}{\top - \i}
    \node[blackV] at ({\xplus+1},\y) {};
    \node[whiteV] at ({\xplus+2},\y) {};
    \pgfmathtruncatemacro{\yup}{\y + 1}
    \draw[edge] (\xplus,\y) -- ({\xplus+1},\y);          
    \draw[edge] ({\xplus+1},\y) -- ({\xplus+2},\y);       
    \draw[edge] ({\xplus+1},\y) -- ({\xplus+2},{\yup});   
  }

  \pgfmathsetmacro{\xminus}{3+4*\p}       
  \pgfmathtruncatemacro{\tminus}{\p + 1}  

  \pgfmathtruncatemacro{\ytopminus}{\tminus}
  \pgfmathtruncatemacro{\ybottomminus}{\tminus - \n}
  \begin{scope}[on background layer]
    \fill[colMinus] (\xminus, \ybottomminus - \ypad) rectangle (\xminus + 2, \ytopminus + \ypad+.5);
  \end{scope}

  \foreach \i in {0,...,\numexpr\n\relax}{
    \pgfmathtruncatemacro{\y}{\tminus - \i}
    \node[whiteV] at (\xminus,\y) {};
  }

  \foreach \i in {0,...,\numexpr\n-1\relax}{
    \pgfmathtruncatemacro{\y}{\tminus - \i}
    \pgfmathtruncatemacro{\ybelow}{\tminus - (\i + 1)}
    \node[blackV] at ({\xminus+1},\y) {};
    \node[whiteV] at ({\xminus+2},\y) {};
    \draw[edge] (\xminus,\y) -- ({\xminus+1},\y);        
    \draw[edge] (\xminus,\ybelow) -- ({\xminus+1},\y);   
    \draw[edge] ({\xminus+1},\y) -- ({\xminus+2},\y);    
  }

} 

\pgfmathtruncatemacro{\xright}{4*\n + 1}   
\pgfmathtruncatemacro{\xcap}{4*\n + 2}     

\foreach \k in {1,...,\numexpr\n\relax}{
  \node[whiteV] at (\xright,\k) {};
  \node[blackV] at (\xcap,\k) {};
  \draw[edge] (\xright,\k) -- (\xcap,\k);
}
\end{tikzpicture}
}
        \caption{The graph $G_n^{vert}$ for $n=3$. Here we use compressed notation $\gamma_{i,j}^{[n]} = a_{i,j}^{[n]}+b_{i,j}^{[n]}, \beta_{i,j}^{[n]} = \frac{a_{i,j}^{[n]}}{a_{i,j}^{[n]}+b_{i,j}^{[n]}}$ and $\bar \beta_{ij}^{[n]}=1-\beta_{ij}^{[n]}.$}
        \label{fig:tilde_aztec}
\end{figure}

\begin{proof}
    First apply \Cref{thm:vertex_expansion} to expand each black vertex into a pair of black vertices connected by a white one. At the rightmost black vertex in each such pair, gauge transform by $1/(a_{i,j}^{[n]}+b_{i,j}^{[n]})$ (with $i,j$ as appropriate), and then gauge transform by $a_{i,j}^{[n]}+b_{i,j}^{[n]}$ at the white vertex immediately to its left\footnote{Note that for the rightmost vertices, this gauge transform adds new weights $a_{n+1,j}^{[n]}$ which were not present in the original Aztec diamond. However, every matching of the resulting graph does not include the edges with these weights, so the partition function and dimer measure are still independent of them---we simply add them as a notational convenience to make this and later constructions more uniform and clear.}. After a cosmetic vertical shear, the resulting graph is as in \Cref{fig:tilde_aztec}.
    
    
    The vertex-expansion moves do not change the partition function, by \Cref{thm:vertex_expansion}, and furthermore by the same result there is a bijection between dimer covers of $G_n^{\Az}$ and $G_n^{vert}$ which fixes the edges around the square faces. While gauge transforms do change the partition function, the $1/(a_{i,j}^{[n]}+b_{i,j}^{[n]})$ and $a_{i,j}^{[n]}+b_{i,j}^{[n]}$ factors cancel. 
\end{proof}

Let us now treat the graph $G_n^{vert}$ a bit more formally; for the rest of this section, all graphs will be drawn such that the vertex sets are subsets of $\Z^2$ and all edges are either horizontal edges from $(x,y)$ to $(x+1,y)$ or up-right edges from $(x,y)$ to $(x+1,y+1)$. Looking at $G_n^{vert}$, one can note that it is composed naturally of columns as in the following definition.

\begin{defi}\label{def:plus_and_minus_cols}
    Given a bipartite graph $G$ drawn on $\Z^2$ as above, we refer to a subgraph as in \Cref{fig:plus_and_minus_cols} (left) as a \emph{$(+)$-column}. In other words, it has white vertices at $(x,y-1),(x,y-2),\ldots,(x,y-m)$, connected horizontally to black vertices at $(x+1,y-1),(x+1,y-2),\ldots,(x+1,y-m)$, then each of these connected to the white vertices at $(x+2,y),(x+2,y-1),\ldots,(x+2,y-m)$ by both horizontal and up-right edges. If the weights on adjacent pairs of right and up-right edges are given by $\beta_i$ and $1-\beta_i$ as in the figure, we will sometimes refer to it as a \emph{$(+)$-column with weights $\beta_1,\ldots,\beta_m$} when emphasizing the weights. We allow the degenerate case $m=0$, for which the $(+)$-column is just a single white vertex at $(x+2,y)$ with no black vertices.

    We refer to a subgraph as in \Cref{fig:plus_and_minus_cols} (right), i.e. white vertices at $(x,y),(x,y-1),\ldots,(x,y-m)$ connected horizontally and up-right to black vertices at $(x+1,y),(x+1,y-1),\ldots,(x+1,y-m+1)$, then these connected horizontally to white vertices at $(x+2,y),(x+2,y-1),\ldots,(x+2,y-m+1)$, as a \emph{$(-)$-column}. If the weights on the black-to-white edges are $\gamma_1,\ldots,\gamma_{m}$ as shown, we will sometimes refer to it as a \emph{$(-)$-column with weights $\gamma_1,\ldots,\gamma_m$} to emphasize the weights. We allow the degenerate case $m=0$, for which the $(-)$-column is just a single white vertex at $(x,y)$.
\end{defi}

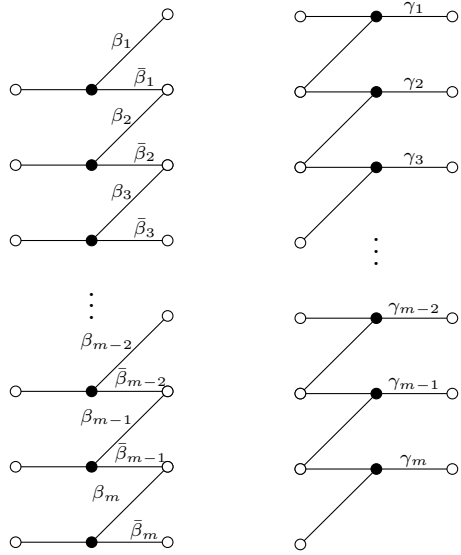
\begin{figure}[t]
 \begin{tikzpicture}[scale=1]
  \foreach \y in {0,...,2} {
            \draw (0,-\y)--(2,-\y);
              \draw (1,-\y)--(2,-\y+1);
  }
   \foreach \y in {1,...,2} {
            \draw[fill=white] (0,-\y) circle(2pt); 
             \draw[fill=black] (1,-\y) circle(2pt); 
             \draw[fill=white] (2,-\y) circle(2pt); 
              \draw[fill=white] (2,-\y+1) circle(2pt); 
              \draw (1.65,-2+\y-.1) node[anchor=south] {\tiny{$\bar{\beta}_{m-\y}$}};
              \draw (1.2,-2+\y+.4) node[anchor=south] {\tiny{$\beta_{m-\y}$}};
              }
  
   \foreach \y in {0} {
            \draw[fill=white] (0,-\y) circle(2pt); 
             \draw[fill=black] (1,-\y) circle(2pt); 
             \draw[fill=white] (2,-\y) circle(2pt); 
              \draw[fill=white] (2,-\y+1) circle(2pt); 
              \draw (1.7,-2+\y-.1) node[anchor=south] {\tiny{$\bar{\beta}_{m}$}};
             \draw (1.2,-2+\y+.4) node[anchor=south] {\tiny{$\beta_{m}$}};
              }
               \foreach \y in {1,...,3} {
            \draw (0,5-\y)--(2,5-\y);
              \draw (1,5-\y)--(2,-\y+6);}
               \foreach \y in {1,...,3} {
            \draw[fill=white] (0,-\y+5) circle(2pt); 
             \draw[fill=black] (1,-\y+5) circle(2pt); 
             \draw[fill=white] (2,-\y+5) circle(2pt); 
              \draw[fill=white] (2,-\y+6) circle(2pt); 
                 \draw (1.7,-\y+4.9) node[anchor=south] {\tiny{$\bar{\beta}_{\y}$}};
              \draw (1.4,-\y+5.4) node[anchor=south] {\tiny{$\beta_{\y}$}};
              }
               
      \draw(1,1.25) node {$\vdots$};

\end{tikzpicture}\qquad \qquad
\begin{tikzpicture}
  \foreach \y in {1,...,3} {
            \draw (2,-\y+2)--(4,-\y+2);
              \draw (2,-\y+1)--(3,-\y+2);
  }
   \foreach \y in {1,...,2} {
            \draw[fill=white] (2,-\y) circle(2pt); 
             \draw[fill=black] (3,-\y+1) circle(2pt); 
             \draw[fill=white] (2,-\y+1) circle(2pt); 
              \draw[fill=white] (4,-\y+1) circle(2pt); 
              \draw (3.5,-2+\y+.9) node[anchor=south] {\tiny{$\gamma_{m-\y}$}};
              }
               \foreach \y in {0} {
            \draw[fill=white] (2,-\y) circle(2pt); 
             \draw[fill=black] (3,-\y+1) circle(2pt); 
             \draw[fill=white] (2,-\y+1) circle(2pt); 
              \draw[fill=white] (4,-\y+1) circle(2pt); 
              \draw (3.5,-2+\y+.9) node[anchor=south] {\tiny{$\gamma_{m}$}};
              }

               \foreach \y in {1,...,3} {
            \draw (2,5-\y+1)--(4,5-\y+1);
              \draw (2,4-\y+1)--(3,-\y+6);}
               \foreach \y in {1,...,3} {
            \draw[fill=white] (2,-\y+5) circle(2pt); 
             \draw[fill=black] (3,-\y+6) circle(2pt); 
             \draw[fill=white] (2,-\y+6) circle(2pt); 
              \draw[fill=white] (4,-\y+6) circle(2pt); 
              \draw (3.5,-\y+5.9) node[anchor=south] {\tiny{$\gamma_{\y}$}};
              }
               
      \draw(3,2) node {$\vdots$};
        \end{tikzpicture}
        \caption{A $(+)$-column (left) and a $(-)$-column (right). To declutter notation we write $\bar{\beta}_j=1-\beta_j$.}
      
        \label{fig:plus_and_minus_cols}
\end{figure}

\begin{defi}\label{def:tG}
    Denote by $G_n^{vert}$ the weighted graph with vertex set contained in $\Z^2$ defined as follows. It has black vertices at $(0,0),(0,-1),\ldots,(0,-n+1)$ which are connected to white vertices at $(1,0),(1,-1),\ldots,(1,-n+1)$ by edges with weights $a_{1,1}^{[n]}+b_{1,1}^{[n]},\ldots,a_{1,n}^{[n]}+b_{1,n}^{[n]}$ (top to bottom). These white vertices then form the leftmost white vertices in a $(+)$-column with weights 
    $$\frac{a_{1,1}^{[n]}}{a_{1,1}^{[n]}+b_{1,1}^{[n]}},\ldots,\frac{a_{1,n}^{[n]}}{a_{1,n}^{[n]}+b_{1,n}^{[n]}}.$$
    The rightmost white vertices of this $(+)$-column are the leftmost white vertices of a $(-)$-column with weights $a_{2,1}^{[n]}+b_{2,1}^{[n]},\ldots,a_{2,n}^{[n]}+b_{2,n}^{[n]}$. The alternating $(+)$- and $(-)$-columns with weights
    $$\frac{a_{i,1}^{[n]}}{a_{i,1}^{[n]}+b_{i,1}^{[n]}},\ldots,\frac{a_{i,n}^{[n]}}{a_{i,n}^{[n]}+b_{i,n}^{[n]}}$$
    for the $(+)$-column and $a_{i+1,1}^{[n]}+b_{i+1,1}^{[n]},\ldots,a_{i+1,n}^{[n]}+b_{i+1,n}^{[n]}$ repeat until $i=n$, at which point there are horizontal edges connecting the rightmost white vertices to $n$ black vertices at coordinates $(4n+2,1),\ldots,(4n+2,n)$.
\end{defi}

The reason why \Cref{def:plus_and_minus_cols} is a useful definition is the following lemma. This is also why we have added the pendant edges on the right and left of $G_n^{vert}$, which might have seemed strange above, as they allow us to write the graph as an alternating sequence of $(+)$- and $(-)$-columns sandwiched between two `trivial' strips of black and white vertices connected by horizontal edges.

\begin{lemma}\label{thm:swap_columns_dimer}
Let $m \geq 1$ and $G$ be a weighted bipartite graph containing a local configuration with a $(+)$ and then a $(-)$-column with weights $\beta_1,\ldots,\beta_m$ and $\gamma_1,\ldots,\gamma_m$ respectively as in the left hand side of \Cref{fig:swap_cols}. Let $G'$ be the graph where the local configuration is replaced by the one on the right hand side with a $(-)$ and then a $(+)$-column with weights $\hgamma_1,\ldots,\hgamma_{m-1}$ and $\hbeta_1,\ldots,\hbeta_{m-1}$ respectively\footnote{In the case $m=1$, $G'$ will have a degenerate $(-)$-column and a degenerate $(+)$-column, so no updated weights are necessary.}, related to the initial weights via 
\begin{align}
\begin{split}
\hgamma_j &:= \beta_j \gamma_j + (1-\beta_{j+1})\gamma_{j+1} \\ 
\hbeta_j &:= \frac{\beta_j\gamma_j}{\beta_j \gamma_j + (1-\beta_{j+1})\gamma_{j+1}}.
\end{split}
\end{align}
Then the marginal distribution (with respect to the dimer measure) of edges outside this local configuration is the same for both $G$ and $G'$. Additionally, both graphs have the same partition function.
\end{lemma}

    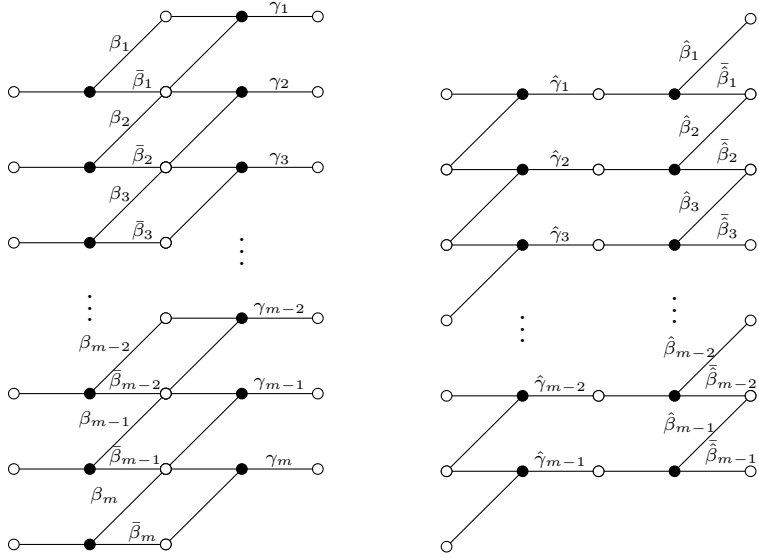
\begin{figure}

        \centering
        \begin{tikzpicture}[scale=1]
  \foreach \y in {0,...,2} {
            \draw (0,-\y)--(2,-\y);
              \draw (1,-\y)--(2,-\y+1);
  }
   \foreach \y in {1,...,2} {
            \draw[fill=white] (0,-\y) circle(2pt); 
             \draw[fill=black] (1,-\y) circle(2pt); 
             \draw[fill=white] (2,-\y) circle(2pt); 
              \draw[fill=white] (2,-\y+1) circle(2pt); 
              \draw (1.6,-2+\y-.1) node[anchor=south] {\tiny{$\bar{\beta}_{m-\y}$}};
              \draw (1.2,-2+\y+.4) node[anchor=south] {\tiny{$\beta_{m-\y}$}};
              }
  
   \foreach \y in {0} {
            \draw[fill=white] (0,-\y) circle(2pt); 
             \draw[fill=black] (1,-\y) circle(2pt); 
             \draw[fill=white] (2,-\y) circle(2pt); 
              \draw[fill=white] (2,-\y+1) circle(2pt); 
              \draw (1.7,-2+\y-.1) node[anchor=south] {\tiny{$\bar{\beta}_{m}$}};
             \draw (1.2,-2+\y+.4) node[anchor=south] {\tiny{$\beta_{m}$}};
              }
               \foreach \y in {1,...,3} {
            \draw (0,5-\y)--(2,5-\y);
              \draw (1,5-\y)--(2,-\y+6);}
               \foreach \y in {1,...,3} {
            \draw[fill=white] (0,-\y+5) circle(2pt); 
             \draw[fill=black] (1,-\y+5) circle(2pt); 
             \draw[fill=white] (2,-\y+5) circle(2pt); 
              \draw[fill=white] (2,-\y+5) circle(2pt); 
                 \draw (1.7,-\y+4.9) node[anchor=south] {\tiny{$\bar{\beta}_{\y}$}};
              \draw (1.4,-\y+5.4) node[anchor=south] {\tiny{$\beta_{\y}$}};
              }
               
      \draw(1,1.25) node {$\vdots$};

  \foreach \y in {1,...,3} {
            \draw (2,-\y+2)--(4,-\y+2);
              \draw (2,-\y+1)--(3,-\y+2);
  }
   \foreach \y in {1,...,2} {
            \draw[fill=white] (2,-\y+1) circle(2pt); 
             \draw[fill=black] (3,-\y+1) circle(2pt); 
             \draw[fill=white] (2,-\y) circle(2pt); 
              \draw[fill=white] (4,-\y+1) circle(2pt); 
              \draw (3.5,-2+\y+.9) node[anchor=south] {\tiny{$\gamma_{m-\y}$}};
              }
               \foreach \y in {0} {
            \draw[fill=white] (2,-\y) circle(2pt); 
             \draw[fill=black] (3,-\y+1) circle(2pt); 
             \draw[fill=white] (2,-\y+1) circle(2pt); 
              \draw[fill=white] (4,-\y+1) circle(2pt); 
              \draw (3.5,-2+\y+.9) node[anchor=south] {\tiny{$\gamma_{m}$}};
              }

               \foreach \y in {1,...,3} {
            \draw (2,5-\y+1)--(4,5-\y+1);
              \draw (2,4-\y+1)--(3,-\y+6);}
               \foreach \y in {1,...,3} {
            \draw[fill=white] (2,-\y+5) circle(2pt); 
             \draw[fill=black] (3,-\y+6) circle(2pt); 
             \draw[fill=white] (2,-\y+6) circle(2pt); 
              \draw[fill=white] (4,-\y+6) circle(2pt); 
              \draw (3.5,-\y+5.9) node[anchor=south] {\tiny{$\gamma_{\y}$}};
              }
               
      \draw(3,2) node {$\vdots$};
        \end{tikzpicture}\qquad \qquad
         \begin{tikzpicture}[scale=1] 
  \foreach \y in {0,...,1} {
            \draw (2,-\y+1)--(4,-\y+1);
              \draw (3,-\y+1)--(4,-\y+2);
  }
   \foreach \y in {1,...,2} {
            \draw[fill=white] (2,-\y+2) circle(2pt); 
             \draw[fill=black] (3,-\y+2) circle(2pt); 
             \draw[fill=white] (4,-\y+2) circle(2pt); 
              \draw[fill=white] (4,-\y+3) circle(2pt); 
              \draw (3.75,\y-0.06-1) node[anchor=south] {\tiny{$\bar{\hat{\beta}}_{m-\y}$}};
              \draw (3.2,\y-.65) node[anchor=south] {\tiny{$\hat{\beta}_{m-\y}$}};
              }

               \foreach \y in {0,...,2} {
            \draw (2,5-\y)--(4,5-\y);
              \draw (3,5-\y)--(4,-\y+6);}
               \foreach \y in {1,...,3} {
            \draw[fill=white] (2,-\y+6) circle(2pt); 
             \draw[fill=black] (3,-\y+6) circle(2pt); 
             \draw[fill=white] (4,-\y+6) circle(2pt); 
              \draw[fill=white] (4,-\y+7) circle(2pt); 
                 \draw (3.7,-\y+5.95) node[anchor=south] {\tiny{$\bar{\hat {\beta}}_{\y}$}};
              \draw (3.2,-\y+6.3) node[anchor=south] {\tiny{$\hat{ \beta}_{\y}$}};
              }
             
      \draw(1,2) node {$\vdots$};

  \foreach \y in {0,...,1} {
            \draw (0,-\y+1)--(2,-\y+1);
              \draw (0,-\y)--(1,-\y+1);
  }
   \foreach \y in {1,...,2} {
            \draw[fill=white] (0,-\y+1) circle(2pt); 
             \draw[fill=black] (1,-\y+2) circle(2pt); 
             \draw[fill=white] (0,-\y+2) circle(2pt); 
              \draw[fill=white] (2,-\y+2) circle(2pt); 
              \draw (1.5,\y-1-.1) node[anchor=south] {\tiny{$\hat{\gamma}_{m-\y}$}};
              }
               
               \foreach \y in {1,...,2} {
            \draw (0,5-\y)--(2,5-\y);
              \draw (0,4-\y)--(1,-\y+5);}
               \foreach \y in {1,...,2} {
            \draw[fill=white] (0,-\y+4) circle(2pt); 
             \draw[fill=black] (1,-\y+5) circle(2pt); 
             \draw[fill=white] (0,-\y+5) circle(2pt); 
              \draw[fill=white] (2,-\y+5) circle(2pt); 
              \draw (1.5,-\y+5.9) node[anchor=south] {\tiny{$\hat{\gamma}_{\y}$}};
              }
                \foreach \y in {3} {
            \draw (0,2+\y)--(2,2+\y);
              \draw (0,1+\y)--(1,\y+2);}
                \foreach \y in {3} {
            \draw[fill=white] (0,\y+1) circle(2pt); 
             \draw[fill=black] (1,\y+2) circle(2pt); 
             \draw[fill=white] (0,\y+2) circle(2pt); 
              \draw[fill=white] (2,\y+2) circle(2pt); 
             \draw (1.5,\y-.1) node[anchor=south] {\tiny{$\hat{\gamma}_{3}$}};
              }
      \draw(3,2.25) node {$\vdots$};

        \end{tikzpicture}
        
        \caption{The portion of $G$ (left) and $G'$ (right) which is changed under \Cref{thm:swap_columns_dimer}. Here we use shorthand $\bar{\beta}_j = 1-\beta_j$ to declutter notation, and edges without weights pictured have weight $1$.}   \label{fig:swap_cols}

\end{figure}

\begin{proof}
First apply urban renewal (\Cref{thm:spider_coupling}) inside each square. The top and bottom of the graph now have $1$-valent white vertices, so remove these together with the black vertex they are incident to. Now gauge transform by $\gamma_j^{-1}$ at the black vertex with the $\gamma_j$ edge incident, and then gauge transform by $\gamma_j$ at the white vertex to its left, for each $j$. Finally, contract all $2$-valent vertices with both incident edges having weight $1$. The result is the graph on the right hand side of \Cref{fig:swap_cols}. Furthermore, the partition function is unchanged, since each urban renewal move changed it by $\beta_j \cdot 1 + (1-\beta_j) \cdot 1 = 1$ and the contribution of each gauge transform by $\gamma_j$ was canceled by the gauge transform by $\gamma_j^{-1}$.
\end{proof}

\begin{lemma}\label{thm:swap_weights_stay_independent}
    In the setup of \Cref{thm:swap_columns_dimer}, suppose the weights $\beta_1,\ldots,\beta_m$ and $\gamma_1,\ldots,\gamma_m$ are mutually independent random variables distributed as 
\begin{align}
\gamma_j &\sim \Gamma(x_j+y_j,1) \\ 
\beta_j &\sim \Beta(x_j,y_j)
\end{align}
for some positive real parameters $x_1,\ldots,x_m,y_1,\ldots,y_m$. Then the weights $\hgamma_1,\ldots,\hgamma_{m-1}$ and $\hbeta_1,\ldots,\hbeta_{m-1}$ are mutually independent with distributions
\begin{align}
\hgamma_j &\sim \Gamma(x_j+y_{j+1},1) \\ 
\hbeta_j &\sim \Beta(x_j,y_{j+1})
\end{align}
\end{lemma}
\begin{proof}
First, by \Cref{thm:XY_lukacs_cor} the collection of random variables $\{\beta_j \gamma_j, (1-\beta_j)\gamma_j: 1 \leq j \leq m\}$ are mutually independent with distributions
\begin{align}
\beta_j   \gamma_j &\sim \Gamma(x_j,1) \\ 
(1-\beta_j)\gamma_j &\sim \Gamma(y_j,1).
\end{align}
It follows by \Cref{thm:lukacs} that the random variables
\begin{align*}
\begin{split}
\hgamma_j &:= \beta_j \gamma_j + (1-\beta_{j+1})\gamma_{j+1} \\ 
\hbeta_j &:= \frac{\beta_j\gamma_j}{\beta_j \gamma_j + (1-\beta_{j+1})\gamma_{j+1}}
\end{split}
\end{align*}
for $j=1,\ldots,m-1$ are mutually independent with distributions as in the statement.
\end{proof}

The column-swap move of \Cref{thm:swap_columns_dimer} is very similar to zipper arguments coming from the Yang-Baxter equation for solvable lattice models. Using it, we will be able to transform $G_n^{vert}$ into graphs as in \Cref{fig:bsw_example}, which we define more formally below (though it may be easier to just look at \Cref{fig:bsw_example}, at least at a first pass). The construction may remind the reader familiar with \cite{chhita2023domino} with matrix commutations in that paper, and we will present an alternative construction in the spirit of \cite{chhita2023domino} in \Cref{appendix:lgv}.

For such column-swap arguments, we find it helpful to schematically visualize sequences of $(+)$- and $(-)$-columns by `staircase' diagrams, where each down step corresponds to a $(-)$-column and each right step corresponds to a $(+)$-column. The Aztec diamond of size $n$, or rather the dilated graph $G_n^{vert}$ of \Cref{def:tG}, corresponds to a staircase of $n$ alternating horizontal and vertical steps, see \Cref{fig:schur_diagram_examples} (left). 

\begin{figure}[H]
        \centering
\includegraphics[scale=.8]{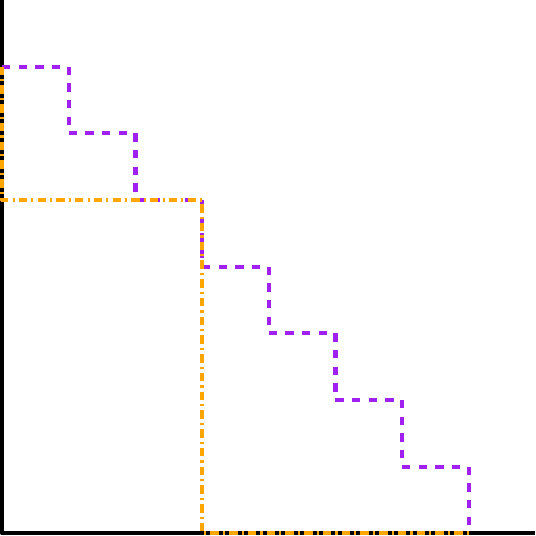}
        \caption{The diagrams of $G_7^{vert}$ (purple) and $G^{v-swap}_{7,3}$ (orange). The latter graph is pictured later in \Cref{fig:bsw_example}.}
                \label{fig:schur_diagram_examples}
\end{figure}

The same diagrams are often used for Schur processes, which indeed are a deterministic limit of our dimer model. The commutation of a $(+)$- and $(-)$-column then corresponds to one of the moves

\begin{center}
    \includegraphics[scale=.6]{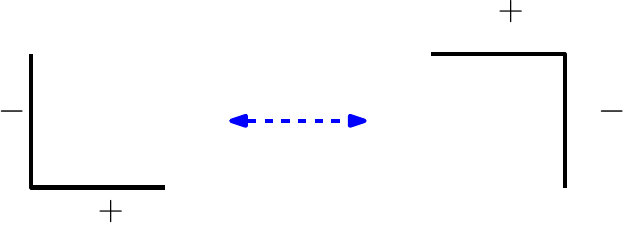}
\end{center}
on this diagram, i.e. adding or subtracting a square. Note that this diagram does not keep track of either weights or matchings. We will use these commutations to turn $G_n^{vert}$ into another graph, which we now define and which is schematically represented in \Cref{fig:schur_diagram_examples} (orange).

\begin{defi}\label{def:gbsw} 
Let $n \geq 1$ and $\ell \in [1,n+1]$ be integers. Let $G^{v-swap}_{n,\ell}$ be the weighted bipartite graph drawn on $\Z^2$ as follows. It has black vertices at $(0,-1),(0,-2),\ldots,(0,-n)$ which are connected to white vertices at $(1,-1),(1,-2),\ldots,(1,-n)$ by edges with weights $a_{1,1}^{[n]}+b_{1,1}^{[n]},\ldots,a_{1,n}^{[n]}+b_{1,n}^{[n]}$ (recall from \Cref{def:plus_and_minus_cols} that we always enumerate weights top to bottom). These white vertices then form the leftmost white vertices of a series of $\ell-1$ $(-)$-columns where the $i\tth$ $(-)$-column (starting from $i=1$ at the left) has weights $a_{1,1}^{[n-i]}+b_{1,1}^{[n-i]},\ldots,a_{1,n-i}^{[n-i]}+b_{1,n-i}^{[n-i]}$. The rightmost white vertices of the $(\ell-1)\tth$ $(-)$-column, which are located at coordinates $(2\ell-1,-1),\ldots,(2\ell-1,-n+\ell-1)$, form the leftmost white vertices of a series of $\ell$ $(+)$-columns where the $i\tth$ column (starting from $i=1$ at the left) has weights
$$\frac{a^{[n-\ell+i]}_{i,1}}{a^{[n-\ell+i]}_{i,1}+b^{[n-\ell+i]}_{i,1}},\ldots,\frac{a^{[n-\ell+i]}_{i,n-\ell+i}}{a^{[n-\ell+i]}_{i,n-\ell+i}+b^{[n-\ell+i]}_{i,n-\ell+i}}.$$ 
The rightmost white vertices of the $\ell\tth$ $(+)$-column in this series, which are at $(4\ell-1,\ell-1),\ldots,(4\ell-1,-n+\ell-1)$, form the leftmost white vertices of a series of $n-\ell+1$ $(-)$-columns, where the $i\tth$ column (starting from $i=1$ at the left) has weights $a^{[n-i+1]}_{\ell+1,1}+b^{[n-i+1]}_{\ell+1,1},\ldots,a^{[n-i+1]}_{\ell+1,n-i+1}+b^{[n-i+1]}_{\ell+1,n-i+1}$. The rightmost white vertices of the $(n-\ell+1)\tth$ $(-)$-column, which are at $(2n+2\ell+1,\ell-1),\ldots,(2n+2\ell+1,0)$, form the leftmost white vertices of a series of $n-\ell$ $(+)$-columns, where the $i\tth$ column (starting from $i=1$ at the left) has weights 
$$\frac{a^{[\ell+i-1]}_{\ell+i,1}}{a^{[\ell+i-1]}_{\ell+i,1}+b^{[\ell+i-1]}_{\ell+i,1}},\ldots,\frac{a^{[\ell+i-1]}_{\ell+i,\ell+i-1}}{a^{[\ell+i-1]}_{\ell+i,\ell+i-1}+b^{[\ell+i-1]}_{\ell+i,\ell+i-1}}.$$
Finally, the rightmost white vertices of the last such $(+)$-column are connected by horizontal edges of weight $1$ to $n$ black vertices, which are located at  $(4n+2,n-1),\ldots,(4n+2,0)$.
\end{defi}

An example is given in \Cref{fig:bsw_example}. In the case $\ell=n+1$ the graph $G^{v-swap}_{n,n+1}$ has two connected components, as pictured in \Cref{fig:disconnected_bsw}.

    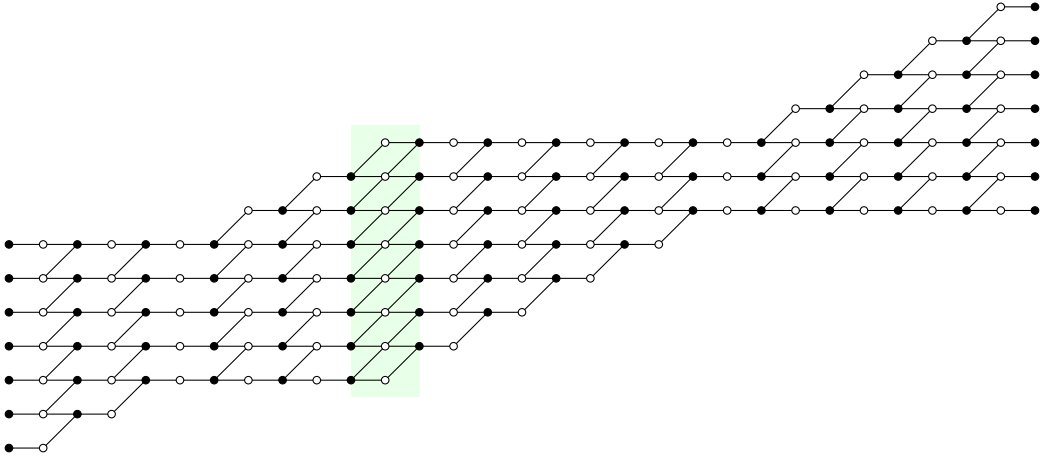
\begin{figure}[H]
        \centering

\begin{tikzpicture}[
    scale=0.45,
    xshift=-60cm,                            
    blackV/.style={circle,draw,fill=black,inner sep=1pt},
    whiteV/.style={circle,draw,fill=white,inner sep=1pt}
]

\fill[green!30,opacity=0.3] (10,-5.5) rectangle (12, 2.5);

%
%
\foreach \y in {-1,-2,-3,-4,-5,-6,-7} {
  \draw (0,\y) -- (1,\y);
}

%
%
\foreach \j in {0,1,2,3,4,5} {
  \draw (1,{ -1 - \j }) -- (2,{ -1 - \j });  
  \draw (1,{ -2 - \j }) -- (2,{ -1 - \j });  
  \draw (2,{ -1 - \j }) -- (3,{ -1 - \j });  
}

%
%
\foreach \j in {0,1,2,3,4} {
  \draw (3,{ -1 - \j }) -- (4,{ -1 - \j });  
  \draw (3,{ -2 - \j }) -- (4,{ -1 - \j });  
  \draw (4,{ -1 - \j }) -- (5,{ -1 - \j });  
}

%
%
\foreach \j in {0,1,2,3,4} {
  \draw (5,{ -1 - \j }) -- (6,{ -1 - \j });       
  \draw (6,{ -1 - \j }) -- (7,{ -1 - \j });       
  \draw (6,{ -1 - \j }) -- (7,{ -\j });           
}

%
%
\foreach \j in {0,1,2,3,4,5} {
  \draw (7,{  0 - \j }) -- (8,{  0 - \j });   
  \draw (8,{  0 - \j }) -- (9,{  0 - \j });   
  \draw (8,{  0 - \j }) -- (9,{  1 - \j });   
}

%
%
\foreach \j in {0,1,2,3,4,5,6} {
  \draw (9,{  1 - \j })  -- (10,{  1 - \j }); 
  \draw (10,{ 1 - \j })  -- (11,{  1 - \j }); 
  \draw (10,{ 1 - \j })  -- (11,{  2 - \j }); 
}

%
%
\foreach \j in {0,1,2,3,4,5,6} {
  \draw (11,{  2 - \j }) -- (12,{  2 - \j }); 
  \draw (11,{  1 - \j }) -- (12,{  2 - \j }); 
  \draw (12,{  2 - \j }) -- (13,{  2 - \j }); 
}

\foreach \j in {0,1,2,3,4,5} {
  \draw (13,{  2 - \j }) -- (14,{  2 - \j }); 
  \draw (13,{  1 - \j }) -- (14,{  2 - \j }); 
  \draw (14,{  2 - \j }) -- (15,{  2 - \j }); 
}

\foreach \j in {0,1,2,3,4} {
  \draw (15,{  2 - \j }) -- (16,{  2 - \j }); 
  \draw (15,{  1 - \j }) -- (16,{  2 - \j }); 
  \draw (16,{  2 - \j }) -- (17,{  2 - \j }); 
}

\foreach \j in {0,1,2,3} {
  \draw (17,{  2 - \j }) -- (18,{  2 - \j }); 
  \draw (17,{  1 - \j }) -- (18,{  2 - \j }); 
  \draw (18,{  2 - \j }) -- (19,{  2 - \j }); 
}

\foreach \j in {0,1,2} {
  \draw (19,{  2 - \j }) -- (20,{  2 - \j }); 
  \draw (19,{  1 - \j }) -- (20,{  2 - \j }); 
  \draw (20,{  2 - \j }) -- (21,{  2 - \j }); 
}

%
%
\foreach \j in {0,1,2} {
  \draw (21,{  2 - \j }) -- (22,{  2 - \j }); 
  \draw (22,{  2 - \j }) -- (23,{  2 - \j }); 
  \draw (22,{  2 - \j }) -- (23,{  3 - \j }); 
}

\foreach \j in {0,1,2,3} {
  \draw (23,{  3 - \j }) -- (24,{  3 - \j }); 
  \draw (24,{  3 - \j }) -- (25,{  3 - \j }); 
  \draw (24,{  3 - \j }) -- (25,{  4 - \j }); 
}

\foreach \j in {0,1,2,3,4} {
  \draw (25,{  4 - \j }) -- (26,{  4 - \j }); 
  \draw (26,{  4 - \j }) -- (27,{  4 - \j }); 
  \draw (26,{  4 - \j }) -- (27,{  5 - \j }); 
}

\foreach \j in {0,1,2,3,4,5} {
  \draw (27,{  5 - \j }) -- (28,{  5 - \j }); 
  \draw (28,{  5 - \j }) -- (29,{  5 - \j }); 
  \draw (28,{  5 - \j }) -- (29,{  6 - \j }); 
}

%
%
\foreach \y in {0,1,2,3,4,5,6} {
  \draw (29,\y) -- (30,\y);
}

%
%
%
\foreach \x/\ya/\yb in {
   0/-1/-7,    2/-1/-6,   4/-1/-5,   6/-1/-5,   8/-5/0,
   10/-5/1,   12/-4/2,   14/-3/2,  16/-2/2,  18/-1/2,
   20/0/2,    22/0/2,   24/0/3,   26/0/4,   28/0/5,   30/0/6
}{
  \foreach \y in {\ya,...,\yb} {
    \node[blackV] at (\x,\y) {};
  }
}

%
\foreach \x/\ya/\yb in {
   1/-1/-7,   3/-1/-6,   5/-1/-5,   7/-5/0,
   9/-5/1,   11/-5/2,  13/-4/2,  15/-3/2,
   17/-2/2,  19/-1/2,  21/0/2,  23/0/3,
   25/0/4,  27/0/5,  29/0/6
}{
  \foreach \y in {\ya,...,\yb} {
    \node[whiteV] at (\x,\y) {};
  }
}

\end{tikzpicture}

        \caption{The graph $G^{v-swap}_{7,3}$. It was obtained from an Aztec diamond $G_7^{\Az}$ by the graph moves described above, which did not alter the $3^{rd}$ column of the original Aztec diamond (highlighted in green).}
        \label{fig:bsw_example}
\end{figure}

    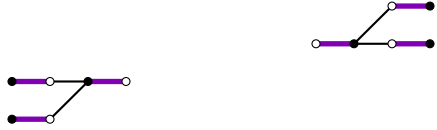
\begin{figure}[H]
        \centering
        \definecolor{purpleEdge}{RGB}{128,0,180}
\begin{tikzpicture}[scale=0.5]

\draw[line width=0.8pt] (0,-1)--(1,-1);
\draw[line width=0.8pt] (0,-2)--(1,-2);
\draw[line width=0.8pt] (1,-1)--(2,-1);
\draw[line width=0.8pt] (1,-2)--(2,-1);
\draw[line width=0.8pt] (2,-1)--(3,-1);

\draw[line width=0.8pt] (8,0)--(9,0);
\draw[line width=0.8pt] (9,0)--(10,0);
\draw[line width=0.8pt] (9,0)--(10,1);
\draw[line width=0.8pt] (10,0)--(11,0);
\draw[line width=0.8pt] (10,1)--(11,1);

\draw[-,line width=2pt,draw=purpleEdge] (0,-1)--(1,-1);
\draw[-,line width=2pt,draw=purpleEdge] (0,-2)--(1,-2);
\draw[-,line width=2pt,draw=purpleEdge] (2,-1)--(3,-1);
\draw[-,line width=2pt,draw=purpleEdge] (8,0)--(9,0);
\draw[-,line width=2pt,draw=purpleEdge] (10,0)--(11,0);
\draw[-,line width=2pt,draw=purpleEdge] (10,1)--(11,1);

\filldraw (0,-1) circle(3pt); \draw[fill=white] (1,-1) circle(3pt);
\filldraw (0,-2) circle(3pt); \draw[fill=white] (1,-2) circle(3pt);
\filldraw (2,-1) circle(3pt); \draw[fill=white] (3,-1) circle(3pt);

\draw[fill=white] (8,0)  circle(3pt);
\filldraw            (9,0)  circle(3pt);
\draw[fill=white] (10,0) circle(3pt);
\draw[fill=white] (10,1) circle(3pt);
\filldraw            (11,0) circle(3pt);
\filldraw            (11,1) circle(3pt);

\end{tikzpicture}

        \caption{The graph $G^{v-swap}_{2,3}$, which has two connected components, shown with its unique dimer cover. In the language of \Cref{def:plus_and_minus_cols} it is composed of two $(-)$-columns followed by two $(+)$-columns, where the second $(-)$-column and first $(+)$-column are degenerate and consist of a single white vertex.}
        \label{fig:disconnected_bsw}
\end{figure}

\begin{prop}\label{thm:aztec_to_udt}
Let $n \geq 1$ and $\ell \in [1,n+1]$ be integers, and let $G_n^{vert}$ be as in \Cref{def:tG} and $G_{n,\ell}^{v-swap}$ be as in \Cref{def:gbsw}, with weights given in terms of the ones in $G_n^{vert}$ via the recurrence \eqref{eq:downshuffle_weights}. Then
\begin{enumerate}
\item There is equality of partition functions $Z_{G_n^{vert}} = Z_{G_{n,\ell}^{v-swap}}$.
\item If furthermore $\ell \leq n$, then there is also an equality of marginal distributions of edges in (a) the $\ell\tth$ column from the left of $n$ squares in $G_n^{vert}$, and (b) the only column of $n$ squares in $G_{n,\ell}^{v-swap}$.\label{item:tG_to_gbsw_dist_eq}
\end{enumerate}
\end{prop}
\begin{proof}
To the left of the $\ell\tth$ column from the left of $n$ squares in $G_n^{vert}$, there are $\ell-1$ $(-)$-columns and $\ell$ $(+)$-columns in alternating order, including the one which intersects this column of squares. Use \Cref{thm:swap_columns_dimer} to swap the order of these to match the order $-,\ldots,-,+,\ldots,+$; note that the rightmost $(+)$-column does not have to be swapped with anything. Now, to the right of our column of squares there are $n-\ell+1$ $(-)$-columns (including the one which intersects this column of squares) and $n-\ell$ $(+)$-columns in alternating order. Again use \Cref{thm:swap_columns_dimer} to swap the order $-,\ldots,-,+,\ldots,+$, and similarly note that the leftmost $(-)$-column did not have to get swapped with anything. By \Cref{thm:swap_columns_dimer}, these swaps did not change the partition function, and since neither the $\ell\tth$ $(+)$-column nor the $\ell\tth$ $(-)$-column were swapped, the marginal distribution of edges in our desired column of squares did not change. This completes the proof.
\end{proof}

\begin{rmk}\label{rmk:more_updowns}
The upshot of the proof of \Cref{thm:aztec_to_udt} is that one can take the modified Aztec graph $G_n^{vert}$ and swap $(+)$- and $(-)$-columns via \Cref{thm:swap_columns_dimer} to obtain related graphs with explicit relations between certain marginals of the dimer measures for the two graphs. We only showed this in the case where we modify $G_n^{\Az}$ to a graph $G_{n,\ell}^{v-swap}$ with columns of the form $-,\ldots,-,+,\ldots,+,-,\ldots,-,+,\ldots,+$. However, at the cost of slightly more notation and keeping track of indices, one can get any arbitrary sequence of $(+)$- and $(-)$-columns by the same argument as the proof of \Cref{thm:aztec_to_udt}. These would correspond to different mixed polymers (see remainder of this section) where the columns of Beta and strict-weak environments interleave in an arbitrary way.
\end{rmk}

In fact, dimer covers of $G_{n,\ell}^{v-swap}$ are deterministically `frozen' on the edges. These frozen portions change the partition function, but have no effect on the dimer measure in the middle portion. 

\begin{lemma}\label{thm:gbsw_frozen_bits}
    Any dimer cover $M$ of the graph $G_{n,\ell}^{v-swap}$ of \Cref{def:gbsw} has the following properties:
    \begin{enumerate}
     \item For each of the first block of $\ell-1$ $(-)$-columns, every horizontal edge with left vertex black and right vertex white is contained in $M$. 
     \item For each of the second block of $n-\ell$ $(+)$-columns, every horizontal edge with left vertex white and right vertex black is contained in $M$. 
     \end{enumerate} 
\end{lemma}
\begin{proof}
The leftmost black vertices are $1$-valent and hence in any dimer cover they must be matched to the white vertices to their right. After removing these vertices and all other edges connected to them, we again have $1$-valent black vertices which must be matched to the white vertices to their right, and so on. This proves the first claim, and the second follows by the same argument proceeding from the right of $G_{n,\ell}^{v-swap}$.
\end{proof}

In particular, this lets us give another proof of independence of factors in the partition function of the Aztec diamond. 

\begin{proof}[Another proof of {\Cref{thm:compute_Z_cor}}]
By the first part of \Cref{thm:aztec_to_udt} together with \Cref{thm:gbsw_frozen_bits}, $Z_{G_n^{\Az}}$ is equal to the weight of the single legal matching of $G^{v-swap}_{n,n+1}$, and from the explicit description of weights in \Cref{def:gbsw} this is given by the right hand side of \eqref{eq:partition_fn}.

For the mutual independence, note that in the proof of \Cref{thm:aztec_to_udt} we obtained $G^{v-swap}_{n,n+1}$ from $G_n^{vert}$ by a sequence of swaps of $(+)$- and $(-)$-columns. By \Cref{thm:swap_weights_stay_independent}, it follows that the weights in $G^{v-swap}_{n,n+1}$ are mutually independent. The marginal laws in \eqref{eq:Z_product_marginals} come \Cref{thm:gamma_to_gamma} and the fact that a sum of two Gamma variables is another Gamma variable (see \Cref{thm:lukacs}).
\end{proof}

\begin{defi}\label{def:gbsw_bar}
    For any $n \geq \ell \geq 1$, let $\bgbsw_{n,\ell}$ be the graph formed by removing both columns of $2$-valent white vertices from $G_{n,\ell}^{v-swap}$, and keeping only the middle component, as in \Cref{fig:bgbsw_coordinates}.
\end{defi}

\begin{figure}
        \centering

\scalebox{.6}{
\begin{tikzpicture}[
    scale=1.5,                             
    labstyle/.style={font=\normalsize,black},
    ticklabel/.style={font=\normalsize,black},
    horizLab/.style={labstyle,above,pos=0.5,xshift=2pt,yshift=-2pt},
    diagLab/.style={labstyle,above left,pos=0.3,xshift=3pt},
    blackV/.style={circle,fill=black,draw,inner sep=1.9pt},
    whiteV/.style={circle,fill=white,draw,inner sep=1.9pt},
    axisstyle/.style={thick,black},
    every path/.style={line width=0.6pt}
]

\draw[->,axisstyle] (5,-6) -- (21,-6) node[below right,ticklabel] {$x$};
\draw[->,axisstyle] (5,-6) -- (5,3)   node[above left,ticklabel]  {$y$};

\foreach \x in {6,7,...,20} {
  \draw[axisstyle] (\x,-6.1) -- (\x,-5.9);
  \node[ticklabel] at (\x,-6.3) {\x};
}

\foreach \y in {-5,-4,...,2} {
  \draw[axisstyle] (4.9,\y) -- (5.1,\y);
  \node[ticklabel,right=4pt] at (4.5,\y) {\y};
}

\foreach \j in {1,...,5}{
  \pgfmathtruncatemacro{\y}{-\j}
  \draw (6,\y) -- node[horizLab]{\(\bar{\beta}_{1,\j}^{[5]}\)} (7,\y);
  \draw (6,\y) -- node[diagLab]{\(\beta_{1,\j}^{[5]}\)}     (7,\y+1);
}
\foreach \j[evaluate=\j as \y using {1-\j}] in {1,...,6}{
  \draw (7,\y) -- (8,\y);
  \draw (8,\y) -- node[horizLab]{\(\bar{\beta}_{2,\j}^{[6]}\)} (9,\y);
  \draw (8,\y) -- node[diagLab]{\(\beta_{2,\j}^{[6]}\)}     (9,\y+1);
}
\foreach \j[evaluate=\j as \y using {2-\j}] in {1,...,7}{
  \draw (9,\y)  -- (10,\y);
  \draw (10,\y) -- node[horizLab]{\(\bar{\beta}_{3,\j}^{[7]}\)} (11,\y);
  \draw (10,\y) -- node[diagLab]{\(\beta_{3,\j}^{[7]}\)}     (11,\y+1);
}

\newcommand{\minuscolumn}[4]{%
  \foreach \j in {1,...,#2}{%
    \pgfmathtruncatemacro{\yB}{2-(\j-1)}
    \draw (#1,\yB)   -- (#1+1,\yB);
    \draw (#1,\yB-1) -- (#1+1,\yB);
    \draw (#1+1,\yB) -- node[labstyle,above]{\(\gamma_{4,\j}^{[#4]}\)} (#1+2,\yB);
  }%
}
\minuscolumn{11}{7}{1}{7}
\minuscolumn{13}{6}{2}{6}
\minuscolumn{15}{5}{3}{5}
\minuscolumn{17}{4}{4}{4}

\foreach \j in {1,...,3}{
  \pgfmathtruncatemacro{\yB}{2-(\j-1)}
  \draw (19,\yB)   -- (20,\yB);
  \draw (19,\yB-1) -- (20,\yB);
}

\foreach \x/\ya/\yb in {
   7/ 0/-5,  9/ 1/-5,  11/2/-5,
   13/2/-4, 15/2/-3, 17/2/-2, 19/2/-1
}{
  \foreach \y in {\ya,...,\yb}{\node[whiteV] at (\x,\y) {};}
}
\foreach \x/\ya/\yb in {
   6/-1/-5,  8/0/-5,   10/1/-5,
   12/2/-4, 14/2/-3,  16/2/-2, 18/2/-1, 20/2/0
}{
  \foreach \y in {\ya,...,\yb}{\node[blackV] at (\x,\y) {};}
}

\end{tikzpicture}%
}

        \caption{The graph $\bgbsw_{7,3}$ as defined in \Cref{def:gbsw_bar}, with coordinates shown. Here we use shorthand $\gamma_{i,j}^{[n]} := a_{i,j}^{[n]}+b_{i,j}^{[n]}$ and $\beta_{i,j}^{[n]} := \frac{a_{i,j}^{[n]}}{a_{i,j}^{[n]}+b_{i,j}^{[n]}}, \bar{\beta}_{i,j}^{[n]} := 1-\beta_{i,j}^{[n]}$ for the weights to declutter notation in the labels.}
        \label{fig:bgbsw_coordinates}
\end{figure}
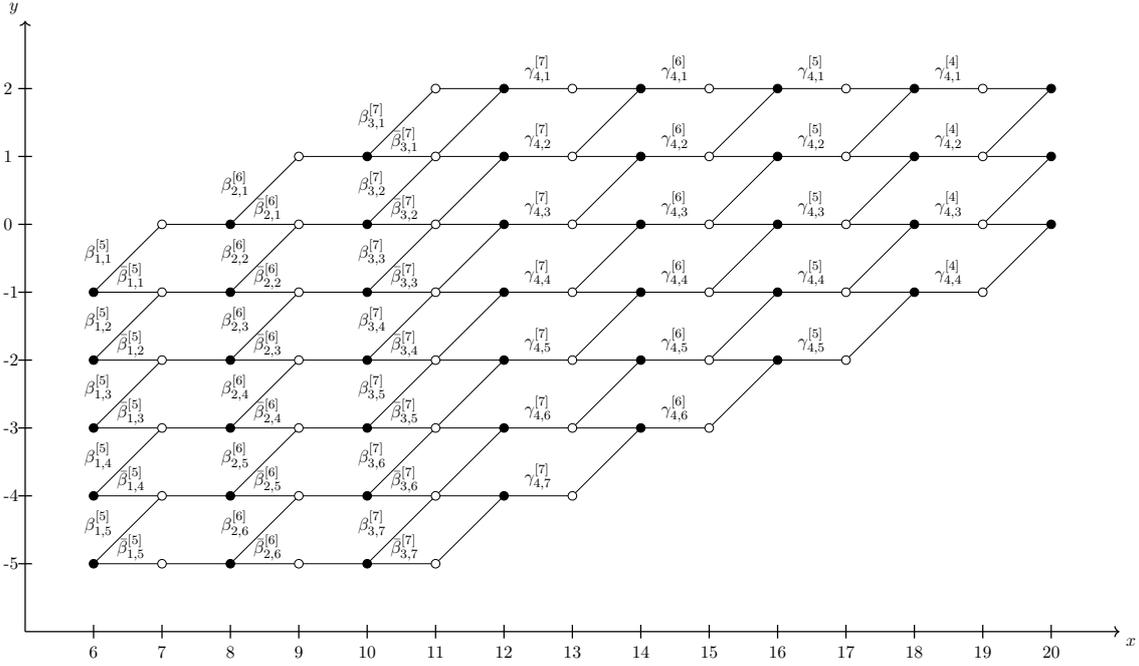

\begin{cor}\label{thm:remove_frozen_bits}
    For any $n \geq \ell \geq 1$, there is a bijection between matchings of $G_{n,\ell}^{v-swap}$ and $\bgbsw_{n,\ell}$ which preserves weights up to an overall factor of 
\begin{equation}
\prod_{i=0}^{\ell-1} \prod_{j=1}^{n-i} (a_{1,j}^{[n-i]}+b_{1,j}^{[n-i]}).
\end{equation}
Consequently, 
    \begin{enumerate}
    \item The dimer measure on the subgraph of $G_{n,\ell}^{v-swap}$ corresponding to $\bgbsw_{n,\ell}$ is the same as the dimer measure on $\bgbsw_{n,\ell}$. \label{item:bgbsw_measure}
    \item The partition function of $\bgbsw_{n,\ell}$ is \label{item:bgbsw_Z} 
    \begin{equation}
     Z_{\bgbsw_{n,\ell}}  = \prod_{k=1}^{n-\ell} \prod_{j=1}^{k} (a_{1,j}^{[k]} + b_{1,j}^{[k]}) = Z_{G_{n-\ell}^{\Az}}.
     \end{equation}
    \end{enumerate}
\end{cor}
\begin{proof}
The first sentence is a direct consequence of \Cref{thm:gbsw_frozen_bits}. The equality of dimer measures in \eqref{item:bgbsw_measure} follows directly from this. For the partition function, we use the same fact together with the explicit partition function of $G_n^{\Az}$ from \Cref{thm:compute_Z_cor}, which is the same as that of $G_n^{vert}$ by \Cref{thm:tG} and the same as that of $G_{n,\ell}^{v-swap}$ by \Cref{thm:aztec_to_udt}.
\end{proof}

\begin{thm}\label{thm:aztec_to_bgbsw}
    Let $n \geq \ell \geq 1$ be integers, and let $G_n^{\Az}$ be the weighted Aztec graph as in \Cref{fig:our_weights} and $\bgbsw_{n,\ell}$ be as in \Cref{def:gbsw_bar}, with weights given in terms of the ones in $G_n^{\Az}$ via the recurrence \eqref{eq:downshuffle_weights}. Then there is an equality of marginal distributions of edges in (a) the $\ell\tth$ column from the left of $n$ squares in $G_n^{\Az}$, and (b) the only column of $n$ squares in $\bgbsw_{n,\ell}$.
\end{thm}
\begin{proof}
    Combining \Cref{thm:tG}, \Cref{thm:aztec_to_udt} Part \eqref{item:tG_to_gbsw_dist_eq}, and \Cref{thm:remove_frozen_bits} Part \eqref{item:bgbsw_measure}, we obtain distributional equalities along the $\ell\tth$ vertical slice from $G_n^{\Az}$ to $G_n^{vert}$, then to $G_{n,\ell}^{v-swap}$, then to $\bgbsw_{n,\ell}$, respectively.
\end{proof}

\subsection{Distributional equality between $\bgbsw_{n,\ell}$ and the $\beta$-$\Gamma$ polymer} \label{subsec:distrbeta} 
The following is the general version of \Cref{def:beta_sw_intro} earlier.

\begin{defi}
    \label{def:beta_sw}
    Let $p,m \geq 1$ be integers, and let $(\tau_i)_{1 \leq i \leq m}, (\eta_j)_{1 \leq j \leq m+p}, (\xi_j)_{1 \leq j \leq m+p}$ be real parameters such that $\eta_j+\tau_i,\xi_j-\tau_i$, and $\eta_j+\xi_{j'}$ are always positive. 

    Consider the weighted, directed graph $G^{\beta \Gamma}_{p,m}$ (\Cref{fig:bgbsw_to_poly} (bottom)) with all vertices lying in the set 
    \begin{multline}\label{eq:bsw_support_set}
        \{(x,y) \in \Z^2: -m \leq x \leq -1, -m-p \leq y \leq -1, x+y \geq -m-p\} \\ 
         \cup \{(x,y) \in \Z^2: -m-p \leq y \leq -1, 0 \leq x \leq \min(p-1, y + m + p)\} 
    \end{multline}
    and edges given as follows:
    \begin{enumerate}
\item For each vertex at $(x,y)$ with $x \leq -1$, there are right and down-right directed edges to vertices $(x+1,y)$ and $(x+1,y-1)$, with weights $\beta_{x,y}$ and $1-\beta_{x,y}$ respectively, where $\beta_{x,y} \sim \Beta(\eta_{-y}+\tau_{-x}, \xi_{-x-y-1}-\tau_{-x})$ are independent Beta variables.  
\item For each vertex at $(x,y)$ with $x \geq 0$ lying in the set \eqref{eq:bsw_support_set}, there are down and right directed edges to vertices $(x,y-1)$ and $(x+1,y)$ (if they also lie in the set \eqref{eq:bsw_support_set}). The downward edges have weight $1$, while the rightward edges have independent Gamma weights $\gamma_{x,y} \sim \Gamma(\eta_{-y}+\xi_{x-y},1)$.
\end{enumerate}
Then the associated \emph{$\beta$-$\Gamma$ polymer partition function} is 
    \begin{equation}
        Z^{\beta \Gamma}_{p,m} := \sum_{\substack{\pi_j:(-m,-j) \to (p-j,-m-j) \\ 1 \leq j \leq p}} \prod_{j=1}^p \wt(\pi_j),
    \end{equation}
    where the sum is over $p$-tuples of nonintersecting\footnote{Meaning that the vertex sets of any two paths, not just the edge sets, are disjoint.} paths $\pi_1,\ldots,\pi_p$ on $G^{\beta \Gamma}_{p,m}$, where $\pi_j$ has start point $(-m,-j)$ and end point $(p-j,-m-j)$, and $\wt(\pi_j)$ denotes the product of edge weights over the edges in $\pi_j$.

    The associated \emph{$\beta$-$\Gamma$ polymer measure} is a probability measure on such $p$-tuples $(\pi_1,\ldots,\pi_p)$ of nonintersecting paths which assigns to each one the probability
    \begin{equation}
        \frac{1}{Z^{\beta \Gamma}_{p,m}} \prod_{j=1}^p \wt(\pi_j).
    \end{equation}
\end{defi}

We recall the following definitions from the Introduction.

\begin{defi}\label{def:X_polymer}
    Each of the polymer paths $\pi_1,\ldots,\pi_p$ in \Cref{def:beta_sw} contains at least one vertex $(0,-y)$, and sometimes several. Let 
    \begin{equation*}
        \pi_j(m) = \min(\{y: (0,-y) \in \pi_j\})
    \end{equation*}
    be the distance of the closest such vertex to the $x$-axis. Then we define
    \begin{equation}
        X^{poly}(\pi_1,\ldots,\pi_p) = \{\pi_j(m): 1 \leq j \leq p\}.
    \end{equation}
\end{defi}

Recall that $|X^{poly}(\pi_1,\ldots,\pi_p)| = p$ since the paths are non-intersecting. We wish to match the distribution of this set with the following analogue for the Aztec diamond.

\begin{defi}\label{def:aztec_matching_column}
    For an Aztec diamond $G_n^{\Az}$ of size $n$, let $\mathcal{M}$ be the set of perfect matchings of $G_n^{\Az}$. For any $1 \leq \ell \leq n$, define functions
    \begin{equation}
        X_\ell^{\Az}: \mathcal{M} \to \binom{[n+1]}{n-\ell+1}
    \end{equation}
    as follows. There are $n$ columns of white vertices in $G_n^{\Az}$; consider the $\ell\tth$ one from left to right. For each perfect matching $M \in \mathcal{M}$, there will be exactly $n-\ell+1$ white vertices in the $\ell\tth$ column which are matched to a black vertex to their left (either up-left or down-left). Labeling the white vertices in this column by $1,\ldots,n+1$ from top to bottom, we let $X_\ell^{\Az}(M)$ be the set of labels of the white vertices which are matched with a black vertex to their left (again, either up-left or down-left).

    Recall that the graph $\bgbsw_{n,\ell}$ has a single column of $n$ faces of degree $4$, which comes from the $\ell\tth$ column of such faces in the Aztec diamond $G_n^{\Az}$. For a matching $\tilde{M}$ of $\bgbsw_{n,\ell}$, we similarly define $\bX(\tilde{M})$ to be the set of labels of white vertices matched with a black vertex to their left (either horizontally left or down-left). 
\end{defi}

See \Cref{fig:mixed_polymer_intro} (left) for an example of $X_\ell^{\Az}$.

\begin{lemma}\label{thm:bgbsw_to_polymer_deterministic}
    Fix integers $n \geq p \geq 1$. Let $\bgbsw_{n,n-p+1}$ be the weighted bipartite graph of \Cref{def:gbsw_bar}, with deterministic weights. Let $G_{p,n-p+1}^{\beta \Gamma}$ be the directed graph of \Cref{def:beta_sw}, but with deterministic weights defined in terms of those of $\bgbsw_{n,n-p+1}$ via
    \begin{align}
    \begin{split} \label{eq:poly_from_aztec_weights}
    \gamma_{x,y} &= a^{[n-x]}_{\ell+1,-y}+b^{[n-x]}_{\ell+1,-y} \\ 
    \beta_{x,y} &= \frac{a^{[n+x+1]}_{\ell+x+1,-y}}{a^{[n+x+1]}_{\ell+x+1,-y}+b^{[n+x+1]}_{\ell+x+1,-y}},
    \end{split}
    \end{align}
    where $\ell = n-p+1$. Then there is a bijection $M \mapsto (\pi_1(M),\ldots,\pi_\ell(M))$ between perfect matchings $M$ on $\bgbsw_{n,n-p+1}$ and $p$-tuples of nonintersecting paths $(\pi_1,\ldots,\pi_p)$ as in \Cref{def:beta_sw} on $G_{p,n-p+1}^{\beta \Gamma}$, which is weight-preserving and satisfies
    \begin{equation}
     \bX(M) = X^{poly}(\pi_1(M),\ldots,\pi_p(M)).
    \end{equation}
\end{lemma}
\begin{proof}
    The bijection is as follows. The graph $\bgbsw_{n,n-p+1}$ can be divided into two halves, the left half consisting of $+$ columns, the right half consisting of $-$ columns, in the sense of \Cref{def:plus_and_minus_cols}. There are four types of edges which, when present in a matching $M$, yield a corresponding segment of path $\pi_j(M)$:
    \begin{enumerate}
        \item Each up-right edge with lower-left vertex black and upper-right vertex white\footnote{Such edges occur only in the left half of $\bgbsw_{n,n-p+1}$.}, corresponds to a horizontal path segment $(x,y) \to (x+1,y)$ ($x < 0$) in the \textbf{left} half of $G_{p,n-p+1}^{\beta \Gamma}$,
        \item Each horizontal edge $b-w$ (black vertex on the left) in the \textbf{left} half of $\bgbsw_{n,n-p+1}$ corresponds to a Southeast path segment $(x,y) \to (x+1,y-1)$ ($x < 0$) in the \textbf{left} half of $G_{p,n-p+1}^{\beta \Gamma}$,
        \item Each horizontal edge $b-w$ (black vertex on the left) in the \textbf{right} half of $\bgbsw_{n,n-p+1}$ corresponds to a horizontal path segment $(x,y) \to (x+1,y)$ ($x \geq 0$) in the \textbf{right} half of $G_{p,n-p+1}^{\beta \Gamma}$,
        \item Each up-right edge with lower-left vertex white and upper-right vertex black\footnote{Such edges occur only in the right half of $\bgbsw_{n,n-p+1}$.}, corresponds to a downward path segment $(x,y) \to (x,y-1)$ ($x \geq 0$) in the \textbf{right} half of $G_{p,n-p+1}^{\beta \Gamma}$.
    \end{enumerate}
    See \Cref{fig:bgbsw_to_poly}. From the definition of $\bgbsw_{n,n-p+1}$ (\Cref{def:gbsw_bar}) we have that the weights of the four types of edges above are 
    \begin{equation}
        \frac{a^{[n+x+1]}_{\ell+x+1,-y}}{a^{[n+x+1]}_{\ell+x+1,-y}+b^{[n+x+1]}_{\ell+x+1,-y}}, \frac{b^{[n+x+1]}_{\ell+x+1,-y}}{a^{[n+x+1]}_{\ell+x+1,-y}+b^{[n+x+1]}_{\ell+x+1,-y}}, a^{[n-x]}_{\ell+1,-y}+b^{[n-x]}_{\ell+1,-y}, \text{ and } 1
    \end{equation}
    respectively, matching the weights of \eqref{eq:poly_from_aztec_weights}.
\end{proof}

\begin{figure}
        \centering

\scalebox{.4}{%
\begin{tikzpicture}[
    scale=2.25,                             
    labstyle/.style={font=\normalsize,text=blue!70!black},
    horizLab/.style={labstyle,above,pos=0.5,xshift=2pt,yshift=-2pt},
    diagLab/.style={labstyle,above left,pos=0.3,xshift=3pt},
    blackV/.style={circle,fill=black,draw,inner sep=1.2pt},
    whiteV/.style={circle,fill=white,draw,inner sep=1.2pt},
    every path/.style={line width=0.6pt},
  matchlu/.style={line width=4pt, red},
  matchlh/.style={line width=4pt, orange},
  matchru/.style={line width=4pt, blue},
  matchblank/.style={line width=4pt, gray},
  matchrh/.style={line width=4pt, blue!60!red},
]

\foreach \j in {1,...,5}{
  \pgfmathtruncatemacro{\y}{-\j}
  \draw (6,\y) -- (7,\y);
  \draw (6,\y) -- (7,\y+1);
}
\foreach \j[evaluate=\j as \y using {1-\j}] in {1,...,6}{
  \draw (7,\y) -- (8,\y);
  \draw (8,\y) -- (9,\y);
  \draw (8,\y) -- (9,\y+1);
}
\foreach \j[evaluate=\j as \y using {2-\j}] in {1,...,7}{
  \draw (9,\y)  -- (10,\y);
  \draw (10,\y) -- (11,\y);
  \draw (10,\y) -- (11,\y+1);
}

\newcommand{\minuscolumn}[4]{%
  \foreach \j in {1,...,#2}{%
    \pgfmathtruncatemacro{\yB}{2-(\j-1)}
    \draw (#1,\yB)   -- (#1+1,\yB);
    \draw (#1,\yB-1) -- (#1+1,\yB);
    \draw (#1+1,\yB) -- (#1+2,\yB);
  }%
}
\minuscolumn{11}{7}{4}{7}
\minuscolumn{13}{6}{4}{6}
\minuscolumn{15}{5}{4}{5}
\minuscolumn{17}{4}{4}{4}

\foreach \j in {1,...,3}{
  \pgfmathtruncatemacro{\yB}{2-(\j-1)}
  \draw (19,\yB)   -- (20,\yB);
  \draw (19,\yB-1) -- (20,\yB);
}


\draw[matchlu] (6,-1) -- (7,0);
\draw[matchlu] (6,-2) -- (7,-1);
\draw[matchlh] (6,-3) -- (7,-3);
\draw[matchlh] (6,-4) -- (7,-4);
\draw[matchlh] (6,-5) -- (7,-5);

\draw[matchlh] (8,0) -- (9,0);
\draw[matchlh] (8,-1) -- (9,-1);
\draw[matchlh] (8,-5) -- (9,-5);

\draw[matchlu] (8,-3) -- (9,-2);
\draw[matchlu] (8,-4) -- (9,-3);

\draw[matchlh] (10,-1) -- (11,-1);
\draw[matchlh] (10,-2) -- (11,-2);
\draw[matchlh] (10,-3) -- (11,-3);
\draw[matchlh] (10,-5) -- (11,-5);

\draw[matchlu] (10,0) -- (11,1);

\draw[matchblank] (7,-2) -- (8,-2);
\draw[matchblank] (9,1) -- (10,1);
\draw[matchblank] (9,-4) -- (10,-4);

\draw[matchblank] (11,2) -- (12,2);
\draw[matchblank] (13,2) -- (14,2);
\draw[matchblank] (15,2) -- (16,2);
\draw[matchblank] (17,2) -- (18,2);
\draw[matchblank] (19,2) -- (20,2);

\draw[matchblank] (19,1) -- (20,1);

\draw[matchblank] (11,0) -- (12,0);
\draw[matchblank] (13,0) -- (14,0);
\draw[matchblank] (15,0) -- (16,0);

\draw[matchblank] (13,-3) -- (14,-3);

\draw[matchru] (11,-4) -- (12,-3);

\draw[matchru] (15,-3) -- (16,-2);

\draw[matchru] (17,-2) -- (18,-1);

\draw[matchru] (19,-1) -- (20,0);

\draw[matchru] (17,0) -- (18,1);


\draw[matchrh] (12,1) -- (13,1);
\draw[matchrh] (12,-1) -- (13,-1);
\draw[matchrh] (12,-2) -- (13,-2);
\draw[matchrh] (12,-4) -- (13,-4);

\draw[matchrh] (14,1) -- (15,1);
\draw[matchrh] (14,-1) -- (15,-1);
\draw[matchrh] (14,-2) -- (15,-2);

\draw[matchrh] (16,1) -- (17,1);
\draw[matchrh] (16,-1) -- (17,-1);

\draw[matchrh] (18,0) -- (19,0);

\foreach \x/\ya/\yb in {
   7/ 0/-5,  9/ 1/-5, 11/2/-5,
  13/2/-4, 15/2/-3, 17/2/-2, 19/2/-1
}{
  \foreach \y in {\ya,...,\yb}{ \node[whiteV] at (\x,\y) {}; }
}
\foreach \x/\ya/\yb in {
   6/-1/-5,  8/0/-5,  10/1/-5,
  12/2/-4, 14/2/-3,  16/2/-2, 18/2/-1, 20/2/0
}{
  \foreach \y in {\ya,...,\yb}{ \node[blackV] at (\x,\y) {}; }
}

\end{tikzpicture}%
}

 \begin{tikzpicture}[
    scale=1,
    every node/.style={circle, fill=black, inner sep=1pt},
    decoration={markings, mark=at position 0.5 with {\arrow{>}}},
    arrowedge/.style={postaction={decorate}, gray},
    infdot/.style={circle, draw=none, fill=gray!50, inner sep=0.5pt}
  ]
\def\n{7}
\def\l{3}

\def\jmax{5}

\foreach \i in {0,...,2} {%
  \pgfmathtruncatemacro{\ii}{\i + 1}%
  \pgfmathtruncatemacro{\jmin}{-\i+1}%
  \foreach \j in {\jmin,...,\jmax} {%
    \pgfmathtruncatemacro{\jm}{\j - 1}%
    \path
      (\i,\j)   coordinate (v-\i-\j)
      (\ii,\j)  coordinate (v-\ii-\j)
      (\ii,\jm) coordinate (v-\ii-\jm);
    \draw[arrowedge] (v-\i-\j) -- (v-\ii-\j);
    \draw[arrowedge] (v-\i-\j) -- (v-\ii-\jm);
  }%
}

\foreach \i in {3,...,6} {%
  \pgfmathtruncatemacro{\ii}{\i + 1}%
  \pgfmathtruncatemacro{\jmin}{-4+\i}%
  \foreach \j in {\jmin,...,\jmax} {%
    \pgfmathtruncatemacro{\jm}{\j - 1}%
      \pgfmathtruncatemacro{\jmm}{\j - 1}
    \path
       (\ii,\j) coordinate (v-\ii-\j);
       (\i,\j)   coordinate (v-\i-\j)
      (\i,\jm)  coordinate (v-\ii-\jm)
     
    \draw[arrowedge] (v-\i-\j) -- (v-\i-\jm);
    \draw[arrowedge] (v-\i-\j) -- (v-\ii-\j);
  }%
}%
\draw[arrowedge] (7,5)--(7,4);
\draw[arrowedge] (7,4)--(7,3);


\draw[line width=3pt, red] (0,5)--(1,5);
\draw[line width=3pt, orange] (1,5)--(2,4);
\draw[line width=3pt, red] (2,4)--(3,4);
\draw[line width=3pt, blue!60!red] (3,4)--(6,4);
\draw[line width=3pt, blue] (6,4)--(6,3);
\draw[line width=3pt, blue!60!red] (6,3)--(7,3);
\draw[line width=3pt, blue] (7,3)--(7,2);

\draw[line width=3pt, red] (0,4)--(1,4);
\draw[line width=3pt, orange] (1,4)--(3,2);
\draw[line width=3pt, blue!60!red] (3,2)--(6,2);
\draw[line width=3pt, blue] (6,2)--(6,1);

\draw[line width=3pt, orange] (0,3)--(1,2);
\draw[line width=3pt,red] (1,2)--(2,2);
\draw[line width=3pt, orange] (2,2)--(3,1);
\draw[line width=3pt, blue!60!red] (3,1)--(5,1);
\draw[line width=3pt, blue] (5,1)--(5,0);

\draw[line width=3pt, orange] (0,2)--(1,1);
\draw[line width=3pt,red] (1,1)--(2,1);
\draw[line width=3pt, orange] (2,1)--(3,0);
\draw[line width=3pt, blue!60!red] (3,-1)--(4,-1);
\draw[line width=3pt, blue] (3,0)--(3,-1);

\draw[line width=3pt,orange] (0,1)--(3,-2);

\foreach \i in {3,...,6} {%
  \pgfmathtruncatemacro{\ii}{\i + 1}%
  \pgfmathtruncatemacro{\jmin}{-4+\i}%
  \foreach \j in {\jmin,...,\jmax} {%
    \pgfmathtruncatemacro{\jm}{\j - 1}%
      \pgfmathtruncatemacro{\jmm}{\j - 1}
    \path
       (\ii,\j) coordinate (v-\ii-\j);
       (\i,\j)   coordinate (v-\i-\j)
      (\i,\jm)  coordinate (v-\ii-\jm)
     
    \node (v\i\j) at (\i,\j) {};
    \node (v\i\jm) at (\i,\jm) {};
    \node (v\ii\j) at (\ii,\j) {};
  }%
}%

\foreach \i in {0,...,2} {%
  \pgfmathtruncatemacro{\ii}{\i + 1}%
  \pgfmathtruncatemacro{\jmin}{-\i+1}%
  \foreach \j in {\jmin,...,\jmax} {%
    \pgfmathtruncatemacro{\jm}{\j - 1}%
    \path
      (\i,\j)   coordinate (v-\i-\j)
      (\ii,\j)  coordinate (v-\ii-\j)
      (\ii,\jm) coordinate (v-\ii-\jm);
    \node (v\i\j) at (\i,\j) {};
    \node (v\ii\j) at (\ii,\j) {};
  }%
}

\node (x,y) at (6,3) {};

\end{tikzpicture}
        \caption{A perfect matching $M$ of $\bgbsw_{7,3}$ and the corresponding path configuration $(\pi_1(M),\ldots,\pi_5(M))$ of $G_{5,3}^{\beta \Gamma}$ under the bijection described in \Cref{thm:bgbsw_to_polymer_deterministic}.} 
        \label{fig:bgbsw_to_poly}
\end{figure}
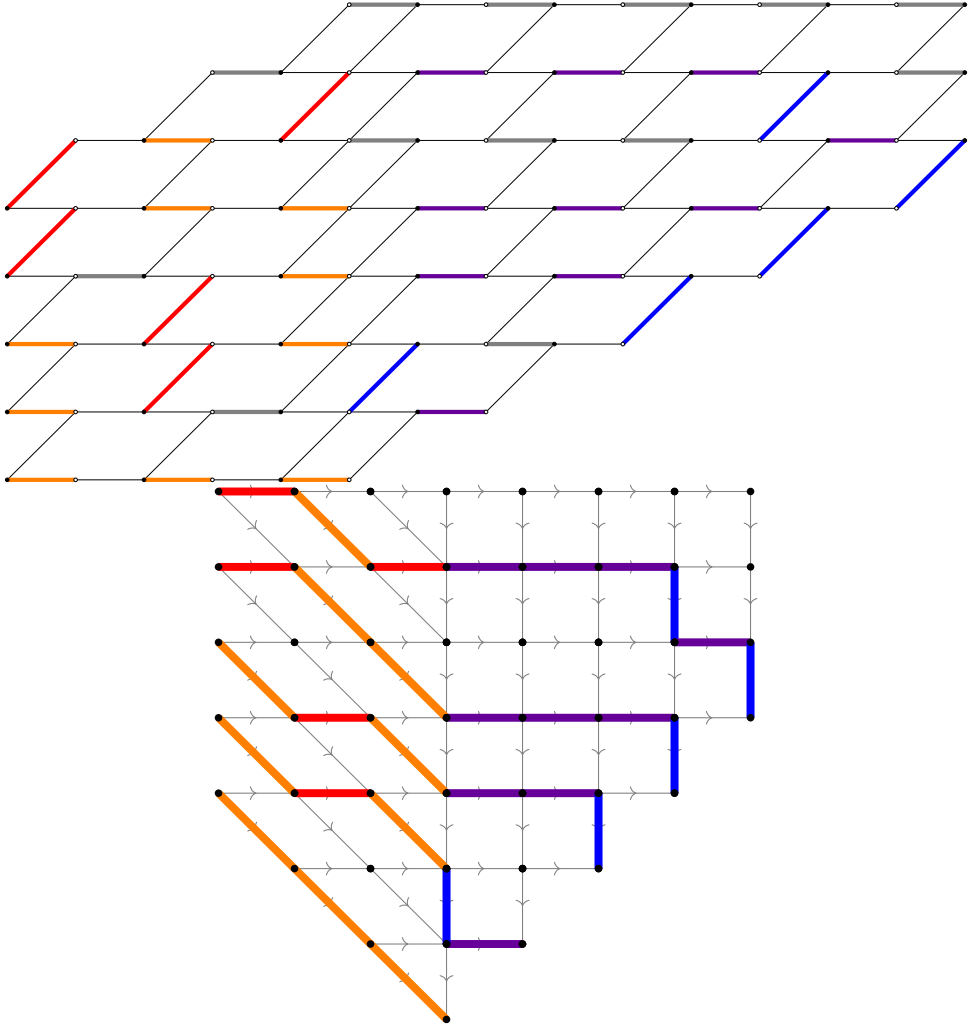

\begin{lemma}\label{thm:vert_slice_polymer_det_weights}
    Let $n \geq 1$ and $G_n^{\Az}$ be an Aztec diamond graph with deterministic positive real weights $\{a_{i,j}^{[n]},b_{i,j}^{[n]}: 1 \leq i,j \leq n\}$ as in \Cref{fig:our_weights}. Fix $1 \leq \ell \leq n$ and let $X_\ell^{\Az}(M)$ be as in \Cref{def:aztec_matching_column} for a random matching $M$ sampled by the dimer measure. Let $p=n-\ell+1$, and let $G_{p,n-p+1}^{\beta \Gamma}$ be the directed graph of \Cref{def:beta_sw}, with weights 
    \begin{align}
    \begin{split} \label{eq:poly_from_aztec_weights2}
    \gamma_{x,y} &= a^{[n-x]}_{\ell+1,-y}+b^{[n-x]}_{\ell+1,-y} \\ 
    \beta_{x,y} &= \frac{a^{[n+x+1]}_{\ell+x+1,-y}}{a^{[n+x+1]}_{\ell+x+1,-y}+b^{[n+x+1]}_{\ell+x+1,-y}},
    \end{split}
    \end{align}
    given in terms of the updated weights of the Aztec defined via \eqref{eq:downshuffle_weights}, and let $\pi_1,\ldots,\pi_p$ be distributed by the polymer measure. Then
    \begin{equation}
        X_\ell^{\Az}(M) = X^{poly}(\pi_1,\ldots,\pi_p) \quad \quad \quad \quad \text{ in distribution.}
    \end{equation} 
\end{lemma}
\begin{proof}
By \Cref{thm:aztec_to_bgbsw}, there is a distributional equality $X_\ell^{\Az}(M) = \bX(\tM)$ between the point configurations defined in \Cref{def:aztec_matching_column}. Combining this with \Cref{thm:bgbsw_to_polymer_deterministic} completes the proof. 
\end{proof}

\begin{thm}
    \label{thm:vert_slice_polymer}
    Fix $n \in \Z_{\geq 1}$ and let $(\phi_{j-n})_{1 \leq j \leq n}, (\psi_j)_{1 \leq j \leq n}, (\theta_i)_{1 \leq i \leq n}$ be real parameters such that $\psi_j+\theta_i$ and $\phi_{j-n}-\theta_i$ are positive for each $1 \leq i,j \leq n$. Consider a Gamma-disordered Aztec diamond with these parameters (\Cref{def:gamma_weights_intro_general}), and let $M$ be a random perfect matching distributed by the corresponding dimer measure.

    Consider also an independent mixed $\beta$-$\Gamma$ polymer on $G^{\beta \Gamma}_{p,m}$ with $p=n-\ell+1$ paths and $m=\ell$ layers in the notation of \Cref{def:beta_sw}, with independent weights\footnote{One can of course match these parameters $\psi_{\cdots},\phi_{\cdots},\theta_{\cdots}$ in the weights to the parameters $\eta_j,\xi_j,\tau_i$ in the notation of \Cref{def:beta_sw}.} 
    \begin{align}\label{eq:beta_in_match}
        \beta_{x,y} &\sim \Beta(\psi_{-y}+\theta_{\ell+1+x},\phi_{-y-n-x-1}-\theta_{\ell+1+x}) \\ 
        \gamma_{x,y} &\sim \Gamma(\psi_{-y}+\phi_{-y+x-n},1).\label{eq:gamma_in_match}
    \end{align}

     Let $\pi_1,\ldots,\pi_p$ be paths (ordered from top to bottom as in \Cref{def:beta_sw}) distributed according to the corresponding polymer measure. Then
    \begin{equation}
        X_\ell^{\Az}(M) = X^{poly}(\pi_1,\ldots,\pi_p) \quad \quad \quad \quad \text{ in distribution.}
    \end{equation}
\end{thm}

\begin{proof}[Proof of \Cref{thm:vert_slice_polymer}]
    Defining updated weights in terms of $ a_{i,j}^{[n]},b_{i,j}^{[n]}$ via \eqref{eq:downshuffle_weights} as usual, we define random variables $\gamma_{x,y},\beta_{x,y}$ in terms of these updated weights via \eqref{eq:poly_from_aztec_weights2}. Then \Cref{thm:vert_slice_polymer_det_weights} gives an equality in distribution of the desired form. So it remains to show that the weights $\gamma_{x,y},\beta_{x,y}$ are all independent and have distributions as in \eqref{eq:beta_in_match}, \eqref{eq:gamma_in_match}. Independence of the weights in $\bgbsw_{n,\ell}$ follows by \Cref{thm:swap_weights_stay_independent}, and the weights in $G^{\beta \Gamma}_{n-\ell+1,\ell}$ have the same joint distribution so the latter weights are independent. To obtain the exact distributions, we note that 
    \begin{align}
        \begin{split}
            a_{i,j}^{[k]}& \sim \Gamma(\psi_j+\theta_i,1)\\ 
            b_{i,j}^{[k]}& \sim \Gamma(\phi_{j-k}-\theta_i,1)
        \end{split}
    \end{align}
    by \eqref{eq:n_level_parameters}. Hence 
    \begin{align}
        \beta_{x,y} &:= \frac{a^{[n+x+1]}_{\ell+x+1,-y}}{a^{[n+x+1]}_{\ell+x+1,-y}+b^{[n+x+1]}_{\ell+x+1,-y}} \sim \Beta(\psi_{-y}+\theta_{\ell+1+x},\phi_{-y-n-x-1}-\theta_{\ell+1+x}) \\ 
        \gamma_{x,y} &:=  a^{[n-x]}_{\ell+1,-y}+b^{[n-x]}_{\ell+1,-y} \sim \Gamma(\psi_{-y}+\phi_{-y+x-n},1)
    \end{align}
    by \Cref{thm:lukacs}. This completes the proof.
\end{proof}

\section{Dynamical matching: shuffling trajectories and polymer paths}\label{sec:dynamical_vert}

The previous section exhibited a distributional equality between vertical slices of an Aztec matching and a multi-path polymer model $G^{\beta \Gamma}_{p,m}$, namely \Cref{thm:vert_slice_polymer}. Now that we have this polymer model, it is natural to ask if the path trajectories on $G^{\beta \Gamma}_{p,m}$ outside of this single slice mean anything for the Aztec diamond. It turns out that they do: they track the trajectory of our chosen slice of the matching as it grows via the shuffling algorithm. We first describe this algorithm, then formulate the main result, \Cref{thm:dynamical_match_vert}.

\subsection{Shuffling algorithm from column-swapping}\label{subsec:shuffling}

The shuffling algorithm is an algorithm to generate a random perfect matching of a size $n$ Aztec diamond with arbitrary weights. In the process of doing so, it actually generates a coupled sequence of matchings of the Aztec diamonds of sizes $1,2,\ldots,n$. We are interested not just in the size $n$ Aztec matching, but in the coupled sequence. Our formulation here will be in terms of commutation of $(+)$- and $(-)$-columns (\Cref{def:plus_and_minus_cols}), by contrast with the discussion in the Introduction and in \Cref{sec:prelim}. 

As in \Cref{sec:vert_polymer} we draw all graphs on $\Z^2$. All graphs considered in this section will be sequences of $(+)$- and $(-)$-columns in some order, capped by columns of $1$-valent black vertices on the right and left, as with $G_n^{vert}$ and $G^{v-swap}_{n,\ell}$ in the previous section.

Because the commutation of $(+)$- and $(-)$-columns (\Cref{thm:swap_columns_dimer}) involves only the spider move and vertex-contraction, \Cref{thm:spider_coupling} and \Cref{thm:vertex_expansion} give a natural coupling between dimer covers before the commutation and after the commutation. Hence there is a natural coupling---or, equivalently, transition kernel between---matchings of two graphs such as those represented in \Cref{fig:shuffling_graphs} (middle and bottom), which are copies of $G_n^{vert}$ and $G_{n+1}^{vert}$ with extra frozen ends.

This transition kernel is the basic step of the shuffling algorithm. Note that one gets two transition kernels, the `up-shuffle' from size $n$ to size $n+1$, and the `down-shuffle' from size $n+1$ to size $n$. These are both equivalent ways of viewing the coupling between the two. The down-shuffle is a deterministic map on tilings, in contrast to the up-shuffle.

To couple multiple matchings of the Aztec diamond at different $n$, we will augment the Aztec diamond by additional $(-)$-columns on the left and $(+)$-columns on the right. These additional columns are frozen (there is only one way to tile them in a perfect matching) as in the proof of \Cref{thm:gbsw_frozen_bits}. We may then repeat the above transition map multiple times. 

So, it is natural to begin with the graph $G^{v-swap}_{n,n+1}$ as defined in \Cref{def:gbsw}, see \Cref{fig:shuffling_graphs} (top). This graph has two connected components and is composed of $n$ $(-)$-columns, followed by $n$ $(+)$-columns, with a column of pendant edges on each side as in $G_n^{vert}$ or  $G^{v-swap}_{n,\ell}$ in the previous section. The weights are as described in \Cref{def:gbsw}.

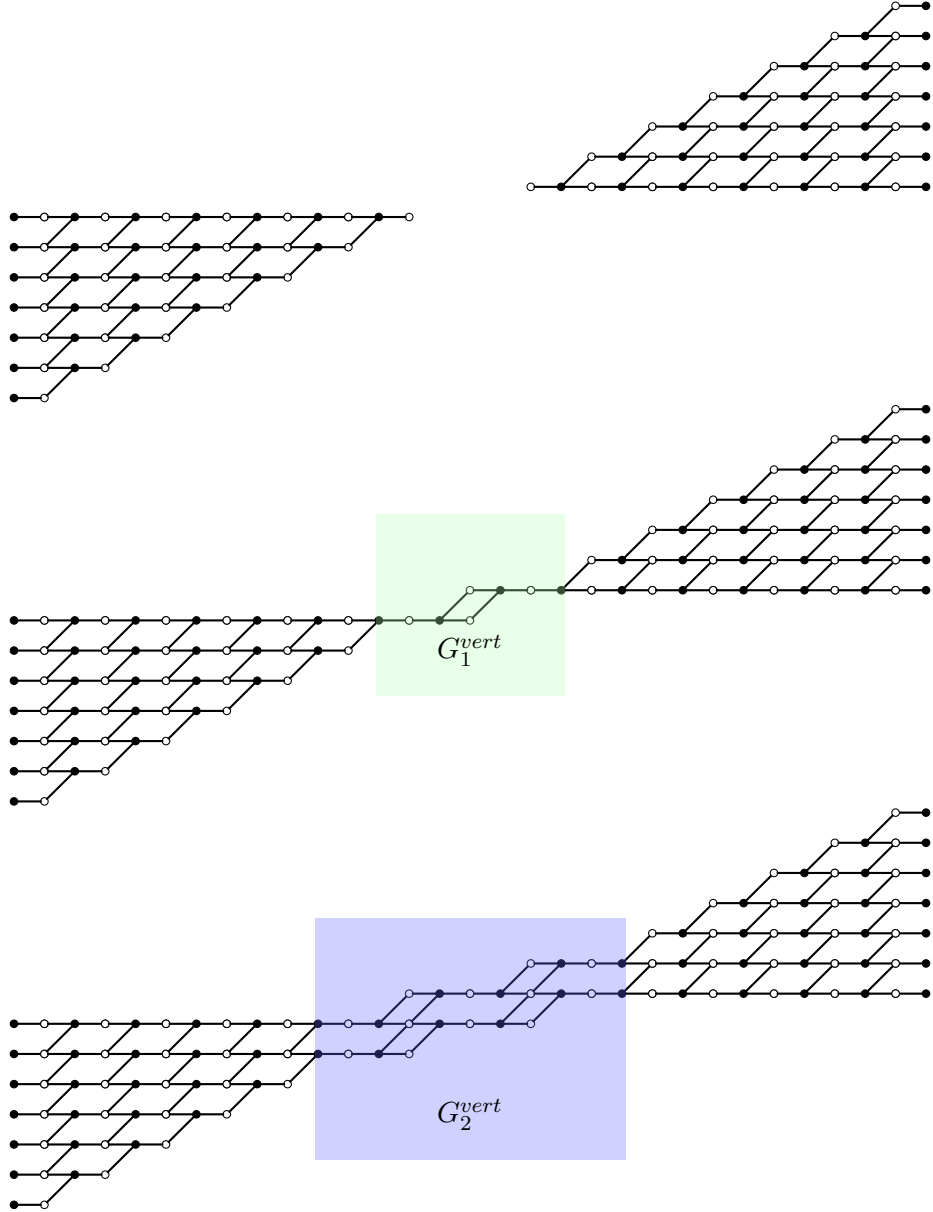
\begin{figure}
        \centering
\begin{tikzpicture}[scale=0.4,
    whiteV/.style={circle,draw,fill=white,inner sep=1pt},
    blackV/.style={circle,draw,fill=black,inner sep=1pt},
    edge/.style={
  line width=0.8pt,
  shorten <=1pt,   
  shorten >=1pt    
}
]

\foreach \y in {-1,-2,-3,-4,-5,-6,-7}{
  \node[blackV] at (0,\y) {};
  \node[whiteV] at (1,\y) {};
  \draw[edge] (0,\y) -- (1,\y);

}

\foreach \i/\m in {1/6,3/5,5/4,7/3,9/2,11/1}{
  \foreach \j in {1,...,\m}{\node[whiteV] at (\i,   -\j) {};}
  \foreach \j in {1,...,\m}{\node[blackV] at (\i+1, -\j) {};}
  \foreach \j in {1,...,\m}{\node[whiteV] at (\i+2, -\j) {};}
  \foreach \j in {1,...,\m}   {\draw[edge] (\i,   -\j) -- (\i+1, -\j);}    
  \foreach \j in {1,...,\m}   {\draw[edge] (\i,   -\j-1) -- (\i+1, -\j);}  
  \foreach \j in {1,...,\m}   {\draw[edge] (\i+1, -\j) -- (\i+2, -\j);}    
}


\foreach \i/\m in {17/1,19/2,21/3,23/4,25/5,27/6}{
  \foreach \j in {1,...,\m}{\node[whiteV] at (\i,   \j-1) {};}
  \foreach \j in {1,...,\m}{\node[blackV] at (\i+1, \j-1) {};}
  \foreach \j in {1,...,\m}{\node[whiteV] at (\i+2, \j-1) {};}
  \foreach \j in {1,...,\m}   {\draw[edge] (\i,   \j-1) -- (\i+1, \j-1);}    
  \foreach \j in {1,...,\m}   {\draw[edge] (\i+1,   \j-1) -- (\i+2, \j);}  
  \foreach \j in {1,...,\m}   {\draw[edge] (\i+1, \j-1) -- (\i+2, \j-1);}    
}


\foreach \y in {0,...,6}{
  \draw[edge] (29,\y) -- (30,\y);
  \node[blackV] at (30,\y) {};
  \node[whiteV] at (29,\y) {};
}

\end{tikzpicture}

\begin{tikzpicture}[scale=0.4,
    whiteV/.style={circle,draw,fill=white,inner sep=1pt},
    blackV/.style={circle,draw,fill=black,inner sep=1pt},
    edge/.style={
  line width=0.8pt,
  shorten <=1pt,   
  shorten >=1pt    
}
]

\foreach \y in {-1,-2,-3,-4,-5,-6,-7}{
  \node[blackV] at (0,\y) {};
  \node[whiteV] at (1,\y) {};
  \draw[edge] (0,\y) -- (1,\y);

}

\foreach \i/\m in {1/6,3/5,5/4,7/3,9/2,11/1}{
  \foreach \j in {1,...,\m}{\node[whiteV] at (\i,   -\j) {};}
  \foreach \j in {1,...,\m}{\node[blackV] at (\i+1, -\j) {};}
  \foreach \j in {1,...,\m}{\node[whiteV] at (\i+2, -\j) {};}
  \foreach \j in {1,...,\m}   {\draw[edge] (\i,   -\j) -- (\i+1, -\j);}    
  \foreach \j in {1,...,\m}   {\draw[edge] (\i,   -\j-1) -- (\i+1, -\j);}  
  \foreach \j in {1,...,\m}   {\draw[edge] (\i+1, -\j) -- (\i+2, -\j);}    
}


\foreach \i/\m in {17/1,19/2,21/3,23/4,25/5,27/6}{
  \foreach \j in {1,...,\m}{\node[whiteV] at (\i,   \j-1) {};}
  \foreach \j in {1,...,\m}{\node[blackV] at (\i+1, \j-1) {};}
  \foreach \j in {1,...,\m}{\node[whiteV] at (\i+2, \j-1) {};}
  \foreach \j in {1,...,\m}   {\draw[edge] (\i,   \j-1) -- (\i+1, \j-1);}    
  \foreach \j in {1,...,\m}   {\draw[edge] (\i+1,   \j-1) -- (\i+2, \j);}  
  \foreach \j in {1,...,\m}   {\draw[edge] (\i+1, \j-1) -- (\i+2, \j-1);}    
}


\foreach \y in {0,...,6}{
  \draw[edge] (29,\y) -- (30,\y);
  \node[blackV] at (30,\y) {};
  \node[whiteV] at (29,\y) {};
}


\foreach \i/\m in {13/0}{
  
  \foreach \j in {0,...,\m}{\node[whiteV] at (\i,   \j-1) {};}
  \foreach \j in {0,...,\m}{\node[blackV] at (\i+1, \j-1) {};}
  \foreach \j in {0,...,\m}{\node[whiteV] at (\i+2, \j-1) {};}
  \foreach \j in {0,...,\m}   {\draw[edge] (\i,   \j-1) -- (\i+1, \j-1);}    
  \foreach \j in {0,...,\m}   {\draw[edge] (\i+1,   \j-1) -- (\i+2, \j);}  
  \foreach \j in {0,...,\m}   {\draw[edge] (\i+1, \j-1) -- (\i+2, \j-1);}    

}


\foreach \i/\m in {15/1}{
  \foreach \j in {1,...,\m}{\node[whiteV] at (\i,   -\j+1) {};}
  \foreach \j in {1,...,\m}{\node[blackV] at (\i+1, -\j+1) {};}
  \foreach \j in {1,...,\m}{\node[whiteV] at (\i+2, -\j+1) {};}
  \foreach \j in {1,...,\m}   {\draw[edge] (\i,   -\j+1) -- (\i+1, -\j+1);}    
  \foreach \j in {1,...,\m}   {\draw[edge] (\i,   -\j+1-1) -- (\i+1, -\j+1);}  
  \foreach \j in {1,...,\m}   {\draw[edge] (\i+1, -\j+1) -- (\i+2, -\j+1);}    
}

\fill[green!30,opacity=0.3] (11.9,-3.5) rectangle (18.1, 2.5);

\node[] at (15,-2) {$G^{vert}_1$};

\end{tikzpicture}

\begin{tikzpicture}[scale=0.4,
    whiteV/.style={circle,draw,fill=white,inner sep=1pt},
    blackV/.style={circle,draw,fill=black,inner sep=1pt},
    edge/.style={
  line width=0.8pt,
  shorten <=1pt,   
  shorten >=1pt    
}
]

\foreach \y in {-1,-2,-3,-4,-5,-6,-7}{
  \node[blackV] at (0,\y) {};
  \node[whiteV] at (1,\y) {};
  \draw[edge] (0,\y) -- (1,\y);

}

\foreach \i/\m in {1/6,3/5,5/4,7/3,9/2}{
  \foreach \j in {1,...,\m}{\node[whiteV] at (\i,   -\j) {};}
  \foreach \j in {1,...,\m}{\node[blackV] at (\i+1, -\j) {};}
  \foreach \j in {1,...,\m}{\node[whiteV] at (\i+2, -\j) {};}
  \foreach \j in {1,...,\m}   {\draw[edge] (\i,   -\j) -- (\i+1, -\j);}    
  \foreach \j in {1,...,\m}   {\draw[edge] (\i,   -\j-1) -- (\i+1, -\j);}  
  \foreach \j in {1,...,\m}   {\draw[edge] (\i+1, -\j) -- (\i+2, -\j);}    
}


\foreach \i/\m in {19/2,21/3,23/4,25/5,27/6}{
  \foreach \j in {1,...,\m}{\node[whiteV] at (\i,   \j-1) {};}
  \foreach \j in {1,...,\m}{\node[blackV] at (\i+1, \j-1) {};}
  \foreach \j in {1,...,\m}{\node[whiteV] at (\i+2, \j-1) {};}
  \foreach \j in {1,...,\m}   {\draw[edge] (\i,   \j-1) -- (\i+1, \j-1);}    
  \foreach \j in {1,...,\m}   {\draw[edge] (\i+1,   \j-1) -- (\i+2, \j);}  
  \foreach \j in {1,...,\m}   {\draw[edge] (\i+1, \j-1) -- (\i+2, \j-1);}    
}


\foreach \y in {0,...,6}{
  \draw[edge] (29,\y) -- (30,\y);
  \node[blackV] at (30,\y) {};
  \node[whiteV] at (29,\y) {};
}


\foreach \i/\m in {11/1}{
  
  \foreach \j in {0,...,\m}{\node[whiteV] at (\i,   \j-1-1) {};}
  \foreach \j in {0,...,\m}{\node[blackV] at (\i+1, \j-1-1) {};}
  \foreach \j in {0,...,\m}{\node[whiteV] at (\i+2, \j-1-1) {};}
  \foreach \j in {0,...,\m}   {\draw[edge] (\i,   \j-1-1) -- (\i+1, \j-1-1);}    
  \foreach \j in {0,...,\m}   {\draw[edge] (\i+1,   \j-1-1) -- (\i+2, \j-1);}  
  \foreach \j in {0,...,\m}   {\draw[edge] (\i+1, \j-1-1) -- (\i+2, \j-1-1);}    

}

\foreach \i/\m in {15/1}{
  
  \foreach \j in {0,...,\m}{\node[whiteV] at (\i,   \j-1) {};}
  \foreach \j in {0,...,\m}{\node[blackV] at (\i+1, \j-1) {};}
  \foreach \j in {0,...,\m}{\node[whiteV] at (\i+2, \j-1) {};}
  \foreach \j in {0,...,\m}   {\draw[edge] (\i,   \j-1) -- (\i+1, \j-1);}    
  \foreach \j in {0,...,\m}   {\draw[edge] (\i+1,   \j-1) -- (\i+2, \j);}  
  \foreach \j in {0,...,\m}   {\draw[edge] (\i+1, \j-1) -- (\i+2, \j-1);}    

}


\foreach \i/\m in {13/2}{
  \foreach \j in {1,...,\m}{\node[whiteV] at (\i,   -\j+1) {};}
  \foreach \j in {1,...,\m}{\node[blackV] at (\i+1, -\j+1) {};}
  \foreach \j in {1,...,\m}{\node[whiteV] at (\i+2, -\j+1) {};}
  \foreach \j in {1,...,\m}   {\draw[edge] (\i,   -\j+1) -- (\i+1, -\j+1);}    
  \foreach \j in {1,...,\m}   {\draw[edge] (\i,   -\j+1-1) -- (\i+1, -\j+1);}  
  \foreach \j in {1,...,\m}   {\draw[edge] (\i+1, -\j+1) -- (\i+2, -\j+1);}    
}

\foreach \i/\m in {17/2}{
  \foreach \j in {1,...,\m}{\node[whiteV] at (\i,   -\j+1+1) {};}
  \foreach \j in {1,...,\m}{\node[blackV] at (\i+1, -\j+1+1) {};}
  \foreach \j in {1,...,\m}{\node[whiteV] at (\i+2, -\j+1+1) {};}
  \foreach \j in {1,...,\m}   {\draw[edge] (\i,   -\j+1+1) -- (\i+1, -\j+1+1);}    
  \foreach \j in {1,...,\m}   {\draw[edge] (\i,   -\j+1+1-1) -- (\i+1, -\j+1+1);}  
  \foreach \j in {1,...,\m}   {\draw[edge] (\i+1, -\j+1+1) -- (\i+2, -\j+1+1);}    
}

\fill[blue!60,opacity=0.3] (9.9,-5.5) rectangle (20.1, 2.5);

\node[] at (15,-4) {$G^{vert}_2$};

\end{tikzpicture}

\caption{The graph $G^{v-swap}_{7,8}$ before shuffling begins (top). Commuting the trivial columns in the middle of $G^{v-swap}_{7,8}$ yields a graph $G^{v-swap}_{7,7}$ with a copy of $G_1^{vert}$ in the middle and frozen regions outside (middle). Commuting the rightmost pair of $(-)$-columns with their right neighbors yields a graph with a copy of $G_2^{vert}$ in the middle (bottom) and frozen regions outside. Weights not pictured, but see \Cref{def:gbsw} for those of $G^{v-swap}_{7,8}$, from which all others are derived.}
                \label{fig:shuffling_graphs}
\end{figure}

\begin{figure}
        \centering
\includegraphics[scale=1]{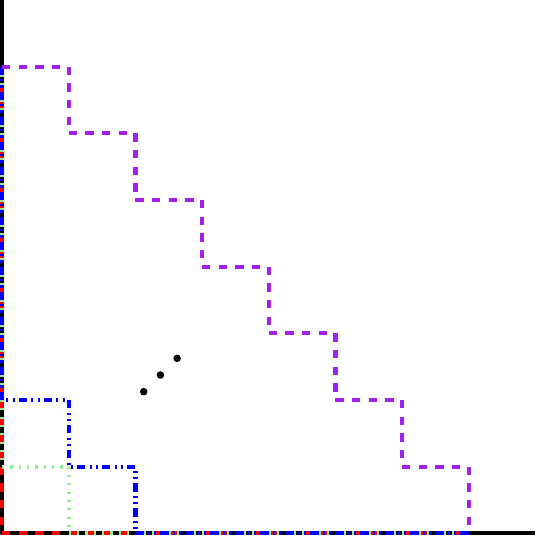}
        \caption{The diagrams for the pre-shuffling graph $G^{v-swap}_{7,8}$ (red), the size $1$ Aztec diamond (green), the size $2$ Aztec diamond (blue), all depicted in \Cref{fig:shuffling_graphs}. Also depicted is the size $7$ Aztec diamond $G^{vert}_7$ (purple). Note that the paths intersect the $y$- and $x$-axis in the first three cases; these portions correspond to frozen blocks of $(-)$- and $(+)$-columns on the corresponding graphs.}
                \label{fig:schur_graphs}
\end{figure}

As in the previous section, we will typically take the weights of the Gamma-disordered Aztec diamond with general parameters (\Cref{def:gamma_weights_intro_general}), but our distributional matchings also hold for arbitrary choices of deterministic real weights on the Aztec diamond, where the polymer weights are given as (deterministic) functions of these. Our random weights simply have the feature that both sets of weights are independent collections of random variables.

Commuting the rightmost $(-)$-column with the leftmost $(+)$-column in \Cref{fig:shuffling_graphs} (top), we obtain a graph $G_1^{vert}$ sandwiched between two frozen blocks of $n-1$ $(-)$- and $(+)$-columns on the left and right respectively, see \Cref{fig:shuffling_graphs} (middle). Commuting the rightmost two $(-)$-columns with the leftmost two $(+)$-columns then yields $G_2^{vert}$ sandwiched between two frozen regions (\Cref{fig:shuffling_graphs} (bottom)). And so on, until the $(+)$ and $(-)$-columns are in alternating order after the $n\tth$ such step. As mentioned, the coupling in \Cref{thm:spider_coupling} yields transition kernels between the dimer measures of each such graph. This is the shuffling algorithm.

The above description actually gives more than just a coupling of matchings on the sequence of graphs $G_j^{vert}, 1 \leq j \leq n$. It gives a coupling of matchings on all graphs coming from interchange of $(+)$- and $(-)$-columns on the original $G^{v-swap}_{n,n+1}$. This is because the coupling in \Cref{thm:spider_coupling} is local, so for any given pair of adjacent $(+)$- and $(-)$-columns, it does not matter which other columns are swapped before vs. after the given pair is swapped.

For instance, two column pairs must be commuted to get between the middle and bottom graph in \Cref{fig:shuffling_graphs}, and locality guarantees that it does not matter which order these swaps happen in. There is actually a coupling between matchings not just on those two graphs, but on those two graphs and the two `intermediate' graphs where either pair of columns is swapped but not the other.

\subsection{Main result} 
Sample a coupled sequence of matchings $M_1,M_2,\ldots$ of the Aztec diamond graphs of size $1,2,\ldots$ according to the shuffling algorithm as described above\footnote{As phrased above, the shuffling algorithm must stop at the graph $G_n^{vert}$, because there are no more $(-)$-columns on the left or $(+)$-columns on the right. However, by either noting that this algorithm is consistent for different $n$ or by simply defining an infinite graph, one may run the shuffling algorithm for all time $n$. Because we will only ever run it for a finite number of steps in what follows, we ignore this detail and simply choose the ambient graph large enough.}. Let $X^{\Az}_\ell(M_\tau)$ be the $\ell\tth$-column point configuration as defined in \Cref{def:aztec_matching_column}, and for convenience extend that definition by the convention 
\begin{equation}
    X^{\Az}_0(M_{n}) := \{1,\ldots,n+1\}
\end{equation} 
for any matching $M_{n}$ of $G^{\Az}_{n}$ or $G_n^{vert}$. Fix also a positive integer $k$, which will index the slice of the Aztec. We are interested in the random sequence 
\begin{equation}
    X^{\Az}_{\tau}(M_{\tau+k-1}), \quad \quad \tau \in \Z_{\geq 0},
\end{equation} 
which we view as a stochastic process in time $\tau$. Its state space is the set $\binom{\N}{k}$ of $k$-element subsets of $\N$, and our convention for $\tau=0$ simply makes $X^{\Az}_{0}(M_{k-1}) = [k]$ a deterministic initial condition. 

Then the first random step of this process corresponds to the first matching $M_k$ which is big enough have a $k\tth$ column, and since the $\tau\tth$ column from the left on $M_{\tau+k-1}$ is the $k\tth$ column from the right, this process describes the evolution of the $k\tth$ column \emph{from the right} of the matching. Our main result of this section, \Cref{thm:dynamical_match_vert}, matches these trajectories with the polymer path trajectories, which we define formally as follows.

\begin{defi}
    \label{def:path_coordinates_dynamical}
    Let $\pi_1,\ldots,\pi_p$ be the paths of a $\beta$-$\Gamma$ polymer on the directed graph $G_{p,m}^{\beta\Gamma}$ from \Cref{def:beta_sw}. For each path $\pi_j$, let $\pi_j(0),\pi_j(1),\ldots,\pi_j(m-1)$ be the negative $y$-coordinates of the vertices $(-m,\cdot),(-m+1,\cdot),\ldots,(-1,\cdot)$ on the path, and let $\pi_j(m)$ be as already defined in \Cref{def:X_polymer}\footnote{Here recall that there are several $y$-coordinates because the paths may move down, and \Cref{def:X_polymer} tells us how to select one.}. For each $0 \leq \tau \leq m$, let $\Pi(\tau) = (\pi_1(\tau),\ldots,\pi_p(\tau))$. For later use in \Cref{sec:hor_slice}, we allow this notation on any directed graph with the same underlying graph as $G_{p,m}^{\beta\Gamma}$ even if the weights are different.
\end{defi}

Note that $\Pi(m) = X^{poly}(\pi_1,\ldots,\pi_p)$ as defined in \Cref{def:X_polymer}.

\begin{example}\label{ex:Pi}
    In the path configuration of \Cref{fig:bgbsw_to_poly}, 
    \begin{align}
        \begin{split}
            \Pi(0) &= (1,2,3,4,5) \\ 
            \Pi(1) &= (1,2,4,5,6) \\ 
            \Pi(2) &= (2,3,4,5,7) \\ 
            \Pi(3) &= (2,4,5,6,8).
        \end{split}
    \end{align}
    $\Pi(3) = X^{poly}(\pi_1,\ldots,\pi_5)$ was already discussed in \Cref{fig:mixed_polymer_intro}.
\end{example}

\begin{thm}
    \label{thm:dynamical_match_vert}
    Let $(\phi_j)_{j \leq 0}, (\psi_j)_{j \geq 1}, (\theta_i)_{i \geq 1}$ be real parameters such that $\psi_j+\theta_i$ and $\phi_j-\theta_i$ are positive for all $i,j$. 
    Sample a coupled sequence of matchings $M_1,M_2,\ldots$ of the Aztec diamonds of size $1,2,\ldots$ according to the shuffling algorithm with these parameters (see \Cref{subsec:shuffling}).
    
    Then the `$k\tth$ column from the right' stochastic process $X^{\Az}_{\tau}(M_{\tau+k-1}), \tau \in \Z_{\geq 0}$, described above, has the following description. Fix any `final time' $T > 0$, and let $G_{k,T}^{\beta \Gamma}$ be a $\beta$-$\Gamma$ polymer with weights defined in terms of the Aztec weights as in \Cref{thm:bgbsw_to_polymer_deterministic}. Let $(\Pi(\tau))_{0 \leq \tau \leq T}$ be the joint polymer path trajectories as in \Cref{def:path_coordinates_dynamical}. Then
        \begin{equation}
        (X^{\Az}_{\tau}(M_{\tau+k-1}))_{0 \leq \tau \leq T} = (\Pi(\tau))_{0 \leq \tau \leq T}
    \end{equation}
    in distribution, for any fixed realization of the weights $a_{i,j}^{[n]},b_{i,j}^{[n]}$. 
\end{thm}

\begin{proof}
    $X_\ell^{\Az}(M)$ is defined on matchings $M$ of $G_n^{\Az}$. However, these are in bijection with matchings on $G_n^{vert}$, so $X_\ell^{\Az}$ gives a function on the latter. This function only depends on the $\ell\tth$ $(+)$-column from the left in such a graph, because $X_\ell^{\Az}$ only depends on the edges in the left part of the $\ell\tth$ column of $G_n^{\Az}$. 

    The random variables $X_\tau^{\Az}(M_{\tau+k-1})$ (when $\tau \geq 1$) thus depend on the first $(+)$-column of $M_k$, second $(+)$-column of $M_{k+1}$, etc. up until the $T\tth$ $(+)$-column of $M_{T+k-1}$, as depicted in \Cref{fig:left_slice_diagram}.

    \begin{figure}[H]
        \centering
\includegraphics[scale=1]{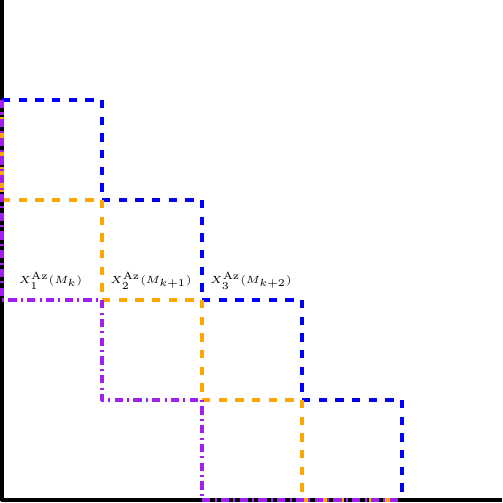}
        \caption{The previous discussion in the case $k=2$, $T = 3$. Purple, orange, and blue lines correspond to Aztec diamonds (embedded in frozen columns as in \Cref{subsec:shuffling}) of sizes $2$, $3$ and $4$ respectively. Horizontal segments corresponding to $(+)$-columns on which the relevant observables $X_\tau^{\Az}(M_{\tau+k-1})$ depend are labeled with those observables.}
                \label{fig:left_slice_diagram}
\end{figure}

    \begin{figure}[H]
        \centering
\includegraphics[scale=1]{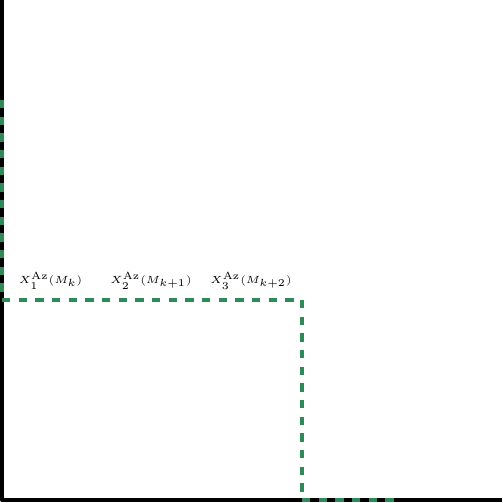}
        \caption{The depiction of $G^{v-swap}_{4,3}$ corresponding to \Cref{fig:left_slice_diagram}. Its non-frozen $(+)$-columns are exactly the same labeled ones in that figure. The segments of the dotted green line intersecting the axes correspond to the two frozen regions.}
                \label{fig:left_slice_gbsw}
\end{figure}

    However, these $(+)$-columns exactly correspond to the first block of $(+)$-columns of $G^{v-swap}_{T+k-1,T}$, see \Cref{fig:left_slice_gbsw}. Furthermore, by the discussion of \Cref{subsec:shuffling}, shuffling gives a coupling between matchings of $G^{v-swap}_{T+k-1,T}$ (or any other intermediate graph) and the matchings $M_1,\ldots,M_{T+k-1}$. Because the coupling in \Cref{thm:spider_coupling} is local, the (non-frozen) $(+)$-columns of $G^{v-swap}_{T+k-1,T}$ have exactly the same configurations of edges as the corresponding ones from Aztec shuffling, namely the first $(+)$-column of $M_k$, second $(+)$-column of $M_{k+1}$, etc. up until the $T\tth$ $(+)$-column of $M_{T+k-1}$.

    We thus obtain a distributional equality between $(X_\tau^{\Az}(M_{\tau+k-1}))_{1 \leq \tau \leq T}$ and the same statistics for $G^{v-swap}_{T+k-1,T}$. Matching the latter graph with the polymer $G^{\beta \Gamma}_{k,T}$ via the combinatorial bijection of \Cref{thm:bgbsw_to_polymer_deterministic} concludes the proof.
\end{proof}

\begin{rmk}
    By exactly the same argument as in the proof of \Cref{thm:dynamical_match_vert}, one may also match the right (strict-weak) half of $G^{\beta \Gamma}_{\cdots}$ with the distribution of the $\ell\tth$ $(-)$-columns from the left in the Aztec diamond. We omit this only to avoid setting up more cumbersome notation; it is no more difficult.
\end{rmk}

\section{Exact matching of East and West turning points with stationary polymers} \label{sec:turning_points}

Two cases of \Cref{thm:vert_slice_polymer} which are of particular interest to us are $\ell = 1$ and $\ell = n$, the leftmost and rightmost vertical slices of the Aztec diamond. These are the West and East turning points defined in \Cref{def:all_turning_points}.

\subsection{The East turning point}

We begin with the second case $\ell=n$, which is slightly easier. The below random walk model was to our knowledge first considered by Korotkikh \cite{korotkikh2022hidden}, deforming the case introduced by Barraquand-Corwin \cite{barraquand2017random}. 

\begin{defi}\label{def:beta_rwre}
    Let $(\psi_j)_{j \geq 1}, (\phi_j)_{j \leq 0}, (\theta_i)_{i \geq 1}$ be real numbers such that $\psi_j+\theta_i$ and $\phi_{1-j}-\theta_i$ are positive for all $i,j \geq 1$. Then the \emph{deformed Beta-random walk in random environment (RWRE)} $X_t, t \in \Z_{\geq 0}$ (with the above deformation parameters) is the random walk on the negative integers with initial condition $X_0=-1$ and transition probabilities
    \begin{align}
    \begin{split}
        \mathbb{P}(X_{t+1} = y | X_t=y) &= B_{y,t}\\ 
        \mathbb{P}(X_{t+1} = y-1 | X_t = y) &= 1-B_{y,t},
    \end{split}
    \end{align}
    where 
    \begin{equation}
        B_{y,t} \sim \Beta(\psi_{-y} + \theta_{t+1},\phi_{-y-t-1} - \theta_{t+1})
    \end{equation}
    are mutually independent Beta variables.
\end{defi}

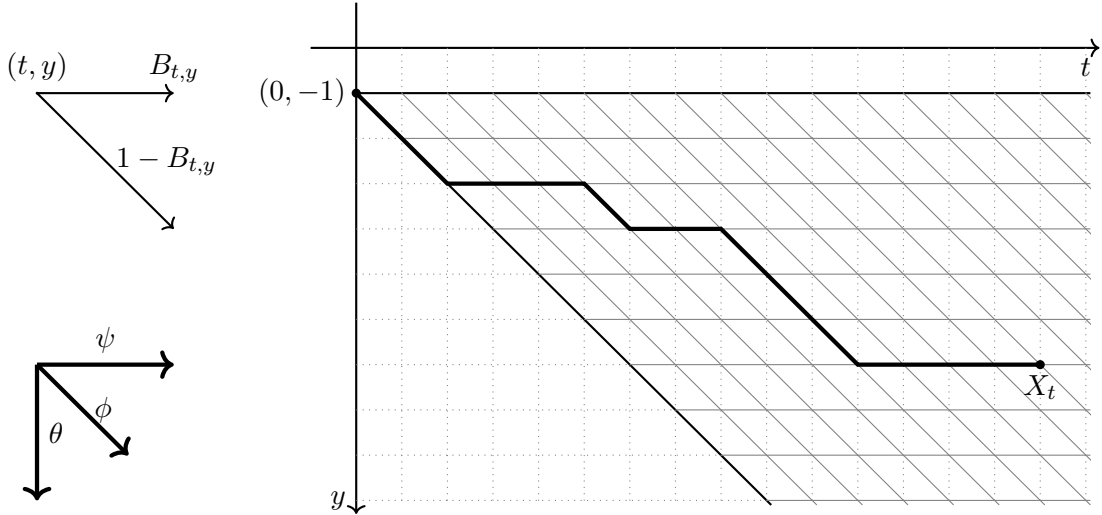
\begin{figure}

 \centering
\begin{tikzpicture}[scale=0.6]

\draw[->, thick] (0,1) -- (0,-10.3);
\draw[->, thick] (-1,0) -- (16.3,0);

\draw[thick] (0,-1) -- (9.1,-10.1);
\draw[thick] (0,-1) -- (16.1,-1);

\foreach \k in {1,2,...,16} {
    \draw[gray, dotted] (\k,0) -- (\k,-10.1);
}

\foreach \k in {2,3,...,10} {
    \draw[gray] (\k-1,-\k) -- (16.1,-\k);
}
\foreach \k in {2,3,...,10} {
    \draw[gray, dotted] (0,-\k) -- (\k-1,-\k);
}

\clip (-8,-10.3) rectangle (16.1,1);
\foreach \k in {1,2,...,16} {
    \draw[gray] (\k,-1) -- (\k+9.1,-10.1);
}

\draw[ultra thick] (0,-1) -- (2,-3) -- (5,-3) -- (6,-4) -- (8,-4) -- (11,-7) -- (15,-7);

\fill (0,-1) circle(0.1);
\fill (15,-7) circle(0.1);


\draw (16,0) node[anchor=north]{$t$};
\draw (0,-10) node[anchor=east]{$y$};
\draw (15,-7) node[anchor=north]{$X_t$};
\draw (0,-1) node[anchor=east]{$(0,-1)$};

\begin{scope}[shift={(-7,-1)}]
\coordinate (P) at (0,0);

\draw[->, thick] (P) -- ++(3,0) node[right, above] {$B_{t,y}$};

\draw[->, thick] (P) -- ++(3,-3) node[midway, right] {$1-B_{t,y}$};

\node[circle, fill=black, inner sep=.1pt, label=above:{$(t,y)$}] (P) at (0,0) {};
\end{scope}

\begin{scope}[shift={(-7,-7)}]
\coordinate (P) at (0,0);

\draw[->, ultra thick] (P) -- ++(3,0) node[midway, above] {$\psi$};

\draw[->, ultra thick] (P) -- ++(2,-2) node[midway, right] {$\phi$};

\draw[->, ultra thick] (P) -- ++(0,-3) node[midway, right] {$\theta$};

\end{scope}





\end{tikzpicture}

\caption{The Beta-RWRE of \Cref{def:beta_rwre} (right), with step distributions (left, upper) and directions along which the $\psi,\phi$ and $\theta$ parameters are constant (left, lower). }\label{fig:beta_rwre}

\end{figure}

\begin{rmk}
    The Beta-RWRE is usually defined as a sheared version of the above with steps $\pm 1$ and state space $\Z$ in works such as \cite{barraquand2017random} and \cite{korotkikh2022hidden}; the reason we shear as above is to match with our previous notation on the Aztec diamond and $\beta$-$\Gamma$ polymer. Our parameters $\psi,\theta,\phi$ correspond to $\tilde{\sigma},-\tilde{\rho},-\tilde{\omega}$ in the notation of \cite[Section 1.3]{korotkikh2022hidden}, with indices translated appropriately.
\end{rmk}
The following result follows from \Cref{thm:dynamical_match_vert}.
\begin{cor}
    \label{thm:right-slice_to_rwre}
    Let $(\phi_j)_{j \leq 0}, (\psi_j)_{j \geq 1}, (\theta_i)_{i \geq 1}$ be real parameters such that $\psi_j+\theta_i$ and $\phi_{1-j}-\theta_i$ are positive for all $i,j \geq 1$. Sample a coupled sequence of matchings $M_1,M_2,\ldots$ of the Aztec diamonds of size $1,2,\ldots$ according to the shuffling algorithm with these parameters (see \Cref{subsec:shuffling}). Then for any $T \geq 0$, 
    \begin{equation}\label{eq:right_slice_to_rwre}
        T^{East}(M_\tau) = -X_\tau - 1 \quad \quad \quad \text{in joint distribution over $0 \leq \tau \leq T$,}
    \end{equation}
    where $X_\tau$ is an independent Beta-RWRE as in \Cref{def:beta_rwre} with the same parameters $\psi_j,\phi_j,\theta_i$ as the Aztec diamond, and we recall the notation $T^{East}$ from \Cref{def:all_turning_points}.
\end{cor}
\begin{proof}
    $T^{East}(M_\tau)+1$ is the unique element of the set $X_\tau^{\Az}(M_\tau)$ in the notation of \Cref{def:aztec_matching_column}. Hence the $k=1$ case of \Cref{thm:dynamical_match_vert} yields that
    \begin{equation}
        (T^{East}(M_\tau)+1)_{0 \leq \tau \leq T} = (\Pi(\tau))_{0 \leq \tau \leq T}
    \end{equation}
    in distribution, where $\Pi(\tau)$ gives the negative $y$-coordinate of a path on $G_{1,T}^{\beta\Gamma}$ (recall \Cref{def:path_coordinates_dynamical}). Since the strict-weak (right) part of $G_{1,T}^{\beta\Gamma}$ consists of a single column, all weights on edges that our directed path may encounter are $1$. So the path in $G_{1,T}^{\beta\Gamma}$ moves through $T$ columns of Beta-distributed directed edges, and then exits the bottom. Furthermore, the parameters on these Beta-distributed edges exactly match those of \Cref{def:beta_rwre}, and this proves \eqref{eq:right_slice_to_rwre}.
    
\end{proof}

\subsection{The West turning point}

We proceed to the $\ell=1$ case, which gives the turning point at the West boundary of the Aztec diamond and proves \Cref{thm:west_matching_intro}. This case is harder than the East turning point because $G_{n,1}^{\beta \Gamma}$ has $n$ paths rather than a single path as in the previous case. However, it turns out that we can relate it to a model with only a single path through a certain complementation bijection. We first restate \Cref{def:stationary_log_gamma_intro} for convenience, slightly generalized to points $(m,n)$ rather than $(n,n)$ but otherwise identical.

\begin{defi}
    \label{def:stationary_log_gamma}
    Let $\alpha,\beta > 0$. Place an independent random weight $Y_{i,j}$ on every point $(i,j) \in \Z_{\geq 0}^2$, where $Y_{0,0}=1$ and those with $(i,j) \neq (0,0)$ have distributions 
    \begin{equation}
        Y_{i,j} \sim \begin{cases}
             \Gamma^{-1}(\beta,1) & j=0 \\ 
             \Gamma^{-1}(\alpha,1) & i = 0 \\ 
             \Gamma^{-1}(\alpha+\beta,1) & i,j > 0
        \end{cases}.
    \end{equation}
    These are known as \emph{stationary log-Gamma polymer} weights. 
    
    Let $\Pi_{m,n}$ be the set of up-right paths with steps $(i,j) \to (i+1,j)$ and $(i,j) \to (i,j+1)$, beginning at $(0,0)$ and ending at $(m,n)$. Then the associated partition function is 
    \begin{equation}
        Z^{stat-\Gamma}_{m,n} := \sum_{\pi \in \Pi_{m,n}} \prod_{(i,j) \in \pi} Y_{i,j},
    \end{equation}
    where the product is over the weights on all lattice points in the path. The associated \emph{quenched} polymer measure on $\Pi_{m,n}$ is 
    \begin{equation}
        Q^{stat-\Gamma}_{m,n}(\pi) = \frac{\prod_{(i,j) \in \pi} Y_{i,j}}{Z^{stat-\Gamma}_{m,n}},
    \end{equation}
    a random probability measure depending on the weights. The associated \emph{annealed} polymer measure on $\Pi_{m,n}$ is 
    \begin{equation}
        \P_{m,n}^{stat-\Gamma}(\pi) = \E[Q^{stat-\Gamma}_{m,n}(\pi)].
    \end{equation}
\end{defi}

The annealed probability measure just gives the distribution of a random path given by first sampling the polymer environment of Gamma variables, then sampling a path with respect to the quenched polymer measure $Q^{stat-\Gamma}_{m,n}$ associated to those weights. The following result is the \emph{raison d'\^etre} of the above weights.

\begin{prop}[Burke property]
    \label{thm:burke}
    In the notation of \Cref{def:stationary_log_gamma}, 
    let
    \begin{align}
        \begin{split}
            U_{m,n} &= \frac{Z^{stat-\Gamma}_{m,n}}{Z^{stat-\Gamma}_{m-1,n}} \\ 
            V_{m,n} &= \frac{Z^{stat-\Gamma}_{m,n}}{Z^{stat-\Gamma}_{m,n-1}}
        \end{split}
    \end{align}
    whenever $m\geq 1, n \geq 0$ (for $U_{m,n}$) or $m \geq 0, n \geq 1$ (for $V_{m,n}$) so both partition functions are defined. Then 
    \begin{align}
        \begin{split}
            U_{m,n}&\sim \Gamma^{-1}(\beta,1) \\ 
            V_{m,n} &\sim \Gamma^{-1}(\alpha,1)
        \end{split}
    \end{align}
    for every $(m,n)$. Furthermore, associating $U_{m,n}$ and $V_{m,n}$ to the directed horizontal edge $(m-1,n) \to (m,n)$ and vertical edge $(m,n) \to (m,n-1)$ respectively, we have that for any down-right path in $\Z_{\geq 0}^2$, the $U$ and $V$ variables associated to its right and down segments are all mutually independent.
\end{prop}
\begin{proof}
    This is a special case of \cite[Theorem 3.3]{seppalainen2012scaling}, which also treats independence of certain additional bulk weights $X_{m,n}$ which we ignore here.
\end{proof}

The source of the name `stationary' in \Cref{def:stationary_log_gamma} is the following result, which follows directly from the Burke property. We will not use it directly until \Cref{sec:turning_points}, but it is natural to prove it here.

\begin{cor}
    \label{thm:shrink_rectangle}
    Let $1 \leq m \leq M$ and $1 \leq n \leq N$ be integers. Let 
    \begin{equation}
        s_{M-m,N-n}(x,y) = (x-(M-m),y-(N-n))
    \end{equation}
    be the shift map. Let $\pi \in \Pi_{M,N}$ and $\pi' \in \Pi_{m,n}$ be distributed by two independent copies of $\Psg_{M,N}$ and $\Psg_{m,n}$ respectively. Then the marginals $\pi' \cap ([1,m] \times [1,n])$ and $s_{M-m,N-n}(\pi \cap [M-m+1,M] \times [N-n+1,N])$ are equal in distribution.
\end{cor}
\begin{proof}
    By \Cref{thm:burke}, the ratios of partition functions $Z_{i,N-n}/Z_{i-1,N-n}$ and $Z_{M-m,j}/Z_{M-m,j-1}$ in the larger rectangle are independent and have the same distributions as the boundary weights $Z_{i,0}/Z_{i-1,0}$ and $Z_{0,j}/Z_{0,j-1}$ in the smaller $m \times n$ rectangle. \Cref{thm:shrink_rectangle} follows immediately. \end{proof}
    
    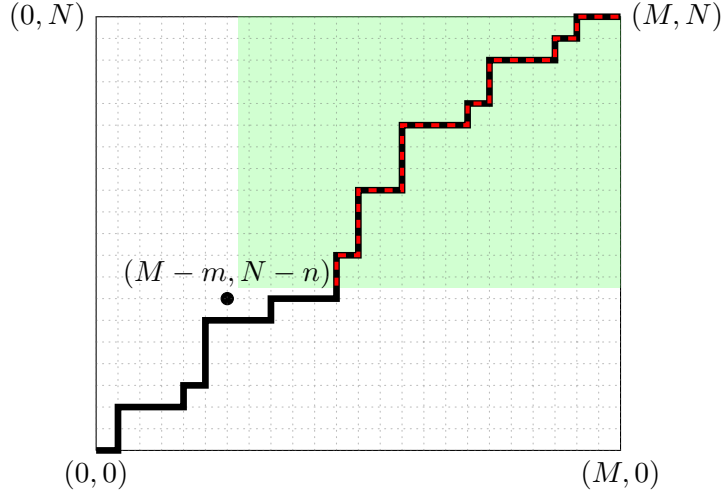
\begin{figure}
            \begin{tikzpicture}[scale=1.15]
\draw (0,0) -- (6,0) -- (6,5) -- (0,5) -- (0,0);

\path[fill=green!60, fill opacity=0.30, draw=none]
    (1.5+.125,1.75+.125) rectangle (6,5);


\node [below] at (6,0) {$(M,0)$};

\node [below] at (0,0) {$(0,0)$};

\node [left] at (0,5) {$(0,N)$};

\node [right] at (6,5) {$(M,N)$};

\node [above] at (1.5,1.75) {$(M-m,N-n)$};

 \draw[fill=black] (1.5,1.75) circle(2pt); 

\draw[step=0.25, dotted, draw=black!60, opacity=0.45, line cap=round] (0,0) grid (6,5);

\draw [line width=2.5pt]
  (0,0) -- (0.25,0) -- (0.25,0.5) -- (1,0.5) -- (1,0.75) -- (1.25,0.75)
  -- (1.25,1.5) -- (2,1.5) -- (2,1.75) -- (2.75,1.75) -- (2.75,2.25) -- (3,2.25)
 -- (3,2.75) -- (3,3) -- (3.5,3) -- (3.5,3.75)
  -- (4.25,3.75) -- (4.25,4) -- (4.5,4) -- (4.5,4.5) -- (5.25,4.5)
  -- (5.25,4.75) -- (5.5,4.75) -- (5.5,5) -- (6,5);

\draw [red,dashed,line width=1.5pt](2.75,1.9) -- (2.75,2.25) -- (3,2.25)
 -- (3,2.75) -- (3,3) -- (3.5,3) -- (3.5,3.75)
  -- (4.25,3.75) -- (4.25,4) -- (4.5,4) -- (4.5,4.5) -- (5.25,4.5)
  -- (5.25,4.75) -- (5.5,4.75) -- (5.5,5) -- (6,5);

\end{tikzpicture}




        \caption{A path $\pi \sim \P^{stat-\Gamma}_{M,N}$ starting at $(0,0)$ (black). \Cref{thm:shrink_rectangle} says that the portion inside the green rectangle (red) as the same (annealed) distribution as the interior portion of a path $\pi' \sim \P^{stat-\Gamma}_{m,n}$ in an independent environment.}
        \label{fig:smaller_box}
    \end{figure}

The quantity of interest to us is the point at which a polymer path intersects the line $x+y=n$.

\begin{defi}
    \label{def:polymer_midpoint}
    For paths $\pi$ as in \Cref{def:stationary_log_gamma} with $n \leq m$, let $x_{mid}(\pi)$ be the unique $x$-value such that $(x,n-x) \in \pi$ (see \Cref{fig:polymer_xmid_intro}).
\end{defi}

\begin{rmk}
     We only prove the `single-time' matching in \Cref{thm:west_matching_intro} because that is all we use in our asymptotics, though we note that it is not hard to prove a dynamical version similar to \Cref{thm:right-slice_to_rwre}.
\end{rmk}

We begin on the Aztec diamond side, first matching the West turning point (with fully general parameters) with a related single-path log-Gamma polymer model, given in \Cref{def:boundary-weighted_log-gamma} and depicted in \Cref{fig:dual_free_weight_gamma_path}, via a path-complementation bijection. We will then specialize parameters and show how to match this polymer with the version in \Cref{def:stationary_log_gamma}. However, for the time being we use a coordinate system for this new polymer which is shifted and does not match the one in \Cref{def:stationary_log_gamma}, as this one is more suited to the coordinates on the Aztec diamond.

\begin{defi}
    \label{def:boundary-weighted_log-gamma}
    Fix $n \in \Z_{\geq 1}$, two sequences of positive real parameters $(\psi_j)_{1 \leq j \leq n},(\phi_{j-n})_{1 \leq j \leq n}$, and $\theta_1 \in \R$ such that $\psi_j+\theta_1 > 0$ and $\phi_{j-n} - \theta_1 > 0$ for all $1 \leq j \leq n$. Equip each shifted lattice point in $\{(x+1/2,y) : y \in \Z_{\leq -1}, x \in \Z_{\geq 0}, y-x \geq -n\}$ with a random weight $1/\gamma_{x,y}$, where
    \begin{equation}\label{eq:def_gamma_in_bwlg}
        \gamma_{x,y}  \sim \Gamma(\psi_{-y}+\phi_{-y+x-2n},1)
    \end{equation}
    are independent. Define another independent set of boundary weights
        \begin{align}\label{eq:ajbj_def}
    \begin{split}
        a_j &\sim \Gamma(\psi_j+\theta_1,1) \\ 
        b_j &\sim \Gamma(\phi_{j-n}-\theta_1,1).
    \end{split}
    \end{align}
    
    Given a path $\tpi: (-1/2,y) \to (n-1/2,-1)$ from a point $(-1/2,y)$ (where $-1 \leq y \leq -n-1$) to $(n-1/2,-1)$ which consists only of horizontal $(x,y) \to (x+1,y)$ and up-right $(x,y) \to (x+1,y+1)$ steps, we define its bulk weight to be 
    \begin{equation}
        \wt(\tpi) = \prod_{(x+1/2,y) \in \tpi} 1/\gamma(x,y)
    \end{equation}
    where the product over $(x+1/2,y) \in \tpi$ is over all points of $(\Z+1/2) \times \Z$ contained in $\tpi$, excluding the two endpoints. We also use the notation $Y(\tpi)$ for the coordinate $y$ of its left endpoint. 

    Then $Q^{edge-\Gamma}_n(\cdot)$ is the probability measure (random, depending on the above weights) on the set of paths defined above, defined by
    \begin{equation}
        Q^{edge-\Gamma}_n(\tpi) \propto \wt(\tpi) \prod_{j=1}^{-Y(\tpi)-1} a_{j}/b_{j}.
    \end{equation} 
\end{defi}

\begin{lemma}
    \label{thm:reduce_to_loggamma}
    Fix $(\psi_j)_{1 \leq j \leq n}, (\phi_{j-n})_{1 \leq j \leq n}, (\theta_i)_{1 \leq i \leq n}$ with $\psi_j + \theta_i > 0$ and $\phi_{j-n} - \theta_i > 0$ for all $i,j$ and let $M$ be a perfect matching from the dimer measure on $G_n^{\Az}$ with random weights given by these parameters. Let $\tpi$ be distributed by an independent copy of $Q^{edge-\Gamma}_n$ with the same parameters $(\psi_j)_{1 \leq j \leq n}, (\phi_{j-n})_{1 \leq j \leq n}, \theta_1$ as the Aztec diamond, and $Y(\tpi)$ is as in \Cref{def:boundary-weighted_log-gamma}. Then
    \begin{equation}
        T^{West}(M) = -Y(\tpi)-1 \quad \quad \quad \text{in distribution,}
    \end{equation}
\end{lemma}

\begin{proof}
    If $a_{i,j}^{[n]},b_{i,j}^{[n]}$ are the weights of our Aztec diamond, then defining 
    \begin{align}
    \begin{split}\label{eq:weights_in_bwlg_deterministic}
        \gamma_{x,y} &= a_{2,-y}^{[n-x]} + b_{2,-y}^{[n-x]} \\ 
        a_{j} &= a_{1,j}^{[n]} \\ 
        b_{j} &= b_{1,j}^{[n]}
    \end{split}
    \end{align}
    (where the weights defining $\gamma_{x,y}$ are defined recursively via shuffling through \eqref{eq:downshuffle_weights}), these will be distributed as in \eqref{eq:def_gamma_in_bwlg} and \eqref{eq:ajbj_def}. As before, we will show distributional equality by showing a finer statement, which is that for any deterministic nonnegative weights $a_{i,j}^{[n]},b_{i,j}^{[n]}$, there is a deterministic bijection between path ensembles $(\pi_1,\ldots,\pi_n)$ on $G_{n,1}^{\beta \Gamma}$ and paths $\tpi$ on the the free-weight log-Gamma polymer with weights as in \eqref{eq:weights_in_bwlg_deterministic}, which preserves weights up to an overall factor of
    \begin{equation}
        \label{eq:bwlg_overall_factor}
        C := \prod_{j=1}^n b_{1,j}^{[n]}/(a_{1,j}^{[n]}+b_{1,j}^{[n]}).
    \end{equation}

    Let us show this bijection. Consider the $n$ paths in $G_{n,1}^{\beta \Gamma}$. Each path takes either a single downward step on the right portion of the graph, or a down-right step on the left portion, as can be seen in \Cref{fig:dual_free_weight_gamma_path}. Call these E-paths and SE-paths, after the type of step they take in the leftmost column.
    
    If a given path $\pi_j$ is an E-path, then all paths $\pi_1,\ldots,\pi_{j-1}$ above it are also E-paths, as if any of them took a down-right step there would be a collision. For each place in which one of the paths $\pi_1,\ldots,\pi_n$ makes a downward step $(x,y) \to (x,y-1)$, draw a segment of dual path $(x+1/2,y) \to (x-1/2,y-1)$. For each vertex $(x,y)$ that has no path incident to it, draw a segment of dual path $(x+1/2,y) \to (x-1/2,y)$. This produces a dual path $\tpi$ from $(n-1/2,-1)$ to $(-1/2, -k-1)$, where $k$ is the number of E-paths.

    Because $\tpi$ picks up exactly the weights $\gamma_{x,y} = a_{2,-y}^{[n-x]} + b_{2,-y}^{[n-x]}$ which the paths $\pi_1,\ldots,\pi_n$ do not pick up, we see that for any $\pi_1,\ldots,\pi_n$ and $\tpi$ the corresponding complement path as above, we have
    \begin{equation}\label{eq:paths_and_complements}
        \frac{1}{\wt(\tpi)} \cdot \prod_{j=1}^n \wt(\pi_j) = \frac{a_{1,1}^{[n]}}{a_{1,1}^{[n]}+b_{1,1}^{[n]}} \cdots  \frac{a_{1,-Y(\tpi)-1}^{[n]}}{a_{1,-Y(\tpi)-1}^{[n]}+b_{1,-Y(\tpi)-1}^{[n]}} \frac{b_{1,-Y(\tpi)}^{[n]}}{a_{1,-Y(\tpi)}^{[n]}+b_{1,-Y(\tpi)}^{[n]}} \cdots \frac{b_{1,n}^{[n]}}{a_{1,n}^{[n]}+b_{1,n}^{[n]}}
    \end{equation}
    where $Y(\tpi)$ is the coordinate of the endpoint as in the end of \Cref{def:boundary-weighted_log-gamma}. The factors on the right hand side of \eqref{eq:paths_and_complements} come from the $\beta$ RWRE portion of $G_{n,1}^{\beta \Gamma}$ (which is just a single column). We see that 
    \begin{equation}
        \text{RHS\eqref{eq:paths_and_complements}} = C \cdot \prod_{j=1}^{-Y(\tpi)-1} a_{1,j}^{[n]}/b_{1,j}^{[n]}
    \end{equation}
    where $C$ is the overall factor in \eqref{eq:bwlg_overall_factor}. Using this and multiplying both sides by $\wt(\tpi)$, we have that 
    \begin{equation}\label{eq:bwlg_weight_pres}
        C \wt(\tpi) \prod_{j=1}^{-Y(\tpi)-1} a_{1,j}^{[n]}/b_{1,j}^{[n]} =  \prod_{j=1}^n \wt(\pi_j).
    \end{equation}
    This shows that our bijection is weight-preserving up to the overall factor $C$.

    \begin{figure}
        \centering

  \begin{tikzpicture}[
    scale=1,
    every node/.style={circle, fill=black, inner sep=1pt},
    decoration={markings, mark=at position 0.5 with {\arrow{>}}},
    arrowedge/.style={postaction={decorate}, gray},
    infdot/.style={circle, draw=none, fill=gray!50, inner sep=0.5pt}
  ]
\def\n{7}
\def\l{3}

\def\jmax{5}

\foreach \i in {2,...,2} {%
  \pgfmathtruncatemacro{\ii}{\i + 1}%
  \pgfmathtruncatemacro{\jmin}{-\i+1}%
  \foreach \j in {\jmin,...,\jmax} {%
    \pgfmathtruncatemacro{\jm}{\j - 1}%
    \path
      (\i,\j)   coordinate (v-\i-\j)
      (\ii,\j)  coordinate (v-\ii-\j)
      (\ii,\jm) coordinate (v-\ii-\jm);
    \draw[arrowedge] (v-\i-\j) -- (v-\ii-\j);
    \node (v\i\j) at (\i,\j) {};
    \node (v\ii\j) at (\ii,\j) {};
    \draw[arrowedge] (v-\i-\j) -- (v-\ii-\jm);
  }%
}

\foreach \i in {3,...,8} {%
  \pgfmathtruncatemacro{\ii}{\i + 1}%
  \pgfmathtruncatemacro{\jmin}{-4+\i}%
  \foreach \j in {\jmin,...,\jmax} {%
    \pgfmathtruncatemacro{\jm}{\j-1}%
    \path
      (\ii,\j) coordinate (v-\ii-\j);
      (\i,\j)   coordinate (v-\i-\j)
      (\i,\jm)  coordinate (v-\ii-\jm)      
    \draw[arrowedge] (v-\i-\j) -- (v-\i-\jm);
     \draw[arrowedge] (v-\i-\j) -- (v-\ii-\j);
    \node (v\i\j) at (\i,\j) {};
    \node (v\ii\j) at (\ii,\j) {};
    \node (v\i\jm) at (\i,\jm) {};
    
  }%
}%
\draw[arrowedge] (9,5)--(9,4);

\foreach \i in {2,...,9} {
 \pgfmathtruncatemacro{\jmin}{-4+\i}%
  \foreach \j in {\jmin,...,\jmax} {%
\node[red] at ({\i+.5},\j)  {};
}
}

\draw[very thick] (2,5)--(9,5)--(9,4);
\draw[very thick] (2,4)--(8,4)--(8,3);
\draw[very thick] (2,3)--(6,3)--(6,2)--(7,2);
\draw[very thick] (2,2)--(4,2)--(4,1)--(6,1);
\draw[very thick] (2,1)--(3,0)--(5,0);
\draw[very thick] (2,0)--(3,-1)--(4,-1);
\draw[very thick] (2,-1)--(3,-2);

\draw[very thick,red] (2.5,1)--(3.5,1)--(4.5,2)--(5.5,2)--(6.5,3)--(7.5,3)--(9.5,5);
\draw (3,-2) node[fill=none,anchor=north west,inner sep=1pt] {\tiny $(0,-8)$};
\draw (4,-1) node[fill=none,anchor=north west,inner sep=1pt] {\tiny $(1,-7)$};
\draw (5,0) node[fill=none,anchor=north west,inner sep=1pt] {\tiny $(2,-6)$};
\draw (6,1) node[fill=none,anchor=north west,inner sep=1pt] {\tiny $(3,-5)$};
\draw (7,2) node[fill=none,anchor=north west,inner sep=1pt] {\tiny $(4,-4)$};
\draw (8,3) node[fill=none,anchor=north west,inner sep=1pt] {\tiny $(5,-3)$};
\draw (9,4) node[fill=none,anchor=north west,inner sep=1pt] {\tiny $(6,-2)$};
\draw (9.5,5) node[fill=none,anchor=west,inner sep=1pt] {\tiny $(6\frac12,-1)$};
\end{tikzpicture}
        \caption{An instance of paths $\pi_1,\ldots,\pi_7$ (black) on the graph $G_{7,1}^{\beta \Gamma}$ (gray), and the free-weight log-Gamma polymer path $\tpi$ obtained by the complementation bijection in the proof of \Cref{thm:reduce_to_loggamma} (red). Here $Y(\tpi) = -5$, and $\pi_1,\ldots,\pi_4$ are E-paths while $\pi_5,\pi_6,\pi_7$ are SE-paths in the notation of the proof.}
        \label{fig:dual_free_weight_gamma_path}
\end{figure}
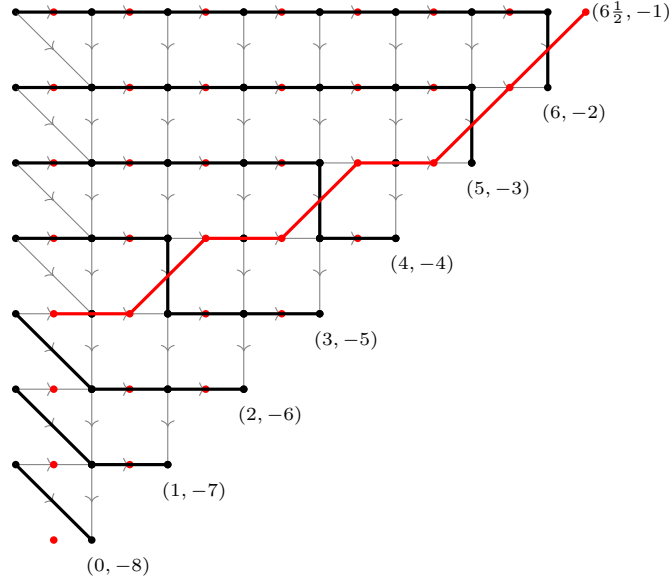
    
    
    Now let us use this bijection to prove the desired distributional equality. The $\ell=1$ case of \Cref{thm:vert_slice_polymer} relates the Aztec diamond to a polymer model with $n$ paths on $G_{n,1}^{\beta \Gamma}$.  More specifically, the random $n$-element set $X_1^{\Az}(M)$ in the notation of \Cref{def:aztec_matching_column} is just $[n+1] \setminus \{T^{West}(M)+1\}$. \Cref{thm:vert_slice_polymer} then gives a distributional equality 
    \begin{equation}\label{eq:X_to_bsw}
        [n+1] \setminus \{T^{West}(M)+1\} = X^{poly}(\pi_1,\ldots,\pi_n),
    \end{equation}
    so $T^{West}(M)$ is the (negative of) the unique $y$-coordinate between $-1$ and $-n-1$ which does \emph{not} have a path entering it from the left (recall \Cref{def:X_polymer}). The weight-preserving bijection we just showed furthermore proves 
    \begin{equation}\label{eq:bsw_to_y}
        X^{poly}(\pi_1,\ldots,\pi_n) = [n+1] \setminus \{-Y(\tpi)\} \quad \quad \quad \quad \text{ in distribution,}
    \end{equation}
    where as usual $Y(\tpi)$ is the $y$-coordinate of the left endpoint of a path $\tpi$ of the free-weight log-Gamma polymer. Combining \eqref{eq:X_to_bsw} with \eqref{eq:bsw_to_y} completes the proof.
\end{proof}

\begin{rmk}
    \label{rmk:relation_to_inverse_beta}
    The path-complementation/particle-hole bijection that we apply to go from a polymer model with $n$ paths on the graph $G_{n,1}^{\beta \Gamma}$ to one with $1$ path as in \Cref{def:boundary-weighted_log-gamma} may also be applied to $G_{p,m}^{\beta \Gamma}$ more generally. It yields a polymer model with $m$ paths, with a random environment which is partially log-Gamma polymer as in \Cref{def:boundary-weighted_log-gamma} (corresponding to the strict-weak portion after path-complementation), and partially the \emph{inverse-Beta polymer} of Thiery-Le Doussal \cite{thiery2015integrable} (corresponding to the Beta-RWRE portion after path-complementation). The full matching will be discussed in upcoming work.
\end{rmk}

Now we have related the West turning point of the Aztec diamond to a log-Gamma polymer model with extra weighting, but this is still not obviously the same as the log-Gamma path midpoint featured in \Cref{thm:west_matching_intro}. The key to matching them is the Burke property. The following lemma, a corollary of the Burke property, shows that the partition functions along the antidiagonal $x+y=n$ differ from one another by quotients of independent Gamma variables. These will match with the $a_j/b_j$ boundary weight factors in \Cref{def:boundary-weighted_log-gamma}.

\begin{lemma}
    \label{thm:burke_along_diagonal}
    Fix $n \in \N$ and let $\alpha, \beta > 0$ and $Z^{stat-\Gamma}_{i,j}$ be as in \Cref{def:stationary_log_gamma}. Then there exist random variables $\{U_{i-1,n-i+1}, V_{i,n-i}: 1 \leq i \leq n\}$, which are independent of each other and of the above-antidiagonal weights $\{Y_{i,j}: i+j \geq n+1\}$, satisfying
    \begin{align}\label{eq:uv_diag_dist}
        \begin{split}
            U_{i,n-i} &\sim \Gamma^{-1}(\beta,1) \\ 
            V_{i,n-i} &\sim \Gamma^{-1}(\alpha,1),
        \end{split}
    \end{align}
    and 
     \begin{equation}\label{eq:uv_Z_recurrence}
        Z^{stat-\Gamma}_{i+1,n-i-1} = \frac{U_{i,n-i}}{V_{i+1,n-i-1}} \cdot Z^{stat-\Gamma}_{i,n-i}
    \end{equation}
    for every $0 \leq i \leq n-1$.
   
\end{lemma}
\begin{proof}
    \eqref{eq:uv_diag_dist} and \eqref{eq:uv_Z_recurrence} follow directly from applying \Cref{thm:burke} along the down-right path composed of the edges $(i,n-i) \to (i,n-i-1)$ and $(i,n-i-1) \to (i+1,n-i-1)$, for $i=0,1,\ldots,n-1$. The independence from $\{Y_{i,j}: i+j \geq n+1\}$ is trivial since the partition functions defining $U_{i,n-i}$ and $V_{i,n-i}$ by their quotients do not involve paths passing through any of the weights $\{Y_{i,j}: i+j \geq n+1\}$.
\end{proof}

\begin{lemma}
    \label{thm:match_gammas}
    Fix $n$ and $\alpha, \beta > 0$, and set the $\psi,\phi,\theta$ parameters of \Cref{def:boundary-weighted_log-gamma} to the homogeneous case 
        \begin{align}\label{eq:homogeneous_mu_xi2}
    \begin{split}
                \theta_i &\equiv 0 \\
                \psi_j &\equiv \alpha  \\ 
                \phi_j &\equiv \beta.
    \end{split}
    \end{align}
    Let $\tpi$ be distributed according to the polymer measure $Q^{edge-\Gamma}_n$ of \Cref{def:boundary-weighted_log-gamma} with these weights $\gamma_{x,y},a_j,b_j$. Let $\pi$ be distributed by a copy of the measure $Q^{stat-\Gamma}_{n,n}$ with weights independent of those of $Q^{edge-\Gamma}_n$. Then
    \begin{equation}\label{eq:match_log-gammas}
        -Y(\tpi) = x_{mid}(\pi)+1 \quad \quad \quad \quad \text{in distribution.}
    \end{equation}
\end{lemma}
\begin{proof}
    The basic idea is portrayed in \Cref{fig:gamma_poly_matching}. The paths $\tpi$ correspond to the portion of paths $\pi$ which lies up-right from the line $i+j=n$, and the boundary weights of the paths $\tpi$ correspond to point-to-point partition functions from $(0,0)$ to a point on the line $i+j=n$ in the stationary log-Gamma model. 

 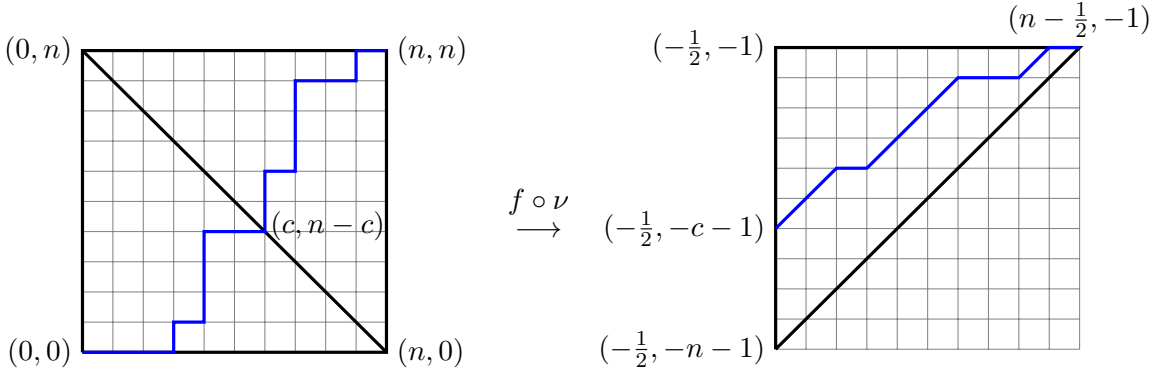
\begin{figure}
        
                \begin{tikzpicture}[scale=.4]
                   \draw[step=1,gray,very thin] (0,0) grid (10,10);
            \draw[very thick] (0,0)--(10,0)--(10,10)--(0,10)--(0,0);
             \draw[very thick] (10,0)--(0,10);
             \draw[blue, very thick] (0,0) --(3,0)--(3,1)--(4,1)--(4,4)--(6,4)--(6,6)--(7,6)--(7,9)--(9,9)--(9,10)--(10,10);

              \node [left] at (0,0) {$(0,0)$};
              \node [left] at (0,10) {$(0,n)$};
              \node [right] at (10,0) {$(n,0)$};
              \node [right] at (10,10) {$(n,n)$}; 
              
            \node at (15,5) {$f \circ \nu$};
             \node at (15,4) {$\longrightarrow$};
     
         \node [right] at (5.85,4.2) {$(c,n-c)$};
        \end{tikzpicture}
        \begin{tikzpicture}[scale=.4]
        \draw[step=1,gray,very thin] (0,0) grid (10,10);
            \draw[very thick] (0,0)--(10,10)--(0,10)--(0,0);
        
             \draw[blue, very thick] (0,4)--(2,6)--(3,6)--(6,9)--(8,9)--(9,10)--(10,10);
                           \node [left] at (0,0) {$(-\frac12,-n-1)$};
              \node [left] at (0,10) {$(-\frac12,-1)$};
               \node [left] at (0,4) {$(-\frac12,-c-1)$};
              \node [above] at (10,10) {$(n-\frac 12,-1)$};
                          
        \end{tikzpicture}
        \caption{ A sample polymer path from $(0,0)$ to $(n,n)$ (left), and its image under $f \circ \nu$ where we recall $f(x,y)=(x+y-n-\frac 12,y-n-1)$. Here $n=10$ and $c=6$.}
        \label{fig:gamma_poly_matching}
    \end{figure}

    More specifically, identifying paths $\pi$ with their vertex set in $\Z_{\geq 0}^2$, we let
    \begin{equation}
        \upsilon(\pi) = \pi \cap \{(i,j): i+j \geq n\}.
    \end{equation}
    This is a path from any point on the antidiagonal $i+j=n$ to the point $(n,n)$. Let $\hat{\pi}$ denote any such path, with starting point $(c,n-c)$. Under the pushforward measure $\upsilon_*(Q^{stat-\Gamma}_{n,n})$, we have 
    \begin{equation}\label{eq:upsilon_pr}
        (\upsilon_*(Q^{stat-\Gamma}_{n,n}))(\hat{\pi}) \propto Z^{stat-\Gamma}_{c,n-c} \cdot \prod_{(i,j) \in \hat{\pi}} Y_{i,j},
    \end{equation}
    with the convention that the product of weights does not include the starting weight $Y_{c,n-c}$ since it already contributes to $Z^{stat-\Gamma}_{c,n-c}$. By \Cref{thm:burke_along_diagonal}, 
    \begin{equation}\label{eq:use_diagonal_burke}
        Z^{stat-\Gamma}_{c,n-c} \cdot \prod_{(i,j) \in \hat{\pi}} Y_{i,j} = Z^{stat-\Gamma}_{0,n} \prod_{i=1}^c \frac{U_{i-1,n-i+1}}{V_{i,n-i}} \prod_{(i,j) \in \hat{\pi}} Y_{i,j}.
    \end{equation}
    Combining \eqref{eq:upsilon_pr} with \eqref{eq:use_diagonal_burke} yields
    \begin{equation}
        (\upsilon_*(Q^{stat-\Gamma}_{n,n}))(\hat{\pi}) \propto \prod_{i=1}^{c(\hat{\pi})} \frac{U_{i-1,n-i+1}}{V_{i,n-i}} \prod_{(i,j) \in \hat{\pi}} Y_{i,j},
    \end{equation}
    for any up-right path $\hat{\pi}$ with starting point $(c(\hat{\pi}),n-c(\hat{\pi}))$ on the diagonal $i+j=n$, and ending point $(n,n)$.

    Let 
    \begin{equation}
        f(x,y) = (x+y-n-\frac{1}{2},y-n-1),
    \end{equation}
    so $f$---which is just the composition of a standard shear and a translation---takes the triangle 
    \begin{equation}
    \{(x,y) \in \R^2: x+y \geq n, x \leq n, y \leq n\}    
    \end{equation}
    (on which the paths $\hat{\pi}$ live) to the triangle 
    \begin{equation}
       \{(x,y) \in \R^2: x \geq 0, y \leq -1, y-x \geq -n-1/2\} 
    \end{equation}
     on which the paths $\tpi$ live. See \Cref{fig:gamma_poly_matching}.
    
    Again identifying paths with their vertex sets, given any up-right path $\pi$ from $(0,0)$ to $(n,n)$, $f \circ \upsilon(\pi)$ is a valid path in the support of $Q^{edge-\Gamma}_n$. For a diagonal point $(c,n-c)$, we have $f(c,n-c) = (-1/2, -c-1)$, which is a valid endpoint of a path $\tpi$ in the support of $Q^{edge-\Gamma}_n$. 

    Furthermore, 
    \begin{equation}
        \prod_{i=1}^c \frac{U_{i-1,n-i+1}}{V_{i,n-i}} = \prod_{j=1}^{-(-c-1)-1} \frac{a_j}{b_j}
    \end{equation}
    in distribution, in fact in joint distribution as $c$ varies over $[0,n]$, since $U_{i-1,n-i+1} = b_i^{-1}$ and $V_{i,n-i} = a_i^{-1}$ in distribution and all these variables are independent. Hence the boundary weights of the paths in the two models, identified via $f$, are equal in distribution. Additionally, the bulk weights
    $$\{Y_{i,j}: (i,j) \in \Z^2, i+j > n, i \leq n, j \leq n\}$$ are equal in joint distribution to 
    $$\{\gamma_{x,y}: (x,y) \in \Z^2,  x \geq 0, y \leq -1, y-x \geq -n\}$$
    with the index sets identified by $f(i,j) = (x,y)$. Hence paths $\hat{\pi}$ and $f(\hat{\pi})$ have weights which are the same in distribution, and this is true for the joint distribution over all paths $\hat{\pi}$. This is exactly the statement that the two random measures $\upsilon_*(Q^{stat-\Gamma}_{n,n})$ and $Q^{edge-\Gamma}_n$ are equal in distribution as random measures. In particular, the marginal distributions of their endpoints are equal. Since
    \begin{equation}
        x_{mid}(\pi) = -Y(f \circ \upsilon(\pi))-1
    \end{equation} 
    for any $\pi$, this yields the desired claim.

\end{proof}

\begin{rmk}
    The above proof yields not just a distributional matching between $Y(\tpi)$ and $y_{mid}(\pi)$, given in \Cref{thm:match_gammas}, but a full distributional matching the whole path $\tpi$ and the portion of the path $\pi$ lying above the line $x+y=n$, after an appropriate shear transformation $f$. In particular, courtesy of \Cref{thm:dynamical_match_vert}, this means that the evolution of the West turning point under shuffling is exactly described by a stationary log-Gamma polymer path. 
\end{rmk}

\begin{proof}[Proof of {\Cref{thm:west_matching_intro}}]
    Combine \Cref{thm:reduce_to_loggamma} with \Cref{thm:match_gammas} (with homogeneous weights $\psi_j \equiv \alpha, \phi_j \equiv \beta, \theta_i \equiv 0$).
\end{proof}

\begin{rmk}
    Barraquand-Le Doussal \cite{barraquand2023stationary} give a $k$-path generalization of the stationary polymer model in \Cref{def:stationary_log_gamma}, which should similarly have a distributional match along antidiagonals with the $k\tth$ slice from the left of a Gamma-disordered Aztec diamond, generalizing \Cref{thm:west_matching_intro}; this will be treated in upcoming work. However, probabilistic asymptotics on path trajectories like the ones in \cite{seppalainen2012scaling} do not seem to be worked out yet in the literature for this multi-path model. Hence, at the moment, this match does not appear to buy any asymptotic information about the Aztec diamond. 
\end{rmk}

\section{Horizontal slices and the $\Gamma \log \Gamma$ polymer}\label{sec:hor_slice}

This section contains the analogues of Sections \ref{sec:vert_polymer}, \ref{sec:dynamical_vert}, and \ref{sec:turning_points}, except that we slice the Aztec diamond horizontally rather than vertically. This results in a different notion of $(+)$- and $(-)$-column and a matching with a different polymer model. We call it the $\Gamma \log \Gamma$ polymer since it consists of a strict-weak (also known as Gamma) polymer grafted to a log-Gamma polymer along the boundary.

Most arguments for these results are quite similar to those for the corresponding results with vertical slices. Because of the similarity, we will be somewhat more terse in this section, and recommend the reader to have finished the previous ones mentioned before starting this one.

\subsection{Definition of {$\Gamma \log \Gamma$} polymer and statement of matching}

\begin{defi}
    \label{def:horiz_polymer_digraph}
    Let $p,m \geq 1$ be integers, and let $(\psi_j)_{1 \leq j \leq m}, (\phi_j)_{1-p \leq j \leq 0}, (\theta_i)_{1 \leq i \leq m+p}$ be parameters such that all sums $\psi_j + \theta_i$ and $\phi_j - \theta_i$ are positive.

    Consider the weighted, directed graph $G^{\Gamma \log \Gamma}_{p,m}$ with the same underlying graph as $G^{\beta \Gamma}_{p,m}$ from \Cref{def:beta_sw} (see \Cref{fig:bgbsw_to_poly} (bottom)), which we recall means that all vertices lie in the set 
    \begin{multline}\label{eq:glog_support_set}
        \{(x,y) \in \Z^2: -m \leq x \leq -1, -m-p \leq y \leq -1, x+y \geq -m-p\} \\ 
         \cup \{(x,y) \in \Z^2: -m-p \leq y \leq -1, 0 \leq x \leq \min(p-1,y+m+p)\}.
    \end{multline}
    Define the edges and their weights as follows
    \begin{enumerate}
\item For each vertex at $(x,y)$ with $x \leq -1$, there are right and down-right directed edges to vertices $(x+1,y)$ and $(x+1,y-1)$, with weights $\rho_{x,y}$ and $1$ respectively, where $\rho_{x,y} \sim \Gamma(\psi_{m+x+1}+\theta_{-y},1)$ are independent Gamma variables.
\item For each vertex at $(x,y)$ with $x \geq 0$ lying in the set \eqref{eq:glog_support_set}, there are down and right directed edges to vertices $(x,y-1)$ and $(x+1,y)$ (if they also lie in the set \eqref{eq:glog_support_set}). Both edges have the same weight $\kappa_{x,y} \sim \Gamma^{-1}(\phi_{1-p+x}-\theta_{-y})$, where the variables $\kappa_{x,y}$ are independent from each other and from the $\rho_{x,y}$ variables.
\end{enumerate}
Then the associated \emph{$\Gamma\log\Gamma$ polymer partition function} is 
    \begin{equation}
        Z^{\Gamma\log\Gamma}_{p,m} := \sum_{\substack{\pi_j:(-m,-j) \to (p-j,-m-j) \\ 1 \leq j \leq p}} \prod_{j=1}^p \wt(\pi_j),
    \end{equation}
    where the sum is over $p$-tuples of nonintersecting\footnote{Meaning that the vertex sets of any two paths, not just the edge sets, are disjoint.} paths $\pi_1,\ldots,\pi_p$ on $G^{\Gamma\log\Gamma}_{p,m}$, where $\pi_j$ has start point $(-m,-j)$ and end point $(p-j,-m-j)$, and $\wt(\pi_j)$ denotes the product of edge weights over the edges in $\pi_j$.

    The associated \emph{$\Gamma\log\Gamma$ polymer measure} is a probability measure on such $p$-tuples $(\pi_1,\ldots,\pi_p)$ of nonintersecting paths which assigns to each one the probability
    \begin{equation}
        \frac{1}{Z^{\Gamma\log\Gamma}_{p,m}} \prod_{j=1}^p \wt(\pi_j).
    \end{equation}
\end{defi}

Note that because the weights on both outgoing edges from any vertex in the right half of the polymer are equal, one may view these weights as associated to the vertex rather than the edges. This polymer environment is then a multi-parameter generalization of that of the log-Gamma polymer of \cite{seppalainen2012scaling}, recalled in \Cref{def:stationary_log_gamma}; such deformed log-Gamma polymers arise from Whittaker processes, see \cite[Definition 5.16]{borodin2014macdonald}.

\begin{rmk}\label{rmk:same_graph_same_defs}
   Recall that \Cref{def:X_polymer} defined a point configuration $X^{poly}(\pi_1,\ldots,\pi_p)$ in terms of the intersections of the paths $\pi_1,\ldots,\pi_p$ with the $y$-axis in $G^{\beta \Gamma}_{p,m}$. Since the underlying graph is the same as $G^{\Gamma \log \Gamma}_{p,m}$, we will freely use $X^{poly}(\pi_1,\ldots,\pi_p)$ for paths $\pi_1,\ldots,\pi_p$ on $G^{\Gamma \log \Gamma}_{p,m}$ as well. We will similarly use $\Pi(0),\ldots,\Pi(m)$ as in \Cref{def:path_coordinates_dynamical}. 
\end{rmk}

This time, we will match $X^{poly}$ with horizontal slices of the Aztec diamond, which we denote with similar notation to $X_\ell^{\Az}$ before (cf. \Cref{def:aztec_matching_column}):

\begin{defi}\label{def:aztec_matching_row}
    For an Aztec diamond $G_n^{\Az}$ of size $n$, let $\mathcal{M}$ be the set of perfect matchings of $G_n^{\Az}$. For any $1 \leq \ell \leq n$, define functions
    \begin{equation}
        Y_\ell^{\Az}: \mathcal{M} \to \binom{[n+1]}{n-\ell+1}
    \end{equation}
    as follows. There are $n$ rows of black vertices in $G_n^{\Az}$; consider the $\ell\tth$ one from top to bottom. For each perfect matching $M \in \mathcal{M}$, there will be exactly $n-\ell+1$ black vertices in the $\ell\tth$ row which are matched to a white vertex above them (either up-right or up-left). Labeling the black vertices in this row by $1,\ldots,n+1$ from left  to right, we let $Y_\ell^{\Az}(M)$ be the set of labels of the black vertices which are matched with a white vertex above them (again, either up-right or up-left). 

\end{defi}

\begin{thm}
    \label{thm:hor_slice_polymer}
    Fix $n \in \Z_{\geq 1}$ and let $(\phi_{j-n})_{1 \leq j \leq n}, (\psi_j)_{1 \leq j \leq n}, (\theta_i)_{1 \leq i \leq n}$ be real parameters such that $\psi_j+\theta_i$ and $\phi_{j-n}-\theta_i$ are positive for each $1 \leq i,j \leq n$. Consider a Gamma-disordered Aztec diamond with these parameters (\Cref{def:gamma_weights_intro_general}), and let $M$ be a random perfect matching distributed by the corresponding dimer measure.

    Consider also an independent $\Gamma \log \Gamma$ polymer on $G^{\Gamma \log \Gamma}_{p,m}$ with $p=n-\ell+1$ paths and $m=\ell$ layers in the notation of \Cref{def:horiz_polymer_digraph}, and the same parameters $\phi_{j-n},\psi_j,\theta_i$. Let $\pi_1,\ldots,\pi_p$ be paths (ordered from top to bottom as in \Cref{def:horiz_polymer_digraph}) distributed according to the corresponding polymer measure. Then 
    \begin{equation}
        Y_\ell^{\Az}(M) = X^{poly}(\pi_1,\ldots,\pi_p) \quad \quad \quad \quad \text{ in distribution.}
    \end{equation}
    Furthermore, this distributional equality holds for the dimer measure associated to any fixed realization of the Aztec weights, if the polymer weights of the corresponding polymer measure are given in terms of them via
 \begin{align}
        \begin{split}
            \rho_{x,y} &\longleftrightarrow a_{-y,\ell+x+1}^{[n-\ell+(\ell+x+1)]} =  a_{-y,\ell+x+1}^{[n+x+1]} \\ 
            \kappa_{x,y} &\longleftrightarrow 1/b_{-y,\ell}^{[n-x]}.
        \end{split}
    \end{align}
\end{thm}

\begin{thm}
    \label{thm:horiz_slice_polymer_matching_full} 
    Let $(\phi_j)_{j \leq 0}, (\psi_j)_{j \geq 1}, (\theta_i)_{i \geq 1}$ be real parameters such that $\psi_j+\theta_i$ and $\phi_j-\theta_i$ are positive for all $i,j$. 
    Sample a coupled sequence of matchings $M_1,M_2,\ldots$ of the Aztec diamonds of size $1,2,\ldots$ according to the shuffling algorithm with these parameters (see \Cref{subsec:shuffling}).
    
    Then the `$k\tth$ row from the top' stochastic process $Y^{\Az}_{\tau}(M_{\tau+k-1}), \tau \in \Z_{\geq 0}$, where $Y_\tau^{\Az}$ is as in \Cref{def:aztec_matching_row}, has the following description. Fix any `final time' $T > 0$, and let $G_{k,T}^{\Gamma \log \Gamma}$ be a $\Gamma \log \Gamma$ polymer with weights given in terms of the Aztec weights via
    \begin{align}
        \begin{split}
            \rho_{x,y} &= a_{-y,k+x+1}^{[T+k+x]} \\ 
            \kappa_{x,y} &= 1/b_{-y,k}^{[T+k-1-x]}.
        \end{split}
    \end{align}
    Let $(\Pi(\tau))_{0 \leq \tau \leq T}$ be the joint polymer path trajectories as in \Cref{def:path_coordinates_dynamical}. Then
        \begin{equation}
        (Y^{\Az}_{\tau}(M_{\tau+k-1}))_{0 \leq \tau \leq T} = (\Pi(\tau))_{0 \leq \tau \leq T}
    \end{equation}
    in distribution, for any fixed realization of the weights $a_{i,j}^{[n]},b_{i,j}^{[n]}$.
\end{thm}

\subsection{Proofs}

We proceed in exactly the same manner as the proof of \Cref{sec:vert_polymer}, so we will just sketch the arguments in this section, assuming the reader has read that one. There is a corresponding (but not the same) analogue of $(+)$- and $(-)$-columns defined below, which obey the same swapping relation (\Cref{thm:horiz_swap_columns_dimer}) and allow us to transform to a graph $\bG^{\Gamma \log \Gamma}_{n,\ell}$, which may be combinatorially matched with the polymer in \Cref{def:horiz_polymer_digraph}.

\begin{defi}
    \label{def:horiz_pm_cols}
     Given a bipartite graph $G$ drawn on $\Z^2$ as above, we refer to a subgraph as in \Cref{fig:horiz_plus_and_minus_cols} (left), i.e. black vertices at $(x,y-1),(x,y-2),\ldots,(x,y-m)$, connected horizontally to white vertices at $(x+1,y-1),(x+1,y-2),\ldots,(x+1,y-m)$, then horizontally and up-right to black vertices at $(x+2,y),(x+2,y-1),\ldots,(x+2,y-m)$, as a \emph{$(+')$-column}. We allow nontrivial weights only on the up-right edges; if they are given by $a_j$ as in the figure, we will sometimes refer to it as a \emph{$(+')$-column with weights $a_1,\ldots,a_m$} when emphasizing the weights. We allow the degenerate case $m=0$, for which the $(+)$-column is just a single black vertex at $(x+2,y)$ with no white vertices.

    We refer to a subgraph as in \Cref{fig:horiz_plus_and_minus_cols} (right), i.e. black vertices at $(x,y),(x,y-1),\ldots,(x,y-m)$ connected horizontally and up-right to white vertices at $(x+1,y),(x+1,y-1),\ldots,(x+1,y-m+1)$, then these connected horizontally to black vertices at $(x+2,y),(x+2,y-1),\ldots,(x+2,y-m+1)$, as a \emph{$(-')$-column}. We allow nontrivial weights only on the right horizontal edges and up-right edges, and require that for each white vertex, the weight on the two such edges incident to it is the same. If they are given by $b_j^{-1}$ as shown, we will sometimes refer to it as a \emph{$(-')$-column with weights $b_1^{-1},\ldots,b_m^{-1}$} to emphasize the weights. We allow the degenerate case $m=0$, for which the $(-')$-column is just a single black vertex at $(x,y)$.
\end{defi}

\begin{figure}
        \centering

 \begin{tikzpicture}[scale=1]
  \foreach \y in {0,...,2} {
            \draw (0,-\y)--(2,-\y);
              \draw (1,-\y)--(2,-\y+1);
  }
   \foreach \y in {1,...,2} {
            \draw[fill=black] (0,-\y) circle(2pt); 
             \draw[fill=white] (1,-\y) circle(2pt); 
             \draw[fill=black] (2,-\y) circle(2pt); 
              \draw[fill=black] (2,-\y+1) circle(2pt); 
              \draw (1.2,-2+\y+.4) node[anchor=south] {\tiny{$a_{m-\y}$}};
              }
  
   \foreach \y in {0} {
            \draw[fill=black] (0,-\y) circle(2pt); 
             \draw[fill=white] (1,-\y) circle(2pt); 
             \draw[fill=black] (2,-\y) circle(2pt); 
              \draw[fill=black] (2,-\y+1) circle(2pt); 
             \draw (1.15,-2+\y+.2) node[anchor=south] {\tiny{$a_{m}$}};
              }
               \foreach \y in {1,...,3} {
            \draw (0,5-\y)--(2,5-\y);
              \draw (1,5-\y)--(2,-\y+6);}
               \foreach \y in {1,...,3} {
            \draw[fill=black] (0,-\y+5) circle(2pt); 
             \draw[fill=white] (1,-\y+5) circle(2pt); 
             \draw[fill=black] (2,-\y+5) circle(2pt); 
              \draw[fill=black] (2,-\y+6) circle(2pt); 
              \draw (1.3,-\y+5.4) node[anchor=south] {\tiny{$a_{\y}$}};
              }
               
      \draw(1,1.25) node {$\vdots$};

\end{tikzpicture}\qquad \qquad
\begin{tikzpicture}
  \foreach \y in {1,...,3} {
            \draw (2,-\y+2)--(4,-\y+2);
              \draw (2,-\y+1)--(3,-\y+2);
  }
   \foreach \y in {1,...,2} {
            \draw[fill=black] (2,-\y) circle(2pt); 
             \draw[fill=white] (3,-\y+1) circle(2pt); 
             \draw[fill=black] (2,-\y+1) circle(2pt); 
              \draw[fill=black] (4,-\y+1) circle(2pt); 
               \draw (2.1,-2+\y+.3) node[anchor=south] {\tiny{$b_{m-\y}^{-1}$}};
                \draw (3.5,-2+\y+.9) node[anchor=south] {\tiny{$b_{m-\y}^{-1}$}};
              }
               \foreach \y in {0} {
            \draw[fill=black] (2,-\y) circle(2pt); 
             \draw[fill=white] (3,-\y+1) circle(2pt); 
             \draw[fill=black] (2,-\y+1) circle(2pt); 
              \draw[fill=black] (4,-\y+1) circle(2pt); 
              \draw (3.5,-2+\y+.9) node[anchor=south] {\tiny{$b_{m}^{-1}$}};
                \draw (2.1,-2+\y+.2) node[anchor=south] {\tiny{$b_{m}^{-1}$}};
              }

               \foreach \y in {1,...,3} {
            \draw (2,5-\y+1)--(4,5-\y+1);
              \draw (2,4-\y+1)--(3,-\y+6);}
               \foreach \y in {1,...,3} {
            \draw[fill=black] (2,-\y+5) circle(2pt); 
             \draw[fill=white] (3,-\y+6) circle(2pt); 
             \draw[fill=black] (2,-\y+6) circle(2pt); 
              \draw[fill=black] (4,-\y+6) circle(2pt); 
              \draw (3.5,-\y+5.9) node[anchor=south] {\tiny{$b_{\y}^{-1}$}};
              \draw (2.2,-\y+5.3) node[anchor=south] {\tiny{$b_{\y}^{-1}$}};
              }         
      \draw(3,2) node {$\vdots$};
        \end{tikzpicture}
        \caption{A $(+')$-column (left) and a $(-')$-column (right).}
        \label{fig:horiz_plus_and_minus_cols}
\end{figure}
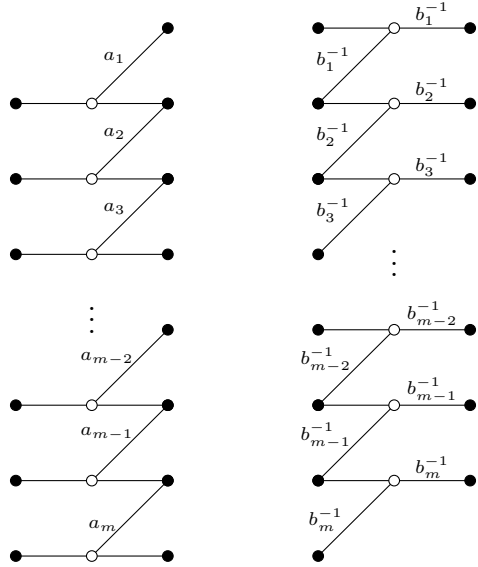

\begin{defi}
    \label{def:Gtilde_horiz}
    Denote by $G^{hor}_n$ the weighted graph with vertex set contained in $\Z^2$ defined as follows. It has white vertices at $(0,0),(0,-1),\ldots,(0,-n+1)$ which are connected to black vertices at $(1,0),(1,-1),\ldots,(1,-n+1)$. These black vertices then form the leftmost black vertices in a $(+')$-column with weights
    $$a_{1,1}^{[n]},\ldots,a_{n,1}^{[n]}.$$
    The rightmost black vertices of this $(+')$-column are the leftmost black vertices of a $(-')$-column with weights  $1/b_{1,1}^{[n]},\ldots,1/b_{n,1}^{[n]}$. The alternating $(+')$- and $(-')$-columns with weights 
    $$a_{1,j}^{[n]},\ldots,a_{n,j}^{[n]}$$
    for the $(+')$-column and 
    $$1/b_{1,j}^{[n]},\ldots,1/b_{n,j}^{[n]}$$
    repeat until $j=n$, at which point there are horizontal edges connecting the rightmost black vertices to $n$ white vertices at coordinates $(4n+2,1),\ldots,(4n+2,n)$. 
\end{defi}

\begin{figure}
        \centering

      \scalebox{.7}{  \begin{tikzpicture}[scale=1.5,
    whiteV/.style={circle,draw,fill=black,inner sep=1.25pt},
    blackV/.style={circle,draw,fill=white,inner sep=1.25pt},
    edge/.style={line width=0.9pt, shorten <=1pt, shorten >=1pt}
]

\draw (2.3,0.4) node[anchor=south] {\small{$a_{11}^{[3]}$}};
\draw (2.3,-.6) node[anchor=south] {\small{$a_{21}^{[3]}$}};
\draw (2.3,-1.6) node[anchor=south] {\small{$a_{31}^{[3]}$}};

\draw (4.5,1) node[anchor=south] {\small{$1/b_{11}^{[3]}$}};
\draw (4.5,0) node[anchor=south] {\small{$1/b_{21}^{[3]}$}};
\draw (4.5,-1) node[anchor=south] {\small{$1/b_{31}^{[3]}$}};

\draw (3.8,.2) node[anchor=south] {\small{$1/b_{11}^{[3]}$}};
\draw (3.8,-.8) node[anchor=south] {\small{$1/b_{21}^{[3]}$}};
\draw (3.8,-1.8) node[anchor=south] {\small{$1/b_{31}^{[3]}$}};

\draw (7.8,1.2) node[anchor=south] {\small{$1/b_{12}^{[3]}$}};
\draw (7.8,.2) node[anchor=south] {\small{$1/b_{22}^{[3]}$}};
\draw (7.8,-.8) node[anchor=south] {\small{$1/b_{32}^{[3]}$}};

\draw (11.8,2.2) node[anchor=south] {\small{$1/b_{13}^{[3]}$}};
\draw (11.8,1.2) node[anchor=south] {\small{$1/b_{23}^{[3]}$}};
\draw (11.8,.2) node[anchor=south] {\small{$1/b_{33}^{[3]}$}};

\draw (6.3,1.4) node[anchor=south] {\small{$a_{12}^{[3]}$}};
\draw (6.3,.4) node[anchor=south] {\small{$a_{22}^{[3]}$}};
\draw (6.3,-.6) node[anchor=south] {\small{$a_{32}^{[3]}$}};

\draw (8.4,2) node[anchor=south] {\small{$1/b_{12}^{[3]}$}};
\draw (8.4,1) node[anchor=south] {\small{$1/b_{22}^{[3]}$}};
\draw (8.4,0) node[anchor=south] {\small{$1/b_{32}^{[3]}$}};

\draw (10.3,2.4) node[anchor=south] {\small{$a_{13}^{[3]}$}};
\draw (10.3,1.4) node[anchor=south] {\small{$a_{23}^{[3]}$}};
\draw (10.3,.4) node[anchor=south] {\small{$a_{33}^{[3]}$}};

\draw (12.5,3) node[anchor=south] {\small{$1/b_{13}^{[3]}$}};
\draw (12.5,2) node[anchor=south] {\small{$1/b_{23}^{[3]}$}};
\draw (12.5,1) node[anchor=south] {\small{$1/b_{33}^{[3]}$}};
\def\n{3}

\definecolor{pluscol}{RGB}{76,175,80}   
\definecolor{minuscol}{RGB}{156,39,176} 
\tikzset{
  colPlus/.style={fill=pluscol,  opacity=0.18},
  colMinus/.style={fill=minuscol,opacity=0.18}
}
\pgfmathsetmacro{\ypad}{0.35} 

\foreach \y in {0,...,\numexpr-\n+1\relax}{
  \node[blackV] at (0,\y) {};
  \node[whiteV] at (1,\y) {};
  \draw[edge] (0,\y) -- (1,\y);
}

\foreach \p in {0,...,\numexpr\n-1\relax}{

  \pgfmathsetmacro{\xplus}{1+4*\p}        
  \pgfmathtruncatemacro{\top}{\p}         

  \pgfmathtruncatemacro{\ybottomplus}{\top - (\n - 1)}
  \pgfmathtruncatemacro{\ytopplus}{\top + 1}
  \begin{scope}[on background layer]
    \fill[colPlus] (\xplus, \ybottomplus - \ypad) rectangle (\xplus + 2, \ytopplus + \ypad+.5);
  \end{scope}

  \foreach \i in {0,...,\numexpr\n-1\relax}{
    \pgfmathtruncatemacro{\y}{\top - \i}
    \node[whiteV] at (\xplus,\y) {};
  }

  \pgfmathtruncatemacro{\topplus}{\top + 1}
  \node[whiteV] at ({\xplus+2},\topplus) {};

  \foreach \i in {0,...,\numexpr\n-1\relax}{
    \pgfmathtruncatemacro{\y}{\top - \i}
    \node[blackV] at ({\xplus+1},\y) {};
    \node[whiteV] at ({\xplus+2},\y) {};
    \pgfmathtruncatemacro{\yup}{\y + 1}
    \draw[edge] (\xplus,\y) -- ({\xplus+1},\y);          
    \draw[edge] ({\xplus+1},\y) -- ({\xplus+2},\y);       
    \draw[edge] ({\xplus+1},\y) -- ({\xplus+2},{\yup});   
  }

  \pgfmathsetmacro{\xminus}{3+4*\p}       
  \pgfmathtruncatemacro{\tminus}{\p + 1}  

  \pgfmathtruncatemacro{\ytopminus}{\tminus}
  \pgfmathtruncatemacro{\ybottomminus}{\tminus - \n}
  \begin{scope}[on background layer]
    \fill[colMinus] (\xminus, \ybottomminus - \ypad) rectangle (\xminus + 2, \ytopminus + \ypad+.5);
  \end{scope}

  \foreach \i in {0,...,\numexpr\n\relax}{
    \pgfmathtruncatemacro{\y}{\tminus - \i}
    \node[whiteV] at (\xminus,\y) {};
  }

  \foreach \i in {0,...,\numexpr\n-1\relax}{
    \pgfmathtruncatemacro{\y}{\tminus - \i}
    \pgfmathtruncatemacro{\ybelow}{\tminus - (\i + 1)}
    \node[blackV] at ({\xminus+1},\y) {};
    \node[whiteV] at ({\xminus+2},\y) {};
    \draw[edge] (\xminus,\y) -- ({\xminus+1},\y);        
    \draw[edge] (\xminus,\ybelow) -- ({\xminus+1},\y);   
    \draw[edge] ({\xminus+1},\y) -- ({\xminus+2},\y);    
  }

} 

\pgfmathtruncatemacro{\xright}{4*\n + 1}   
\pgfmathtruncatemacro{\xcap}{4*\n + 2}     

\foreach \k in {1,...,\numexpr\n\relax}{
  \node[whiteV] at (\xright,\k) {};
  \node[blackV] at (\xcap,\k) {};
  \draw[edge] (\xright,\k) -- (\xcap,\k);
}
\end{tikzpicture}}
        \caption{The graph $G^{hor}_n$ for $n=3$. }
        \label{fig:horiz_Gtilde}
\end{figure}
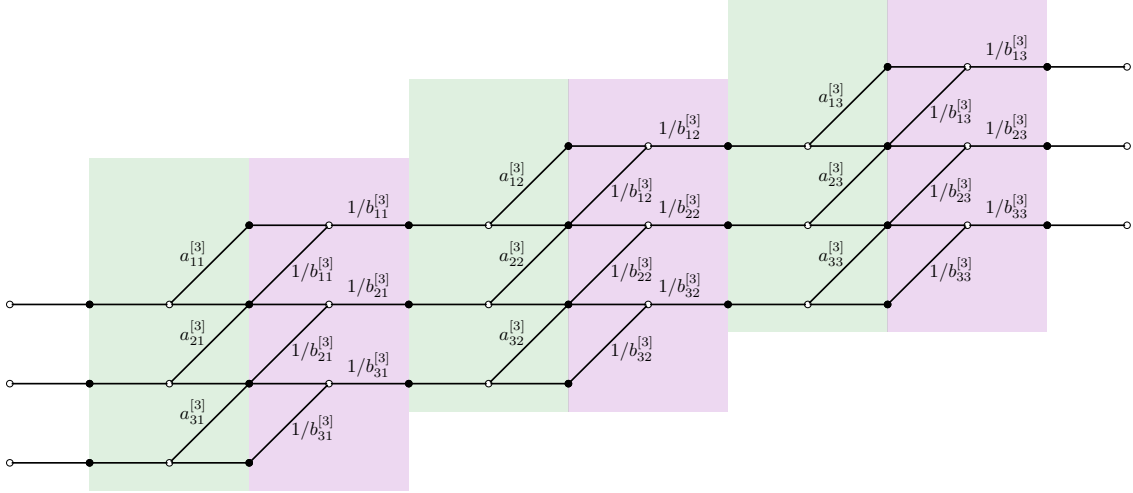

As before, dimer covers of $G_n^{\Az}$ are trivially in weight-preserving bijection to those on $G^{hor}_n$ by reflecting the graph and then applying \Cref{thm:vertex_expansion}.

\begin{lemma}\label{thm:horiz_swap_columns_dimer}
Let $m \geq 1$ and $G$ be a weighted bipartite graph containing a local configuration with a $(+')$ and then a $(-')$-column with weights $a_{1,j}^{[m]},\ldots,a_{m,j}^{[m]}$ and $1/b_{1,j}^{[m]},\ldots,1/b_{m,j}^{[m]}$ respectively as in the left hand side of \Cref{fig:horiz_swap_cols}. Let $G'$ be the graph where the local configuration is replaced by the one on the right hand side with a $(-')$ and then a $(+')$-column with weights $1/b_{1,j-1}^{[m-1]},\ldots,1/b_{m-1,j-1}^{[m-1]}$ and $a_{1,j}^{[m-1]},\ldots,a_{m-1,j}^{[m-1]}$ respectively\footnote{In the case $m=1$, $G'$ will have a degenerate $(-')$-column and a degenerate $(+')$-column, so no updated weights are necessary.}, related to the initial weights via \eqref{eq:downshuffle_weights}.

Then the marginal distribution (with respect to the dimer measure) of edges outside this local configuration is the same for both $G$ and $G'$. 
\end{lemma}

    \begin{figure}
        \centering

 \begin{tikzpicture}[scale=1]
  \foreach \y in {0,...,2} {
            \draw (0,-\y)--(2,-\y);
              \draw (1,-\y)--(2,-\y+1);
  }
   \foreach \y in {1,...,2} {
            \draw[fill=black] (0,-\y) circle(2pt); 
             \draw[fill=white] (1,-\y) circle(2pt); 
             \draw[fill=black] (2,-\y) circle(2pt); 
              \draw[fill=black] (2,-\y+1) circle(2pt); 
              \draw (1.2,-2+\y+.4) node[anchor=south] {\tiny{$a_{m-\y,j}^{[m]}$}};
              }
  
   \foreach \y in {0} {
            \draw[fill=black] (0,-\y) circle(2pt); 
             \draw[fill=white] (1,-\y) circle(2pt); 
             \draw[fill=black] (2,-\y) circle(2pt); 
              \draw[fill=black] (2,-\y+1) circle(2pt); 
             \draw (1.15,-2+\y+.2) node[anchor=south] {\tiny{$a_{m,j}^{[m]}$}};
              }
               \foreach \y in {1,...,3} {
            \draw (0,5-\y)--(2,5-\y);
              \draw (1,5-\y)--(2,-\y+6);}
               \foreach \y in {1,...,3} {
            \draw[fill=black] (0,-\y+5) circle(2pt); 
             \draw[fill=white] (1,-\y+5) circle(2pt); 
             \draw[fill=black] (2,-\y+5) circle(2pt); 
              \draw[fill=black] (2,-\y+5) circle(2pt); 
              \draw (1.3,-\y+5.4) node[anchor=south] {\tiny{$a_{\y,j}^{[m]}$}};
              }
               
      \draw(1,1.25) node {$\vdots$};

  \foreach \y in {1,...,3} {
            \draw (2,-\y+2)--(4,-\y+2);
              \draw (2,-\y+1)--(3,-\y+2);
  }
   \foreach \y in {1,...,2} {
            \draw[fill=black] (2,-\y) circle(2pt); 
             \draw[fill=white] (3,-\y+1) circle(2pt); 
             \draw[fill=black] (2,-\y+1) circle(2pt); 
              \draw[fill=black] (4,-\y+1) circle(2pt); 
               \draw (3.3,-2+\y+.3) node[anchor=south] {\tiny{$1/b_{m-\y,j}^{[m]}$}};
                \draw (3.5,-2+\y+.9) node[anchor=south] {\tiny{$1/b_{m-\y,j}^{[m]}$}};
              }
               \foreach \y in {0} {
            \draw[fill=black] (2,-\y) circle(2pt); 
             \draw[fill=white] (3,-\y+1) circle(2pt); 
             \draw[fill=black] (2,-\y+1) circle(2pt); 
              \draw[fill=white] (4,-\y+1) circle(2pt); 
              \draw (3.5,-2+\y+.9) node[anchor=south] {\tiny{$1/b_{m,j}^{[m]}$}};
                \draw (3.1,-2+\y+.2) node[anchor=south] {\tiny{$1/b_{m,j}^{[m]}$}};
              }

               \foreach \y in {1,...,3} {
            \draw (2,5-\y+1)--(4,5-\y+1);
              \draw (2,4-\y+1)--(3,-\y+6);}
               \foreach \y in {1,...,3} {
            \draw[fill=black] (2,-\y+5) circle(2pt); 
             \draw[fill=white] (3,-\y+6) circle(2pt); 
             \draw[fill=black] (2,-\y+6) circle(2pt); 
              \draw[fill=black] (4,-\y+6) circle(2pt); 
              \draw (3.5,-\y+5.9) node[anchor=south] {\tiny{$1/b_{\y,j}^{[m]}$}};
              \draw (3.1,-\y+5.3) node[anchor=south] {\tiny{$1/b_{\y,j}^{[m]}$}};
              }         
      \draw(3,2) node {$\vdots$};
        \end{tikzpicture}
\quad \quad
 \begin{tikzpicture}[scale=1]
  \foreach \y in {0,...,1} {
            \draw (2,-\y)--(4,-\y);
              \draw (3,-\y)--(4,-\y+1);
  }
  
   \foreach \y in {1,...,2} {
            \draw[fill=black] (2,-\y+1) circle(2pt); 
             \draw[fill=white] (3,-\y+1) circle(2pt); 
             \draw[fill=black] (4,-\y+1) circle(2pt); 
              \draw[fill=black] (4,-\y+2) circle(2pt); 
               \draw (3.2,-2+\y+.4) node[anchor=south] {\tiny{$a_{m-\y,j}^{[m-1]}$}};
              }

               \foreach \y in {1,...,3} {
            \draw (2,5-\y+1)--(4,5-\y+1);
              \draw (3,6-\y)--(4,-\y+7);}
               \foreach \y in {1,...,3} {
            \draw[fill=black] (2,-\y+6) circle(2pt); 
             \draw[fill=white] (3,-\y+6) circle(2pt); 
             \draw[fill=black] (4,-\y+6) circle(2pt); 
              \draw[fill=black] (4,-\y+7) circle(2pt); 
              \draw (3.3,-\y+6.4) node[anchor=south] {\tiny{$a_{\y,j}^{[m-1]}$}};
              }
               
      \draw(3,2.25) node {$\vdots$};

  \foreach \y in {2,...,3} {
            \draw (0,-\y+2)--(2,-\y+2);
              \draw (0,-\y+1)--(1,-\y+2);
  }
   \foreach \y in {1,...,2} {
            \draw[fill=black] (0,-\y) circle(2pt); 
             \draw[fill=white] (1,-\y+1) circle(2pt); 
             \draw[fill=black] (0,-\y+1) circle(2pt); 
              \draw[fill=black] (2,-\y+1) circle(2pt); 
               \draw (-.23,-2+\y+.3-1) node[anchor=south] {\tiny{$1/b_{m-\y,j-1}^{[m-1]}$}};
                \draw (1.5,-2+\y+.9-1) node[anchor=south] {\tiny{$1/b_{m-\y,j-1}^{[m-1]}$}};
              }

               \foreach \y in {1,...,3} {
            \draw (0,5-\y+1)--(2,5-\y+1);
              \draw (0,4-\y+1)--(1,-\y+6);}
               \foreach \y in {1,...,3} {
            \draw[fill=black] (0,-\y+5) circle(2pt); 
             \draw[fill=white] (1,-\y+6) circle(2pt); 
             \draw[fill=black] (0,-\y+6) circle(2pt); 
              \draw[fill=black] (2,-\y+6) circle(2pt); 
              \draw (1.5,-\y+5.9) node[anchor=south] {\tiny{$1/b_{\y,j-1}^{[m-1]}$}};
              \draw (0.0,-\y+5.3) node[anchor=south] {\tiny{$1/b_{\y,j-1}^{[m-1]}$}};
              }         
      \draw(1,2) node {$\vdots$};
        \end{tikzpicture}
        \caption{The portion of $G$ (left) and $G'$ (right) which is changed under \Cref{thm:horiz_swap_columns_dimer}.   }   \label{fig:horiz_swap_cols}
        
\end{figure}
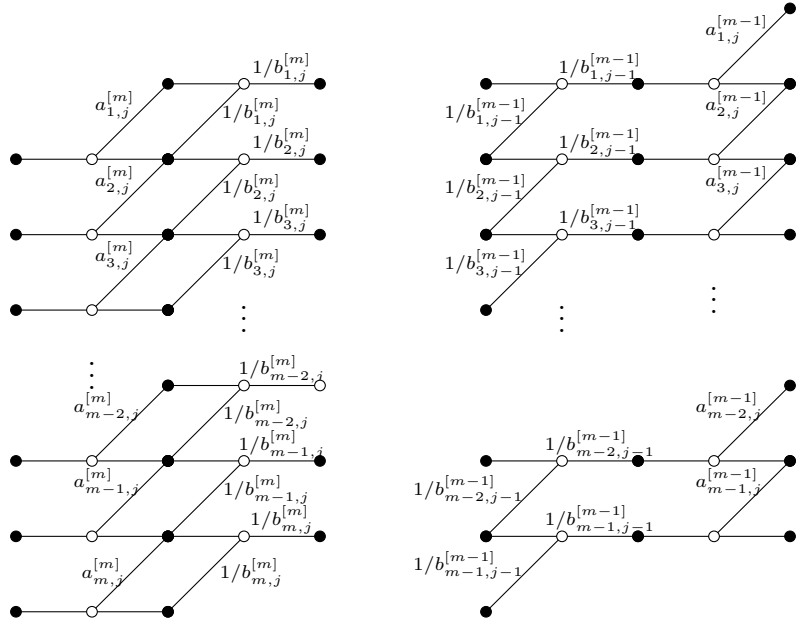

\begin{proof}
Exactly analogous to the proof of \Cref{thm:swap_columns_dimer}. It is easiest to see by first gauge-transforming by each $b_{i,j}^{[m]}$ at the white vertices of the $(-')$-columns, so all nontrivial edge weights are either $b_{i,j}^{[m]}$ or $a_{i,j}^{[m]}$ for some $i$. Now apply urban renewal (\Cref{thm:spider_coupling}) inside each square. As before, the top and bottom of the graph have $1$-valent vertices (this time black ones), so remove these and the edges incident to them. Also, contract the $1$-valent vertices joining those of the new squares which are above one another via \Cref{thm:vertex_expansion}. Each of the new squares has weights 
\begin{equation}
    \frac{1}{a_{i,j}^{[m]}+b_{i,j}^{[m]}}, \frac{a_{i,j}^{[m]}}{a_{i,j}^{[m]}+b_{i,j}^{[m]}}, \frac{b_{i,j}^{[m]}}{a_{i,j}^{[m]}+b_{i,j}^{[m]}}, \frac{1}{a_{i,j}^{[m]}+b_{i,j}^{[m]}}
\end{equation}
at its North, East, South and West faces respectively, for some $i$. Gauge transform by $a_{i,j}^{[m]}+b_{i,j}^{[m]}$ at the upper white vertex of each such square. After doing this at each square, each one has weights
\begin{equation}
    1, a_{i,j}^{[m-1]}, b_{i,j-1}^{[m-1]}, 1
\end{equation}
ordered as before. Contract all $2$-valent white vertices on the sides to single black vertices, and expand each white vertex in each square horizontally. Finally, gauge transform by $1/b_{i,j-1}^{[m-1]}$ at each white vertex which is incident to an edge of weight $b_{i,j-1}^{[m-1]}$. The resulting graph is $G'$, and the distributional equality statements in \Cref{thm:spider_coupling} and \Cref{thm:vertex_expansion} imply the desired one. 
\end{proof}

The analogues of $G^{v-swap}_{n,\ell}$ and $\bG^{v-swap}_{n,\ell}$ from Definitions \ref{def:gbsw} and \ref{def:gbsw_bar} are as follows.

\begin{defi}\label{def:g_h-swap} 
Let $n \geq 1$ and $\ell \in [1,n+1]$ be integers. Let $G^{h-swap}_{n,\ell}$ be the weighted bipartite graph drawn on $\Z^2$ as follows. It has white vertices at $(0,-1),(0,-2),\ldots,(0,-n)$ which are connected to black vertices at $(1,-1),(1,-2),\ldots,(1,-n)$ by edges with weights $1$. These black vertices then form the leftmost black vertices of a series of $\ell-1$ $(-')$-columns where the $j\tth$ $(-')$-column (starting from $j=1$ at the left) has weights $1/b_{i,0}^{[n-j]}, 1 \leq i \leq n-j$. The rightmost black vertices of the $(\ell-1)\tth$ $(-')$-column, which are located at coordinates $(2\ell-1,-1),\ldots,(2\ell-1,-n+\ell-1)$, form the leftmost black vertices of a series of $\ell$ $(+')$-columns where the $j\tth$ column (starting from $j=1$ at the left) has weights $a_{i,j}^{[n-\ell+j]}, 1 \leq i \leq n-\ell+j$. The rightmost black vertices of the $\ell\tth$ $(+')$-column in this series, which are at $(4\ell-1,\ell-1),\ldots,(4\ell-1,-n+\ell-1)$, form the leftmost black vertices of a series of $n-\ell+1$ $(-')$-columns, where the $j\tth$ column (starting from $j=1$ at the left) has weights $1/b_{i,\ell}^{[n-j+1]}, 1 \leq i \leq n-j+1$. The rightmost black vertices of the $(n-\ell+1)\tth$ $(-')$-column, which are at $(2n+2\ell+1,\ell-1),\ldots,(2n+2\ell+1,0)$, form the leftmost black vertices of a series of $n-\ell$ $(+')$-columns, where the $j\tth$ column (starting from $j=1$ at the left) has weights $a_{i,\ell+j}^{[\ell+j-1]}, 1 \leq j \leq \ell+j-1$. Finally, the rightmost black vertices of the last such $(+')$-column are connected by horizontal edges of weight $1$ to $n$ white vertices, which are located at $(4n+2,n-1),\ldots,(4n+2,0)$.
\end{defi}

\begin{defi}\label{def:g_h-swap_bar}
    For any $n \geq \ell \geq 1$, let $\bG^{h-swap}_{n,\ell}$ be the graph formed by removing both columns of $2$-valent white vertices from $G_{n,\ell}^{h-swap}$, and keeping only the middle component (the underlying graph is the same as the graph from \Cref{def:gbsw_bar} depicted in \Cref{fig:bgbsw_coordinates}, but the weights are different).
\end{defi}

\begin{proof}[Proof of {\Cref{thm:hor_slice_polymer}}]
    The proof is exactly analogous to that of \Cref{thm:vert_slice_polymer}. First, we may define $Y_\ell^{\Az}(M)$ for matchings $M$ on $G_n^{hor}$ rather than $G_n^{\Az}$, as the $\ell\tth$ column of squares from the left on the former corresponds to the $\ell\tth$ row of $n$ squares from the top in the latter after reflection over the diagonal from top-left to bottom-right and vertex-dilation. Then, \Cref{thm:horiz_swap_columns_dimer} allows us to swap the $(+')$ and $(-')$ columns to the left and right of the $\ell\tth$ column, while keeping the distribution of edges in the $\ell\tth$ column invariant. This gives a distributional equality between the edges on the $\ell\tth$ column of squares in $G_n^{hor}$ and the only column of squares in $G^{h-swap}_{n,\ell}$, analogous to \Cref{thm:aztec_to_udt}. Any dimer cover of $G^{h-swap}_{n,\ell}$ must have the same configuration of all horizontal edges in the first block of $(-')$ and second block of $(+')$ columns; \Cref{thm:gbsw_frozen_bits} showed this for $G^{v-swap}_{n,\ell}$, which is the same underlying graph as $G^{h-swap}_{n,\ell}$ with different weights. This gives a distributional equality with the column of squares in $\bG^{h-swap}_{n,\ell}$. Dimer covers of $\bG^{h-swap}_{n,\ell}$ are in bijection with path configurations $\pi_1,\ldots,\pi_p$ of $\glg_{n-\ell+1,\ell}$: the bijection was given for the same underlying graph and digraph in \Cref{thm:bgbsw_to_polymer_deterministic}. 
    
    It remains to check that the parameters of the weights match up. The series of $\ell$ $(+')$-columns in the left half of $\bG^{h-swap}_{n,\ell}$ have weights such that the $j\tth$ column (starting from $j=1$ at the left) has weights $a_{i,j}^{[n-\ell+j]}, 1 \leq i \leq n-\ell+j$, enumerated top to bottom. On the right half of $\bG^{h-swap}_{n,\ell}$ we have a series of $n-\ell+1$ $(-')$-columns, where the $j\tth$ column (starting from $j=1$ at the left) has weights $1/b_{i,\ell}^{[n-j+1]}, 1 \leq i \leq n-j+1$. Hence the weights on $\glg_{n-\ell+1,\ell}$ correspond to those on $\bG^{h-swap}_{n,\ell}$ via
   \begin{align}
        \begin{split}
            \rho_{x,y} &\longleftrightarrow a_{-y,\ell+x+1}^{[n-\ell+(\ell+x+1)]} =  a_{-y,\ell+x+1}^{[n+x+1]} \\ 
            \kappa_{x,y} &\longleftrightarrow 1/b_{-y,\ell}^{[n-x]}
        \end{split}
    \end{align}
    These have distribution
    \begin{align}
        \begin{split}
             a_{-y,\ell+x+1}^{[n+x+1]} &\sim \Gamma(\psi_{\ell+x+1}+\theta_{-y},1) \\ 
            b_{-y,\ell}^{[n-x]} & \sim \Gamma(\phi_{x+\ell-n} - \theta_{-y},1).
        \end{split}
    \end{align}
    Taking $m=\ell, p = n-\ell+1$, this yields the distributions in \Cref{def:horiz_polymer_digraph} and completes the proof.

\end{proof}

\begin{proof}[Proof of {\Cref{thm:horiz_slice_polymer_matching_full}}]
    Exactly analogous to the proof of \Cref{thm:dynamical_match_vert} in \Cref{sec:dynamical_vert}, now that we have proven the single-time matching \Cref{thm:hor_slice_polymer}.
\end{proof}

\subsection{Matching North and South turning points with stationary strict-weak polymer} \label{subsec:north_south_matching}

We will actually only treat the South turning point, as the North is the same by flipping the Aztec diamond around the horizontal axis and interchanging the $a$ and $b$ variables.

The stationary strict-weak polymer, also known as the Gamma polymer, was introduced in \cite{corwin2015strict}. We will instead follow the notation of the later work \cite{chaumont2018fluctuation}, partly to keep notation similar to what we had for the stationary log-Gamma polymer, and partly because we will use probabilistic results from \cite{chaumont2018fluctuation} as input. The Burke property given in this section was proven in \cite{corwin2015strict}, but we will use the formulation in \cite{chaumont2018fluctuation}.

\begin{defi}
    \label{def:stat_sw}
    Let $\alpha,\beta > 0$. On every up-right directed edge $(i,j) \to (i+1,j)$ or $(i,j) \to (i,j+1)$ with $i,j \geq 0$, place an independent random weight with distributions as follows. Boundary vertical edges $e = (0,j) \to (0,j+1)$ have weights $\wt(e) = R_j \sim \Beta^{-1}(\beta,\alpha)$, while all other vertical edges have weights $1$. Horizontal edges $e = (i,j) \to (i+1,j)$ have weights
    \begin{equation}
        \wt(e) = Y_{i,j} \sim \begin{cases}
            \Gamma(\alpha,1) & j \geq 1 \\ 
            \Gamma(\alpha+\beta,1) & j=0
        \end{cases}.
    \end{equation}
     These are known as \emph{stationary strict-weak polymer} weights. 
    
    Let $\Pi_{m,n}$ be the set of up-right paths with steps $(i,j) \to (i+1,j)$ and $(i,j) \to (i,j+1)$, beginning at $(0,0)$ and ending at $(m,n)$. Then the associated partition function is 
    \begin{equation}
        Z^{SW}_{m,n} := \sum_{\pi \in \Pi_{m,n}} \prod_{e \in \pi} \wt(e),
    \end{equation}
    where the product is over the weights on all lattice edges in the path. The associated \emph{quenched} polymer measure on $\Pi_{m,n}$ is 
    \begin{equation}
        Q^{SW}_{m,n}(\pi) = \frac{\prod_{e \in \pi} \wt(e)}{Z^{SW}_{m,n}},
    \end{equation}
    a random probability measure depending on the weights. The associated \emph{annealed} polymer measure on $\Pi_{m,n}$ is 
    \begin{equation}
        \P_{m,n}^{SW}(\pi) = \E[Q^{SW}_{m,n}(\pi)].
    \end{equation}
\end{defi}

\begin{prop}[Burke property]
    \label{thm:burke_sw}
    In the notation of \Cref{def:stat_sw}, 
    let
    \begin{align}
        \begin{split}
            U_{m,n} &= \frac{Z^{SW}_{m,n}}{Z^{SW}_{m-1,n}} \\ 
            V_{m,n} &= \frac{Z^{SW}_{m,n}}{Z^{SW}_{m,n-1}}
        \end{split}
    \end{align}
    whenever $m\geq 1, n \geq 0$ (for $U_{m,n}$) or $m \geq 0, n \geq 1$ (for $V_{m,n}$) so both partition functions are defined. Then 
    \begin{align}
        \begin{split}
            U_{m,n} &\sim \Gamma(\alpha+\beta,1) \\ 
            V_{m,n} &\sim \Beta^{-1}(\beta,\alpha)
        \end{split}
    \end{align}
    for every $(m,n)$. Furthermore, associating $U_{m,n}$ and $V_{m,n}$ to the directed horizontal edge $(m-1,n) \to (m,n)$ and vertical edge $(m,n) \to (m,n-1)$ respectively, we have that for any down-right path in $\Z_{\geq 0}^2$, the $U$ and $V$ variables associated to its right and down segments are all mutually independent.
\end{prop}
\begin{proof}
    Checked in \cite[Lemma 6.3]{corwin2015strict} in slightly different coordinates, see \cite[Proposition 2.3]{chaumont2018fluctuation} for the statement in our coordinates. Our $\alpha,\beta$ correspond to $\theta$ and $\mu-\theta$ in the notation of \cite{chaumont2018fluctuation}.
\end{proof}

We will use the notation $x_{mid}$ of \Cref{def:polymer_midpoint} for strict-weak polymer paths as well as for log-Gamma, since the path set is the same. Finally, we can state the desired matching.

\begin{thm}
    \label{thm:south_endpoint}
    Fix $n \in \N$ and $\alpha,\beta > 0$. Let $M$ be a perfect matching of the Gamma-disordered Aztec diamond $G_n^{\Az}$ with homogeneous parameters 
    \begin{align}
    \begin{split}
                \theta_i &\equiv 0 \\
                \psi_j &\equiv \alpha  \\ 
                \phi_j &\equiv \beta.
    \end{split}
    \end{align}
    Let $\pi$ be distributed according to the polymer measure $Q^{SW}_{n,n}$ of an independent stationary strict-weak polymer with the same parameters $\alpha,\beta$. Then 
    \begin{equation}
        T^{South}(M) = x_{mid}(\pi)\quad \quad \quad \quad \text{in (annealed) distribution}
    \end{equation}
    where $T^{South}(M)$ is as in \Cref{def:all_turning_points}.
\end{thm}

\begin{proof}
    The proof of \Cref{thm:south_endpoint} is essentially the same as that of \Cref{thm:west_matching_intro}, except that we do not have to carry out the path-complementation bijection of \Cref{thm:reduce_to_loggamma} first. The basic idea is that the partition function ratios along the antidiagonal $y=n-x$ of the stationary strict-weak polymer correspond to the variables $\kappa_{i,j}$ in the log-Gamma portion of the $\Gamma \log \Gamma$ polymer, while the weights above it correspond to the variables $\rho_{i,j}$ of the $\Gamma$ portion.

    The exact analogue of \Cref{thm:burke_along_diagonal} for the stationary strict-weak polymer follows directly from \Cref{thm:burke_sw}. Namely,
    \begin{equation}
        \label{eq:sw_Z_recurrence}
        Z^{SW}_{i+1,n-i-1} = \frac{U_{i,n-i}}{V_{i+1,n-i-1}} \cdot Z^{SW}_{i,n-i},
    \end{equation}
    where 
    \begin{align}\label{eq:uv_diag_dist_sw}
        \begin{split}
            U_{i,n-i} &\sim \Gamma(\alpha+\beta,1) \\ 
            V_{i,n-i} &\sim \Beta^{-1}(\beta,\alpha)
        \end{split}
    \end{align}
    are independent. Hence 
    \begin{equation}\label{eq:sw_Z_product}
        Z^{SW}_{n-i,i} = \prod_{j=0}^{i-1} \frac{V_{n-j,j}}{U_{n-j-1,j+1}} Z^{SW}_{n,0}.
    \end{equation}
    Furthermore, from the `in this case' part of \Cref{thm:XY_lukacs_cor}, 
    \begin{equation} \label{eq:sw_z_ratios}
        \frac{V_{n-j,j}}{U_{n-j-1,j+1}} \sim \Gamma^{-1}(\beta,1).
    \end{equation}

    By \Cref{thm:hor_slice_polymer}, 
        \begin{equation}
        Y_n^{\Az}(M) = X^{poly}(\pi)\quad \quad \quad \quad \text{in (annealed) distribution}
    \end{equation}
    where $\pi$ is a polymer path on $G_{1,n}^{\Gamma \log \Gamma}$. Both sides are singleton sets, and it is clear from comparing \Cref{def:aztec_matching_row} with \Cref{def:all_turning_points} that $Y_n^{\Az}(M) = \{T_{South}(M)+1\}$. 
    
    The remainder of the proof is the same as that of \Cref{thm:match_gammas}, and we simply give a sketch. The partition function ratios of \eqref{eq:sw_z_ratios} are equal in distribution to the weights $\kappa_{0,y} \sim \Gamma^{-1}(\phi_{0})$ of \Cref{def:horiz_polymer_digraph}, since we have set $\phi_j \equiv \beta$. Additionally, the weights $1$ and $Y_{\ldots}$ on the vertical and horizontal edges of \Cref{def:stat_sw} are equal in distribution to the weights $\rho_{\ldots} \sim \Gamma(\alpha,1)$ of \Cref{def:horiz_polymer_digraph}, since we have set $\psi_j \equiv \alpha$. The portions of paths in $\Pi_{n,n}$ lying above the antidiagonal $y=n-x$ are in bijection with paths on $G^{\Gamma \log \Gamma}_{1,n}$, and this bijection is weight-preserving and takes $X^{poly}(\pi)$ to $\{x_{mid}(\pi)+1\}$.
\end{proof}

\section{Fluctuations of polymer paths and turning points}\label{sec:fluctuations_of_turning_points}

Having made exact matchings between the turning points of the Gamma-disordered Aztec diamond and known polymer models in Theorems~\ref{thm:west_matching_intro}, \ref{thm:right-slice_to_rwre}, and~\ref{thm:south_endpoint}, in this section we will use known theorems for these polymer models to establish our asymptotics on the turning points. The proofs for the North and South turning points are identical, and also essentially identical to those for the West turning point, while the East turning point has different asymptotics. We analyze the East turning point in \Cref{subsec:east}, analyze the West turning point in \Cref{subsec:west}, and explain the few minor changes to this argument needed for the North and South turning points in \Cref{subsec:north_and_south}.

\subsection{East} \label{subsec:east}

\begin{thm}
    \label{thm:beta_rwre_limit_dynamical}
    Fix $\alpha,\beta > 0$ and set parameters $\psi_j \equiv \alpha T, \phi_j \equiv \beta T, \theta_i \equiv 0$ in the Aztec diamond weights. Sample a coupled sequence of matchings $M_1,M_2,\ldots$ according to the shuffling algorithm with these parameters (see \Cref{subsec:shuffling}). 
    
    For each $n$, define the continuous time random walk
    \begin{equation}
        B_n(t) = \frac{T^{East}(M_{\floor{nt}}) - \frac{\beta}{\alpha+\beta}\floor{nt}}{\sqrt{n}}
    \end{equation}
    where $T^{East}$ is as in \Cref{def:all_turning_points}. Then for almost every realization of the weights, the law of $B_n(t), t \in [0,1]$ converges as $n \to \infty$ to $\frac{\sqrt{\alpha \beta}}{\alpha+\beta}B_t, t \in [0,1]$, where $B_t$ is a standard Brownian motion, and convergence is with respect to the Skorokhod topology. In particular, for almost every realization of the weights,
    \begin{equation}\label{eq:gaussian_turning_point}
        \frac{T^{East}(M_n) - \frac{\beta}{\alpha+\beta} \cdot n}{\sqrt{n}} \to \mc{N}\left(0,\frac{\alpha \beta}{(\alpha+\beta)^2}\right) \quad \quad \quad \quad \text{ in distribution as $n \to \infty$.}
    \end{equation}
\end{thm}
\begin{proof}
    By \Cref{thm:right-slice_to_rwre},
    \begin{equation}
        T^{East}(M_{\floor{nt}}) = -X_{\floor{nt}}-1
    \end{equation}
    in joint distribution, where $X_\tau, \tau \in \Z_{\geq 0}$ is a Beta-RWRE. A quenched invariance principle for almost every environment for the Beta-RWRE is already proven as a consequence of the general result \cite[Theorem 1]{rassoul2005almost}, so we will just explain how to check the hypotheses of that result in our specific case. In the notation there (see \cite[page 300]{rassoul2005almost}), we have have an iid product-type random environment with
    \begin{itemize}
        \item $d=2$
        \item $\pi_{0,e_1+z} = \begin{cases}
           \frac{a_{1,1}^{[1]}}{a_{1,1}^{[1]}+b_{1,1}^{[1]}}  & z = 0 \\ 
            \frac{b_{1,1}^{[1]}}{a_{1,1}^{[1]}+b_{1,1}^{[1]}}  & z = -e_2 \\ 
            0 & \text{otherwise}
        \end{cases}$
        \item     
        $
        p(0,z) := \E[\pi_{0,e_1+z}] =  \begin{cases}
            \frac{\alpha}{\alpha+\beta} & z = 0 \\ 
            \frac{\beta}{\alpha+\beta} & z = -e_2 \\ 
            0 & \text{otherwise}
        \end{cases}
        $
        \item $v = e_1 - \frac{\beta}{\alpha+\beta} e_2$
        \item $
            \mf{D} := \sum_{z \in \{0,-e_2\}} (e_1+z-v)(e_1+z-v)^T p(0,z) = \frac{\alpha}{\alpha+\beta} \frac{\beta}{\alpha+\beta} \begin{pmatrix}
                0 & 0 \\ 0 & 1
            \end{pmatrix}
        $
    \end{itemize}
    The Hypothesis (ME) needed for \cite[Theorem 1]{rassoul2005almost} is trivially checked, since
    \begin{equation}
        \sum_z |z|^2 \E[\pi_{0,e_1+z}] = \frac{\beta}{\alpha+\beta} < \infty 
    \end{equation}
    and 
    \begin{equation}
        \mathbb{P}(\sup_z \pi_{0,e_1+z} < 1) = \Pr\left(\max\left\{\frac{a_{1,1}^{[1]}}{a_{1,1}^{[1]}+b_{1,1}^{[1]}},\frac{b_{1,1}^{[1]}}{a_{1,1}^{[1]}+b_{1,1}^{[1]}}\right\} < 1\right) = 1 > 0.
    \end{equation}
    
    This shows Skorokhod convergence. Since a standard Brownian motion is continuous almost surely, we have, for any fixed $t\in [0,1]$, that point evaluation $\pi_t:f \mapsto f(t)$ is continuous at $B_t$ in the Skorokhod topology almost surely \cite[Thm 12.5]{bil1999convergence}. Then by the Mapping Theorem \cite[Thm 2.7]{bil1999convergence} we also have that, for any fixed $t \in [0,1]$, the variable $B_n(t)$ converges in law to $B(t)$. Applying this to $t=1$ yields \eqref{eq:gaussian_turning_point}. 
\end{proof}

\subsection{West}\label{subsec:west}

Recall the stationary log-Gamma polymer of \Cref{def:stationary_log_gamma}. By \Cref{thm:west_matching_intro}, we wish to consider the intersection point $x_{mid}(\pi)$ of a path in this polymer with the antidiagonal $x+y=n$. Specifically, we will show in the following two results that (1) $x_{mid}(\pi)$ lies within a large enough $n^{2/3}$-size interval with high probability, and (2) $x_{mid}(\pi)$ lies outside an $n^{2/3}$-size interval with non-negligible probability. Together, these show that $x_{mid}(\pi)$ fluctuates on the scale $n^{2/3}$.

\begin{prop}[Tail bounds on the $n^{2/3}$ scale]
    \label{thm:loggamma_midpoint_tails}
    Fix parameters $0 < \alpha,\beta <  \infty$. For each $n \in \N$ let $\pi_n$ be a path distributed by the random polymer measure $Q^{stat-\Gamma}_{n,n}$ of \Cref{def:stationary_log_gamma} with these parameters, and let $x_{mid}(\pi_n)$ be as in \Cref{def:polymer_midpoint}. Then there exist constants $b_0,C$ depending on $\alpha,\beta$ such that for every $b \geq b_0$ and $n \in \N$,
    \begin{equation}\label{eq:loggamma_midpoint_tails}
        \P^{stat-\Gamma}_{n,n}\left(\abs*{x_{mid}(\pi_n) - \frac{\Psi_1(\beta)}{\Psi_1(\alpha)+\Psi_1(\beta)} \cdot n} \geq bn^{2/3}\right) \leq \frac{C}{b^3}.
    \end{equation}
    Here $\Psi_1$ is the trigamma function.
\end{prop}

\begin{prop}[Anticoncentration on the $n^{2/3}$ scale]
    \label{thm:loggamma_midpoint_anticoncentration}
    Assume the same setup as \Cref{thm:loggamma_midpoint_tails}. Then there exist constants $n_0,C_0,C_1 > 0$ depending on $\alpha,\beta$ such that for all $n \geq n_0$,
    \begin{equation}\label{eq:xmid_anticoncentration}
       \P^{stat-\Gamma}_{n,n}\left(\abs*{x_{mid}(\pi_n) - \frac{\Psi_1(\beta)}{\Psi_1(\alpha)+\Psi_1(\beta)} \cdot n} > C_1 n^{2/3}\right) \geq C_0.
    \end{equation}
    As before, $\Psi_1$ is the trigamma function.
\end{prop}

\begin{rmk}
    It is instructive to consider extreme cases of the macroscopic location $\frac{\Psi_1(\beta)}{\Psi_1(\alpha)+\Psi_1(\beta)} \cdot n$ of the turning points. Since the trigamma function is decreasing, when $\alpha \gg \beta$ (so the $a_{i,j}^{[n]}$ weights are much larger than the $b_{i,j}^{[n]}$ weights, we have $\frac{\Psi_1(\beta)}{\Psi_1(\alpha)+\Psi_1(\beta)} \approx 1$ and the turning point is essentially in the Southwest corner of the Aztec diamond. Similarly, if $\alpha \ll \beta$ the turning point is essentially in the Northwest corner. If $\alpha = T \balpha$ and $\beta = T \bbeta$, then $\Psi_1(T \balpha) \approx \tfrac{1}{T\balpha}$ for large $T$ by \eqref{eq:trigamma_bounds}, and consequently
    \begin{equation}
        \lim_{T \to \infty} \frac{\Psi_1(T\bbeta)}{\Psi_1(T\balpha)+\Psi_1(T\bbeta)} = \frac{\balpha}{\balpha+\bbeta}.
    \end{equation}
    This corresponds to the deterministic limit $a_{i,j}^{[n]}=\balpha, b_{i,j}^{[n]} = \bbeta$ of our weights.
\end{rmk}

The hard work needed for these estimates is contained in \cite{seppalainen2012scaling} and \cite{chaumont2018fluctuation}, but some translation is needed since those works study a different quantity. They do not study the intersection with the antidiagonal, but rather intersections with certain vertical and horizontal lines which we define now.

\begin{defi}
    \label{def:v_w}
    Given a path $\pi \in \Pi_{m,n}$ in the notation of \Cref{def:stationary_log_gamma}, and $0 \leq l \leq n, 0 \leq k \leq m$, we define
    \begin{align}
        \begin{split}
            v_0(l) &:= \min\{i: (i,l) \in \pi\} \\ 
            v_1(l) &:= \max\{i: (i,l) \in \pi\}
        \end{split}
    \end{align}
    and
        \begin{align}
        \begin{split}
            w_0(k) &:= \min\{j: (k,j) \in \pi\} \\ 
            w_1(k) &:= \max\{j: (k,j) \in \pi\}.
        \end{split}
    \end{align}
    Note that the dependence on $\pi$ is implicit in the notation.
\end{defi}

\begin{figure}[H]
\centering
\begin{tikzpicture}[scale=1.15]
\draw (0,0) -- (6,0) -- (6,5) -- (0,5) -- (0,0);
\draw [ultra thick] (0,2) -- (6,2);
\draw [ultra thick] (4,0) -- (4,5);

\node [below] at (4,0) {$k$};
\node[left] at (0,2) {$l$};

\draw [line width=2.5pt] (0,0) -- (.5,0) -- (.5,.75) -- (.75,.75) -- (.75,1) -- (1.5,1) -- (1.5,1.25) -- (1.75,1.25) -- (1.75,1.75) -- (2,1.75) -- (2,2) -- (2.75,2) -- (2.75,2.5) -- (3.25,2.5) -- (3.25, 3.25) -- (3.75, 3.25) -- (3.75,3.5) -- (4,3.5) -- (4,4.25) -- (4.75,4.25) -- (4.75,4.75) -- (5,4.75) -- (5,5) -- (6,5);

\draw [dashed] (2,0) -- (2,2);
\draw [dashed] (2.75,0) -- (2.75,2);
\draw [dashed] (0,3.5) -- (4,3.5);
\draw [dashed] (0,4.25) -- (4,4.25);

\node [below] at (2.75,0) {$v_1(l)$};
\node [below] at (2,0) {$v_0(l)$};
\node [below] at (6,0) {$m$};

\node [left] at (0,3.5) {$w_0(k)$};
\node [left] at (0,4.25) {$w_1(k)$};
\node [left] at (0,5) {$n$};

\end{tikzpicture}
\caption{Example path with $v_0,v_1, w_0,w_1$ illustrated. Reproduced from \cite[Figure 3]{chaumont2018fluctuation}.}
\label{fig:v_w}
\end{figure}
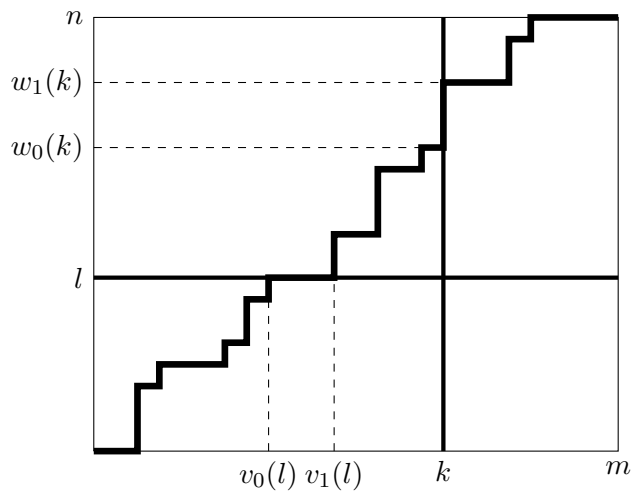

We will use the following theorem on the fluctuation exponents of these quantities.

\begin{thm}[Tail bounds and anticoncentration on the $n^{2/3}$ scale]
    \label{thm:loggamma_upper_bound_sepp}
    Recall the setup of \Cref{def:stationary_log_gamma} and fix the parameters $\alpha,\beta$. For each $N \in \N$, let 
    \begin{equation}\label{eq:m_n}
        m_N = \floor*{\frac{\Psi_1(\alpha)}{\Psi_1(\beta)}N}
    \end{equation}
    where $\Psi_1$ is the trigamma function. Then there exist positive constants $b_0, C, c_0, c_1, N_0$ depending only on the parameters $\alpha,\beta$ and on $\tau$, such that for all $b \geq b_0$ and $N \in \N$,
\begin{align}
    \P_{m_N,N}^{stat-\Gamma}(v_0(\floor{\tau N}) \leq \tau m_N - b N^{2/3} \text{ or }v_1(\floor{\tau N}) \geq \tau m_N + b N^{2/3}) &\leq \frac{C}{b^3} \\ 
    \P_{m_N,N}^{stat-\Gamma}(w_0(\floor{\tau m_N}) \leq \tau N - b N^{2/3} \text{ or }w_1(\floor{\tau m_N}) \geq \tau N + b N^{2/3}) &\leq \frac{C}{b^3}
\end{align}
    and for all $N \geq N_0$, 
    \begin{equation}
        \P_{m_N,N}^{stat-\Gamma}(v_1(\floor{\tau N}) \geq \tau m_N + c_1 N^{2/3} \text{ or }w_1(\tau m_N) \geq \tau N + c_1 N^{2/3}) \geq c_0.
    \end{equation}
    Furthermore, for any compact set $S \subset [0,1)$, the constants $b_0, C, c_0, c_1, N_0$ may be chosen uniformly over all $\tau \in S$.
\end{thm}
\begin{proof}
    This is a special case of the log-Gamma case of \cite[Theorem 1.5]{chaumont2018fluctuation}, where we recall that our $\alpha,\beta$ correspond to $\theta,\mu-\theta$ in their notation. Specifically, one must take the parameters $N,n_N,m_N$ in that result to be $\floor{N/\Psi_1(\beta)}$, $N$ and $\floor{N\Psi_1(\alpha)/\Psi_1(\beta)}$ in terms of our variables. Note that the fact that our $N$ is a constant times the one in \cite[Theorem 1.5]{chaumont2018fluctuation} changes the constants $c_1, b, C$ in that result by a constant, but this does not change the result.

    One must then verify that our $n_N$ and $m_N$ satisfy \cite[(1.10)]{chaumont2018fluctuation}. In our notation, this condition translates to the bounds
    \begin{equation}
        \abs*{\floor*{\frac{\Psi_1(\alpha)}{\Psi_1(\beta)}N} - \floor*{\frac{N}{\Psi_1(\beta)}} \cdot \Var(\log Y_{0,j})} \leq \gamma N^{2/3} 
    \end{equation}
    and
    \begin{equation}
        \abs*{N - \floor*{\frac{N}{\Psi_1(\beta)}} \cdot \Var(\log Y_{j,0})} \leq \gamma N^{2/3} 
    \end{equation}
    holding for all $N$, for some constant $\gamma$. These bounds follow directly from the formula
    \begin{equation}
        \Var \log X = \Psi_1(\chi) 
    \end{equation}
    when $X \sim \Gamma^{-1}(\chi,1)$, and the fact that $Y_{0,j} \sim \Gamma^{-1}(\alpha,1)$ and $Y_{j,0} \sim \Gamma^{-1}(\beta,1)$ in \Cref{def:stationary_log_gamma}. The quantities we wish to bound by $\gamma N^{2/3}$ are actually just $\mathcal{O}(1)$ error terms coming from different placement of floor functions in the two terms being compared.
\end{proof}

Essentially the same theorem, but without the uniformity in $\tau$ (which we need), was proven first in {\cite[Theorem 2.3]{seppalainen2012scaling}}. We note that \cite[Theorem 1.5]{chaumont2018fluctuation} also proves uniformity with respect to the Gamma parameters $\alpha,\beta$, though we do not need this.

To relate $v_i$ to $x_{mid}$ for \Cref{thm:loggamma_midpoint_tails}, we use a basic deterministic fact:
\begin{lemma}
    \label{thm:v_w_to_xmid}
    Let $1 \leq n \leq m$ and let $\pi \in \Pi_{m,n}$ be any path, let $0 \leq \ell \leq n$, and let $x_{mid}(\pi)$ and $v_i(\ell)$ be as in \Cref{def:polymer_midpoint} and \Cref{def:v_w} respectively. Then 
    \begin{equation}\label{eq:x_bound_by_max}
        |x_{mid}(\pi) - (n-\ell)| \leq \max\{v_1(\ell)-(n-\ell),(n-\ell)-v_0(\ell)\}       
    \end{equation}
\end{lemma}
\begin{proof}
    The segment between $(v_0(\ell),\ell)$ and $(v_1(\ell),\ell)$ either lies to the left of the line $x+y=n$, to the right of the line, or intersects the line. In other words, these three cases correspond to the conditions
    \begin{enumerate}
        \item $v_1(\ell) \leq n-\ell$
        \item $v_0(\ell) \geq n-\ell$
        \item $v_0(\ell) \leq n-\ell \leq v_1(\ell)$
    \end{enumerate}
    respectively.

    In the first case, we additionally have 
    \begin{equation}
        v_1(\ell) \leq x_{mid}(\pi) \leq n-\ell
    \end{equation}
    because the path $\pi$ moves only up and right from the point $(v_1(\ell),\ell)$ so the $x$-coordinate can only increase before it hits the line $x+y=n$. Similarly, in the second case
    \begin{equation}
        n-\ell \leq x_{mid}(\pi) \leq v_0(\ell).
    \end{equation}
    In the third case, since $\pi$ passes horizontally between $(v_0(\ell),\ell)$ and $(v_1(\ell),\ell)$, its intersection with the line $x+y=n$ occurs at $(n-\ell,\ell)$, and so $x_{mid}(\pi) = n-\ell$.

    In each case, \eqref{eq:x_bound_by_max} clearly holds.
\end{proof}

\begin{proof}[Proof of {\Cref{thm:loggamma_midpoint_tails}}]
    We note that interchanging $\alpha$ with $\beta$ and simultaneously interchanging the $x$ and $y$ axes in \Cref{def:stationary_log_gamma} yields distributionally equal weights, since the boundary weights are interchanged by both transformations. This transformation also replaces
    \begin{align}
        x_{mid}(\pi_n) &\mapsto n - x_{mid}(\pi_n) \\ 
        \frac{\Psi_1(\beta)}{\Psi_1(\alpha) + \Psi_1(\beta)} &\mapsto 1-\frac{\Psi_1(\beta)}{\Psi_1(\alpha) + \Psi_1(\beta)}.
    \end{align}
    Hence we may assume without loss of generality that $\Psi_1(\alpha) \geq \Psi_1(\beta)$, since the case $\Psi_1(\alpha) < \Psi_1(\beta)$ of \Cref{thm:loggamma_midpoint_tails} follows from this case by symmetry.
    
    For each $n$, let $N_n$ be an integer such that 
    \begin{equation}\label{eq:n_from_N}
        n = \floor*{\frac{\Psi_1(\alpha)}{\Psi_1(\beta)}N_n}.
    \end{equation}
    By our assumption $\Psi_1(\alpha) \geq \Psi_1(\beta)$, we may always take $N_n \leq n$.

    Let $\pi_n$ be as in the statement. We first relate it to a path $\pi_n'$ in a smaller $n \times N_n$ rectangular box, on which \Cref{thm:loggamma_upper_bound_sepp} applies, see \Cref{fig:box_and_line}. Rather than $x_{mid}(\pi_n)$, we will consider $\min(x_{mid}(\pi_n),N_n)$ for now, and deal with the event $\min(x_{mid}(\pi_n),N_n) \neq x_{mid}(\pi_n)$ later. 

    We next claim that as a deterministic fact, for any path $\pi_n$, the function of $\pi_n$ given by $\min(x_{mid}(\pi_n),N_n)$ depends only on the intersection $\pi_n \cap ([1,n] \times [n-N_n+1,n])$. We have that $x_{mid}(\pi_n)$ lies in $[0,N_n-1]$ if and only if the intersection of $\pi_n$ with the line $x+y=n$ lies inside the box $[0,n] \times [n-N_n+1,n]$ (see \Cref{fig:box_and_line}). In the other case, $x_{mid}(\pi_n) \geq N_n$, $\pi_n$ must hit the line $x+y=n$ at a point outside this box, and hence will enter the box at a point with $x$-coordinate $\geq N_n$. Hence $\min(x_{mid}(\pi_n), N_n)$ depends only on the part of the path $\pi_n \cap ([0,n] \times [n-N_n+1,n])$ lying inside this box. In fact, it depends only on $\pi_n \cap ([1,n] \times [n-N_n+1,n])$, since $\pi_n$ passes through $(0,n)$ if and only if it passes through $(1,n)$ but not $(1,n-1)$, and the latter event depends only on $\pi_n \cap ([1,n] \times [n-N_n+1,n])$, proving the claim.

    \begin{figure}
        \centering 

        \begin{tikzpicture}[scale=1.15]
\draw (0,0) -- (5,0) -- (5,5) -- (0,5) -- (0,0);

\path[fill=green!60, fill opacity=0.30, draw=none]
    (0,1.75) rectangle (5,5);

\draw[blue,ultra thick] (0,5) -- (5,0); 

\node [left] at (0,1.75) {$(0,n-N_n)$};

\node [below] at (5,0) {$(n,0)$};

\node [below] at (0,0) {$(0,0)$};

\node [left] at (0,5) {$(0,n)$};

\node [right] at (5,5) {$(n,n)$};

\node [right] at (3.25,1.75) {$(N_n,n-N_n)$};

 \draw[fill=black] (3.25,1.75) circle(2pt); 

\draw[step=0.25, dotted, draw=black!60, opacity=0.45, line cap=round] (0,0) grid (5,5);

\draw [line width=2.5pt]
  (0,0) -- (0.25,0) -- (0.25,0.5) -- (1,0.5) -- (1,0.75) -- (1.25,0.75)
  -- (1.25,1.5) -- (1.5,1.5) -- (1.5,1.75) -- (2.25,1.75) -- (2.25,2)
  -- (2.5,2) -- (2.5,2.75) -- (3,2.75) -- (3,3) -- (3.5,3) -- (3.5,3.75)
  -- (4.25,3.75) -- (4.25,4) -- (4.5,4) -- (4.5,4.5) -- (5,4.5) -- (5,5);

\draw [red,dashed,line width=1.5pt]
  (0,1.75) -- (2.25,1.75) -- (2.25,1.75) -- (2.25,2)
  -- (2.5,2) -- (2.5,2.75) -- (3,2.75) -- (3,3) -- (3.5,3) -- (3.5,3.75)
  -- (4.25,3.75) -- (4.25,4) -- (4.5,4) -- (4.5,4.5) -- (5,4.5) -- (5,5);

\end{tikzpicture}
        \caption{A path $\pi_n \sim \P^{stat-\Gamma}_{n,n}$ (black). The portion inside the green $n \times N_n$ rectangle, when completed to a path $\pi_n'$ beginning at $(0,n-N_n)$ by adding a horizontal portion (red, dashed), has distribution $\P^{stat-\Gamma}_{n,N_n}$ after shifting, by \Cref{thm:shrink_rectangle}. In this case, $\min(x_{mid}(\pi_n),N_n) = x_{mid}(\pi_n)$. If instead the black path had intersected the blue line below the green rectangle, we would have $\min(x_{mid}(\pi_n),N_n)=N_n=x_{mid}(\pi_n')$.}
        \label{fig:box_and_line}
    \end{figure}
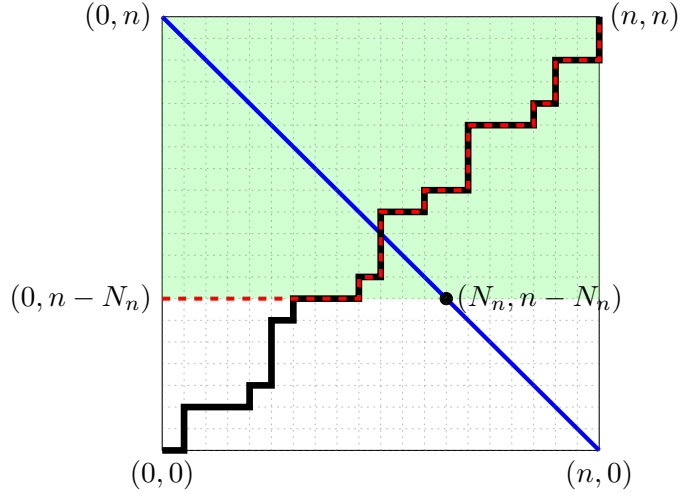

    Let $\pi_n'$ be $\Psg_{n,N_n}$-distributed. Similarly, $x_{mid}(\pi_n')$ only depends on $\pi_n' \cap ([1,n] \times [1,N_n])$. By the distributional equality of \Cref{thm:shrink_rectangle}, $\pi_n \cap ([1,n] \times [n-N_n+1,n])$ and $\pi_n' \cap ([1,n] \times [1,N_n])$ are equal in distribution up to a vertical coordinate shift by $N_n$, which does not affect $x_{mid}$. Hence by the previous paragraph,
    \begin{equation}\label{eq:pi_pi'_dist_eq}
        \min\{x_{mid}(\pi_n), N_n\} = x_{mid}(\pi_n') \quad \quad \quad \quad \text{in (annealed) distribution.}
    \end{equation}

    However, \Cref{thm:loggamma_upper_bound_sepp} will give us control on $x_{mid}(\pi_n')$. Setting $N = N_n$ and $m_N = n$ in \Cref{thm:loggamma_upper_bound_sepp}, there exist positive constants $b_0, C$ depending only on the parameters $\alpha,\beta$ and on $\tau$, such that for all $b \geq b_0$ and $n \in \N$,
\begin{align}\label{eq:cite_n_N_n}
    \P_{n,N_n}^{stat-\Gamma}(v_0(\floor{\tau N_n}) \leq \tau n - b N_n^{2/3} \text{ or }v_1(\floor{\tau N_n}) \geq \tau n + b N_n^{2/3}) &\leq \frac{C}{b^3} 
\end{align}
    By changing $b_0$ and $C$, \eqref{eq:cite_n_N_n} remains true with $\tau n$ replaced by $\ceil{\tau n}$ and $b N_n^{2/3}$ replaced by $b n^{2/3}$. Hence, setting 
    \begin{equation}
        \tau = \frac{\Psi_1(\beta)}{\Psi_1(\alpha)+\Psi_1(\beta)},
    \end{equation}
    we have 
    \begin{equation}
        \tau \cdot n = N_n - \tau \cdot N_n,
    \end{equation}
    up to irrelevant $1+o_n(1)$ factors from the floor function in \eqref{eq:n_from_N} which we ignore. Hence \eqref{eq:cite_n_N_n} implies 
    \begin{equation}\label{eq:Psg_max}
        \Psg_{n,N_n}(\max\{(N_n - \floor{\tau N_n}) - v_0(\floor{\tau N_n}), v_1(\floor{\tau N_n}) - (N_n - \floor{\tau N_n})\} \geq b n^{2/3}) \leq \frac{C}{b^3}.
    \end{equation}
    Taking $(m,n) = (n,N_n)$ and $\ell = \floor{\tau N_n}$ in \Cref{thm:v_w_to_xmid}, we have 
    \begin{equation}
        \abs*{x_{mid}(\pi_n') - (N_n - \floor{\tau N_n})}\leq \max\{(N_n - \floor{\tau N_n}) - v_0(\floor{\tau N_n}), v_1(\floor{\tau N_n}) - (N_n - \floor{\tau N_n})\},
    \end{equation}
    so
    \begin{equation}\label{eq:Psg_x_mid}
        \Psg_{n,N_n}(\abs*{x_{mid}(\pi_n') - (N_n - \floor{\tau N_n})} \geq b n^{2/3}) \leq \frac{C}{b^3}
    \end{equation}
    for all $b \geq b_0$. 

    To complete the proof, we must deal with the minimum on the left hand side of \eqref{eq:pi_pi'_dist_eq}. Note that \eqref{eq:Psg_x_mid} implies
        \begin{equation}\label{eq:Psg_x_mid2}
        \Psg_{n,N_n}(x_{mid}(\pi_n') \geq (1-\tau)N_n + b n^{2/3}) \leq \frac{C}{b^3}.
    \end{equation}
    Since $\tau > 0$ and $n$ is a constant times $N_n$, \eqref{eq:Psg_x_mid2} implies that $\Psg_{n,N_n}(x_{mid}(\pi_n') = N_n) = o_n(1)$ as $n \to \infty$. Hence by \eqref{eq:pi_pi'_dist_eq} and \eqref{eq:Psg_x_mid},
    \begin{equation}\label{eq:Psg_x_mid_noprime}
        \Psg_{n,n}(x_{mid}(\pi_n) \geq (1-\tau)N_n + b n^{2/3}) \leq \frac{C}{b^3} + o_n(1).
    \end{equation}
    Note also that by \eqref{eq:n_from_N},
\begin{equation}\label{eq:get_psi_ratio}
    (1-\tau)N_n = \frac{\Psi_1(\beta)}{\Psi_1(\alpha)+\Psi_1(\beta)} \cdot n \cdot (1+o_n(1)),
\end{equation}
    where the $1+o_n(1)$ comes from ignoring the floor function in \eqref{eq:n_from_N}. After possibly changing $b_0$ and $C$ to dispense with the $1+o_n(1)$ terms, substituting \eqref{eq:get_psi_ratio} into \eqref{eq:Psg_x_mid_noprime} yields \eqref{eq:loggamma_midpoint_tails}.
\end{proof}

For \Cref{thm:loggamma_midpoint_anticoncentration}, we similarly begin with an elementary deterministic lemma.
\begin{lemma}
    \label{thm:x_mid_bigger_vw}
    Let $n \leq m$, let $\pi \in \Pi_{m,n}$, and for any $1 \leq \ell \leq n$ and $1 \leq k \leq n$ let $v_1(\ell)$ and $w_1(k)$ be as in \Cref{def:v_w}. Then for any such $\pi$, the implications
\begin{align}
    v_1(\ell) \geq n-\ell &\implies x_{mid}(\pi) \geq n-\ell \\ 
    w_1(k) \geq n-k &\implies x_{mid}(\pi) \leq k
\end{align}
    hold.
\end{lemma}
\begin{proof}
    If $v_1(\ell) \geq n-\ell$, then the point $(v_1(\ell),\ell) \in \pi$ lies weakly up-right of the line $x+y=n$. Hence the intersection of $\pi$ with $x+y=n$ occurs at $y$-coordinate $\leq \ell$, because $\pi$ is an up-right path. Thus it occurs at $x$-coordinate $\geq n- \ell$, so $x_{mid}(\pi) \geq n-\ell$. 
    
    Similarly, if $w_1(k) \geq n-k$ then the point $(k,w_1(k)) \in \pi$ lies weakly up-right of the line $x+y=n$. Hence the intersection of $\pi$ with $x+y=n$ occurs at $x$-coordinate $\leq k$. 
\end{proof}

\begin{proof}[Proof of {\Cref{thm:loggamma_midpoint_anticoncentration}}]
    As in the proof of \Cref{thm:loggamma_midpoint_tails}, it suffices to consider the case $\Psi_1(\alpha) \geq \Psi_1(\beta)$. Let $N_n$, $\pi_n$, and $\pi_n'$ be as in that proof, and recall that
    \begin{equation}\label{eq:pi_pi'_dist_eq_restate}
        \min\{x_{mid}(\pi_n), N_n\} = x_{mid}(\pi_n') \quad \quad \quad \quad \text{in (annealed) distribution.}
    \end{equation}
    For any $C_1$, by \eqref{eq:n_from_N}
    \begin{equation}
        N_n > \frac{\Psi_1(\beta)}{\Psi_1(\alpha)+\Psi_1(\beta)}n + C_1 n^{2/3}
    \end{equation}
    for all sufficiently large $n$, and together with \eqref{eq:pi_pi'_dist_eq_restate} this implies
    \begin{multline}\label{eq:pi_to_pi'_prob}
         \P^{stat-\Gamma}_{n,n}\left(\abs*{x_{mid}(\pi_n) - \frac{\Psi_1(\beta)}{\Psi_1(\alpha)+\Psi_1(\beta)} \cdot n} > C_1 n^{2/3}\right) \\ 
         =  \P^{stat-\Gamma}_{n,N_n}\left(\abs*{x_{mid}(\pi_n') - \frac{\Psi_1(\beta)}{\Psi_1(\alpha)+\Psi_1(\beta)} \cdot n} > C_1 n^{2/3}\right)
    \end{multline}
    for all sufficiently large $n$. So it remains to bound the right-hand side of \eqref{eq:pi_to_pi'_prob}.

    Let 
    \begin{equation}
        \tau_0 = \frac{\Psi_1(\beta)}{\Psi_1(\alpha)+\Psi_1(\beta)},
    \end{equation}
    fix some $\delta \in (0,\tau_0)$, and let $S = [\tau_0-\delta,\tau_0]$. Then
    \begin{equation}\label{eq:tau_n}
        \tau_n := \tau_0 - c_1 \frac{N_n^{\frac{2}{3}}}{N_n+n} \in S
    \end{equation}
    for all $n$ large enough. Setting $N=N_n$ and $m_N = n$ in \Cref{thm:loggamma_midpoint_anticoncentration}, there exist positive constants $c_1,c_0,N_0$ such that
    \begin{equation}\label{eq:use_anticoncentration_vw}
        \P_{n,N_n}^{stat-\Gamma}\left(v_1(\floor{\tau N_n}) \geq \tau n + c_1 N_n^{2/3} \text{ or }w_1(\floor{\tau n}) \geq \tau N_n + c_1 n^{2/3}\right) \geq c_0.
    \end{equation}
    for all $n \geq N_0$ and $\tau \in S$.

    $\tau_0$ satisfies the elementary relations
    \begin{equation}\label{eq:tau_0_swaps}
        \tau_0 n = (1-\tau_0)N_n \cdot (1+o_n(1))
    \end{equation}
    and 
    \begin{equation}
        \tau_0 = \frac{N_n}{N_n+n}\cdot(1+o_n(1)),
    \end{equation}
    where as usual, $1+o_n(1)$ comes from ignoring a floor function. From these a simple computation shows
    \begin{equation}
        \label{eq:asymptotic_for_v}
        N_n-\floor{\tau_n N_n} = \left(\tau_0 n + \tau_0 c_1 N_n^{2/3}\right) \cdot (1+o_n(1))
    \end{equation}
    and
    \begin{equation}
        \label{eq:asymptotic_for_w}
        N_n - \floor{\tau_n n} = \left(\tau_0 N_n + (1-\tau_0) \pfrac{\Psi_1(\beta)}{\Psi_1(\alpha)}^{2/3} c_1 n^{2/3} \right) \cdot (1+o_n(1)).
    \end{equation}
    Because $\tau_0 < 1$, \eqref{eq:asymptotic_for_v} implies that
    \begin{equation}
        \label{eq:bound_for_v}
        N_n-\floor{\tau_n N_n} < \tau_0 n + c_1 N_n^{2/3}
    \end{equation}
    for all large enough $n$. Similarly, $1-\tau_0 \in (0,1)$, and by our assumption $\Psi_1(\alpha) \geq \Psi_1(\beta)$ we have that 
    \begin{equation}
        \pfrac{\Psi_1(\beta)}{\Psi_1(\alpha)}^{2/3} \in (0,1),
    \end{equation}
    so \eqref{eq:asymptotic_for_w} implies that
    \begin{equation}
        \label{eq:bound_for_w}
        N_n - \floor{\tau_n n} < \tau_0 N_n + c_1 n^{2/3}
    \end{equation}
    for all large enough $n$. Combining \eqref{eq:bound_for_v} and \eqref{eq:bound_for_w} with \eqref{eq:use_anticoncentration_vw}, we have
    \begin{equation}
        \label{eq:anticonc_ready_for_lemma}
        \P_{n,N_n}^{stat-\Gamma}\left(v_1(\floor{\tau_n N_n}) \geq N_n-\floor{\tau_n N_n} \text{ or }w_1(\floor{\tau_n n}) \geq \tau_n N_n - \floor{\tau_n n}\right) \geq c_0
    \end{equation}
    for all sufficiently large $n$. By \Cref{thm:x_mid_bigger_vw}, the implications
    \begin{equation}\label{eq:v_to_xmid}
        v_1(\floor{\tau_n N_n}) \geq N_n-\floor{\tau_n N_n} \quad \quad \implies \quad \quad  x_{mid}(\pi_n') \geq N_n-\floor{\tau_n N_n}
    \end{equation}
    and
    \begin{equation}
        \label{eq:w_to_xmid}
        w_1(\floor{\tau_n n}) \geq \tau_n N_n - \floor{\tau_n n} \quad \quad \implies \quad \quad  x_{mid}(\pi_n') \leq \floor{\tau_n n}
    \end{equation}
    hold for any path $\pi_n'$. Hence \eqref{eq:anticonc_ready_for_lemma} implies
    \begin{equation}
        \label{eq:prob_lowerbound_xmid}
        \P_{n,N_n}^{stat-\Gamma}\left(x_{mid}(\pi_n') \geq N_n-\floor{\tau_n N_n} \text{ or }x_{mid}(\pi_n') \leq \floor{\tau_n n}\right) \geq c_0.
    \end{equation}
    By \eqref{eq:asymptotic_for_v} and the definition of $N_n$,
    \begin{equation}
        \label{eq:new_v_asymptotic}
        N_n - \floor{\tau_n N_n} =  \left(\tau_0 n + \tau_0 c_1 \pfrac{\Psi_1(\beta)}{\Psi_1(\alpha)}^{2/3} n^{2/3}\right) \cdot (1+o_n(1)).
    \end{equation}
    By \eqref{eq:asymptotic_for_w} and 
    \eqref{eq:tau_0_swaps}, 
    \begin{equation}
        \label{eq:tau_asymptotic}
        \floor{\tau_n n} = \left( \tau_0 n - (1-\tau_0) \pfrac{\Psi_1(\beta)}{\Psi_1(\alpha)}^{2/3} c_1 n^{2/3}\right) \cdot (1+o_n(1)).
    \end{equation}
    Let $C_1$ be any real number satisfying
    \begin{equation}\label{eq:def_C_1}
        0 < C_1 < \min\left\{\tau_0 c_1 \pfrac{\Psi_1(\beta)}{\Psi_1(\alpha)}^{2/3}, (1-\tau_0) \pfrac{\Psi_1(\beta)}{\Psi_1(\alpha)}^{2/3} c_1 \right\},
    \end{equation}
    so that 
    \begin{equation}
        \label{eq:C_1_lower}
        \tau_0 n + C_1 n^{2/3} < \left( \tau_0 n - (1-\tau_0) \pfrac{\Psi_1(\beta)}{\Psi_1(\alpha)}^{2/3} c_1 n^{2/3}\right) \cdot (1+o_n(1))
    \end{equation}
    and 
    \begin{equation}
            \label{eq:C_1_upper}
            \tau_0 n - C_1 n^{2/3} >
        \left( \tau_0 n - (1-\tau_0) \pfrac{\Psi_1(\beta)}{\Psi_1(\alpha)}^{2/3} c_1 n^{2/3}\right) \cdot (1+o_n(1)) 
    \end{equation}
    for all sufficiently large $n$, say $n > n_0$ where we also take $n_0 > N_0$. Then taking the bound \eqref{eq:prob_lowerbound_xmid}, substituting the asymptotic expressions \eqref{eq:new_v_asymptotic} and \eqref{eq:tau_asymptotic}, and bounding these via \eqref{eq:C_1_lower} and \eqref{eq:C_1_upper}, we have that
    \begin{equation}
        \label{eq:prob_lowerbound_xmid_final}
        \P_{n,N_n}^{stat-\Gamma}\left(x_{mid}(\pi_n') \geq \tau_0 n + C_1 n^{2/3} \text{ or }x_{mid}(\pi_n') \leq \tau_0 n - C_1 n^{2/3}\right) \geq c_0
    \end{equation}
    for all $n > n_0$. After taking $C_0=c_0$ in \eqref{eq:prob_lowerbound_xmid_final}, recalling the definition of $\tau_0$, and replacing $\Psg_{n,N_n}$ and $\pi_n'$ by $\Psg_{n,n}$ and $\pi_n$ via \eqref{eq:pi_to_pi'_prob} (possibly increasing $n_0$ so that the latter holds), \eqref{eq:prob_lowerbound_xmid_final} yields \eqref{eq:xmid_anticoncentration}.
\end{proof}

\subsection{North and South} \label{subsec:north_and_south}

Because the North turning point is related to the South turning point by symmetry, we will again just consider the South one. Recall that $\Psi_1$ is the trigamma function.

\begin{prop}[Tail bounds on the $n^{2/3}$ scale]
    \label{thm:sw_midpoint_tails}
    Fix parameters $\alpha,\beta > 0$. For each $n \in \N$ let $\pi_n$ be a path distributed by the random polymer measure $Q^{SW}_{n,n}$ of \Cref{def:stat_sw} with these parameters, and let $x_{mid}(\pi_n)$ be as in \Cref{def:polymer_midpoint}. Then there exist constants $b_0,C$ depending on $\alpha,\beta$ such that for every $b \geq b_0$ and $n \in \N$,
    \begin{equation}
        \P^{SW}_{n,n}\left(\abs*{x_{mid}(\pi_n) - \frac{\Psi_1(\alpha+\beta)}{\Psi_1(\beta)}  \cdot n} \geq bn^{2/3}\right) \leq \frac{C}{b^3}.
    \end{equation}
\end{prop}

\begin{prop}[Anticoncentration on the $n^{2/3}$ scale]
    \label{thm:sw_midpoint_anticoncentration}
    Assume the same setup as \Cref{thm:sw_midpoint_tails}. Then there exist constants $n_0,C_0,C_1 > 0$ depending on $\alpha,\beta$ such that for all $n \geq n_0$,
    \begin{equation}\label{eq:sw_xmid_anticoncentration}
       \P^{SW}_{n,n}\left(\abs*{x_{mid}(\pi_n) - \frac{\Psi_1(\alpha+\beta)}{\Psi_1(\beta)} \cdot n} > C_1 n^{2/3}\right) \geq C_0.
    \end{equation}
\end{prop}

The only needed new input is the analogue of \Cref{thm:loggamma_upper_bound_sepp}, which we give now.

\begin{thm}[Tail bounds and anticoncentration on the $n^{2/3}$ scale]
    \label{thm:sw_bounds_from_cn}
    Recall the setup of \Cref{def:stat_sw} and fix the parameters $\alpha,\beta$. For each $N \in \N$, let 
    \begin{equation}\label{eq:m_n_sw}
        m_N = \floor*{\frac{\Psi_1(\beta)-\Psi_1(\alpha+\beta)}{\Psi_1(\alpha+\beta)}N}
    \end{equation}
    where $\Psi_1$ is the trigamma function. Then there exist positive constants $b_0, C, c_0, c_1, N_0$ depending only on the parameters $\alpha,\beta$ and on $\tau$, such that for all $b \geq b_0$ and $N \in \N$,
\begin{align}
    \P_{m_N,N}^{SW}(v_0(\floor{\tau N}) \leq \tau m_N - b N^{2/3} \text{ or }v_1(\floor{\tau N}) \geq \tau m_N + b N^{2/3}) &\leq \frac{C}{b^3} \\ 
    \P_{m_N,N}^{SW}(w_0(\floor{\tau m_N}) \leq \tau N - b N^{2/3} \text{ or }w_1(\floor{\tau m_N}) \geq \tau N + b N^{2/3}) &\leq \frac{C}{b^3}
\end{align}
    and for all $N \geq N_0$, 
    \begin{equation}
        \P_{m_N,N}^{SW}(v_1(\floor{\tau N}) \geq \tau m_N + c_1 N^{2/3} \text{ or }w_1(\tau m_N) \geq \tau N + c_1 N^{2/3}) \geq c_0.
    \end{equation}
    Furthermore, for any compact set $S \subset [0,1)$, the constants $b_0, C, c_0, c_1, N_0$ may be chosen uniformly over all $\tau \in S$.
\end{thm}
\begin{proof}
    This is a special case of the strict-weak case of \cite[Theorem 1.5]{chaumont2018fluctuation}, where we recall that our $\alpha,\beta$ correspond to $\mu,\theta$ in their notation, and we set their $\beta$ to be $1$. Specifically, one must take the parameters $N,n_N,m_N$ in that result to be $\floor{N/\Psi_1(\beta)}$, $N$ and $\floor{N\Psi_1(\alpha)/\Psi_1(\beta)}$ in terms of our variables. Note that the fact that our $N$ is a constant times the one in \cite[Theorem 1.5]{chaumont2018fluctuation} changes the constants $c_1, b, C$ in that result by a constant, but this does not change the result.

    One must then verify that our $n_N$ and $m_N$ satisfy \cite[(1.10)]{chaumont2018fluctuation}. In our notation, this condition translates to the bounds
    \begin{equation}
        \abs*{\floor*{\frac{\Psi_1(\beta)-\Psi_1(\alpha+\beta)}{\Psi_1(\alpha+\beta)}N} - \floor*{\frac{N}{\Psi_1(\alpha+\beta)}} \cdot \Var(\log R_j)} \leq \gamma N^{2/3} 
    \end{equation}
    and
    \begin{equation}
        \abs*{N - \floor*{\frac{N}{\Psi_1(\alpha+\beta)}} \cdot \Var(\log Y_{i,0})} \leq \gamma N^{2/3} 
    \end{equation}
    holding for all $N$, for some constant $\gamma$. Recalling the distributions of $R_j$ and $Y_{i,0}$ from \Cref{def:stat_sw}, it is well known that
    \begin{align}
        \Var(\log R_j) &= \Psi_1(\beta)-\Psi_1(\alpha+\beta) \\ 
        \Var(\log Y_{i,0}) &= \Psi_1(\alpha+\beta),
    \end{align}
    and this proves the above bounds.
\end{proof}

\begin{proof}[Proof of {\Cref{thm:sw_midpoint_tails}}]
    First assume that $\alpha,\beta$ are such that 
    \begin{equation}\label{eq:sw_nice_case}
        \Psi_1(\beta) - \Psi_1(\alpha+\beta) \geq \Psi_1(\alpha+\beta).
    \end{equation}
    This is the analogue of the condition $\Psi_1(\alpha) \geq \Psi_1(\beta)$ in the proof of \Cref{thm:loggamma_midpoint_tails}. The proof is now exactly the same as that of \Cref{thm:loggamma_midpoint_tails} after replacing $\Psi_1(\alpha)$ and $\Psi_1(\beta)$ (the variances of the logarithms of the boundary weights in the stationary log-Gamma model) by $\Psi_1(\beta) - \Psi_1(\alpha+\beta)$ and $\Psi_1(\alpha+\beta)$ (the variances of the logarithms of the boundary weights in the stationary strict-weak model). The only needed input is \Cref{thm:sw_bounds_from_cn} in place of \Cref{thm:loggamma_upper_bound_sepp}, and the deterministic \Cref{thm:v_w_to_xmid} which is valid for both polymers.

    Unlike the log-Gamma case, where the case $\Psi_1(\alpha) \geq \Psi_1(\beta)$ and its negation could be related by symmetry, the case 
        \begin{equation}\label{eq:sw_bad_case}
        \Psi_1(\beta) - \Psi_1(\alpha+\beta) < \Psi_1(\alpha+\beta).
    \end{equation}
    cannot be dealt with by symmetry, but it can be dealt with by the same argument. Instead of choosing $N_n$ by
    \begin{equation}
        n=\floor*{\frac{\Psi_1(\beta) - \Psi_1(\alpha+\beta)}{\Psi_1(\alpha+\beta)}N_n}
    \end{equation}
    as we do in the case \eqref{eq:sw_nice_case}, we choose $N_n$ to be an integer such that
    \begin{equation}\label{eq:new_N_n}
        n=\floor*{\frac{\Psi_1(\alpha+\beta)}{\Psi_1(\beta) - \Psi_1(\alpha+\beta)}N_n},
    \end{equation}
    and \eqref{eq:sw_bad_case} ensures we may choose $N_n \leq n$. Rather than considering the wide short box $[0,n] \times [n-N_n,n]$ inside $[0,n]^2$ as before, we reflect the picture and consider the tall thin box $[n-N_n,n] \times [0,n] $. The result \Cref{thm:shrink_rectangle} applies mutatis mutandis to the stationary strict-weak polymer, since its proof just uses the Burke property. So the portion of our polymer path inside $[n-N_n,n] \times [0,n] $ is also distributed as a stationary strict-weak polymer.
    
    Now, the ratio of this box's dimensions is still $\frac{\Psi_1(\alpha+\beta)}{\Psi_1(\beta) - \Psi_1(\alpha+\beta)}\cdot (1+o(1))$, which is the right ratio for \Cref{thm:sw_bounds_from_cn} to apply. From there we obtain that the polymer path crosses the line $y=n-x$ roughly at its intersection with the line 
    \begin{equation}
        y-n = \frac{n}{N_n}(x-n),
    \end{equation}
    which is the diagonal of the box $[n-N_n,n] \times [0,n]$, together with appropriate quantitative bounds as before. Solving for the $x$-coordinate of the intersection between these two lines yields 
    \begin{equation}
        x = \frac{\Psi_1(\alpha+\beta)}{\Psi_1(\beta)} \cdot n \cdot (1+o(1))
    \end{equation}
    where the $1+o(1)$ comes from the floor function in \eqref{eq:new_N_n} when we rewrite $N_n$ in terms of $n$. The details of establishing the fluctuation results using \Cref{thm:sw_bounds_from_cn} are the same as the case \eqref{eq:sw_nice_case} after reflecting about the line $x=y$.
\end{proof}

\begin{proof}[Proof of {\Cref{thm:sw_midpoint_anticoncentration}}]
    In the case \eqref{eq:sw_nice_case} the proof is identical to that of \Cref{thm:loggamma_midpoint_anticoncentration}, again after substituting $\Psi_1(\alpha)$ and $\Psi_1(\beta)$ in that argument for $\Psi_1(\beta) - \Psi_1(\alpha+\beta)$ and $\Psi_1(\alpha+\beta)$. In the case \eqref{eq:sw_bad_case} we define $N_n$ by \eqref{eq:new_N_n} and make the same argument as before but reflected about the line $y=x$.
\end{proof}

\subsection{Proof of {\Cref{thm:turning_points_intro}}}

\begin{proof}
    For the West turning point, the result follows by matching the turning point with the stationary log-Gamma polymer midpoint by \Cref{thm:west_matching_intro}, and then applying  \Cref{thm:loggamma_midpoint_tails} and \Cref{thm:loggamma_midpoint_anticoncentration}. For the South turning point, we similarly match with the strict-weak polymer by \Cref{thm:south_endpoint} and then apply \Cref{thm:sw_midpoint_tails} and \Cref{thm:sw_midpoint_anticoncentration}. The North turning point is the same after reflecting the Aztec diamond about the horizontal axis, which interchanges $\alpha$ and $\beta$. In principle we have different constants for the bounds for each turning point, but by taking the worst case of the three we have the uniform constants in \Cref{thm:turning_points_intro}.
\end{proof}

\section{Asymptotics of the free energy}\label{sec:free_energy}

Our next goal is to understand the asymptotic behavior of the free energy.

\subsection{Results}

Recall from the Introduction that the  free energy is defined as
\begin{equation}
    F_n = T \log Z_n.
\end{equation}
The free energy is a central quantity in statistical mechanics, though it is often challenging to compute explicitly. In our model, however, the partition function is given as a product of independent random variables (see \Cref{thm:compute_Z_cor}), which allows us to employ standard probabilistic tools. Below we present our main results; the derivations are provided afterward.

We begin by computing the averaged quenched and annealed free energies. The  \emph{anneaeled} free energy is
\[
F_n^a = T \log \mathbb{E}[Z_n].
\]
By Jensen's inequality, we always have \( F_n^a \geq \mathbb E F_n \). Since \( Z_n \) is a product of independent random variables, both quantities can be computed explicitly.

\begin{thm}\label{prop:annealed_quenched_free_energy}
Consider the averaged quenched and annealed free energies, $\mathbb E F_n$  and $F_n^a$ for the Gamma-disordered  Aztec diamond as defined in \Cref{def:gamma_weights_intro_general}. We have
\begin{equation}\label{eq:annealed_free_energy_Gammas}
    F_n^a = T \sum_{k=1}^n \sum_{j=1}^k \log(\psi_j + \phi_{j-k}),
\end{equation}
and
\begin{equation}\label{eq:quenched_free_energy_Gammas}
    \mathbb E F_n = T \sum_{k=1}^n \sum_{j=1}^k \Psi_0(\psi_j + \phi_{j-k}),
\end{equation}
where \( \Psi_0(x) \) denotes the digamma function.
\end{thm}

The proof is given in Section~\ref{sec:proof_free_energy} below.

To compare \( F_n^a \) and \( \mathbb E F_n \), we use the following classical inequality for the digamma function:
\[
\frac{1}{2x} < \log x - \Psi_0(x) < \frac{1}{x}, \quad x > 0.
\]
This shows that when all parameters \( \psi_j + \phi_{j-k} \) are large, the annealed and quenched free energies are close. However, when these parameters are close to zero, the difference between \( F_n^a \) and \( \mathbb EF_n  \) becomes significant. To track this in a uniform way we scale $\psi_j$ and $\phi_{j-n}$ parameters by \( T \): 
\[
\psi_j= \bar\psi_jT, \qquad  \phi_{j-n}=\bar \phi_{j-n} T.
\]

With this scaling, we obtain the following result.

\begin{cor}\label{cor:boundsenergy}
Let \( n \in \mathbb{N} \), and \( \bar \psi_j, \bar \phi_j, T > 0 \).  Consider the averaged quenched and annealed free energies, $\mathbb E F_n$ and $F_n^a$ for the Gamma-disordered  Aztec diamond as defined in \Cref{def:gamma_weights_intro_general} with $\psi_j=T \bar \psi_j$ and $\phi_{j-n}=T \bar \phi_{j-n}.$ 
Then
\[
\frac{1}{2n^2} \sum_{k=1}^n \sum_{j=1}^k \frac{1}{\bar \psi_j + \bar \phi_{j-k}} 
<\frac{1}{n^2}\left( F_n^a - \mathbb E [F_n] \right)< 
\frac{1}{n^2} \sum_{k=1}^n \sum_{j=1}^k \frac{1}{\bar \psi_j + \bar \phi_{j-k}}.
\]
\end{cor}

We conclude with a result on the fluctuations of the free energy.

\begin{thm} \label{prop:CLT_free_energy}
In the same setting as \Cref{cor:boundsenergy}, as \( n \to \infty \) we have
\[
\mathbb{P}\left(\frac{F_n - \mathbb E[ F_n]}{\sqrt{\sum_{k=1}^n \sum_{j=1}^k T^2 \Psi_1(T(\bar \psi_j + \bar \phi_{j-k}))}} < s \right)
= \frac{1}{\sqrt{2\pi}} \int_{-\infty}^s e^{-x^2/2} \, dx + \mathcal{O}(n^{-1}),
\]
where the error term is uniform for \( T \in (0,\infty) \), and for all sequences \( \{\bar \psi_j\}, \{\bar \phi_k\} \) contained in compact subsets of \( (0,\infty) \). Here, \( \Psi_1(x) =  \Psi_0'(x) \) denotes the trigamma function.
\end{thm}

The uniformity of the error term with respect to \( T \in (0,\infty) \) implies, in particular, that \( T \) may vary arbitrarily with \( n \) without affecting the nature of the fluctuations. Consequently, the fluctuations remain Gaussian in all regimes; there is no scaling limit of \( T \) with \( n \) that leads to non-Gaussian behavior. 

Finally, we recall that the trigamma function \( \Psi_1 \) is strictly positive and satisfies the bounds
\begin{equation}
    \label{eq:trigamma_bounds}
    \frac{1}{x} + \frac{1}{2x^2} \leq \Psi_1(x) \leq \frac{1}{x} + \frac{1}{x^2}, \quad x > 0.
\end{equation}
These estimates imply that the variance of \( F_n \) behaves asymptotically as \( \sim n^2 \) for small \( T > 0 \), and as \( \sim n^2 T \) when both \( n \) and \( T \) are large.

\subsection{Preliminaries on Gamma random variables and their logarithms}
We begin with some basic observations about logarithms of Gamma-distributed random variables. First, note that if \( X \sim  \Gamma(\xi, 1) \) , then
\begin{equation} \label{eq:mean_gamma}
\mathbb{E}[X] = \xi.
\end{equation}
Next we recall that the characteristic function \( \mathbb{E}[X^{i t}] \) of \( \log X \) is given by
\[
\frac{\Gamma(\xi + i t)}{\Gamma(\xi)},
\]
and hence the cumulants of \( \log X \) are
\begin{equation} \label{eq:cumulants_loggamma}
\kappa_k = \Psi_{k-1}(\xi), \quad k \geq 2,
\end{equation}
where \( \Psi_k = \frac{d^{k+1}}{dx^{k+1}} \log \Gamma(x) \) denotes the polygamma function. Moreover, the following asymptotic expansions hold \cite[Chapter 5]{DLMF}: 
\begin{equation} \label{eq:asymptotics_Psik_near_0}
\Psi_k(\xi) = \frac{(-1)^{k+1}k!}{\xi^{k+1}} + \mathcal{O}(1), \quad \xi \downarrow 0,
\end{equation}
and
\begin{equation} \label{eq:asymptotics_Psik_near_infty}
\Psi_k(\xi) = \frac{(-1)^{k+1}k!}{\xi^k} + \mathcal{O}(\xi^{-k-1}), \quad \xi \to \infty.
\end{equation}

These expressions show that the variance of \( \log X \) is of order \( \sim 1/\xi^2 \) for small \( \xi \), and \( \sim 1/\xi \) for large \( \xi \). After appropriate rescaling, we obtain the following. We do not actually use it in this paper (for the desired Gaussianity we make a stronger argument which involves the rate of convergence), but it is worth recording.

\begin{prop}
Let \( X \sim \Gamma(\xi,1) \). Then, as \( \xi \to \infty \),
\begin{equation} \label{eq:loggamma_to_normal}
\frac{\log X - \Psi_0(\xi)}{\sqrt{\Psi_1(\xi)}} \overset{d}{\longrightarrow} Z \sim \mathcal{N}(0,1),
\end{equation}
and as \( \xi \downarrow 0 \),
\begin{equation} \label{eq:loggamma_to_exp}
\frac{\log X - \Psi_0(\xi)}{\sqrt{\Psi_1(\xi)}} - 1 \overset{d}{\longrightarrow} Z \sim -\mathrm{Exp}(1),
\end{equation}
in distribution. 
\end{prop}

\begin{proof}
The cumulants \( \tilde{\kappa}_k \) of
\[
\frac{\log X - \Psi_0(\xi)}{\sqrt{\Psi_1(\xi)}}
\]
are given by
\[
\tilde{\kappa}_1 = 0, \qquad
\tilde{\kappa}_2 = 1, \qquad
\tilde{\kappa}_k = \frac{\Psi_{k-1}(\xi)}{(\Psi_1(\xi))^{k/2}} \quad \text{for } k \geq 3.
\]

From \eqref{eq:asymptotics_Psik_near_infty}, it follows that \( \tilde{\kappa}_k \to 0 \) as \( \xi \to \infty \) for all \( k \geq 3 \). Hence, the rescaled logarithm of a Gamma distribution converges in distribution to a standard normal distribution, proving \eqref{eq:loggamma_to_normal}.

Similarly, from \eqref{eq:asymptotics_Psik_near_0}, we find that as \( \xi \downarrow 0 \), the cumulants satisfy
\[
\tilde{\kappa}_k \to (-1)^k (k-1)! \quad \text{for } k \geq 2,
\]
which are the cumulants of the negative of an exponential distribution with parameter 1. Therefore, after shifting the mean appropriately, we obtain convergence to \( -\mathrm{Exp}(1) \), establishing~\eqref{eq:loggamma_to_exp}.
\end{proof}
\subsection{Proofs of Theorem \ref{prop:annealed_quenched_free_energy} and \ref{prop:CLT_free_energy}} \label{sec:proof_free_energy}
We recall that by Corollary \ref{thm:compute_Z_cor} we have
$$
        Z_n = \prod_{k=1}^n \prod_{j=1}^{k}X_{j,k},
$$
where $X_{j,k}$, for $j=1,\ldots,k$ and $k=1, \ldots ,n$ are independent random variables with distribution 
$$X_{j,k} \sim \Gamma(\psi_j+\phi_{j-k},1).$$
The independence makes  it easy to compute the the expectation and fluctuations of the free energy. Let us start with the mean and then compute the fluctuations.

\begin{proof}[Proof of Theorem \ref{prop:annealed_quenched_free_energy}]
Using the independence of the random variables and equation~\eqref{eq:mean_gamma}, we obtain
\[
\mathbb{E}[Z_n] = \prod_{k=1}^n \prod_{j=1}^{k} (\psi_j + \phi_{j-k}).
\]
Taking the logarithm yields equation~\eqref{eq:annealed_free_energy_Gammas}.

To compute the averaged quenched free energy, we first take the logarithm and then the expectation. Applying equation~\eqref{eq:cumulants_loggamma}, we find
\[
\mathbb{E}[\log Z_n] = \sum_{k=1}^n \sum_{j=1}^k \Psi_0(\psi_j + \phi_{j-k}),
\]
which establishes equation~\eqref{eq:quenched_free_energy_Gammas}.
\end{proof}

\begin{proof}[Proof of Theorem \ref{prop:CLT_free_energy}]
It is clear from the independence of the random variables that a Central Limit Theorem holds for any fixed choice of parameters. However, to obtain a uniform convergence, it is useful to apply the Berry–Esseen theorem. Recall that this theorem states that for independent (but not necessarily identically distributed) random variables $X_j$, for $j=1,\ldots,N$, we have
\[
\left|\mathbb{P}\left(\frac{\sum_{j=1}^N (X_j - \mathbb{E}X_j)}{\left(\sum_{j=1}^N \sigma_j^2\right)^{1/2}} < s\right) - \frac{1}{\sqrt{2\pi}} \int_{-\infty}^s e^{-x^2/2} \, dx \right|
\leq C \frac{\sum_{j=1}^N \mathbb{E}[|X_j - \mathbb{E}X_j|^3]}{\left(\sum_{j=1}^N \sigma_j^2\right)^{3/2}},
\]
where $C$ is a universal constant and $\sigma_j^2 = \operatorname{Var}(X_j)$. 

Using the Cauchy–Schwarz inequality, we have
\[
\mathbb{E}[|X_j - \mathbb{E}X_j|^3] \leq \sqrt{\mathbb{E}[|X_j - \mathbb{E}X_j|^2]} \cdot \sqrt{\mathbb{E}[|X_j - \mathbb{E}X_j|^4]}.
\]
Applying Cauchy–Schwarz again to the sum, we obtain
\[
\left|\mathbb{P}\left(\frac{\sum_{j=1}^N (X_j - \mathbb{E}X_j)}{\left(\sum_{j=1}^N \sigma_j^2\right)^{1/2}} < s\right) - \frac{1}{\sqrt{2\pi}} \int_{-\infty}^s e^{-x^2/2} \, dx \right|
\leq C \frac{\left(\sum_{j=1}^N \mathbb{E}[|X_j - \mathbb{E}X_j|^4]\right)^{1/2}}{\sum_{j=1}^N \sigma_j^2}.
\]

Next, observe that
\[
\mathbb{E}[|X_j - \mathbb{E}X_j|^4] = \kappa_{4,j} + 3\kappa_{2,j}^2,
\]
where $\kappa_{k,j}$ denotes the $k$th cumulant of $X_j$. This is particularly helpful in our setting, as the cumulants are expressed in terms of polygamma functions (see equation~\eqref{eq:cumulants_loggamma}).

Applying this to our case (include the factor $T$) yields:
\begin{multline*}
\left|\mathbb{P}\left(\frac{F_n - \mathbb E F_n}{\sqrt{\sum_{k=1}^n \sum_{j=1}^k T^2 \Psi_1(T(\bar \psi_j + \bar \phi_{j-k}))}} < s\right) - \frac{1}{\sqrt{2\pi}} \int_{-\infty}^s e^{-x^2/2} \, dx \right| \\
\leq C \frac{\left(\sum_{k=1}^n \sum_{j=1}^k T^4 \Psi_3(T(\bar \psi_j + \bar \phi_{j-k})) + 3 T^2 \Psi_1(T(\bar \psi_j + \bar \phi_{j-k}))^2 \right)^{1/2}}{\sum_{k=1}^n \sum_{j=1}^k T^2 \Psi_1(T(\bar \psi_j + \bar \phi_{j-k}))}.
\end{multline*}

We have also used equation~\eqref{eq:cumulants_loggamma} to compute both second and fourth moments (i.e., cumulants). We can bound the expression as follows:
\begin{multline*}
\left|\mathbb{P}\left(\frac{F_n - \mathbb E F_n}{\sqrt{\sum_{k=1}^n \sum_{j=1}^k T^2\Psi_1(T(\bar \psi_j + \bar \phi_{j-k}))}} < s\right) - \frac{1}{\sqrt{2\pi}} \int_{-\infty}^s e^{-x^2/2} \, dx \right| \\
\leq \frac{C\sqrt{2}}{\sqrt{n(n+1)}} \cdot \frac{\max_{j,k} \sqrt{T^4\Psi_3(T(\bar \psi_j + \bar \phi_{j-k})) + 3T^2 \Psi_1(T(\bar \psi_j + \bar \phi_{j-k}))^2}}{\min_{j,k} T^2 \Psi_1(T(\bar \psi_j + \bar \phi_{j-k}))}.
\end{multline*}
Now note that from \eqref{eq:asymptotics_Psik_near_0} we find
$$T^4\Psi_3(T(\bar \psi_j+\bar \phi_{j-k}))+3T^2\Psi_1(T(\bar \psi_j + \bar \phi_{j-k}))^2 \sim 1,$$
and 
$$
T^2\Psi_1(T(\bar \psi_j + \bar \phi_{j-k})\sim 1,$$
both as $T \downarrow 0.$ Hence 
\begin{equation}\label{eq:psisatT0}
\frac{\max_{j,k} \sqrt{T^4\Psi_3(T(\bar \psi_j + \bar \phi_{j-k})) + 3T^2\Psi_1(T(\bar \psi_j + \bar \phi_{j-k}))^2}}{\min_{j,k} T^2 \Psi_1(T(\bar \psi_j + \bar \phi_{j-k}))}\sim 1,
\end{equation}
as $T\downarrow 0$.

Similarly, from \eqref{eq:asymptotics_Psik_near_infty} we find 
$$T^4\Psi_3(T(\bar \psi_j+\bar \phi_{j-k}))+3T^2\Psi_1(T(\bar \psi_j + \bar \phi_{j-k}))^2 \sim T,$$
and 
$$
T^2 \Psi_1(T(\bar \psi_j + \bar \phi_{j-k})\sim T,$$
both as $T\to \infty$. Hence,
\begin{equation}\label{eq:psisatTinfty}
\frac{\max_{j,k} \sqrt{T^4\Psi_3(T(\bar \psi_j + \bar \phi_{j-k})) + 3T^2\Psi_1(T(\bar \psi_j + \bar \phi_{j-k}))^2}}{\min_{j,k} T^2 \Psi_1(T(\bar \psi_j + \bar \phi_{j-k}))}\sim 1,
\end{equation}
as $T\to \infty$. Finally, by continuity, the fact that $\Psi_1(x)>0$ for $x>0$, \eqref{eq:psisatT0} and \eqref{eq:psisatTinfty} we find that there exist constants $c_1, c_2 > 0$ such that
\[
c_1 < \frac{\max_{j,k} \sqrt{T^4\Psi_3(T(\bar \psi_j + \bar \phi_{j-k})) + 3T^2\Psi_1(T(\bar \psi_j + \bar \phi_{j-k}))^2}}{\min_{j,k} T^2 \Psi_1(T(\bar \psi_j + \bar \phi_{j-k}))} < c_2,
\]
for all $T \in (0, \infty)$. This completes the proof.
\end{proof}
\begin{proof}[Proof of Theorem \ref{thm:free_energy_difference_intro}]
  The theorem  follows directly from Theorem \ref{prop:annealed_quenched_free_energy} and Corollary \ref{cor:boundsenergy} by setting $\bar \psi_j=\bar \alpha$ and $\bar \phi_{j-n}= \bar \beta$. 
\end{proof}

\appendix
\section{LGV paths}\label{appendix:lgv}

An important strategy for studying the Aztec diamond dimer model is by using its bijection to non-intersecting paths on a directed graph, and utilizing a theorem due to Lindstr\"om-Gessel-Viennot \cite{GesselViennot1985,Lindstrom1973} and Eynard-Mehta \cite{EynardMehta1998} to study the correlation functions. In fact, it was shown recently \cite{chhita2023domino} that the shuffling algorithm is related to commutation rules for the transition matrices for the non-intersecting paths. The reader familiar with \cite{chhita2023domino} may have noticed the resemblance between switching $(+)$- and $(-)$-columns, as done in \Cref{sec:vert_polymer}, and these commutation rules. The purpose of this appendix is to outline how these points of view are indeed equivalent (as one expects). 

\subsection{Non-intersecting paths}
We start by introducing the model of non-intersecting paths on a directed graph, beginning with defining the latter. 

The graph \( G_{LGV}^n = (V_{LGV}^n, E_{LGV}^n) \) is defined as follows. The vertex set is
\[
V_{LGV}^n = \left\{ (i, j) \mid i \in \{0, 1, \ldots, 2n\},\ j \in \mathbb{Z} \right\}.
\]

The directed edges are given by the following rules, for all \( i = 0, 1, \ldots, n-1 \) and all \( j \in \mathbb{Z} \):
\begin{itemize}
  \item From each vertex \( (2i, j) \), there are directed edges to \( (2i+1, j) \) and \( (2i+1, j-1) \).
  \item From each vertex \( (2i+1, j) \), there are directed edges to \( (2i+2, j) \) and \( (2i+1, j-1) \).
\end{itemize}

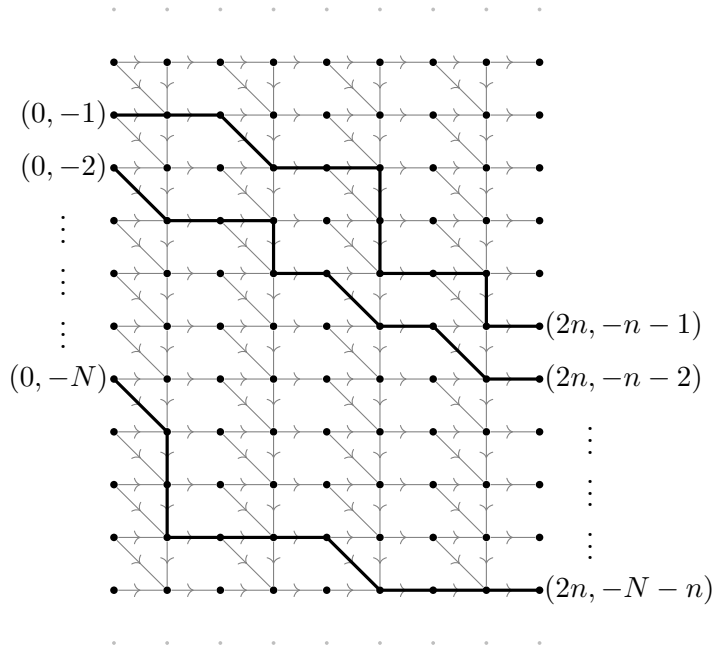
\begin{figure}
\begin{center}
\begin{tikzpicture}[
    scale=0.7,
    every node/.style={circle, fill=black, inner sep=1pt},
    decoration={markings, mark=at position 0.5 with {\arrow{>}}},
    arrowedge/.style={postaction={decorate}, gray},
    infdot/.style={circle, draw=none, fill=gray!50, inner sep=0.5pt}
  ]

\def\n{4}        
\def\jmin{-6}    
\def\jmax{4}     

\foreach \i in {0,...,8} {
  \foreach \j in {\jmin,...,\jmax} {
    \node (v\i\j) at (\i, \j) {};
  }
}

\foreach \k in {0,...,3} {
  \pgfmathtruncatemacro{\even}{2*\k}
  \pgfmathtruncatemacro{\odd}{2*\k+1}
  \pgfmathtruncatemacro{\evenNext}{2*\k+2}
  
  \foreach \j in {\jmin,...,\jmax} {
    \pgfmathtruncatemacro{\jp}{\j+1}
    \pgfmathtruncatemacro{\jm}{\j-1}

    \draw[arrowedge] (v\even\j) -- (v\odd\j);
    \ifnum\j>\jmin
      \draw[arrowedge] (v\even\j) -- (v\odd\jm);
    \fi

    \ifnum\evenNext<9
      \draw[arrowedge] (v\odd\j) -- (v\evenNext\j);
    \fi
    \ifnum\j>\jmin
      \draw[arrowedge] (v\odd\j) -- (v\odd\jm);
    \fi
  }
}

\foreach \i in {0,...,8} {
  \node[infdot] at (\i, \jmax+1.0) {};
  \node[infdot] at (\i, \jmin-1.0) {};

}

\draw (0,3) node[fill=none,anchor=east,inner sep=1pt] {$(0,-1)$};
\draw (8,-1) node[fill=none,anchor=west,inner sep=1pt] {$(2n,-n-1)$};
\draw (0,2) node[fill=none,anchor=east,inner sep=1pt] {$(0,-2)$};
\draw (8,-2) node[fill=none,anchor=west,inner sep=1pt] {$(2n,-n-2)$};
\draw (0,1) node[fill=none,anchor=east,inner sep=8pt] {$\vdots$};
\draw (8,-3) node[fill=none,anchor=west,inner sep=8pt] {$\vdots$};
\draw (0,0) node[fill=none,anchor=east,inner sep=8pt] {$\vdots$};
\draw (8,-4) node[fill=none,anchor=west,inner sep=8pt] {$\vdots$};
\draw (0,-1) node[fill=none,anchor=east,inner sep=8pt] {$\vdots$};
\draw (8,-5) node[fill=none,anchor=west,inner sep=8pt] {$\vdots$};
\draw (0,-2) node[fill=none,anchor=east,inner sep=1pt] {$(0,-N)$};
\draw (8,-6) node[fill=none,anchor=west,inner sep=1pt] {$(2n,-N-n)$};

\draw[very thick] (0,3)--(1,3)--(2,3)--(3,2)--(5,2)--(5,0)--(7,0)--(7,-1)--(8,-1);
\draw[very thick] (0,2)--(1,1)--(2,1)--(3,1)--(3,0)--(4,0)--(5,-1)--(6,-1)--(7,-2)--(8,-2);
\draw[very thick] (0,-2)--(1,-3)--(1,-4)--(1,-5)--(4,-5)--(5,-6)--(8,-6);
\end{tikzpicture}
\caption{Illustration of the non-intersecting paths on the graph $G_{LGV}^n$ (here $n=4$). }
\end{center} 
\end{figure}
The graph can be interpreted as an alternating arrangement of two types of infinite vertical columns. The type I columns connect vertices to their two right neighbors, one horizontally across and one diagonally down.  
The type II columns send arrows horizontally to the neighbors on the right, and also downward to the neighbors directly below.

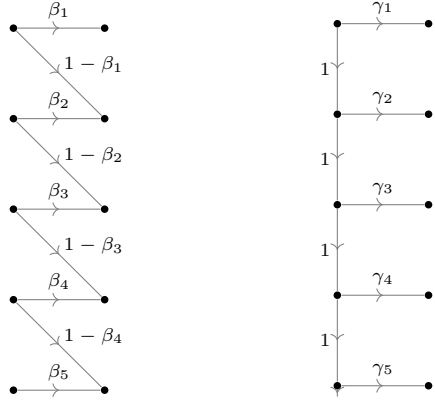
\begin{figure}[t]
\begin{center}
  
\begin{tikzpicture}[
    scale=1.2,
    every node/.style={circle, fill=black, inner sep=1pt},
    decoration={markings, mark=at position 0.5 with {\arrow{>}}},
    arrowedge/.style={postaction={decorate}, gray}
  ]

\foreach \i in {0,1} {
  \foreach \j in {-1,0,1,2,3} {
    \node (v\i\j) at (\i, \j) {};
  }
}

\foreach \j in {-1,0,1,2,3} {
  \draw[arrowedge] (v0\j) -- (v1\j);
  \ifnum\j>-1
    \pgfmathtruncatemacro{\jp}{\j-1}
    \draw[arrowedge] (v0\j) -- (v1\jp);
  \fi
}

\draw (0.5,3) node[fill=none,anchor=south,inner sep=1pt] {\tiny $\beta_1$};
\draw (0.5,2.6) node[fill=none,anchor=west,inner sep=1pt] {\tiny $1-\beta_1$};
\draw (0.5,2) node[fill=none,anchor=south,inner sep=1pt] {\tiny $\beta_2$};
\draw (0.5,1.6) node[fill=none,anchor=west,inner sep=1pt] {\tiny $1-\beta_2$};
\draw (0.5,1) node[fill=none,anchor=south,inner sep=1pt] {\tiny $\beta_3$};
\draw (0.5,0.6) node[fill=none,anchor=west,inner sep=1pt] {\tiny $1-\beta_3$};
\draw (0.5,0) node[fill=none,anchor=south,inner sep=1pt] {\tiny $\beta_4$};
\draw (0.5,-.4) node[fill=none,anchor=west,inner sep=1pt] {\tiny $1-\beta_4$};
\draw (0.5,-1) node[fill=none,anchor=south,inner sep=1pt] {\tiny $\beta_5$};

\end{tikzpicture}
\qquad  \qquad \qquad 
\begin{tikzpicture}[
    scale=1.2,
    every node/.style={circle, fill=black, inner sep=1pt},
    decoration={markings, mark=at position 0.5 with {\arrow{>}}},
    arrowedge/.style={postaction={decorate}, gray}
  ]

\foreach \i in {0,1} {
  \foreach \j in {-1,0,1,2,3} {
    \node (v\i\j) at (\i, \j) {};
  }
}

\foreach \j in {-1,0,1,2,3} {
  \draw[arrowedge] (v0\j) -- (v1\j);
  \pgfmathtruncatemacro{\jp}{\j-1}
    \draw[arrowedge] (v0\j) -- (v0\jp);
}
\draw (0.5,3) node[fill=none,anchor=south,inner sep=1pt] {\tiny $\gamma_1$}; 
\draw (0,2.5) node[fill=none,anchor=east,inner sep=1pt] {\tiny $1$};
\draw (0.5,2) node[fill=none,anchor=south,inner sep=1pt] {\tiny $\gamma_2$};
\draw (0,1.5) node[fill=none,anchor=east,inner sep=1pt] {\tiny $1$};
\draw (0.5,1) node[fill=none,anchor=south,inner sep=1pt] {\tiny $\gamma_3$};
\draw (0,0.5) node[fill=none,anchor=east,inner sep=1pt] {\tiny $1$};
\draw (0.5,0) node[fill=none,anchor=south,inner sep=1pt] {\tiny $\gamma_4$};
\draw (0,-.5) node[fill=none,anchor=east,inner sep=1pt] {\tiny $1$};
\draw (0.5,-1) node[fill=none,anchor=south,inner sep=1pt]{\tiny $\gamma_5$};
\end{tikzpicture}
\caption{Part of a type I column (left) and part of a type II column (right).}
\end{center}
\end{figure}
For $N \geq n$ we then consider collections of paths $\{\pi_j\}_{j=1}^N$ on this directed graph such that 
\begin{itemize}
    \item Each path $\pi_j$ starts at $(0,-j)$ and ends at $(2n,-j-n)$.
    \item The paths are mutually disjoint, i.e. they have no vertices in common. 
\end{itemize}
The last  is the non-intersecting condition.

The connection between the matchings of the Aztec diamond and these non-intersecting paths has been discussed at length in various places in the literature. It is important to note that, despite the fact that all are based on the same principle, there are different conventions in the literature. We will briefly discuss here the viewpoint of \cite{chhita2023domino,BerggrenDuits2019}, and refer to these papers for further details.

Instead of considering just one Aztec diamond of size $n$, we will consider a region that consists of gluing three individual parts as shown in Figure \ref{fig:tower_graph_part_I}:  two Aztec diamond graphs, one of size $n$ and one of size $n-1$,  with a simple corridor of vertical length $N-n$ in between. We place the graph such that the top-left black vertex has coordinate $(\frac14,-\frac 54)$ and the top-left white vertex has coordinate $(\frac34,-\frac34)$. 

We claim that there is a bijection between all matching of this larger graph and non-intersecting paths on the graph $G_{LGV}^n$ described above. This is done by drawing red path on the Northwest, Southeast and Southwest dominoes  as indicated in Figure \ref{fig:tower:graph_part_II}. By removing the dominoes and augmenting the paths with trivial horizontal parts as indicated in the picture on the right, we obtain precisely a collection of non-intersecting paths on $G_{LGV}^n$ with the starting and endpoint described above. 

An important feature of the larger (slightly odd-looking) bipartite graph is that a matching cannot have an edge that has one vertex in a corridor and one in one of the Aztec diamonds. More precisely, each matching decouples into two matchings of the individual Aztec diamond graphs and a frozen configuration in the corridor in between. 

For the paths this means  that each path \(\pi_j\) has to pass through the point 
\[
  (2n - 2j + 1,\,-n - 1)\,,\qquad j = 1,\dots,n.
\]
Moreover, there is an explicit bijection between the portions of the paths\footnote{This portion by itself are referred to in the literature as the DR paths \cite{johansson2002non}. They start from one side of the Aztec, end at a neighboring side and never intersect.} lying above the line
\[
  y = -n - 1
\]
and dimer configurations of the Aztec diamond of size \(n\). That the matching is frozen in the corridor implies that  segments of the paths between the horizontal lines
\[
  y = -n - 1
  \quad\text{and}\quad
  y = -N  - 1
\]
are deterministic zig–zag paths.

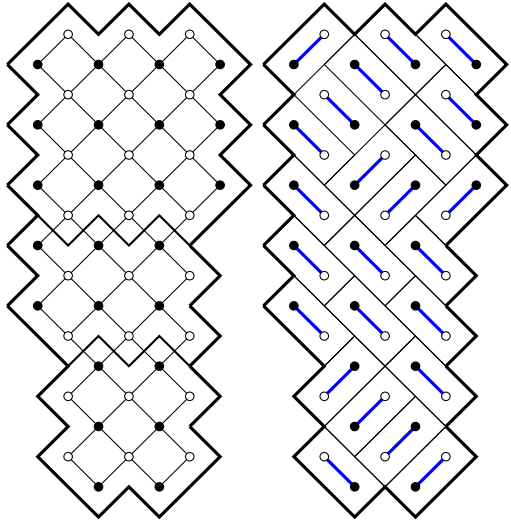
\begin{figure}[t]
  \begin{tikzpicture}[scale=.4]
  \foreach \i in {0,...,3} {
    \draw[very thick] (0,{-2*\i}) -- (1,{-2*\i-1}) -- (0,{-2*\i-2});
  }
   \foreach \i in {0,...,3} {
    \draw[very thick] (6,{-6-2*\i}) -- (7,{-7-2*\i}) -- (6,{-8-2*\i});
  }
  \draw [very thick] (0,0)--(2,2)--(3,1)--(4,2)--(5,1)--(6,2)--(8,0)--(7,-1)--(8,-2)--(7,-3)--(8,-4)--(6,-6);
  \draw [very thick] (0,-8)--(2,-10)--(1,-11)--(2,-12)--(1,-13)--(3,-15)--(4,-14)--(5,-15)--(6,-14);
  
  \draw[thick](1,-5)--(2,-6)--(3,-5)--(4,-6)--(5,-5)--(6,-6);
   \draw[thick] (2,-10)--(3,-9)--(4,-10)--(5,-9)--(6,-10);

\draw (1,0) -- (6,-5);
\draw (4,1) -- (7,-2);
\draw (6,1) -- (7,0);
\draw (2,1) -- (7,-4);
\draw (1,-2) -- (6,-7);
\draw (1,-4) -- (6,-9);
\draw (1,-6) -- (6,-11);
\draw (1,-8) -- (6,-13);
\draw (2,-11) -- (5,-14);

\draw (1,0) -- (2,1);
\draw (1,-2) -- (4,1);
\draw (1,-4) -- (6,1);
\draw (1,-6) -- (7,0);
\draw (1,-8) -- (7,-2);
\draw (2,-9) -- (7,-4);
\draw (2,-11) -- (6,-7);
\draw (2,-13) -- (6,-9);
\draw (3,-14) -- (6,-11);
\draw (5,-14) -- (6,-13);
\draw (2,-13) -- (3,-14);
 
  \foreach \i in {0,1,2} {
   \foreach \j in {0,...,7} {
      \draw[fill=white] ({2+2*\i},{1-2*\j}) circle(4pt);
      
   }
}
\foreach \i in {0,1} {
   \foreach \j in {0,...,7} {
      \filldraw ({3+2*\i},{-2*\j}) circle(4pt);
   }
}
\foreach \j in {0,...,4} {
      \filldraw ({1},{-2*\j}) circle(4pt);
   }
   \foreach \j in {0,...,2} {
      \filldraw ({7},{-2*\j}) circle(4pt);
   }
\end{tikzpicture}
  \begin{tikzpicture}[scale=.4]
  \foreach \i in {0,...,3} {
    \draw[very thick] (0,{-2*\i}) -- (1,{-2*\i-1}) -- (0,{-2*\i-2});
  }
   \foreach \i in {0,...,3} {
    \draw[very thick] (6,{-6-2*\i}) -- (7,{-7-2*\i}) -- (6,{-8-2*\i});
  }
  \draw [very thick] (0,0)--(2,2)--(3,1)--(4,2)--(5,1)--(6,2)--(8,0)--(7,-1)--(8,-2)--(7,-3)--(8,-4)--(6,-6);
  \draw [very thick] (0,-8)--(2,-10)--(1,-11)--(2,-12)--(1,-13)--(3,-15)--(4,-14)--(5,-15)--(6,-14);
  
  \draw [help lines] (1,-5)--(2,-6)--(3,-5)--(4,-6)--(5,-5)--(6,-6);
   \draw [help lines] (2,-10)--(3,-9)--(4,-10)--(5,-9)--(6,-10);

   \draw (0,0) \eastdomino;   
    \draw (0,-2) \southdomino ; 
     \draw (0,-4) \southdomino;  
\draw (0,-6) \southdomino;  
\draw (0,-8) \southdomino;  

\draw (2,-6) \southdomino; 
\draw (2,-8) \southdomino;  
\draw (4,-6) \southdomino;  
\draw (4,-8) \southdomino;

\draw (5,1) \northdomino; 
\draw (3,1) \northdomino; 
\draw (5,-1) \northdomino; 

\draw (2,-4) \eastdomino;  
\draw (3,-5) \westdomino;  
\draw (5,-5) \westdomino;  
\draw (4,-2) \southdomino;  
\draw (2,0) \southdomino; 

\draw (1,-13) \northdomino;
\draw (1,-11) \westdomino;
\draw (2,-12) \eastdomino;
\draw (3,-13) \westdomino;
\draw (4,-14) \eastdomino;
\draw (4,-10) \southdomino;

 \draw (0,0) \eastdimer;   
    \draw (0,-2) \southdimer ; \draw (1,-1) \northdimer; 
     \draw (0,-4) \southdimer;  
\draw (0,-6) \southdimer;  
\draw (0,-8) \southdimer;  

\draw (2,-6) \southdimer; 
\draw (2,-8) \southdimer;  
\draw (4,-6) \southdimer;  
\draw (4,-8) \southdimer;

\draw (5,1) \northdimer; 
\draw (3,1) \northdimer; 
\draw (5,-1) \northdimer; 

\draw (2,-4) \eastdimer;  
\draw (3,-5) \westdimer;  
\draw (5,-5) \westdimer;  
\draw (4,-2) \southdimer;  
\draw (2,0) \southdimer; 

\draw (1,-13) \northdimer;
\draw (1,-11) \westdimer;
\draw (2,-12) \eastdimer;
\draw (3,-13) \westdimer;
\draw (4,-14) \eastdimer;
\draw (4,-10) \southdimer;

  \foreach \i in {0,1,2} {
   \foreach \j in {0,...,7} {
      \draw[fill=white] ({2+2*\i},{1-2*\j}) circle(4pt);
      
   }
}
\foreach \i in {0,1} {
   \foreach \j in {0,...,7} {
      \draw[fill=black] ({3+2*\i},{-2*\j}) circle(4pt);
   }
}
\foreach \j in {0,...,4} {
      \draw[fill=black] ({1},{-2*\j}) circle(4pt);
   }
   \foreach \j in {0,...,2} {
     \draw[fill=black] ({7},{-2*\j}) circle(4pt);
   }

\end{tikzpicture}
   
\caption{The leftmost picture is the bipartite graph that consists of two Aztec diamond graphs, one of size 3 and another of size 2, and a corridor of vertical length 2 between them. Every dimer configuration necessarily decouples into a dimer cover of the three separate regions.   } \label{fig:tower_graph_part_I}
\end{figure}
\qquad \qquad 
\begin{figure}[t]
  
    \begin{tikzpicture}[scale=.4]
  \foreach \i in {0,...,3} {
    \draw[very thick] (0,{-2*\i}) -- (1,{-2*\i-1}) -- (0,{-2*\i-2});
  }
   \foreach \i in {0,...,3} {
    \draw[very thick] (6,{-6-2*\i}) -- (7,{-7-2*\i}) -- (6,{-8-2*\i});
  }
  \draw [very thick] (0,0)--(2,2)--(3,1)--(4,2)--(5,1)--(6,2)--(8,0)--(7,-1)--(8,-2)--(7,-3)--(8,-4)--(6,-6);
  \draw [very thick] (0,-8)--(2,-10)--(1,-11)--(2,-12)--(1,-13)--(3,-15)--(4,-14)--(5,-15)--(6,-14);
  
  \draw [help lines] (1,-5)--(2,-6)--(3,-5)--(4,-6)--(5,-5)--(6,-6);
   \draw [help lines] (2,-10)--(3,-9)--(4,-10)--(5,-9)--(6,-10);
   
  \foreach \i in {0,1,2} {
   \foreach \j in {0,...,7} {
      \draw ({2+2*\i},{1-2*\j}) circle(4pt);
      
   }
}
\foreach \i in {0,1} {
   \foreach \j in {0,...,7} {
      \filldraw ({3+2*\i},{-2*\j}) circle(4pt);
   }
}
\foreach \j in {0,...,4} {
      \filldraw ({1},{-2*\j}) circle(4pt);
   }
   \foreach \j in {0,...,2} {
      \filldraw ({7},{-2*\j}) circle(4pt);
   }

   \draw (0,0) \eastdomino;  \draw (0,0) \eastdominoplus;
    \draw (0,-2) \southdomino ; \draw (0,-2) \southdominoplus;
     \draw (0,-4) \southdomino;  \draw (0,-4) \southdominoplus;
\draw (0,-6) \southdomino;  \draw (0,-6) \southdominoplus;
\draw (0,-8) \southdomino;  \draw (0,-8) \southdominoplus;

\draw (2,-6) \southdomino;  \draw (2,-6) \southdominoplus;
\draw (2,-8) \southdomino;  \draw (2,-8) \southdominoplus;
\draw (4,-6) \southdomino;  \draw (4,-6) \southdominoplus;
\draw (4,-8) \southdomino;  \draw (4,-8) \southdominoplus;

\draw (5,1) \northdomino; 
\draw (3,1) \northdomino; 
\draw (5,-1) \northdomino; 

\draw (2,-4) \eastdomino;  \draw (2,-4) \eastdominoplus;
\draw (3,-5) \westdomino;  \draw (3,-5) \westdominoplus;
\draw (5,-5) \westdomino;  \draw (5,-5) \westdominoplus;
\draw (4,-2) \southdomino;  \draw (4,-2) \southdominoplus;
\draw (2,0) \southdomino;  \draw (2,0) \southdominoplus;

\draw (1,-13) \northdomino;
\draw (1,-11) \westdomino;
\draw (2,-12) \eastdomino;
\draw (3,-13) \westdomino;
\draw (4,-14) \eastdomino;
\draw (4,-10) \southdomino;

\draw (1,-11) \westdominoplus;
\draw (2,-12) \eastdominoplus;
\draw (3,-13) \westdominoplus;
\draw (4,-14) \eastdominoplus;
\draw (4,-10) \southdominoplus;

\end{tikzpicture}
\qquad  
\begin{tikzpicture}[scale=.4]
 \draw (0.5,0.5) node[fill=none,anchor=east,inner sep=3pt] {$(0,-1)$};
 \draw (0.5,-7.5) node[fill=none,anchor=east,inner sep=3pt] {$(0,-N)$};
  \draw (12.5,-5.5) node[fill=none,anchor=west,inner sep=3pt] {$(2n,-n-1)$};
 \draw (12.5,-13.5) node[fill=none,anchor=west,inner sep=3pt] {$(2n,-N-n)$};
 
  \draw (0,0) \eastdominoplus;
\draw (0,-2) \southdominoplus;
  \draw (0,-4) \southdominoplus;
 \draw (0,-6) \southdominoplus;
  \draw (0,-8) \southdominoplus;

  \draw (4,-6) \southdominoplus;
 \draw (4,-8) \southdominoplus;
\draw (8,-6) \southdominoplus;
\draw (8,-8) \southdominoplus;

 \draw (4,-4) \eastdominoplus;
\draw (5,-5) \westdominoplus;
 \draw (9,-5) \westdominoplus;
 \draw (8,-2) \southdominoplus;
 \draw (4,0) \southdominoplus;

\draw (1,-11) \westdominoplus;
\draw (4,-12) \eastdominoplus;
\draw (5,-13) \westdominoplus;
\draw (8,-14) \eastdominoplus;
\draw (8,-10) \southdominoplus;

\draw[very thick] (2.5,0.5)--(4.5,0.5);
\draw[very thick] (2.5,-3.5)--(4.5,-3.5);
\draw[very thick] (2.5,-5.5)--(4.5,-5.5);
\draw[very thick] (2.5,-7.5)--(4.5,-7.5);
\draw[very thick] (2.5,-11.5)--(4.5,-11.5);

\draw[very thick] (6.5,-1.5)--(8.5,-1.5);
\draw[very thick] (6.5,-5.5)--(8.5,-5.5);
\draw[very thick] (6.5,-7.5)--(8.5,-7.5);
\draw[very thick] (6.5,-9.5)--(8.5,-9.5);
\draw[very thick] (6.5,-13.5)--(8.5,-13.5);

\draw[very thick] (10.5,-5.5)--(12.5,-5.5);
\draw[very thick] (10.5,-7.5)--(12.5,-7.5);
\draw[very thick] (10.5,-9.5)--(12.5,-9.5);
\draw[very thick] (10.5,-11.5)--(12.5,-11.5);
\draw[very thick] (10.5,-13.5)--(12.5,-13.5);
\end{tikzpicture}
\caption{By drawing paths on the dominoes in the dimer matching in Figure \ref{fig:tower_graph_part_I} we obtain non-intersecting paths. By augmenting these paths with trivial horizontal lines, we obtain configuration of non-intersecting paths. This augmentation by horizontal lines corresponds to the vertex-dilation to go from $G^{\Az}_n$ to $G^{vert}_n$ in \Cref{sec:vert_polymer}. } 
\label{fig:tower:graph_part_II}
\end{figure}
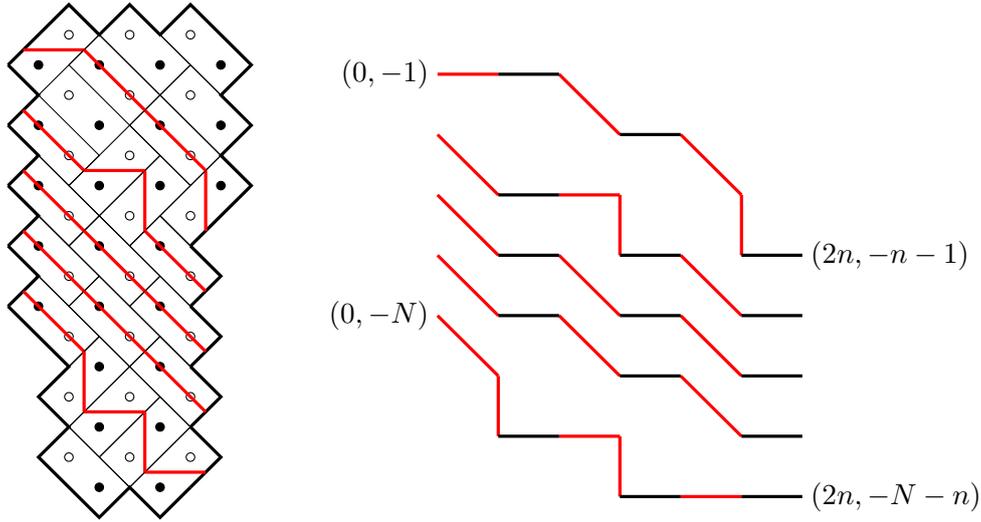
\subsection{The LGV theorem}

The next step is to define a probability measure on the collection of non-intersecting paths. Before we do that, we first introduce some notation.

Let $S = (S_{i,j})_{i,j \in \Z}$ be the shift matrix 
$$
S_{i,j}=
  \begin{cases}1, & j=i-1,\\
 0 & j\neq i-1.
   \end{cases}
$$
For  $\boldsymbol{a}=\{a_j\}_{j=\infty}^\infty$ we define $D(\boldsymbol{a})$ as the $\mathbb Z\times \mathbb Z$ diagonal matrix with diagonal entries given by $a_j$, but in reversed order:
$$
D(\boldsymbol{a})=\mathop{{\mathrm{diag}}} (\ldots,a_2,a_1,a_0,a_{-1},a_{-2},\ldots).
$$
In other words, $(D(\boldsymbol{a}))_{jk}=a_{-k}$ if $j=k$ and $(D(\boldsymbol{a}))_{jk}=0$ otherwise.
 We also define infinite matrices $\phi(\boldsymbol{\beta})$ and $\psi(\boldsymbol{\gamma})$ as follows:
$$
 \phi(\boldsymbol{\beta})= D(\boldsymbol{\beta})+(I-D(\boldsymbol{\beta}))S , \quad  \text{and} \qquad  \psi(\boldsymbol{\gamma})= \left(\sum_{j=0}^\infty S^j\right)D(\boldsymbol{\gamma}).
$$
These matrices encode the weighting of the two different type columns that were used to construct the graph $G_{LGV}^{n}$. 

We use these matrices to define weights on the edges of the graph $G_{LGV}^{n}$ as follows. For $i=1,\ldots,n$ let $\boldsymbol{\beta}_i^{[n]}=\{{\beta}_{i,j}^{[n]}\}_{j=-\infty}^\infty$ and $\boldsymbol{\gamma}_{i+1}^{[n]}=\{\gamma_{i+1,j}^{[n]}\}_{j=-\infty}^\infty$ be two families of sequences of positive numbers, and set 
$$
\phi_{i}=\phi(\boldsymbol{\beta}_{i}^{[n]}), \quad \text{and } \psi_{i}=\psi(\boldsymbol{\gamma}_{i+1}^{[n]}).
$$
Then we define edge weights
$$
w((2(i-1),j)\to (2i-1,k))=\left(\phi_{i}\right)_{j,k},
$$
and 
$$
w((2i-1,j)\to (2i,k))=\left(\psi_i\right)_{j,k}
$$
for $i=1,\ldots,n$.
Using these weights, we then define a probability measure on the set of all non-intersecting paths by defining the probability of a collection $\{\pi_j\}_{j=1}^N$ as the product of the weights of all edges:
$$
\mathbb P\left(\{\pi_j\}_{j=1}^N\right) = \frac1{Z_{n,N}}\prod_{j=1}^N \prod_{e \in \pi_j} w(e),
$$
where $Z_{n,N}$ is a normalizing constant. 
\begin{rmk}
    The slightly odd convention that the sequences of $\gamma_j$'s start with $\gamma_2$ and end with $\gamma_{n+1}$ is done to match the conventions that we made earlier in the paper. Indeed, by choosing 
    \begin{align}
        \label{eq:ab_to_betagamma}
        \begin{split}
 \gamma_{i,j}^{[n]}&=a_{i,j}^{[n]}+b_{i,j}^{[n]} \\
            \beta_{i,j}^{[n]}&=\frac{a_{i,j}^{[n]}}{a_{i,j}^{[n]}+b_{i,j}^{[n]}},
        \end{split}
    \end{align}
the mapping from non-intersecting paths to matchings of the Aztec diamond with weights $a^{[n]}_{i,j}$ and $b_{i,j}^{[n]}$ outlined in \Cref{fig:tower:graph_part_II} is weight-preserving (up to an overall factor that does not change the probability measure).
\end{rmk}

\subsection{Vertical slices}

For each $\ell \in \{0,\ldots,2n\}$ and $j =1,\ldots,N$ the $\pi_j$ has at least one (but possibly more) vertex with horizontal coordinate $\ell$, and we define $\pi_j(\ell$) as the vertical coordinate of the \textit{highest} vertex with horizontal coordinate $\ell$. We will be interested in the joint law of the positions  $\pi_j(2\ell-1)$ for $j=1,\ldots,N$, at given  odd horizontal  coordinate $2\ell-1$ for some $\ell \in \{1,2,\ldots, n\}$.

Before describing the law we first define, for $\ell=1,2,\ldots,n$, the matrices, 
$$
\Phi_\ell= \phi_1 \psi_1 \phi_2 \psi_2\cdots \phi_{\ell-1} \psi_{\ell-1} \phi_{\ell} =\left(\prod_{j=1}^{\ell-1} \phi_j \psi_j\right) \phi_{\ell}, 
$$
and
$$\Psi_\ell=\psi_{\ell} \phi_{\ell+1} \psi_{\ell+1} \phi_{\ell+2} \psi_{\ell+2} \cdots \phi_{n} \psi_{n} \phi_{\ell}=\psi_{\ell}\prod_{j=\ell+1}^n \phi_j \psi_j.
$$
Then the celebrated theorem by Lindström-Gessel-Viennot (the LGV Theorem) \cite{GesselViennot1985,Lindstrom1973} tells us that the law of the $\pi_{j}(2\ell-1)$ is given by a product of determinants:
\begin{equation}
\label{eq:product_of_determinants}
\mathbb P(\pi_j(2 \ell-1)=x_j, j=1, \ldots, N)=\frac{1}{Z_{\ell,n,N}}\det ((\Phi_{\ell})_{-i,x_j})_{i,j=1}^N \det ((\Psi_\ell)_{x_j,-j-n})_{i,j=1}^N
\end{equation}
where $Z_{\ell,n,N}$ is a normalizing constant. 

\subsection{Reversing the order of the columns}
Now the important point is that one can reverse the order in the product in the definition of $\Phi_\ell$ and $\Psi_\ell$ based on the following lemma, which is the analogue of Lemma \ref{thm:swap_columns_dimer}
\begin{lemma} \label{lem:swappingoperators}
Let $\boldsymbol{\beta}= \{\beta_j\}_{j=-\infty}^\infty,\boldsymbol{\gamma}=\{\gamma_j\}_{j=-\infty}^\infty $ be two infinite sequences. Then 
$$\phi(\boldsymbol{\beta}) \psi(\boldsymbol{\gamma})= \psi(\hat {\boldsymbol{\gamma}}) \phi(\hat {\boldsymbol{\beta}}),$$
where $\hat \beta, \hat \gamma$ are defined by 
\begin{equation} \label{eq:updaterulebetagamma}
\begin{cases}
\hat \gamma_j &= \beta_j \gamma_j + (1-\beta_{j+1})\gamma_{j+1} \\ 
\hat \beta_j &= \frac{\beta_j\gamma_j}{\beta_j \gamma_j + (1-\beta_{j+1})\gamma_{j+1}},
\end{cases}
\end{equation}
for $j \in \mathbb Z$.
\end{lemma}

\begin{proof}
The verification is a straightforward matrix multiplication. We will show that 
\begin{equation}\label{eq:matrix_commutation_with_inverses}
\psi(\hat {\boldsymbol{\gamma}})^{-1} \phi(\boldsymbol{\beta}) = \phi(\hat{ \boldsymbol{\beta})} \psi(\boldsymbol{\gamma})^{-1}.
\end{equation}

First note that 
\begin{multline}\label{eq:phibetapsigamma}
\psi(\hat {\boldsymbol{\gamma}})^{-1} \phi(\boldsymbol{\beta})= D(\hat {\boldsymbol{\gamma}})^{-1}(I-S)(D(\boldsymbol{\beta})+(I-D(\boldsymbol{\beta}))S)\\= D(\hat {\boldsymbol{\gamma}})^{-1} D(\boldsymbol{\beta}) -D(\hat {\boldsymbol{\gamma}})^{-1} S D(\boldsymbol{\beta})+D(\hat {\boldsymbol{\gamma}})^{-1}(I-D (\boldsymbol{\beta}))S-D(\hat{\boldsymbol{ \gamma}})^{-1}S (I-D (\boldsymbol{\beta}))S,
\end{multline}
and 
\begin{multline}\label{eq:phihatbetapsigamma}
 \phi(\hat {\boldsymbol{\beta}})\psi( \boldsymbol{\gamma})^{-1}= (D(\hat {\boldsymbol{\beta}})+(I-D(\hat {\boldsymbol{\beta}}))S) D( \boldsymbol{\gamma})^{-1}(I-S) \\=  D(\hat {\boldsymbol{\beta}})D( \boldsymbol{\gamma})^{-1} - D(\hat {\boldsymbol{\beta}}) D( \boldsymbol{\gamma})^{-1} S+(I-D (\hat {\boldsymbol{\beta}}))SD( \boldsymbol{\gamma})^{-1}- (I-D (\hat {\boldsymbol{\beta}}))SD(\boldsymbol{\gamma})^{-1}S.
\end{multline}
Now using 
$$
S D(\boldsymbol{\beta})=D(\sigma(\boldsymbol{\beta})S,$$
where $\sigma(\boldsymbol{\beta})_j=\beta_{j+1}$, we can write the right sides of \eqref{eq:phibetapsigamma} and \eqref{eq:phihatbetapsigamma} as quadratic polynomials in $S$, and after comparing coefficients
we find that \eqref{eq:matrix_commutation_with_inverses} is equivalent to the system
$$
\begin{cases}
    D(\hat{\boldsymbol{\gamma}})^{-1} D(\boldsymbol {\beta})=D(\hat {\boldsymbol{\beta}})D( \boldsymbol{\gamma})^{-1}\\
    -D(\hat{\boldsymbol{\gamma}})^{-1} D(\sigma (\boldsymbol{\beta}))+D(\hat{\boldsymbol{ \gamma}})^{-1}(I-D (\boldsymbol{\beta}))=-D(\hat {\boldsymbol{\beta}}) D( \boldsymbol{\gamma})^{-1} +(I-D (\hat {\boldsymbol{\beta}}))D( \sigma(\boldsymbol{\gamma}))^{-1}\\
    D(\hat{ \boldsymbol{\gamma}})^{-1} (I-D (\sigma(\boldsymbol{\beta})))=(I-D (\hat {\boldsymbol{\beta}}))D(\sigma(\boldsymbol{\gamma}))^{-1}.
\end{cases}
$$
These three equations can be reduced to two, since the second equation is the difference between the third and first equations. From the latter we obtain (where multiplication between two vectors is taken componentwise) 
\begin{equation}
   \begin{cases} 
        \hat {\boldsymbol{\gamma}}^{-1} \boldsymbol{\beta}=\hat{\boldsymbol{ \beta}}\boldsymbol{ \gamma}^{-1},\\
        \hat \gamma^{-1} (1-\sigma(\boldsymbol{\beta}))=(1-\hat {\boldsymbol{\beta}}) \sigma (\boldsymbol{\gamma})^{-1},
    \end{cases}
\end{equation}
which we can rewrite as 
\begin{equation}
   \begin{cases} 
        \boldsymbol{ \beta}\boldsymbol{\gamma} =\hat{ \boldsymbol{\beta}}\hat {\boldsymbol{\gamma}},\\
         (1-\sigma(\boldsymbol{\beta}))\sigma (\boldsymbol{\gamma})=(1-\hat {\boldsymbol{\beta}}) \hat {\boldsymbol{\gamma}}.
    \end{cases}
\end{equation}
Substituting the first into the second gives
\begin{equation*} 
   \begin{cases} 
         \boldsymbol{\beta}\boldsymbol{\gamma} =\hat {\boldsymbol{\beta}}\hat {\boldsymbol{\gamma}},\\
         (1-\sigma(\boldsymbol{\beta}))\sigma (\boldsymbol{\gamma}) +\boldsymbol{\beta} \boldsymbol{\gamma}= \hat{\boldsymbol{\gamma}}.
    \end{cases}
\end{equation*}
From here it is easy to see that $\hat {\boldsymbol{\gamma}}$ and $\hat {\boldsymbol{\beta}}$ are as in \eqref{eq:updaterulebetagamma}, which finishes the proof.
\end{proof}

By iteratively applying this lemma, we see that we can rewrite $\Phi_\ell$ and $\Psi_{\ell}$ as in the following lemma. 
\begin{lemma} We can write $\Phi_\ell$ and $\Psi_\ell$ as
\begin{equation}\label{eq:refactor1}
\Phi_\ell= \prod_{i=1}^{\ell-1} \psi(\boldsymbol{\gamma}_1^{[n-i]})\prod_{i=1}^\ell \phi(\boldsymbol{\beta}_i^{[n-\ell+i]})
\end{equation}
and 
\begin{equation} \label{eq:refactor2}
\Psi_\ell= \prod_{i=\ell}^n \psi(\boldsymbol{\gamma}_{\ell+1} ^{[n+\ell-i]})\prod_{i=\ell+1}^n \phi(\boldsymbol{\beta}_i^{[i-1]})
\end{equation}
Here the weights $\boldsymbol{\beta}_i^{[k]}$ and $\boldsymbol{\gamma}_i^{[k]}$ are defined as\footnote{Note that not all the updated parameters appear in \eqref{eq:refactor1} and \eqref{eq:refactor2}. In particular, note that only $\boldsymbol{\gamma}_1^{[k]}$ appears in $\Phi_\ell$, and no other $\boldsymbol{\gamma}_{i}^{[k]}$. Similarly, only $\boldsymbol{\gamma}_{\ell+1}^{[k]}$  appears in $\Psi_\ell$.}
\begin{equation} \label{eq:updaterulebetagammaattlevelk}
\begin{cases} 
\gamma_{i,j}^{[k-1]} &= \beta_{i,j}^{[k]} \gamma_{i+1,j}^{[k]} + (1-\beta_{i,j+1}^{[k]})\gamma_{i+1,j+1}^{[k]} \\ 
 \beta_{i,j}^{[k-1]} &= \frac{\beta_{i,j}^{[k]}\gamma_{i+1,j}^{[k]}}{\beta_{i,j} ^{[k]}\gamma_{i+1,j}^{[k]} + (1-\beta_{i,j+1}^{[k]})\gamma_{i+1,j+1}^{[k]}}
\end{cases}.
\end{equation}
\end{lemma}
\begin{proof}
By Lemma \ref{lem:swappingoperators} and \eqref{eq:updaterulebetagammaattlevelk} we have 
$$
\phi(\boldsymbol{\beta}_i^{[k]})\psi(\boldsymbol{\gamma}_{i+1}^{[k]})=\psi(\boldsymbol{\gamma}_{i}^{[k-1]})\phi(\boldsymbol{\beta}_i^{[k-1]}).$$ By applying this rule iteratively (in arbitrary order) we can reorder the product in the definition of $\Phi_{\ell}$ and $\Psi_\ell$ such that all $\psi$'s come first followed by all the $\phi$'s. The result is then as stated in the formula.
\end{proof}
The reordering in \eqref{eq:refactor1} and \eqref{eq:refactor2} is not just algebraic, but also has an combinatorial interpretation of switching a type I with a type II column. This leads us to define the graph $\tilde G_{LGV}^{n,\ell}=(\tilde V_{LGV}^{n,\ell},\tilde E_{LGV}^{n,\ell})$ as follows. We start by gluing $\ell-1$ columns of type II, followed by $\ell$ columns of type I, then $n-\ell+1$ columns of type II, and, finally, $n-\ell$ columns of type I. See also Figure \ref{fig:reorderedpaths}.
The weights $\tilde w(e)$ that we will put on the edges are determined as follows: 
$$
\tilde w((i,j)\to (i+1,k))=\left(\psi(\gamma_{1}^{[n-i-1]})\right)_{j,k}
$$
for $i=0,\ldots,\ell-2$, and 
$$
\tilde w((\ell+i,j)\to (\ell+i+1,k))=\left(\phi(\beta_{i+1}^{[n-\ell+j+1]})\right)_{j,k}
$$
for $i=0,\ldots,\ell-1$,
and 
$$
\tilde w((2\ell+i,j)\to (2\ell+i+1,k))=\left(\psi(\gamma_{\ell+1}^{[n-i]})\right)_{j,k}
$$
for $i=0,\ldots,n-\ell$, and 
$$
\tilde w((n+\ell+i,j)\to (n+\ell+i+1,k))=\left(\phi(\beta_{\ell+i+1}^{[\ell+i]})\right)_{j,k}
$$
for $i=0,\ldots,n-\ell-1$. The $\tilde G_{LGV}^{n,\ell}$ is the analogue of the graph $G^{v-swap}_{n,\ell}$. 
\begin{figure}
\begin{center}
  \begin{tikzpicture}[
    scale=0.7,
    every node/.style={circle, fill=black, inner sep=1pt},
    decoration={markings, mark=at position 0.5 with {\arrow{>}}},
    arrowedge/.style={postaction={decorate}, gray},
    infdot/.style={circle, draw=none, fill=gray!50, inner sep=0.5pt}
  ]
\def\n{7}
\def\l{3}
\def\jmin{-7}
\def\jmax{4}

\foreach \i in {0,...,10} {
  \foreach \j in {\jmin,...,\jmax} {
    \node (v\i\j) at (\i,\j) {};
  }
}

\foreach \i in {0,...,9} {%
  \foreach \j in {\jmin,...,\jmax} {%
    \ifnum \i<2
      \draw[arrowedge] (v\i\j) -- (v\the\numexpr\i+1\relax\j);
      \ifnum\j>\jmin
        \pgfmathtruncatemacro{\jm}{\j-1}
        \draw[arrowedge] (v\i\j) -- (v\i\jm);
      \fi
    \else
      \ifnum \i<5\relax
        \draw[arrowedge] (v\i\j) -- (v\the\numexpr\i+1\relax\j);
        \ifnum\j>\jmin
          \pgfmathtruncatemacro{\jm}{\j-1}
          \draw[arrowedge] (v\i\j) -- (v\the\numexpr\i+1\relax\jm);
        \fi
      \else
        \ifnum \i<8\relax
          \draw[arrowedge] (v\i\j) -- (v\the\numexpr\i+1\relax\j);
          \ifnum\j>\jmin
            \pgfmathtruncatemacro{\jm}{\j-1}
            \draw[arrowedge] (v\i\j) -- (v\i\jm);
          \fi
        \else
          \draw[arrowedge] (v\i\j) -- (v\the\numexpr\i+1\relax\j);
          \ifnum\j>\jmin
            \pgfmathtruncatemacro{\jp}{\j-1}
            \draw[arrowedge] (v\i\j) -- (v\the\numexpr\i+1\relax\jp);
          \fi
        \fi
      \fi
    \fi
  }
}

\foreach \i in {0,...,10} {
  \node[infdot] at (\i,\jmax+0.7) {};
  \node[infdot] at (\i,\jmax+1.4) {};
  \node[infdot] at (\i,\jmin-0.7) {};
  \node[infdot] at (\i,\jmin-1.4) {};
}

\foreach \i in {0,...,3} {
  \draw[red,very thick] (\i+5,\i-3) circle(4pt) {};
}
\foreach \i in {0,...,3} {
  \draw[red,very thick] (-\i+5,\i-3) circle(4pt) {};
}
\foreach \i in {0,...,3} {
  \draw[red,very thick] (2,\i) circle(4pt) {};
}

\draw (0,3) node[fill=none,anchor=east,inner sep=1pt] {$(0,-1)$};
\draw (10,-2) node[fill=none,anchor=west,inner sep=1pt] {$(2n,-n-1)$};
\draw (0,2) node[fill=none,anchor=east,inner sep=1pt] {$(0,-2)$};
\draw (10,-3) node[fill=none,anchor=west,inner sep=1pt] {$(2n,-n-2)$};
\draw (0,1) node[fill=none,anchor=east,inner sep=1pt] {$(0,-n+\ell-1)$};
\draw (10,-4) node[fill=none,anchor=west,inner sep=1pt] {$(2n,-2n+\ell-1)$};
\draw (0,0) node[fill=none,anchor=east,inner sep=1pt] {$(0,-n+\ell-2)$};
\draw (10,-5) node[fill=none,anchor=west,inner sep=1pt] {$(2n,-2n+\ell-2)$};
\draw (0,-1) node[fill=none,anchor=east,inner sep=8pt] {$\vdots$};
\draw (10,-6) node[fill=none,anchor=west,inner sep=8pt] {$\vdots$};
\draw (0,-2) node[fill=none,anchor=east,inner sep=1pt] {$(0,-N)$};
\draw (10,-7) node[fill=none,anchor=west,inner sep=1pt] {$(2n,-N-n)$};

\draw[very thick] (0,3)--(1,3)--(2,3)--(3,3)--(4,3)--(5,2)--(7,2)--(7,0)--(8,0)--(10,-2);
\draw[very thick] (0,2)--(3,2)--(3,2)--(4,1)--(5,1)--(5,0)--(6,0)--(6,-1)--(7,-1)--(8,-1)--(8,-1)--(9,-2)--(10,-3);
\draw[very thick] (0,-2)--(1,-2)--(1,-4)--(1,-5)--(4,-5)--(5,-6)--(9,-6)--(10,-7);
\draw[very thick] (0,0)--(2,0)--(5,-3)--(8,-3)--(10,-5);
\draw[very thick] (0,1)--(2,1)--(3,0)--(4,0)--(5,-1)--(5,-2)--(8,-2)--(10,-4);
\end{tikzpicture}
\caption{Illustration of the non-intersecting paths on the graph $\tilde G_{LGV}^{n,\ell}$ (here $\ell=3$ and $n=5$). The graph $\tilde G_{LGV}^{n,\ell}$ is the analogue of $G^{v-swap}_{n,\ell}$. } \label{fig:reorderedpaths}
\end{center} 
\end{figure}
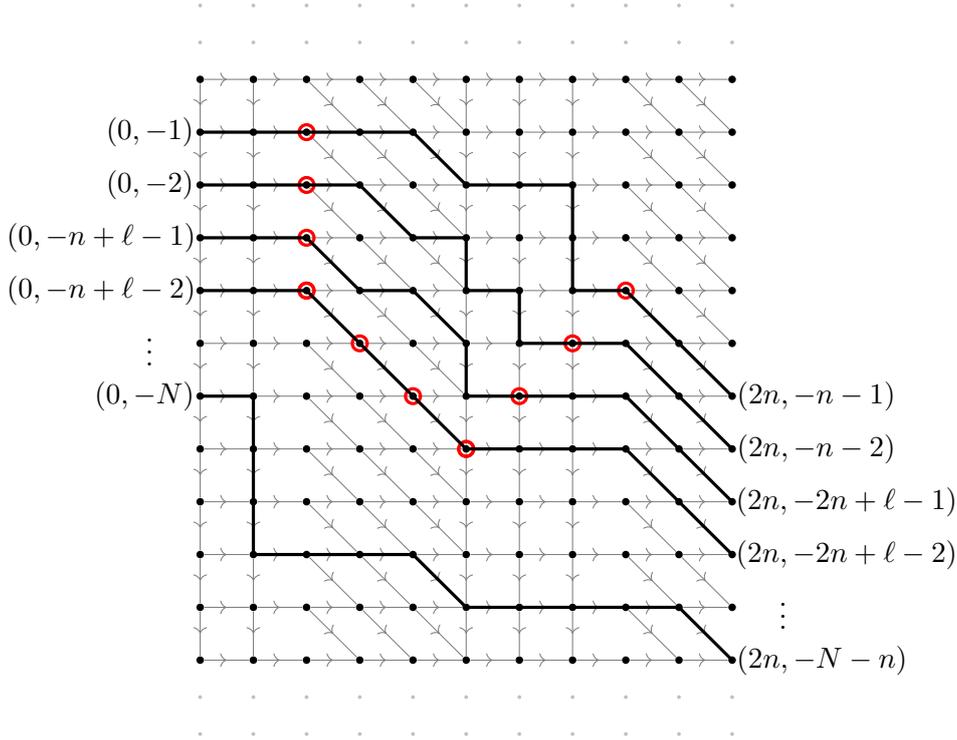

We consider again non-intersecting paths $\{\tilde \pi_j\}_{j=1}^N$ starting at $(0,-j-n)$ and ending at $(2n,-j)$ for $j=1,\ldots, N$, but now on the graph $\tilde G_{LGV}^{n,\ell}$. The following is the equivalent of Proposition \ref{thm:aztec_to_udt}.

\begin{prop} Fix $\ell \in \{1,2,\ldots, n-1\}$. Then the  joint of law of $\{\pi_j(2 \ell-1)\}_{j=1}^N$ and the joint law of $\{\tilde \pi_j(2 \ell-1)\}_{j=1}^N$ are the same. 
\end{prop}
\begin{proof}
This is a triviality since the LGV Theorem states that both laws are given by \eqref{eq:product_of_determinants} (but based on different factorizations of $\Phi_\ell$ and $\Psi_\ell$.)
\end{proof}

\subsection{Identifying and removing frozen parts}
An important feature of the graph $G^{v\text{-}swap}_{n,\ell}$ was that every matching was frozen in the first $\ell-1$ and last $n-\ell$ columns, as stated in Lemma \ref{thm:gbsw_frozen_bits}.
Similarly, we show that for every configuration of non-intersecting paths on $\tilde G_{LGV}^{n,\ell}$, the top paths are frozen in the first $\ell-1$ columns and the last $n-\ell$ columns. In other words, these paths must necessarily be straight horizontal lines until we arrive at the type I  columns. Whereas Lemma \ref{thm:gbsw_frozen_bits} followed from the fact that there were $1$-valent vertices, in the case of non-intersecting paths this is due to a blocking principle given by the following inequality relation on the type II columns:
\begin{equation}\label{eq:restriction_on_paths}
\tilde \pi_j(k)\geq \tilde \pi_{j+1}(k-1)+1,
\end{equation}
for $j=1, \ldots, N-1$ and 
$
k=1,\ldots,\ell-1$, and $k= 2\ell, \ldots, n+\ell
$.

The following is then the equivalent of \Cref{thm:gbsw_frozen_bits}.
\begin{lemma}
 Let $\ell\in \{1,\ldots,n\}$ and $N\geq  n$, and consider a configuration of non-intersecting path $\{\tilde \pi_j\}_{j=1}^N$ on the graph $\tilde G_{LGV}^{n,\ell}$.
For $j=1, \ldots, N-n+1$, the path $\tilde  \pi_{j}$ consists of a straight  horizontal line from $(0,-j)$ to $(\ell-1,-j)$, then a path from  $(\ell-1,-j)$ to $(n+\ell,-j-\ell)$, and then a straight diagonal line from $(n+\ell,-j-\ell)$ to $(2n,-j-n)$.
\end{lemma}
\begin{proof} 
We first sketch the idea. By the blocking principle \eqref{eq:restriction_on_paths} for the type II columns, only the $\ell-1$ paths in the bottom can move (in the first step only the bottommost path can jump, and if it does then it gives space for the next path to move in the second step, etc.). The top $N-\ell+1$ paths have to be frozen in these first columns and, consequently, the path $\tilde \pi_j$ cannot reach a position below $-j-\ell$ after passing through the next $\ell$ columns of type I for $1 \leq j \leq N-\ell+1$. Then again by the fact that type II columns have the blocking principle \eqref{eq:restriction_on_paths}, the top paths are restricted by these positions at the vertical slice $2\ell-1$ giving a lower bound for their positions at slice $n+\ell$. This lower bound is tight enough to ensure that  the only way to end up at the endpoint at slice $2n$  is that this lower bound at slice $n+\ell$ is attained and that the paths continue diagonally down towards their final destinations. The proof below formalizes this argument. 

From \eqref{eq:restriction_on_paths} we see that, for $j=1,\ldots,N-\ell+1$ and $k=1, \ldots, \ell-1$, 
$$
\tilde \pi_j(k) \geq \tilde\pi_{j+k}(0)+ k=-j,
$$
where we used the initial condition in the last step. But since paths only jump down  and start at $\tilde \pi_j(0)=-j$, we also have  $\tilde \pi_j(k)\leq-j$. We thus find 
$$\tilde \pi_j(k)=-j, $$
for all $j=1,\ldots,N-\ell+1$ and $k=1,\ldots,\ell-1$. The top $N-\ell+1$ paths are therefore straight lines in the first $\ell-1$ columns.

Now we show that the last $n-\ell$ steps in the paths $\tilde \pi_j$ with $j=1,\ldots,n-\ell+1$ have to be diagonal straight lines, which is somewhat harder to see from the illustrations. 
We first note that from the first half of this proof, we also find 
\begin{equation} \label{eq:lowerboundaftertwoell}
\tilde \pi_j(2 \ell-1) \geq -j-\ell,
\end{equation}
for $j=1\ldots, N-\ell+1$, since the paths can step at most $1$ down at each step between   $\ell-1$ and $2\ell-1$.  
Now, note from \eqref{eq:restriction_on_paths} that 
$$
\tilde \pi_j(2\ell-1+k)\geq \tilde \pi_{j+k}(2\ell-1)+k,
$$
for $j=1,\ldots, N-k$ and $k=1, \ldots,n-\ell+1$.  Combining the last two inequalities, we find 
$$
\tilde \pi_j(2\ell-1+k) \geq -j-\ell,
$$
for $j=1, \ldots, N-n+1$ and $k=1,\ldots,n-\ell+1$. 
Finally, since the last $n-\ell$ steps of the paths can go down at most $1$, and since the paths end at $(2n,-j-n)$ we  must have 
$$
\tilde \pi_{j}(n+\ell) \leq -j-\ell.
$$
for all $j=1\ldots,N$.  Therefore
\begin{equation}\label{eq:endpointsafternminell}
\tilde \pi_{j}(n+\ell) = -j-\ell.
\end{equation}
for $j=1, \ldots, N-n+1$, and the only way these paths can end up from here at $(2n,-j-n)$ is by diagonal straight lines. \end{proof}
We are not done yet. There are further parts of the path configurations that are frozen  and this will be important for us.  We start with the following lemma, which says that the middle paths are all frozen, not only in the first $\ell-1$ columns of type II, but also in the  subsequent collection of type I columns. Note that this only holds for sufficiently large $N$, and we therefore upgrade our assumption and take\footnote{Assuming $N\geq 2n$ is sufficient for all upcoming claims, but part of these claims will hold under weaker bounds on $N$. Since this is however not relevant for our purposes, we will make the overall assumption $N\geq 2n$.} $N\geq 2n$.

\begin{lemma}\label{lem:frozen_first_parts}
Let $\ell\in \{1,\ldots,n\}$ and take $N\geq  2n$. 
    The positions $\tilde \pi_{n-\ell+1+j}(k)$ for $j=n-\ell+2,n-\ell+3,\ldots, N-n+1$ and $k \in \{0,\ldots,2\ell-1\}$ are fully frozen. More precisely,  up to the vertical line with horizontal coordinate $2 \ell-1$ the path $\tilde \pi_j$ consists of a straight horizontal line from $(0,-j)$ to $(\ell-1,-j)$, then a straight diagonal line to $(2 \ell-1,-j+\ell)$, for $j=n-\ell+2,n-\ell+3,\ldots, N-n+1$.
\end{lemma}

\begin{proof}
Note that we already know from the previous lemma that,  for $j \leq N-n+1$,  the paths  have horizontal straight parts from $(0,-j)$ to $(\ell-1,-j)$. It remains to show that they continue as diagonal straight lines, if we further assume $j\geq n-\ell+2$. 

We first consider the location of $\tilde \pi_{n-\ell+2}(2\ell-1).$ First note that by \eqref{eq:lowerboundaftertwoell} we have $\tilde \pi_{n-\ell+2}(2\ell-1)\geq -n-2$.  On the other hand, combining this inequality repeatedly with \eqref{eq:restriction_on_paths} we find 
\begin{equation} \label{eq:help_frozen_end}
-n-2 \leq \tilde  \pi_{n-\ell+2}(2\ell-1) \leq  \tilde \pi_{n-\ell+1}(2\ell)-1  \leq\tilde  \pi_{1}(n+\ell)-n+\ell-1=-n-2,
\end{equation}
 where we used \eqref{eq:endpointsafternminell} in the last step. Therefore, we have $\tilde \pi_{n-\ell+2}(2\ell)=-n-2.$  This means in particular that $\tilde \pi_{n-\ell+2}$ from $(\ell-1,-n+ \ell-2)$ has to continue as  diagonal straight line to $(2 \ell-1,-n-2)$. 

 By the non-intersecting condition we then also see that all $\tilde \pi_j$ with $j=n-\ell+2, \ldots,N-n+1$ have to be straight diagonal lines. This finishes the proof.
\end{proof}
\begin{cor}  Let $\ell\in \{1,\ldots,n\}$ and take $N\geq  2n$. The paths $\tilde \pi_j$, with $j=n-\ell+1,\ldots,N-2n+\ell-1$ are fully frozen and given by a straight horizontal paths in type II columns and straight diagonal lines in type I columns.
\end{cor}
\begin{proof}
    For the first $\ell-1$ type II columns and first $\ell$ type I columns this is (part of) Lemma \ref{lem:frozen_first_parts}. But this means also that, by the blocking principle \eqref{eq:restriction_on_paths} and the fact that the middle paths have consecutive locations at slice $2\ell-1$, the middle paths with index $j=n-\ell+1,\ldots,N-2n+\ell$, must be frozen in the next $n-\ell+1$ type II columns. It also implies that the only way to end up at their final destinations is by taking straight diagonal lines in the remain $n-\ell$ columns of type I. 
\end{proof}
We have thus proven that the collection of paths $\tilde \pi_j$ can be divided in a top, middle and bottom block. The middle block is fully frozen and the top and bottom have randomness. But the fact that the middle block is frozen also implies that the top and bottom blocks are independent. 

We also need the following lemma describing frozen points that the top paths have to go through.
\begin{lemma} \label{lem:frozen_diagonal_parts}
    For $j=1, \ldots, n-\ell+1$, the path $\tilde \pi_j$ has to pass through the point 
    $(n+\ell+1-j,-\ell-j)$.
\end{lemma}
\begin{proof}
Combining $\tilde \pi_{n-\ell+2}(2\ell-1)=-n-2$ from \eqref{eq:help_frozen_end} with \eqref{eq:restriction_on_paths} yields
$$
\tilde \pi_{n-\ell+2-k}(2\ell+k-1)\geq -n-2+k,
$$
for $k=0,\ldots,n-\ell+1$. On the other hand, applying \eqref{eq:restriction_on_paths} starting from $\pi_1(n+\ell)=-1-\ell$, gives
$$
\tilde \pi_{n-\ell+2-k}(2\ell+k-1)\leq -n-2+k.
$$
for $k=0,\ldots,n-\ell+1$. Combining the two inequalities gives an equality, and after inserting $k=n-\ell+2-j$ we find the statement.
\end{proof}

\begin{rmk}
    We find it useful  to think of the limit $N \to \infty$. Then we will have an infinite number of paths but for $j \geq n-\ell+2$ the path but the only non-frozen parts of the paths lie fully inside the region that is cut out by the red point in Figure \ref{fig:reorderedpaths} . This limit should be taken with some care, and this is why we formulate the results for finite $N$.
\end{rmk}

\subsection{Relation to polymer measure}

Lemmas \ref{lem:frozen_first_parts} and \ref{lem:frozen_diagonal_parts} tell us that the part of the paths inside region of the graph that is cut out by the frozen locations indicated by the red dots in Figure \ref{fig:reorderedpaths}  is independent from the parts of the paths outside this region. Since this region contains precisely the paths that we are interested in, and in particular the locations $\tilde \pi_j(2 \ell-1)$, we eliminate all other parts from the graph. What remains is a graph, denoted by $tG^{LGV}_{n,\ell}$, consisting of (parts of) $\ell$ columns of type I and then (parts of) $n-\ell+1$ columns of type II. We will place this graph such that the top-left coordinate is $(-\ell,-1)$  and then consider $n-\ell+1$ non-intersecting paths $\hat \pi_j$ that start at $(-\ell,-j)$  and end at $(n-\ell+1-j,-\ell-j)$. See also Figure \ref{fig:cut_out_part}.  

Summarizing, we have obtained the following proposition.

\begin{prop}
Let $\ell\in \{1,\ldots, n\}$ and $N\geq 2n-\ell+1$. Then the laws of $\{\pi(2 \ell-1)\}_{j=1}^{n-\ell+1}$ and  $\{\hat \pi_j(\ell)\}_{j=1}^{n-\ell+1}$ are the same.
\end{prop}

This is the path translation of \Cref{thm:vert_slice_polymer}. 

Note that each of the paths has to end with a final horizontal edge. By removing these edges, which does not change the probability measure elsewhere, we obtain the same hybrid polymer models on $G^{\beta \Gamma}_{\ldots}$ from \Cref{sec:vert_polymer}.

\subsection{Concluding remarks}
The above is an alternative construction to relate dimer configuration of the Aztec diamond to $\beta$-$\Gamma$ polymer, that fits with existing recent literature connecting spider moves to matrix refactorizations. One downside of this construction is the introduction of auxilliary  paths  and having to choose $N$ large enough (often one takes $N\to \infty$). The column swaps in \Cref{sec:vert_polymer} do not require additional paths, and we have therefore chosen to keep that as our main construction in this paper. 

\begin{figure}
\begin{center}
  \begin{tikzpicture}[
    scale=0.7,
    every node/.style={circle, fill=black, inner sep=1pt},
    decoration={markings, mark=at position 0.5 with {\arrow{>}}},
    arrowedge/.style={postaction={decorate}, gray},
    infdot/.style={circle, draw=none, fill=gray!50, inner sep=0.5pt}
  ]
\def\n{7}
\def\l{3}

\def\jmax{3}

\foreach \i in {0,...,2} {%
  \pgfmathtruncatemacro{\ii}{\i + 1}%
  \pgfmathtruncatemacro{\jmin}{-\i+1}%
  \foreach \j in {\jmin,...,\jmax} {%
    \pgfmathtruncatemacro{\jm}{\j - 1}%
    \path
      (\i,\j)   coordinate (v-\i-\j)
      (\ii,\j)  coordinate (v-\ii-\j)
      (\ii,\jm) coordinate (v-\ii-\jm);
    \draw[arrowedge] (v-\i-\j) -- (v-\ii-\j);
    \node (v\i\j) at (\i,\j) {};
    \node (v\ii\j) at (\ii,\j) {};
    \draw[arrowedge] (v-\i-\j) -- (v-\ii-\jm);
  }%
}

\foreach \i in {3,...,5} {%
  \pgfmathtruncatemacro{\ii}{\i + 1}%
  \pgfmathtruncatemacro{\jmin}{-4+\i}%
  \foreach \j in {\jmin,...,\jmax} {%
    \pgfmathtruncatemacro{\jm}{\j - 1}%
    \path
      (\i,\j)   coordinate (v-\i-\j)
      (\i,\jm)  coordinate (v-\ii-\jm)
      (\ii,\jm) coordinate (v-\ii-\jm);
    \draw[arrowedge] (v-\i-\j) -- (v-\i-\jm);
    \node (v\i\j) at (\i,\j) {};
    \node (v\i\jm) at (\i,\jm) {};
    \draw[arrowedge] (v-\i-\jm) -- (v-\ii-\jm);
    \node (v\ii\jm) at (\ii,\jm) {};
  }%
}%
\draw[arrowedge] (4,3)--(5,3);

\draw[arrowedge] (5,3)--(6,3);
\foreach \i in {1,...,3} {
  \draw[red,very thick] (\i+3,\i-3) circle(4pt) {};
}
\foreach \i in {1,...,3} {
  \draw[red,very thick] (0,\i) circle(4pt) {};
}

\node (x,y) at (6,3) {};

\draw[very thick] (0,3)--(1,3)--(2,3)--(3,3)--(4,3)--(4,2)--(5,2)--(5,0)--(6,0);
\draw[very thick] (0,2)--(3,2)--(3,2)--(3,1)--(4,1)--(4,-1)--(5,-1);
\draw[very thick] (0,1)--(2,1)--(3,0)--(3,-2)--(4,-2);
\end{tikzpicture}
\qquad
\qquad 
  \begin{tikzpicture}[
    scale=0.7,
    every node/.style={circle, fill=black, inner sep=1pt},
    decoration={markings, mark=at position 0.5 with {\arrow{>}}},
    arrowedge/.style={postaction={decorate}, gray},
    infdot/.style={circle, draw=none, fill=gray!50, inner sep=0.5pt}
  ]
\def\n{7}
\def\l{3}

\def\jmax{3}

\foreach \i in {0,...,2} {%
  \pgfmathtruncatemacro{\ii}{\i + 1}%
  \pgfmathtruncatemacro{\jmin}{-\i+1}%
  \foreach \j in {\jmin,...,\jmax} {%
    \pgfmathtruncatemacro{\jm}{\j - 1}%
    \path
      (\i,\j)   coordinate (v-\i-\j)
      (\ii,\j)  coordinate (v-\ii-\j)
      (\ii,\jm) coordinate (v-\ii-\jm);
    \draw[arrowedge] (v-\i-\j) -- (v-\ii-\j);
    \node (v\i\j) at (\i,\j) {};
    \node (v\ii\j) at (\ii,\j) {};
    \draw[arrowedge] (v-\i-\j) -- (v-\ii-\jm);
  }%
}

\foreach \i in {3,...,4} {%
  \pgfmathtruncatemacro{\ii}{\i + 1}%
  \pgfmathtruncatemacro{\jmin}{-4+\i}%
  \foreach \j in {\jmin,...,\jmax} {%
    \pgfmathtruncatemacro{\jm}{\j-1}%
    \path
    (\ii,\jm) coordinate (v-\ii-\jm);
      (\i,\j)   coordinate (v-\i-\j)
      (\i,\jm)  coordinate (v-\i-\jm)
      
    \draw[arrowedge] (v-\i-\j) -- (v-\i-\jm);
    \node (v\i\j) at (\i,\j) {};
    \node (v\i\jm) at (\i,\jm) {};
    \draw[arrowedge] (v-\i-\j) -- (v-\ii-\j);
    \node (v\ii\j) at (\i,\j) {};
  }%
}%
\draw[arrowedge] (4,3)--(5,3);
 
    \node (5,3) at (5,3) {};
    \node (5,2) at (5,2) {};
    \node (5,1) at (5,1) {};
    \node (5,0) at (5,0) {};

\draw[arrowedge] (5,3)--(5,2);
\foreach \i in {1,...,3} {
  \draw[red,very thick] (\i+2,\i-3) circle(4pt) {};
}
\foreach \i in {1,...,3} {
  \draw[red,very thick] (0,\i) circle(4pt) {};
}

\draw[very thick] (0,3)--(1,3)--(2,3)--(3,3)--(4,3)--(4,2)--(5,2)--(5,0);
\draw[very thick] (0,2)--(3,2)--(3,2)--(3,1)--(4,1)--(4,-1);
\draw[very thick] (0,1)--(2,1)--(3,0)--(3,-2);
\end{tikzpicture}
\caption{The left picture shows the non-frozen parts of paths $\tpi_1,\ldots,\tpi_{n-\ell+1}$ on $\tG^{LGV}_{n,\ell}$ when $n=5, \ell=3$. Note that, after removing the $1$-valent edges on the right, this graph is exactly the underlying graph $G^{\beta \Gamma}_{3,3}$ from \Cref{def:beta_sw} shown on the right, and the weights will also have the same distribution if we make the choice \eqref{eq:ab_to_betagamma}.}
\label{fig:cut_out_part}
\end{center} 
\end{figure}
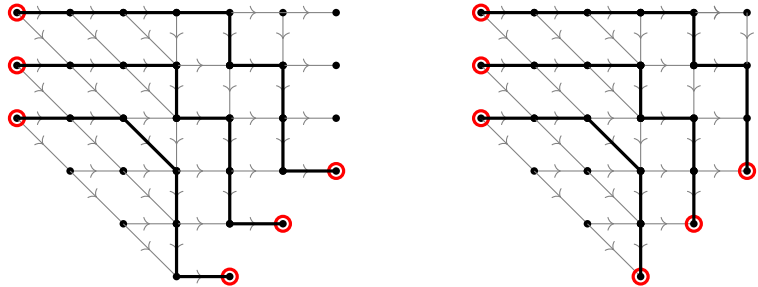

\section{Deterministic limit to Fock's weights}\label{appendix:deterministic}

We recall from \eqref{eq:deterministic_n_level_parameters_intro} that  in the limit $T\to \infty$ the random weights for the Gamma disordered Aztec diamond with general parameters become deterministic and are given by 
\begin{equation}\label{eq:weight_limit}
a_{i,j}^{[n]}=\psi_j+\theta_i, \qquad b_{i,j}^{[n]}=\phi_{j-n}-\theta_i.
\end{equation}
The purpose of this appendix is to show that these weights are a special case of Fock's weights in the genus 0 case \cite{BoutillierDeTiliere2024}. They also appear in the context of critical weights for isoradial graphs \cite{Kenyon2002Laplacian}. Note that if $\theta_i=0$, then these are exactly the general Schur process weights on the Aztec diamond; one can also get this family by taking all $\psi_j\equiv \psi$ and $\phi_j\equiv \phi$ but leaving the $\theta_i$ distinct. We thank C\'edric Boutillier for helpful comments related to this appendix. 
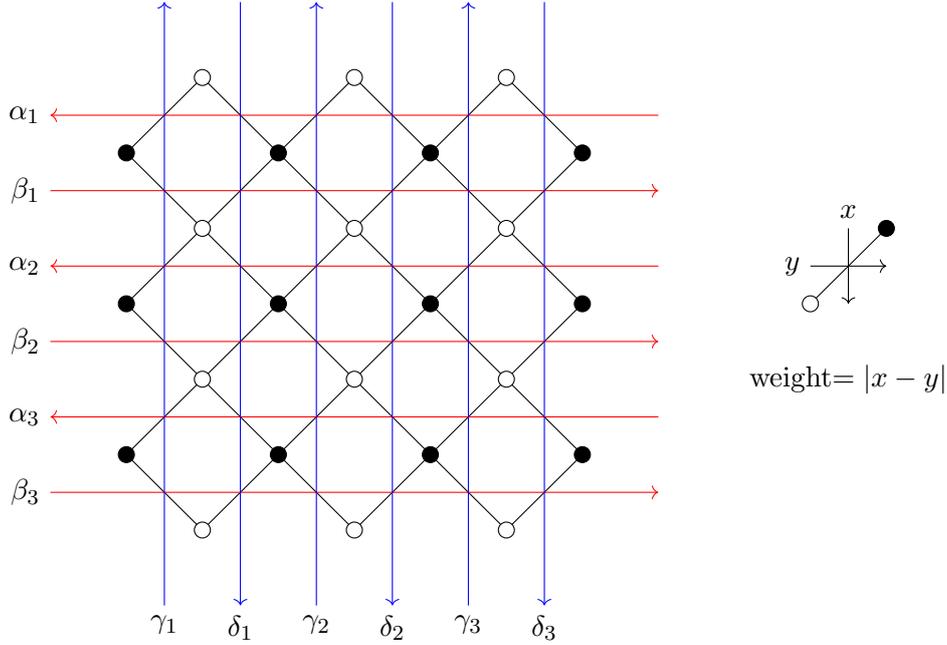
\begin{figure}[t]
    \centering

    \begin{tikzpicture}[scale=1]
  
\foreach \x in {0,...,2} {
  \foreach \y in {1,...,3} {
   \draw (2*\x,2*\y)-- (2*\x+1,2*\y+1);
    \draw (2*\x,2*\y)-- (2*\x+1,2*\y-1);
\draw (2*\x+2,2*\y)-- (2*\x+1,2*\y+1);
    \draw (2*\x+2,2*\y)-- (2*\x+1,2*\y-1);
  }
}

\foreach \x in {0,...,3} {
  \foreach \y in {1,...,3} {
   \filldraw (2*\x,2*\y) circle(3pt);
    \draw[fill=white] (2*\y-1,2*\x+1) circle(3pt);
  }
}
\draw[red,<-] (-1,6.5)--(7,6.5);
\draw[anchor=east] (-1,6.5) node {$\alpha_1$};
\draw[red,<-] (-1,4.5)--(7,4.5);
\draw[anchor=east] (-1,4.5) node {$\alpha_2$};
\draw[red,<-] (-1,2.5)--(7,2.5);
\draw[anchor=east] (-1,2.5) node {$\alpha_3$};
\draw[red,->] (-1,5.5)--(7,5.5);
\draw[anchor=east] (-1,5.5) node {$\beta_1$};
\draw[red,->] (-1,3.5)--(7,3.5);
\draw[anchor=east] (-1,3.5) node {$\beta_2$};
\draw[red,->] (-1,1.5)--(7,1.5);
\draw[anchor=east] (-1,1.5) node {$\beta_3$};

\draw[blue,->] (4.5,0)--(4.5,8);
\draw[anchor=north] (4.5,0) node {$\gamma_3$};
\draw[blue,->] (2.5,0)--(2.5,8);
\draw[anchor=north] (2.5,0) node {$\gamma_2$};
\draw[blue,->] (0.5,0)--(0.5,8);
\draw[anchor=north] (0.5,0) node {$\gamma_1$};
\draw[blue,<-] (5.5,0)--(5.5,8);
\draw[anchor=north] (5.5,0) node {$\delta_3$};
\draw[blue,<-] (3.5,0)--(3.5,8);
\draw[anchor=north] (3.5,0) node {$\delta_2$};
\draw[blue,<-] (1.5,0)--(1.5,8);
\draw[anchor=north] (1.5,0) node {$\delta_1$};

\draw (9,4)--(10,5);
 \draw[fill=white] (9,4) circle(3pt);
  \draw[fill=black] (10,5) circle(3pt);
\draw[->] (9.5,5)--(9.5,4);
\draw[->] (9,4.5)--(10,4.5);
\draw[anchor=south] (9.5,5) node {$x$};
\draw[anchor=east] (9,4.5) node {$y$};
\draw (9.5,3) node {weight$=|x-y|$};
\end{tikzpicture}

    \caption{A rail yard graph and the corresponding edge weights given in \Cref{def:fock_weights}.}
    \label{fig:railyard}
\end{figure}

\begin{defi}\label{def:fock_weights}
   Fix sequences $(\alpha_i)_{i=1}^n$, $(\gamma_j)_{j=1}^n$, $(\beta_k)_{k=1}^n$, $(\delta_\ell)_{\ell=1}^n$ of real numbers such that 
\[
\alpha_i < \gamma_j < \beta_k < \delta_\ell \qquad \text{for all } 1 \le i,j,k,\ell \le n.
\]
We draw the tracks of the rail yard graph on the Aztec diamond as in Figure~\ref{fig:railyard}. 
There are vertical and horizontal tracks, alternating in direction. 
We label the horizontal tracks directed to the left by $\alpha_j$ and those directed to the right by $\beta_j$. 
The vertical tracks directed upward and downward are labeled $\gamma_i$ and $\delta_i$,  respectively. 
Note that the intersection point of any horizontal and vertical track lies on an edge. 
We then define the weight of each edge to be the absolute value of the difference of the labels of the two intersecting tracks, as indicated in Figure~\ref{fig:railyard}.
\end{defi}

We note that the parameters $\psi_j, \phi_j, \theta_i$ in \eqref{eq:weight_limit} can be changed by taking
\begin{align}
    \begin{split}
        \phi_j &\mapsto \phi_j + c \\ 
        \psi_j &\mapsto \psi_j - c \\ 
        \theta_i &\mapsto \theta_i + c
    \end{split}
\end{align}
without affecting the edge weights. Since one must have $\phi_j > \theta_i$ for each $i,j$ for positivity already, this means that without loss of generality one may take
\begin{equation}\label{eq:weight_inequality}
    \phi_j > \theta_i > \psi_j \quad \quad \quad \text{ for all }i,j \in \Z.
\end{equation}

\begin{prop} Fix $n$ and let $\P_\infty$ be the dimer measure on the Aztec diamond $G_n^{\Az}$ with weights as in \eqref{eq:weight_limit} for parameters satisfying \eqref{eq:weight_inequality}. Let $\P_\delta$ be the dimer measure on the same Aztec diamond defined by \Cref{def:fock_weights} with $\delta_j=\delta$, $\gamma_i=\theta_i$, $\alpha_j=-\psi_j$ and $\beta_j=\phi_{j-n}$, for any real $\delta$ large enough to satisfy the condition of \Cref{def:fock_weights}. Then
\[
\lim_{\delta \to \infty} \P_\delta = \P_\infty
\]
as functions on the finite set $\mc{M}$ of perfect matchings.
\end{prop}
\begin{proof}
It is well known\footnote{One way to see this is that any two dimer configurations can be obtained from one another by a sequence of moves where each move replaces the edges around a given face of the graph by their complement. The change in the weight of the dimer configuration for each flip corresponds to multiplying or dividing by the corresponding face weight.}  that the dimer measure is determined by the face weights. Hence, we need only show that the face weights defined by the Fock weights converge to the face weights defined by the weights \eqref{eq:weight_limit}. Now let us define these face weights.

Note that there are two types of faces for the Aztec diamond: the even faces, that have black vertices at the left and right, and the odd faces, that have the black vertices on top and bottom. For the faces, we associate weights $F$ according to
\begin{equation}
    \tikz{ 
    
    \draw (0,0)--(1,1)--(2,0)--(1,-1)--(0,0);
    
    \draw[fill=white] (1,1) circle(3pt);
    \draw[fill=white] (1,-1) circle(3pt);
    \draw[fill=black] (0,0) circle(3pt);
    \draw[fill=black] (2,0) circle(3pt); 
        \draw[anchor=south east] (.5,.5) node {$x$};
       \draw[anchor=south west] (1.5,.5) node {$y$};
       \draw[anchor=north east] (.5,-.5) node {$z$};
        \draw[anchor=north west] (1.5,-.5) node {$w$};
         \draw (1,0) node {$F$};
        \draw (4,0) node {$F=\frac{xw}{yz}$};
    }
    \end{equation}
    for the even faces, and
\begin{equation}
    \tikz{ 
    
    \draw (0,0)--(1,1)--(2,0)--(1,-1)--(0,0);
    
    \draw[fill=black] (1,1) circle(3pt);
    \draw[fill=black] (1,-1) circle(3pt);
    \draw[fill=white] (0,0) circle(3pt);
    \draw[fill=white] (2,0) circle(3pt); 
      \draw[anchor=south east] (.5,.5) node {$x$};
       \draw[anchor=south west] (1.5,.5) node {$y$};
       \draw[anchor=north east] (.5,-.5) node {$z$};
        \draw[anchor=north west] (1.5,-.5) node {$w$};
        \draw (1,0) node {$F$};
        \draw (4,0) node {$F=\frac{yz}{xw}$};
    }
    \end{equation}
    for the odd faces.

For the weights yielding $\P_\infty$, we get face weights
\begin{equation}\label{eq:ours_even}
    \tikz{ 
    
    \draw (0,0)--(1,1)--(2,0)--(1,-1)--(0,0);
    
    \draw[fill=white] (1,1) circle(3pt);
    \draw[fill=white] (1,-1) circle(3pt);
    \draw[fill=black] (0,0) circle(3pt);
    \draw[fill=black] (2,0) circle(3pt); 
        \draw[anchor=south east] (.5,.5) node {$a_{i,j}^{[n]}$};
       \draw[anchor=south west] (1.5,.5) node {$1$};
       \draw[anchor=north east] (.5,-.5) node {$b_{i,j}^{[n]}$};
        \draw[anchor=north west] (1.5,-.5) node {$1$};
         \draw (1,0) node {$F$};
        \draw[anchor=west]   (4,0) node {$F=\frac{a_{i,j}^{[n]}}{b_{i,j}^{[n]}}=\frac{\psi_j+\theta_i}{\phi_{j-n}-\theta_i}$};
    }
    \end{equation}
    and 
    \begin{equation}\label{eq:ours_odd}
    \tikz{ 
    
    \draw (0,0)--(1,1)--(2,0)--(1,-1)--(0,0);
    
    \draw[fill=black] (1,1) circle(3pt);
    \draw[fill=black] (1,-1) circle(3pt);
    \draw[fill=white] (0,0) circle(3pt);
    \draw[fill=white] (2,0) circle(3pt); 
        \draw[anchor=south east] (.5,.5) node {$1$};
       \draw[anchor=south west] (1.5,.5) node {$b_{i+1,j}^{[n]}$};
       \draw[anchor=north east] (.5,-.5) node {$1$};
        \draw[anchor=north west] (1.5,-.5) node {$a_{i+1,j+1}^{[n]}$};
         \draw (1,0) node {$F$};
        \draw[anchor=west] (4,0) node {$F=\frac{b_{i+1,j}^{[n]}}{a_{i+1,j+1}^{[n]}}=\frac{\phi_{j-n}-\theta_{i+1}}{\psi_{j+1}+\theta_{i+1}}$};
    }.
    \end{equation}
The corresponding face weights for $\P_\delta$ are given by 
\begin{equation}\label{eq:fock_even}
    \tikz{ 
    
    \draw (0,0)--(1,1)--(2,0)--(1,-1)--(0,0);
    
    \draw[fill=white] (1,1) circle(3pt);
    \draw[fill=white] (1,-1) circle(3pt);
    \draw[fill=black] (0,0) circle(3pt);
    \draw[fill=black] (2,0) circle(3pt); 
        \draw[anchor=south east] (.5,.5) node {$\gamma_i-\alpha_j$};
       \draw[anchor=south west] (1.5,.5) node {$\delta_i-\alpha_j$};
       \draw[anchor=north east] (.5,-.5) node {$\beta_j-\gamma_i$};
        \draw[anchor=north west] (1.5,-.5) node {$\delta_i-\beta_j$};
         \draw (1,0) node {$F$};
        \draw[anchor=west]   (4,0) node {$F=\frac{(\gamma_i-\alpha_j)(\delta_i-\beta_j)}{(\delta_i-\alpha_j)(\beta_j-\gamma_i)}$};
    }
    \end{equation}
    and
    \begin{equation}\label{eq:fock_odd}
    \tikz{ 
    
    \draw (0,0)--(1,1)--(2,0)--(1,-1)--(0,0);
    
    \draw[fill=black] (1,1) circle(3pt);
    \draw[fill=black] (1,-1) circle(3pt);
    \draw[fill=white] (0,0) circle(3pt);
    \draw[fill=white] (2,0) circle(3pt); 
        \draw[anchor=south east] (.5,.5) node {$\delta_i-\beta_j$};
       \draw[anchor=south west] (1.5,.5) node {$\beta_{j}-\gamma_{i+1}$};
       \draw[anchor=north east] (.5,-.5) node {$\delta_i-\alpha_{j+1}$};
        \draw[anchor=north west] (1.5,-.5) node {$ \gamma_{i+1}-\alpha_{j+1} $};
         \draw (1,0) node {$F$};
        \draw[anchor=west] (4,0) node {$F=\frac{(\beta_{j}-\gamma_{i+1})(\delta_i-\alpha_{j+1})}{(\delta_i-\beta_j)(\gamma_{i+1}-\alpha_{j+1})}$};
    }.
    \end{equation}

Under our identification $\delta_i=\delta$, $\gamma_i=\theta_i$, $\alpha_j=-\psi_j$ and $\beta_j=\phi_{j-n}$, as $\delta \to \infty$ the weight in \eqref{eq:fock_even} converges to the one in \eqref{eq:ours_even}, and similarly \eqref{eq:fock_odd} converges to \eqref{eq:ours_odd}. This concludes the proof.
\end{proof}

\begin{rmk}
   In the literature, the parameters $\alpha_j, \beta_j, \gamma_j, \delta_j$ are more commonly taken to lie on the unit circle rather than on the real line. However, there is a remarkable principle: if one applies any Möbius transformation to these parameters, one obtains a new weight that is \emph{gauge equivalent} to the original one. Consequently, we may apply a Möbius transformation that maps the unit circle to the real line. See also \cite[Remark~8]{BoutillierDeTiliere2024} and \cite[Section 5]{KenyonOkounkov2006}.
\end{rmk}

\bibliographystyle{alpha_abbrvsort}
\bibliography{references.bib}

\end{document}

%% file: preamble.tex
\usepackage{anysize}
\usepackage{amsmath}
\usepackage{amsthm}
\usepackage{amsmath,amscd}
\usepackage[utf8]{inputenc}
\usepackage{amssymb}
\usepackage{stmaryrd}
\usepackage{wasysym}
\usepackage{mathrsfs}
\usepackage{subcaption}
\usepackage[pagebackref=true]{hyperref} 
\renewcommand*{\backref}[1]{}
\renewcommand*{\backrefalt}[4]{({\tiny%
   \ifcase #1 Not cited.%
         \or Cited on page~#2.%
         \else Cited on pages #2.%
   \fi%
   })}
\usepackage{cleveref}
\usepackage{enumitem}
\usepackage{epstopdf}
\usepackage{relsize} 

\usepackage{pstricks,pst-node,graphics}

\usepackage{dsfont}
\usepackage{soul}
\usepackage{filecontents}
\hypersetup{colorlinks=true,linkcolor=black,citecolor=black}
\usepackage{color}
\usepackage[all]{xy}
\usepackage{float}
\usepackage{bm}
\usepackage{mathtools}
\usepackage{thmtools}
\usepackage{setspace}
\usepackage{comment}
\usepackage{tikz}
\usetikzlibrary{calc,backgrounds,math}
\usepackage{ytableau}
\usepackage{mathdots}
\usepackage[myheadings]{fullpage}

\numberwithin{equation}{section}
\newcommand\mtop{1in}
\newcommand\mbottom{1in}
\newcommand\mleft{1in}
\newcommand\mright{1in}
\usepackage[top = \mtop, bottom = \mbottom, left = \mleft, right=\mright]{geometry}

\newtheorem{thm}{Theorem}[section]
\newtheorem{prop}[thm]{Proposition}

\newtheorem{lemma}[thm]{Lemma}
\newtheorem{cor}[thm]{Corollary}

\theoremstyle{definition}
\newtheorem{defi}{Definition}[section]

\newtheorem{rmk}{Remark}[section]
\newtheorem{example}{Example}[section]

\usepackage{scalerel,stackengine}
\stackMath
\newcommand\reallywidehat[1]{%
\savestack{\tmpbox}{\stretchto{%
  \scaleto{%
    \scalerel*[\widthof{\ensuremath{#1}}]{\kern-.6pt\bigwedge\kern-.6pt}%
    {\rule[-\textheight/2]{1ex}{\textheight}}
  }{\textheight}%
}{0.5ex}}%
\stackon[1pt]{#1}{\tmpbox}%
}
\DeclareSymbolFont{bbold}{U}{bbold}{m}{n}
\DeclareSymbolFontAlphabet{\mathbbold}{bbold}

\makeatletter
\def\@tocline#1#2#3#4#5#6#7{\relax
  \ifnum #1>\c@tocdepth 
  \else
    \par \addpenalty\@secpenalty\addvspace{#2}%
    \begingroup \hyphenpenalty\@M
    \@ifempty{#4}{%
      \@tempdima\csname r@tocindent\number#1\endcsname\relax
    }{%
      \@tempdima#4\relax
    }%
    \parindent\z@ \leftskip#3\relax \advance\leftskip\@tempdima\relax
    \rightskip\@pnumwidth plus4em \parfillskip-\@pnumwidth
    #5\leavevmode\hskip-\@tempdima
      \ifcase #1
       \or\or \hskip 1em \or \hskip 2em \else \hskip 3em \fi%
      #6\nobreak\relax
    \hfill\hbox to\@pnumwidth{\@tocpagenum{#7}}\par
    \nobreak
    \endgroup
  \fi}
\makeatother

\makeatletter
\newcommand{\subalign}[1]{%
  \vcenter{%
    \Let@ \restore@math@cr \default@tag
    \baselineskip\fontdimen10 \scriptfont\tw@
    \advance\baselineskip\fontdimen12 \scriptfont\tw@
    \lineskip\thr@@\fontdimen8 \scriptfont\thr@@
    \lineskiplimit\lineskip
    \ialign{\hfil$\m@th\scriptstyle##$&$\m@th\scriptstyle{}##$\hfil\crcr
      #1\crcr
    }%
  }%
}
\makeatother

\DeclarePairedDelimiter{\abs}{\lvert}{\rvert} 
\DeclarePairedDelimiter{\floor}{\lfloor}{\rfloor}
\DeclarePairedDelimiter{\ceil}{\lceil}{\rceil}

\newcommand{\R}{\mathbb{R}}
\newcommand{\Z}{\mathbb{Z}}

\newcommand{\N}{\mathbb{N}}

\newcommand{\E}{\mathbb{E}}

\newcommand{\mf}{\mathfrak}

\newcommand{\mc}{\mathcal}

\newcommand{\pfrac}[2]{\left(\frac{#1}{#2}\right)}

\newcommand{\tth}{^{th}}

\newcommand{\bX}{\bar{X}}

\newcommand{\balpha}{\bar{\alpha}}
\newcommand{\bbeta}{\bar{\beta}}

\newcommand{\y}{\mathsf{y}}

\DeclareMathOperator{\tM}{\tilde{M}}

\newcommand{\Beta}{\mathrm{Beta}}

\renewcommand{\l}{\lambda}

\newcommand{\tG}{\tilde{G}}
\newcommand{\bG}{\bar{G}}

\newcommand{\hbeta}{\hat{\beta}}
\newcommand{\hgamma}{\hat{\gamma}}
\DeclareMathOperator{\Az}{Az}

\DeclareMathOperator{\wt}{wt}
\newcommand{\bgbsw}{\bar{G}^{v-swap}}
\newcommand{\tpi}{\tilde{\pi}}

\renewcommand{\P}{\mathbb{P}}
\newcommand{\Psg}{\P^{stat-\Gamma}}

\newcommand{\glg}{G^{\Gamma \log \Gamma}}

\newcommand{\eastdimer} {[blue,very thick] ++(01,0)--++(1,1)}
\newcommand{\westdimer} {[blue,very thick] ++(1,0)--++(1,1)}
\newcommand{\southdimer} {[blue,very thick] ++(1,0)--++(1,-1)}
\newcommand{\northdimer} {[blue,very thick] ++(1,0)--++(1,-1)}

\newcommand{\eastdomino} {--++(2,2)--++(1,-1)--++(-2,-2)--++(-1,1)}
\newcommand{\westdomino}  {--++(2,2)--++(1,-1)--++(-2,-2)--++(-1,1)}
\newcommand{\southdomino}  {--++(2,-2)--++(1,1)--++(-2,2)--++(-1,-1)}
\newcommand{\northdomino}  {--++(2,-2)--++(1,1)--++(-2,2)--++(-1,-1)}

\newcommand{\eastdominoplus} {[very thick,red] ++(.5,.5)--++(2,0)}
\newcommand{\westdominoplus} {[very thick,red]++(1.5,1.5)--++(0,-2)}
\newcommand{\southdominoplus}  {[very thick,red] ++(0.5,0.5)--++(2,-2)}